%% file: singleObjOpt.tex
\documentclass[nohyperref]{article}

\usepackage{microtype}
\usepackage{graphicx}
\usepackage{subfigure}
\usepackage{booktabs} 

\usepackage{hyperref}

\usepackage[accepted]{icml2022}

\usepackage{amsmath}
\usepackage{amssymb}
\usepackage{mathtools}
\usepackage{amsthm}

\usepackage{amsfonts}      %

\usepackage[capitalize,noabbrev]{cleveref}

\theoremstyle{plain}
\newtheorem{theorem}{Theorem}[section]
\newtheorem{proposition}[theorem]{Proposition}
\newtheorem{lemma}[theorem]{Lemma}

\theoremstyle{definition}

\theoremstyle{remark}

\newtheorem{example}{Example}

\usepackage[textsize=tiny]{todonotes}

\usepackage{rotating}
\usepackage{tabularx}
\usepackage{xstring} 
\usepackage{wasysym} 
\usepackage{MnSymbol} 
\usepackage{float}
\usepackage{ifthen}

\usepackage{xcolor}

\newcommand{\bbobdatapath}{} 
\input{\bbobdatapath cocopp_commands.tex} 
\graphicspath{{\bbobdatapath\algsfolder}}

\newcommand{\DIM}{\ensuremath{\mathrm{DIM}}}
\newcommand{\ERT}{\ensuremath{\mathrm{ERT}}}

\newcommand{\Df}{\ensuremath{\Delta f}}

\newcommand{\fopt}{\ensuremath{f_\mathrm{opt}}}
\newcommand{\ftarget}{\ensuremath{f_\mathrm{t}}}

\icmltitlerunning{Parallel black-box optimization of expensive high-dimensional multimodal functions via magnitude}

\begin{document}

\twocolumn[
\icmltitle{Parallel black-box optimization of expensive \\ high-dimensional multimodal functions via magnitude}



\icmlsetsymbol{equal}{*}

\begin{icmlauthorlist}
\icmlauthor{Steve Huntsman}{str}
\end{icmlauthorlist}

\icmlaffiliation{str}{Systems and Technology Research, Arlington, Virginia, USA}

\icmlcorrespondingauthor{Steve Huntsman}{steve.huntsman@str.us}

\icmlkeywords{magnitude,optimization}

\vskip 0.3in
]




\begin{abstract}
Building on the recently developed theory of magnitude, we introduce the optimization algorithm \texttt{EXPLO2} and carefully benchmark it. \texttt{EXPLO2} advances the state of the art for optimizing high-dimensional ($D \gtrapprox 40$) multimodal functions that are expensive to compute and for which derivatives are not available, such as arise in hyperparameter optimization or via simulations. 
\end{abstract}

\section{\label{sec:Introduction}Introduction}

It is a truism that machine learning problems are optimization problems \cite{sun2019survey}. Some of the most challenging problems in machine learning involve optimizing multimodal functions without information about derivatives on high-dimensional domains, with hyperparameter optimization \cite{loshchilov2016cma,koch2018autotune} an example \emph{par excellence}. More generally, functions whose values are obtained from complex simulations or experiments are frequently multimodal, and optimizing them is a foundational engineering problem. 

Most current optimization algorithms for this regime are not suited for parallel optimization, or have runtime that scales quadratically with dimension; we outline an approach that overcomes both problems simultaneously while outperforming large-scale algorithms (i.e., algorithms that use only sparse linear algebra on any matrices that scale with problem dimension) that do the same. Our \texttt{EXPLO2} algorithm
\footnote{
After initial writing, we discovered that \cite{clerc:hal-01930529,clerc2019iterative} discuss an optimization algorithm called \texttt{Explo2}, with exactly the same etymology of trading off between formalized notions of exploration and exploitation. We hope that context and the use of all caps are sufficient to distinguish these.
}
is a wrapper around an arbitrary large-scale optimizer and serves not least as an initial demonstration of ideas involving the recently developed notion of \emph{magnitude} \cite{leinster2017magnitude,leinster2021entropy} that are likely to be more applicable to a wide range of problems in optimization, sampling, machine learning, and other areas. 

The paper is structured as follows: \S \ref{sec:Weightings} introduces the basic concepts of weightings and magnitude that underlie \texttt{EXPLO2}; \S \ref{sec:Intuition} gives intuition; and \S \ref{sec:backgroundOptimization} gives background on black-box optimization and related work. We describe \texttt{EXPLO2} in \S \ref{sec:Algorithm} and benchmark it in \S \ref{sec:Benchmarking} before making closing remarks in \S \ref{sec:Remarks}. Appendices are after the references.

\section{\label{sec:Weightings}Weightings and magnitude}

For background on this section, see \S 6 of \cite{leinster2021entropy}.

A square nonnegative matrix $Z$ is a \emph{similarity matrix} iff its diagonal is strictly positive. For example, let $t \in (0,\infty)$ and let $d$ be a square matrix whose entries are in $[0,\infty]$ and satisfy the triangle inequality. Then $Z = \exp[-td]$ (where the $[ \cdot ]$ notation indicates a function applied to the individual entries of a matrix) is a similarity matrix. In this paper, all similarity matrices will be of this form, and $d$ will always be a distance matrix for a finite subset of Euclidean space.

We say that a column vector $w$ is a \emph{weighting} iff $Zw = 1$, where $1$ denotes a vector of all ones. The transpose of a weighting for $Z^T$ is called a \emph{coweighting}. If $Z$ has both a well-defined weighting $w$ and a well-defined coweighting $v$, then its \emph{magnitude} is $\text{Mag}(Z) := \sum_j w_j = \sum_k v_k$.

In the event that $Z = \exp[-td]$ and $d$ is the distance matrix of a finite subset of $\mathbb{R}^D$, $Z$ is positive definite \cite{steinwart2008support}, hence invertible, and so its weighting and magnitude are well-defined and unique. More generally, $\text{Mag}(Z) = \sum_{jk} (Z^{-1})_{jk}$ if $Z$ is invertible. The \emph{magnitude function} $\text{Mag}(t;d)$ is the map $t \mapsto \text{Mag}(\exp[-td])$.  

Magnitude is a very general scale-dependent notion of effective size (for finite sets, the effective number of points) that encompasses both cardinality and Euler characteristic, as well as encoding other rich geometrical data \cite{leinster2017magnitude}. Meanwhile, the components of a weighting meaningfully encode a notion of effective size per point (that can be negative ``just behind boundaries'') at a given scale \cite{willerton2009heuristic,meckes2015magnitude,bunch2020practical}. 

\begin{example}
\label{ex:3PointSpace}
Let $\{x_j\}_{j=1}^3 \subset \mathbb{R}^2$ have pairwise distances $d_{jk} := d(x_j,x_k)$ given by $d_{12} = d_{13} = 1 = d_{21} = d_{31}$ and $d_{23} = \delta = d_{32}$ with $\delta<2$. It turns out that 
$$w_1 = \frac{e^{(\delta+2)t}-2e^{(\delta+1)t}+e^{2t}}{e^{(\delta+2)t}-2e^{\delta t}+e^{2t}};$$
$$w_2 = w_3 = \frac{e^{(\delta+2)t}-e^{(\delta+1)t}}{e^{(\delta+2)t}-2e^{\delta t}+e^{2t}}.$$
This is shown in Figure \ref{fig:3pointSpace20210402} for $\delta = 10^{-3}$. At $t = 10^{-2}$, the effective sizes of $x_2$ and $x_3$ are $\approx 0.25$; that of $x_1$ is $\approx 0.5$, so the effective number of points is $\approx 1$. At $t = 10$, these effective sizes are respectively $\approx 0.5$ and $\approx 1$, so the effective number of points is $\approx 2$. Finally, at $t = 10^4$, the effective sizes are all $\approx 1$, so the effective number of points is $\approx 3$. 


\begin{figure}[h]
  \centering
  \includegraphics[trim = 40mm 105mm 40mm 105mm, clip, width=\columnwidth,keepaspectratio]{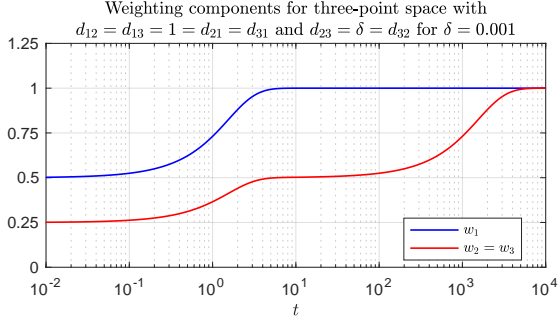}
  \caption{Weighting components for an ``isoceles'' metric space. 
  }
  \label{fig:3pointSpace20210402}
\end{figure}
\end{example}

\section{\label{sec:Intuition}Intuition}

For large enough values of the scale parameter $t$, the weighting of a finite metric space is proportional to the distribution on the space that maximizes an axiomatically supported notion of diversity \cite{leinster2016maximizing,leinster2021entropy}. Along with this special case, the more general intuition that components of a weighting measure an effective size per point suggests maximizing the differential magnitude due to a new point (which we compute in \S \ref{sec:DifferentialMagnitude}) as a mechanism for efficiently exploring an ambient space. 

This idea for ``exploration'' dovetails with another idea underlying ``exploitation.'' In $\mathbb{R}^D$, $Z = \exp[-td]$ is a \emph{radial basis function} (RBF) interpolation matrix \cite{buhmann2003radial} and the equation $Zw = 1$ amounts to defining a weighting $w$ as the vector whose components are coefficients for interpolating the unit function. That is, if $\{x_j\}_{j=1}^n$ are points in $\mathbb{R}^D$ with distance matrix $d$ and we have a weighting $w$ satisfying $\sum_k w_k \exp(-td_{jk}) = 1$, then in fact 
$u(x) := \sum_k w_k \exp(-t|x-x_k|) \approx 1$,
where $\approx$ indicates an optimal interpolation in the sense of a representer theorem \cite{scholkopf2001generalized}. In particular, the triangle inequality implies that $u(x_j + \delta x_j) \ge \exp(-t|\delta x_j|)$. If now also $f : \mathbb{R}^D \rightarrow \mathbb{R}$ and $y_j := f(x_j)$, then its RBF interpolation is
\begin{equation}
\label{eq:rbfInterpolation}
f(x) \approx y Z^{-1} \zeta(x)
\end{equation}
where $\zeta_k(x) := \exp(-t|x-x_k|)$ and we treat $y$ and $\zeta$ respectively as row and column vectors.

We seek to optimize a function by selecting new evaluation points in a way that progressively shifts from exploration (embodied by the differential magnitude due to a new point, for which see Proposition \ref{prop:magnitudePlusOneProp}) to exploitation (embodied by \eqref{eq:rbfInterpolation}, which we obtain at marginal cost). This shift is controlled by an interpretable regularization parameter, and the interpolations are designed to incorporate recent knowledge while requiring constant runtime, even as more points are successively evaluated. For details, see Algorithm \ref{alg:EXPLO2}.

\section{\label{sec:backgroundOptimization}Background on black-box optimization}

Unconstrained global optimization algorithms can be usefully categorized according to characteristics of the intended objective function. 
\footnote{
The no free lunch theorem \cite{wolpert1997no} requires that optimization algorithms be understood in the context of a specific set of problems they are designed to address.
}
For example, (semi-) differentiable functions are readily optimized using gradient descent or quasi-Newton methods such as \texttt{L-BFGS} \cite{liu1989limited}, 
whereas if we dispense with any regularity assumptions, a random search is as good an approach as any other. 

One intermediate regime of interest is where the objective function is structured and reasonably well behaved (e.g., continuous almost everywhere; Lipschitz continuous, etc.) but only accessible via oracle queries. This regime is referred to in the literature under the terms \emph{deriviative-free} or \emph{black-box optimization} \cite{rios2013derivative,audet2017derivative}. 
Meanwhile, in an ``expensive'' regime, the appropriate measure of performance is the current best function value for a given number of function evaluations. For example, the objective function might be defined in terms of a computer simulation that takes a significant amount of time to execute, or an experiment that must be performed. 
\footnote{
In the ``non-expensive'' regime, the appropriate measure of performance is the number of function evaluations required to achieve a given target.
}

At a high level of abstraction, global optimization algorithms for expensive black-box functions are about making tradeoffs between exploration of the solution space and exploitation of information that has already been acquired through the course of the algorithm's execution. A practical and reasonably complete taxonomy of useful techniques is i) metaheuristics such as \texttt{CMA-ES} \cite{li2018fast,varelas2018comparative}; 
\footnote{
\texttt{CMA-ES} is a \emph{de facto} standard algorithm for problems of up to tens of dimensions and function evaluation budgets as low as tens of thousands. It can be pushed into more demanding regimes, but has runtime quadratic in problem dimension \cite{li2018fast,varelas2018comparative} and usually algorithm variants are employed.
}
ii) deterministic algorithms such as \texttt{DIRECT} \cite{jones2021direct} or \texttt{MCS} \cite{huyer1999global}; 
and iii) \emph{surrogate or metamodel-assisted algorithms}. 
In dimension $\lessapprox 5$, metaheuristics and deterministic algorithms are broadly competitive with each other \cite{sergeyev2018efficiency}, but surrogates can yield substantial improvements 
\cite{haftka2016parallel,vu2017surrogate,stork2020open,xia2020gops}.

Surrogate-assisted algorithms can be broadly classified into statistically and geometrically-informed interpolation techniques. The former class of \emph{Bayesian optimization} \cite{frazier2018bayesian} is exemplified by Gaussian process regression (``kriging''); the latter class is exemplified by RBF interpolation.

In Bayesian optimization, a Gaussian process prior is placed on the objective, 
and after initialization points are selected for evaluation/sampling according to an \emph{acquisition function} that embodies an explore/exploit tradeoff. 
While the Bayesian optimization framework is flexible and theoretically appealing, it scales poorly for problems with more than a few tens of dimensions and execution budgets in the thousands. In this regime, statistics are dimensionally cursed, and more geometrically-oriented approaches are needed.
\footnote{
In dimension $\lessapprox 20$, mirroring the ideas of this paper by using an exponential kernel in Bayesian optimization for both surrogate construction and diversity optimization (the latter as an acquisition function) may be fruitful. However, since statistical approaches are intrinsically disadvantaged in the high-dimensional/low execution budget regime that interests us, we do not pursue this idea here.
}

Meanwhile, RBF surrogate techniques for global optimization were introduced in \cite{gutmann2001radial} and elaborated in \cite{bjorkman2000global}: here, after initialization, points are selected for evaluation according to the expected ``bumpiness'' of the resulting surrogate. Subsequent RBF surrogate techniques \cite{regis2005constrained,regis2007improved} performed exploration by considering distances from previously evaluated points. 
These RBF surrogate techniques all cycle through a range of distance lower bounds to make exploration/exploitation tradeoffs.

\subsection{\label{sec:related}Related work}

Our approach is fairly close to a hybrid of \cite{regis2005constrained} and \cite{ulrich2011maximizing}. The underlying motivation is that the principled constructs for diversity optimization (``exploration'') and RBF interpolation (to facilitate surrogate ``exploitation'')
that these approaches respectively build on are actually the same provided only that we use an exponential kernel for the former. Because our notion of diversity optimization is nonlocal, and our unconventional choice of kernel acts as a ``gentle funnel,'' it is natural to expect that our approach is well-suited to optimizing functions with global structure, as well as multimodal functions, and particularly those functions exhibiting both characteristics.

To keep the runtime per timestep to an acceptably low constant, we borrow an idea of \cite{booker1999rigorous} to use a ``balanced'' set of points that includes points with the largest errors for the previous surrogate as well as points with the current most optimal values. For the sake of simplicity, the balancing between these two subsets is determined by the regularizer. Meanwhile, we follow the advice of \cite{villanueva2013locating} for expensive problems by using multiple starting points in conjunction with surrogates.

Work that is only tangentially related but should be mentioned here applies magnitude to primitives in machine learning such as geometry-aware information theory \cite{posada2020gait} and boundary detection \cite{bunch2020practical}. These and the present paper indicate that magnitude is fertile ground for growing new ideas in machine learning.

\section{\label{sec:Algorithm}Algorithm}

The ``explore/exploit'' (\texttt{EXPLO2}) algorithm described in Algorithm \ref{alg:EXPLO2} tries to minimize $f : \mathbb{R}^D \rightarrow \mathbb{R}$ on $B:= \prod_{j=1}^D [\ell_j,u_j]$ with a budget of $N$ function evaluations. \texttt{EXPLO2} is basically a wrapper around another optimization algorithm (e.g., MATLAB's default \texttt{fmincon}, which is a large-scale algorithm) that acts on surrogates that gradually shift the balance of a trade between 
\begin{itemize}
	\item[i)] exploration, as measured by the differential magnitude of a new point at which to evaluate $f$ relative to the set of points at which $f$ has already been evaluated, and 
	\item[ii)] exploitation, as measured by an exponential RBF interpolation of $f$ at a mix of points where a previous interpolation had a) the largest errors and b) the most optimal function values. 
\end{itemize}
Both this mix of interpolation points and the tradeoff between exploration and exploitation are controlled by a single regularization term $\lambda : [0,1] \rightarrow \mathbb{R}$, nominally decreasing from $1$ to $0$, e.g. our default choice $\lambda(\tau) = 1-\tau$.
\footnote{Our experiments suggest that the details of $\lambda$ do not make much difference in practice, though it may be slightly more advantageous to use a ``sawtooth'' with ``teeth'' of decaying width.
}
The course of the algorithm's execution is shown for the two-dimensional Rastrigin function in Figure \ref{fig:2dRastriginSchematic}.

\begin{figure}[h]
  \centering
  \includegraphics[trim = 0mm 115mm 0mm 110mm, clip, width=\columnwidth,keepaspectratio]{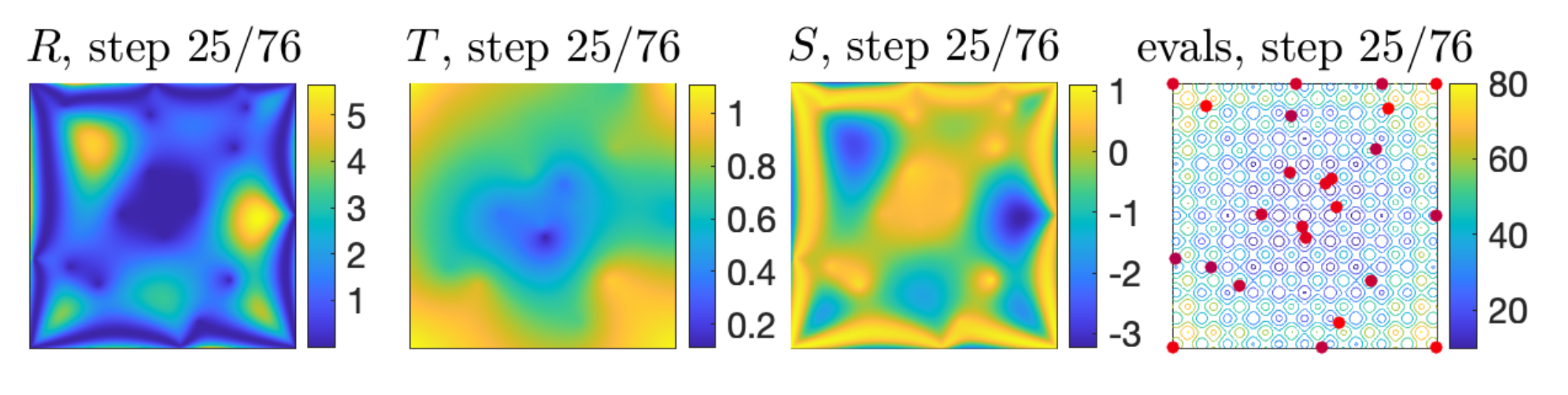} \\
  \includegraphics[trim = 0mm 115mm 0mm 110mm, clip, width=\columnwidth,keepaspectratio]{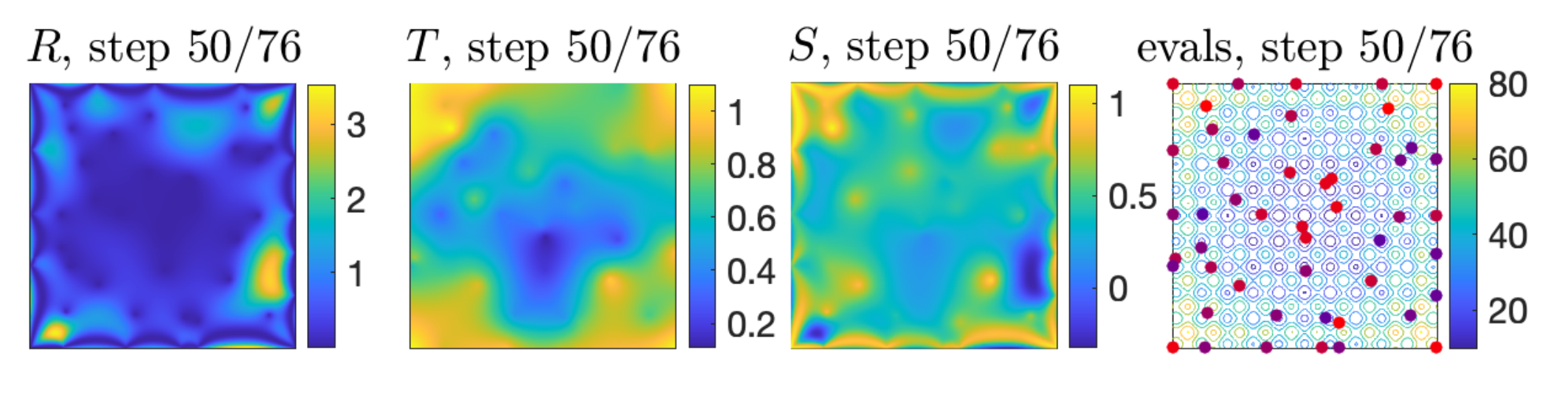} \\
  \includegraphics[trim = 0mm 115mm 0mm 110mm, clip, width=\columnwidth,keepaspectratio]{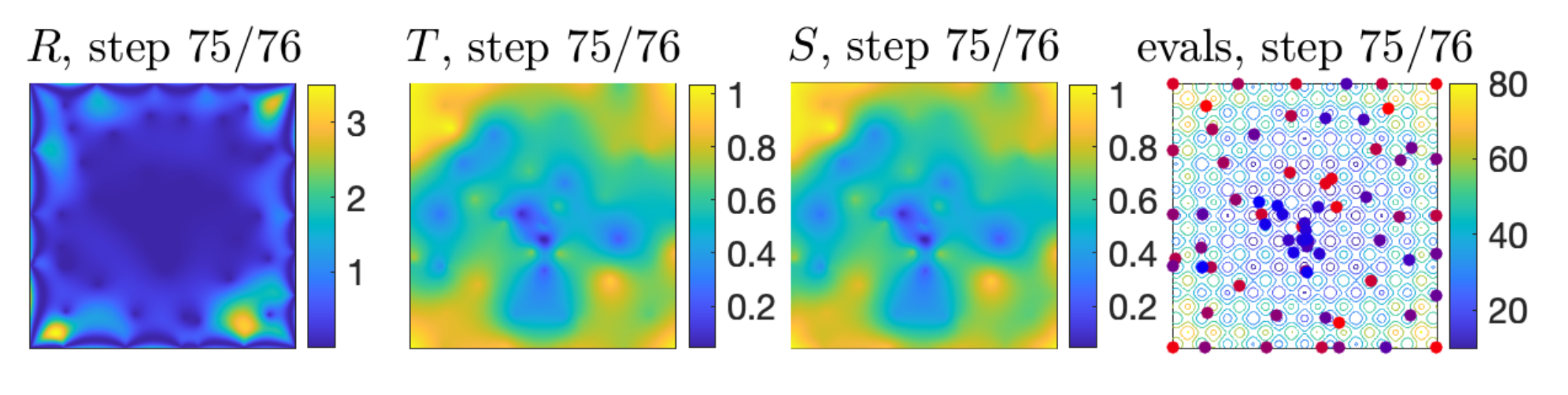}
  \caption{(Top left; center left; center right) Differential magnitude $R$; RBF interpolation $T$; surrogate $S$ after the 25th of 76 function evaluations for the two-dimensional Rastrigin function on $[-5.12,5.12]^2$. (Top right) First 25 evaluation points for \texttt{EXPLO2} overlaid on a contour plot of the objective. (Middle; bottom) As in the top row, but for 50th and 75th of 76 evaluations. Evaluation points ordering is indicated as a transition from {\color{red}red for old points} to {\color{blue}blue for new points}.
  }
  \label{fig:2dRastriginSchematic}
\end{figure}

The rationale behind an exponential RBF interpolation is simply that this is computed anyway for exploration. That is, the interpolation coefficients are of the form $yZ^{-1}$, where $y$ and $Z = \exp[-td]$ respectively indicate function values and the similarity matrix of evaluation points (see \eqref{eq:rbfInterpolation} and Algorithm \ref{alg:EXPLO2}). Meanwhile, the limit $t \downarrow 0$ corresponds to an interpolation using very shallow decaying exponentials, which simultaneously avoids degeneracies that arise for $t$ bounded away from zero, mitigates nondifferentiable behavior at interpolation points, and helps make (surrogate) global optima more easily accessible to the internal optimizer. 

\definecolor{shadecolor}{RGB}{200,200,200}
\renewcommand{\algorithmicrequire}{\textbf{Input:}}
\renewcommand{\algorithmicensure}{\textbf{Output:}}
\begin{algorithm}
  \caption{Explore/exploit (\texttt{EXPLO2}) optimizer}
  \label{alg:EXPLO2}
  \begin{algorithmic}[1]
    \REQUIRE Function $f$, lower bounds $\ell$, upper bounds $u$, evaluation budget $N$, {\bf optional parameters:} number $n_\parallel$ of parallel function evaluations (default: $n_\parallel = 1$), number $n_\sigma$ of downsampling points (default: $n_\sigma = 100$), number $n_\leftrightarrow$ of points to estimate exploration range (default: $n_\leftrightarrow = 100$), number $n_\downarrow$ of tries for each surrogate optimization from uniformly random initial points (default: $n_\downarrow = 3$), explore/exploit regularizer $\lambda$ (default: $\lambda(\tau) = 1-\tau$), and initialization strategy (default: uniformly random)    
    \STATE $t \leftarrow \sqrt{\varepsilon}$ \hfill \emph{// Machine epsilon (see line 13)} \label{line:epsilon}	
    \STATE Select $N_0 \leftarrow D+1$ initial points
    \STATE Evaluate $f$ on the $N_0$ initial points
    \STATE $n \leftarrow N_0+1$
    \WHILE{$n < N$}
      \IF{$n > n_\sigma$}
        \STATE Form $I_\sigma$ using $n_\rho \approx n_\sigma \min(1, \lambda(\frac{n}{N})/\lambda(\frac{1}{N}))$ points with greatest interpolation relative error and $n_\mu = n_\sigma -n_\rho$ points with least function values \hfill \label{line:downsampleIndices}
      \ELSE
        \STATE $I_\sigma \leftarrow \{1,2,\dots,n-1\}$
      \ENDIF
      \STATE $x_\sigma \leftarrow (x_i)_{i \in I_\sigma}$, $y_\sigma \leftarrow (y_i)_{i \in I_\sigma}$ \hfill \emph{// Downsample}\label{line:downsample}
      \STATE $d \leftarrow $ distance matrix for $x_\sigma$ \label{line:d}
      \STATE $Z \leftarrow \exp[-td]$; \hfill \emph{// Use $t \downarrow 0$ limit if convenient} \label{line:Z} 
      \STATE $w \leftarrow Z^{-1}1$ \label{line:w}
      \STATE $\zeta_{i'}(x) \leftarrow \exp(-td(x_{\sigma,i'},x))$ \hfill \emph{// see \eqref{eq:rbfInterpolation} and Prop. \ref{prop:magnitudePlusOneProp}} \label{line:zeta}
      \STATE $T(x) \leftarrow y_\sigma Z^{-1} \zeta(x)$ \hfill \emph{// RBF exploi\underline{T}ation}: see \eqref{eq:rbfInterpolation} \label{line:exploit}
      \FOR{$j$ from 1 to $n_\parallel$}
        \STATE $R \leftarrow \frac{(1-\zeta^T w)^2}{1-\zeta^T Z^{-1} \zeta}$ \hfill \emph{// explo\underline{R}ation: see Prop. \ref{prop:magnitudePlusOneProp}} \label{line:explore}
	\STATE $C \leftarrow \min \{2^D, n_\leftrightarrow\}$ corners of bounding box
        \STATE $R^\vee \leftarrow \max_{x \in C} R(x)$
        \STATE $S \leftarrow \frac{T}{\max y_\sigma - \min y_\sigma} - \lambda(\frac{n}{N}) \frac{R}{R^\vee}$ \hfill \emph{// Surrogate} 
        \label{line:surrogate}
        \FOR{$k$ from 1 to $n_\downarrow$}
          \STATE Minimize $S$ \hfill \emph{// Using, e.g., \texttt{fmincon}} \label{line:optimize}
          \STATE Keep result iff best so far in current loop
        \ENDFOR
      \STATE Adjoin best result from line 23 to $x_\sigma$ 
      \STATE Update $d$, $Z$, $w$ and $\zeta$ as in lines 12-15
      \ENDFOR
      \STATE Memorialize/evaluate in parallel the $n_\parallel$ new points
      \STATE Get relative errors of RBF interpolation for downsampling per lines 7, 11, and 16
      \STATE $n \leftarrow n+n_\parallel$
    \ENDWHILE
    \ENSURE $(x_i,y_i)_{i = 1}^N$
  \end{algorithmic}
\end{algorithm}

The definition of $R(x)$ in Algorithm \ref{alg:EXPLO2} arises from
\begin{proposition}
\label{prop:magnitudePlusOneProp}
Let $Z$ be positive definite and consider a positive definite matrix of the form $Z[\zeta] := \left ( \begin{smallmatrix} Z & \zeta \\ \zeta^T & 1 \end{smallmatrix} \right )$. Then
\end{proposition}
\begin{equation}
\text{Mag}(Z[\zeta]) = \text{Mag}(Z) + \frac{(1-\zeta^T w)^2}{1-\zeta^T Z^{-1} \zeta}. \qed
\end{equation}

\texttt{EXPLO2} ``just'' requires a positive definite metric space in the sense of \cite{meckes2013positive} 
\footnote{
In fact, an even weaker requirement suffices here: roughly, that the space is endowed with an extended quasipseudometric $d'$ such that the similarity matrices $Z = \exp[-td]$ are all invertible for some finite interval $[0,\varepsilon]$, where $d$ denotes a restriction of $d'$ to an arbitrary finite set. However, we are not aware of an example of such a space that is not positive definite. A good starting place to try to construct such a space may be \cite{gurvich2012characterizing}.
}
and a suitable solver for the surrogate functions. In particular, any subset of Euclidean space is positive definite, as is any ultrametric space or weighted tree. In the Euclidean setting, we can always use an extremum of a continuous RBF interpolation to approximate an extremum of the underlying black-box function: while there are no guarantees on the errors that result, in practice such errors can usually be reasonably expected to be tolerably small. \texttt{EXPLO2} is thus applicable \emph{de novo} to a wide range of discrete problems.

It is important to note that the runtime of \texttt{EXPLO2} is not directly affected by the dimension $D$. The impact of dimension is felt mainly through the inner optimizer, which for large-scale algorithms such as \texttt{fmincon} is typically linear. However, while we use the default value $n_\sigma = 100$ at all times, it is conceivable that one might want $n_\sigma = O(D)$, which would nominally introduce a cubic runtime dependence on $D$ (i.e., even worse than \texttt{CMA-ES}). To avoid this, we note that experiments on time series (described in \S \ref{sec:HighDimTimeSeries}) indicate that it is possible to get good approximations to weightings from nearest neighbors alone, even in problems with hundreds of thousands of dimensions. Along similar lines, an efficient approximate nearest neighbor algorithm \cite{datar2004locality,andoni2018approximate,li2020approximate} could be used to efficiently create a sparse similarity matrix $Z$, yielding a large-scale algorithm, and probably one with little performance penalty.

\section{\label{sec:Benchmarking}Benchmarking}

\subsection{\label{sec:Comparison}Algorithms for comparison}

Most black-box optimization algorithms are tailored to problems in dimension $\lessapprox 20$ and/or inexpensive functions. On one hand, many optimization algorithms suitable for higher-dimensional applications (including variants of the popular \texttt{CMA-ES} and \texttt{L-BFGS} algorithms) require thousands or tens of thousands of function evaluations per dimension to yield acceptable results \cite{varelas2018comparative,varelas2019benchmarking}. On the other hand, even state-of-the-art Bayesian optimization algorithms such as \texttt{SMAC-BBOB} \cite{hutter2013evaluation} or \cite{eriksson2019scalable} are generally not competitive on problems with more than a few tens of dimensions. In short, the list of candidate algorithms for optimizing high-dimensional expensive functions is not long, and it becomes shorter when considering multimodal functions. 

For low evaluation budgets in dimension $\lessapprox 40$, \texttt{NEWUOA} \cite{powell2006newuoa,ros2009benchmarking}, \texttt{MCS}, and \texttt{GLOBAL} \cite{csendes1988nonlinear,csendes2008global} are the best algorithms on the \texttt{BBOB} testbed \cite{hansen2010fun,hansen2010comparing}. Meanwhile, \cite{brockhoff2015comparison} points out that for expensive black-box optimization using surrogates in dimension $\lessapprox 40$,
\begin{quote}
the three algorithms \texttt{SMAC-BBOB} (for very low budgets below $\approx 3 \cdot D$ function evaluations), \texttt{NEWUOA} (for medium budgets), and \texttt{lmm-CMA-ES}
\footnote{
See \cite{kern2006local,auger2013benchmarking}; also \texttt{DTS-CMA-ES} \cite{pitra2016doubly} and \texttt{lq-CMA-ES} \cite{hansen2019global}.
}
(for relatively large budgets of $\ge 30 \cdot D$ evaluations) build a good portfolio that constructs the upper envelope over all compared algorithms for almost all problem groups.
\end{quote}

With this in mind, a representative set of algorithms to benchmark against for low-dimensional, expensive, multimodal functions is \texttt{NEWUOA}, \texttt{MCS}, \texttt{GLOBAL}, \texttt{SMAC-BBOB}, and \texttt{*-CMA-ES} with \texttt{*} $\in$ \{\texttt{lmm}, \texttt{DTS}, \texttt{lq}\}. Of these, only \texttt{NEWUOA} is well-suited to problems with hundreds of dimensions (indeed, most of these algorithms have not even been benchmarked for dimension $40$ in \texttt{COCO} \cite{hansen2021coco} as of this writing, and 
\begin{itemize}
	\item while \texttt{NEWUOA} scales to hundreds of dimensions, its time complexity ranges between quadratic and quintic in dimension--with quadratic or cubic the case in practice and as benchmarked \cite{ros2009benchmarking}; 
	\item \texttt{MCS} 
	is unsuitable for high-dimensional problems as it relies on partitioning the search space;
	\item \texttt{GLOBAL} is unsuitable for high-dimensional problems because it relies on clustering \cite{assent2012clustering};
	\item as a Bayesian optimization technique, \texttt{SMAC-BBOB} is not suited to high-dimensional problems;
	\item for ``noisy and multimodal functions, the speedup [of surrogate-assisted variants of \texttt{CMA-ES} relative to \texttt{CMA-ES}] \dots vanishes with increasing dimension'' \cite{kern2006local}; i.e., ``\texttt{lmm-} and \texttt{DTS-CMA-ES} become quickly computationally infeasible with increasing dimension, hence their main application domain is in moderate dimension'' \cite{hansen2019global}, where they embody the state of the art for relatively large budgets of $\ge 30 \cdot D$ evaluations \cite{bajer2019gaussian}.
\end{itemize}

Finally, since \texttt{EXPLO2} is (as benchmarked) a wrapper around \texttt{fmincon},
\footnote{Of course, other internal optimizers could be chosen, and this offers a way to either accentuate the strengths or mitigate the weaknesses of the search and surrogate aspects of \texttt{EXPLO2}.}
it also makes sense to compare the performance of these two algorithms.

\subsection{\label{sec:Results}Results}


Results from experiments according to \cite{hansen2016exp,hansen2016perfass} on the
\texttt{BBOB} benchmark functions given in \cite{hansen2010fun} are
presented in Figures \ref{fig:ECDFsingleOne20}-\ref{fig:ECDFsingleOne40last} (see also \S \ref{sec:ERTscaling}-\ref{sec:mixint}).
The experiments were performed and plots were produced with \texttt{COCO} \cite{hansen2021coco}, version 2.4.
\footnote{
\label{foot:plots}
These plots (using the \texttt{--expensive} option) have a rigid predetermined format. We also produced fixed-budget plots of cumulative best values using the same data (and substituting, e.g. \texttt{DTS-CMA-ES} for \texttt{lmm-CMA-ES}) via the web interface of \texttt{IOHanalyzer} \cite{doerr2018iohprofiler} at \url{https://iohanalyzer.liacs.nl/}, but these plots produced neither additional nor conflicting insights. (We provide an exhaustive set of these plots in \S \ref{sec:iohAnalyzer}.)
}

Because of the runtime overhead of \texttt{EXPLO2}, our experiments used a fixed budget of 25 function evaluations per dimension.
\footnote{
While \cite{tuvsar2017anytime} points out how an anytime benchmark of a budget-dependent algorithm such as \texttt{EXPLO2} can be performed with linear overhead, the cost-to-benefit ratio in our case was still prohibitive. In any event, this would not have made larger budgets any easier (or much more relevant) to obtain.
}
As \cite{hansen2016exp} points out, algorithms are only comparable up to the smallest budget given to any of them, corresponding in this case to $\log_{10} 25 \approx 1.40$ on the horizontal axis for Figures 
\ref{fig:ECDFsingleOne20} and \ref{fig:ECDFsingleOne40}. This corresponds in our case exactly to the location of crosses ($\times$), which indicate where bootstrapping of experimental data begins to estimate results for larger numbers of function evaluations. At the same time, we used $n_\parallel = 32$, and $\log_{10} (25/32) \approx -0.107$. Thus allowing for parallel resources (see \S \ref{sec:Parallel}) suggests comparing \texttt{EXPLO2} at the value $\log_{10} 25$ on the horizontal axes with the other algorithms at $\log_{10} (25/32)$, except for \texttt{*-CMA-ES}, which is parallelizable.
\footnote{
\label{foot:parallelCMAES}
While \texttt{*-CMA-ES} is easily parallelizable \cite{hansen2003reducing,khan2018parallel}, the quadratic scaling of non-large-scale variants with dimension creates a serious disadvantage for high-dimensional problems. In high dimensions it is appropriate to compare \texttt{EXPLO2} to large-scale variants of \texttt{CMA-ES} \cite{varelas2018comparative,varelas2019benchmarking}, and as we discuss below our (necessarily limited) experiments in this regard yielded good results.
}

\begin{figure}
\centering
\begin{tabular}{@{}c@{}c@{}c@{}c@{}}
\includegraphics[width=0.238\textwidth]{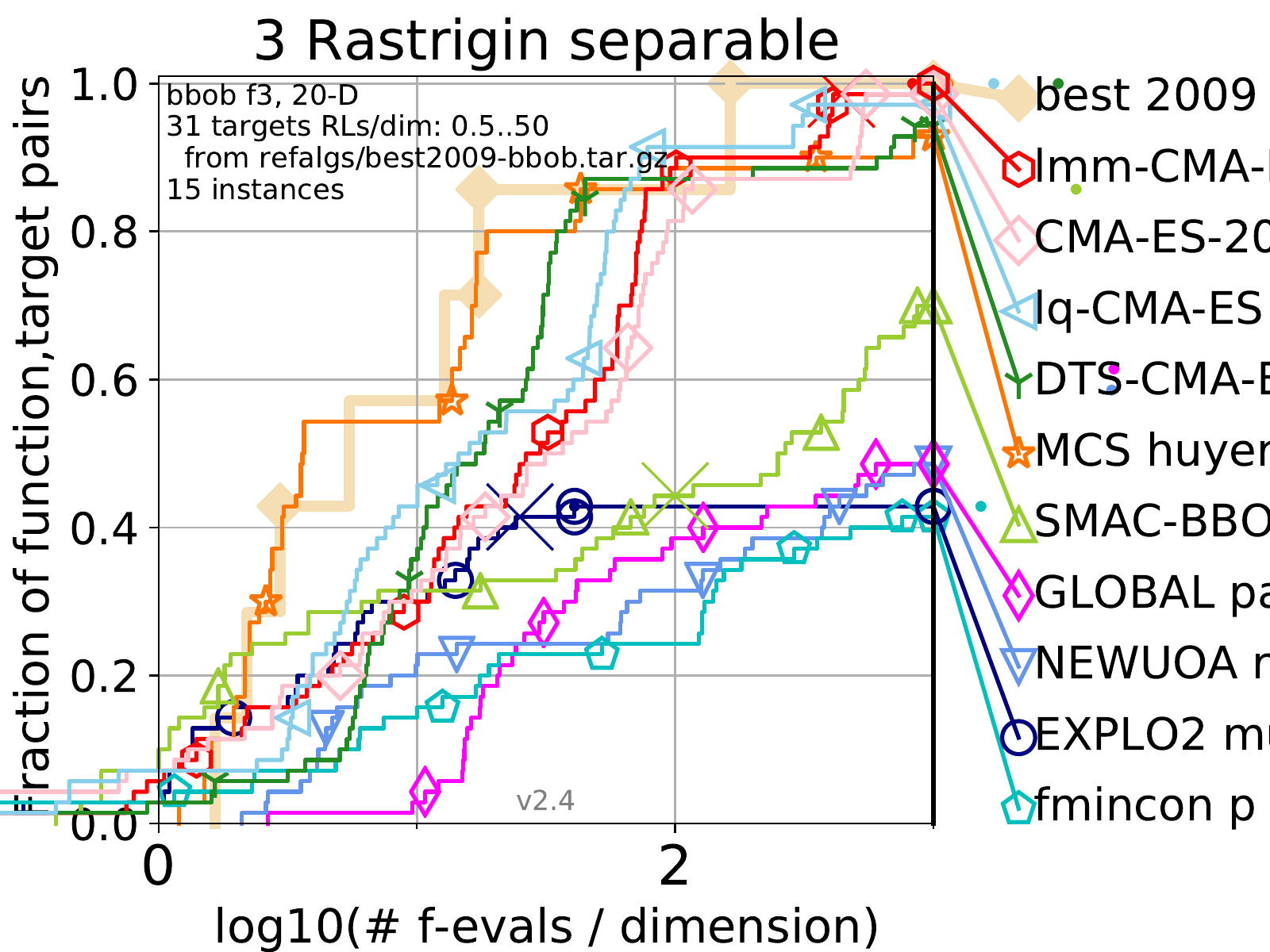}&
\includegraphics[width=0.238\textwidth]{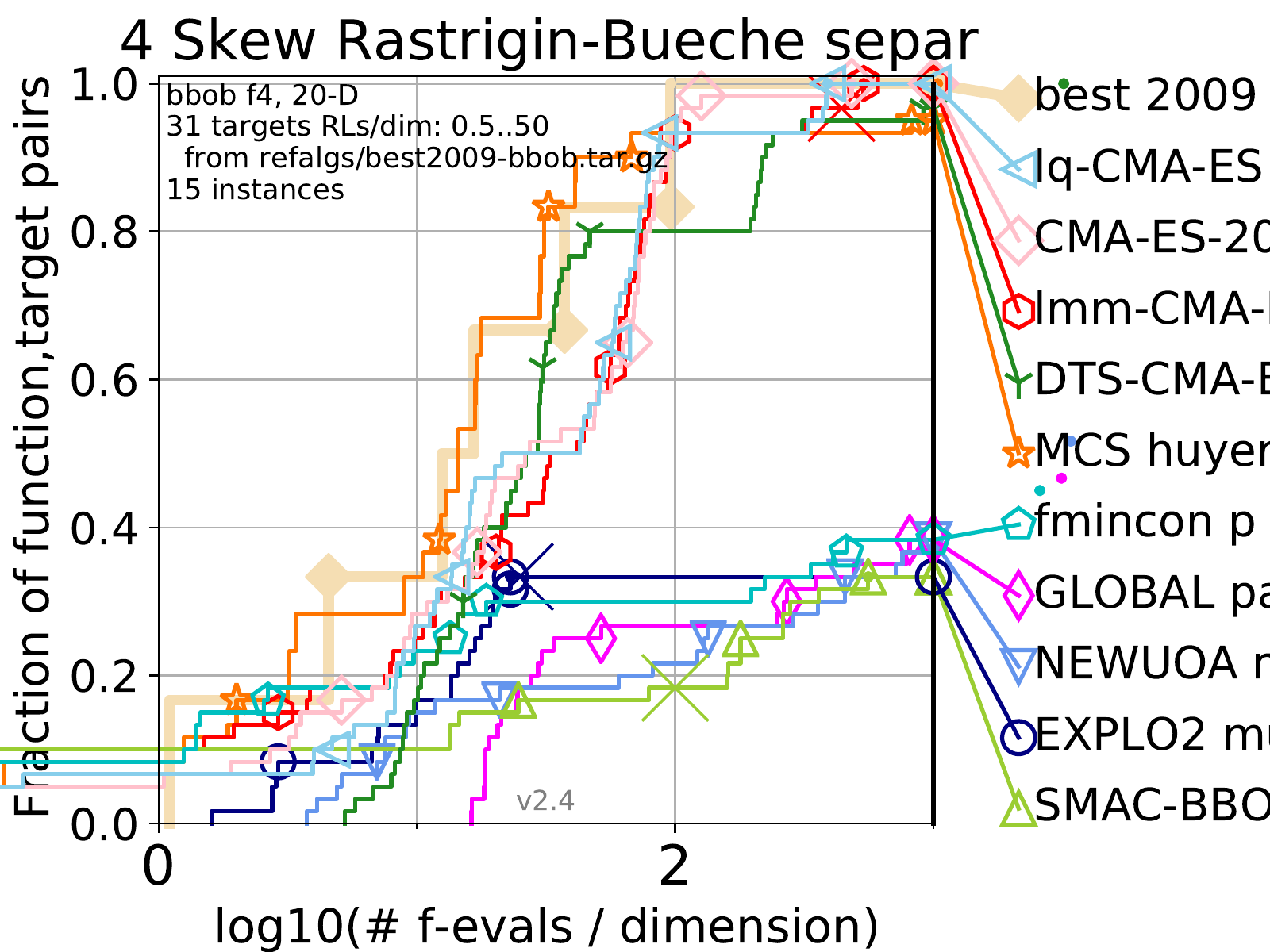}\\
\end{tabular}
\vspace*{-0.2cm}
\caption{
\label{fig:ECDFsingleOne20}
\bbobecdfcaptionsinglefunctionssingledim{20} {\bf Crosses ($\times$) indicate where experimental data ends and bootstrapping begins; algorithms are not comparable after this point.} \texttt{EXPLO2} used default options except for $n_\parallel = 32$. While \texttt{*-CMA-ES} outperforms \texttt{EXPLO2} here, in high dimensions only large-scale variants of \texttt{CMA-ES} are appropriate for most purposes, and \texttt{EXPLO2} outperforms the best of these on structured multimodal functions: see Figures \ref{fig:RastriginVDCMAESvsEXPLO220}-\ref{fig:GRF8F2VDCMAESvsEXPLO2320}.
}
\end{figure}

\begin{figure}
\centering
\begin{tabular}{@{}c@{}c@{}c@{}c@{}}
\includegraphics[width=0.238\textwidth]{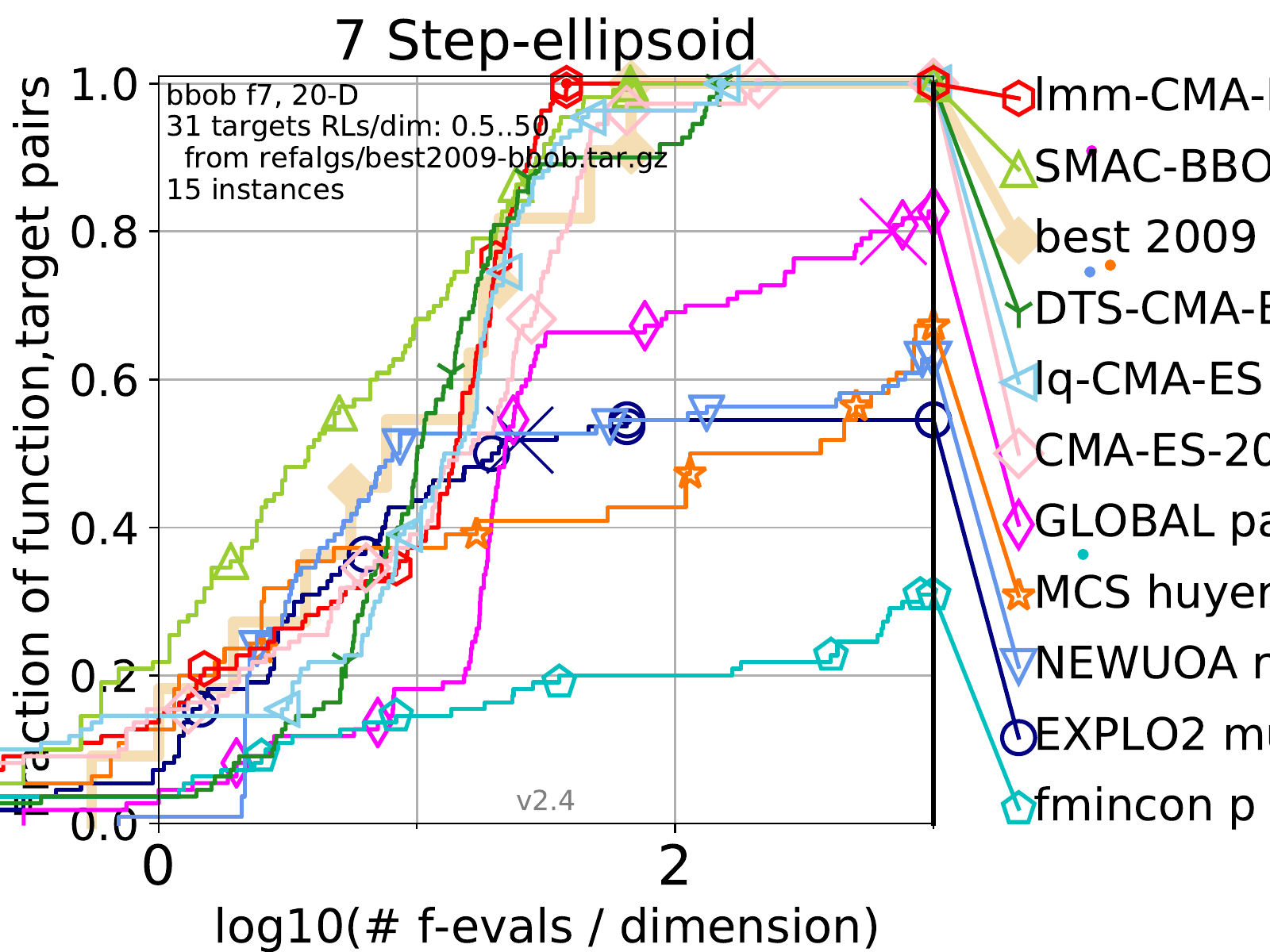}&
\includegraphics[width=0.238\textwidth]{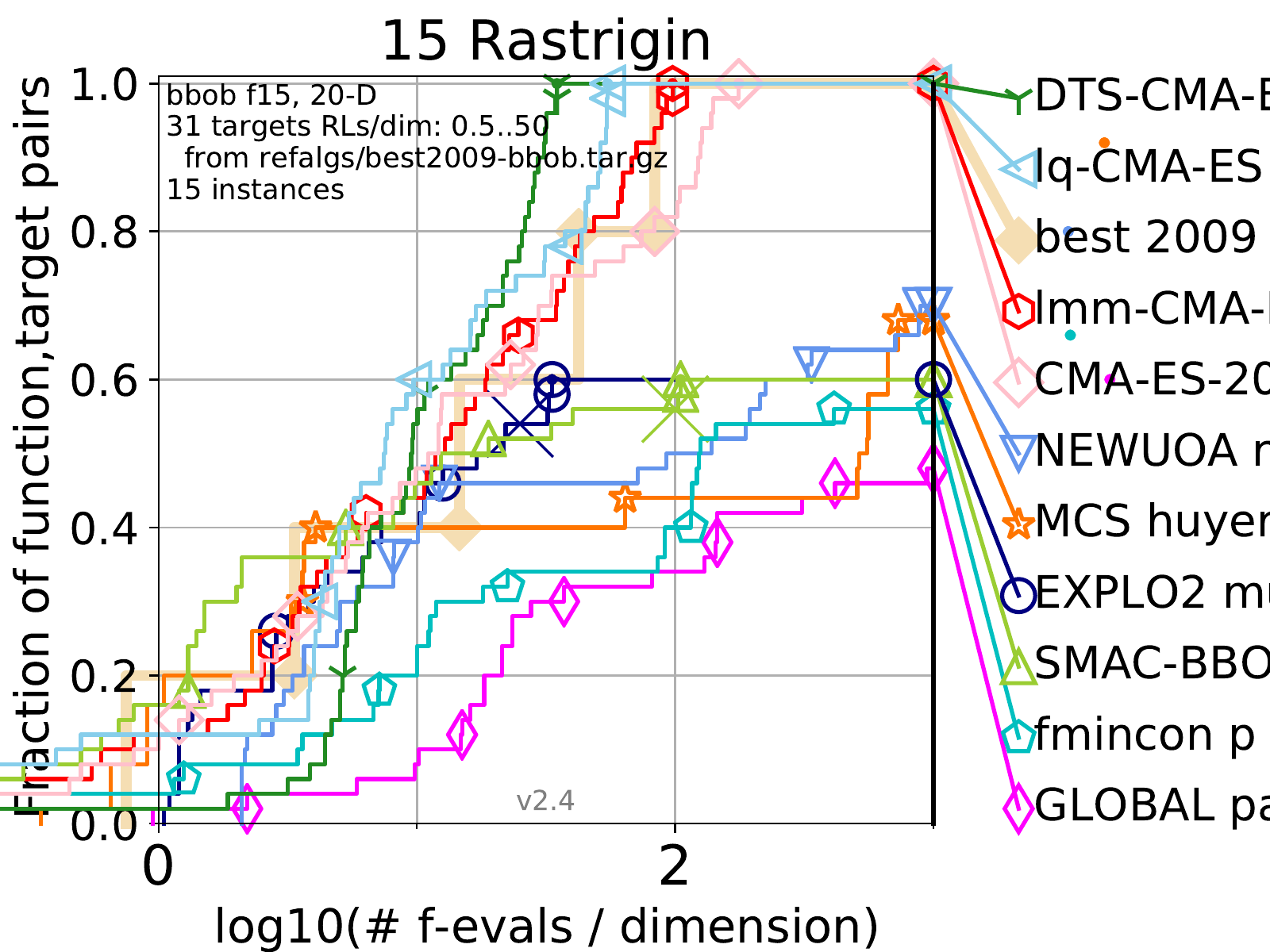}\\
\end{tabular}
\vspace*{-0.2cm}
\caption{
As in Figure \ref{fig:ECDFsingleOne20}, but for $f_{7}$ and $f_{15}$.
}
\end{figure}

\begin{figure}
\centering
\begin{tabular}{@{}c@{}c@{}c@{}c@{}}
\includegraphics[width=0.238\textwidth]{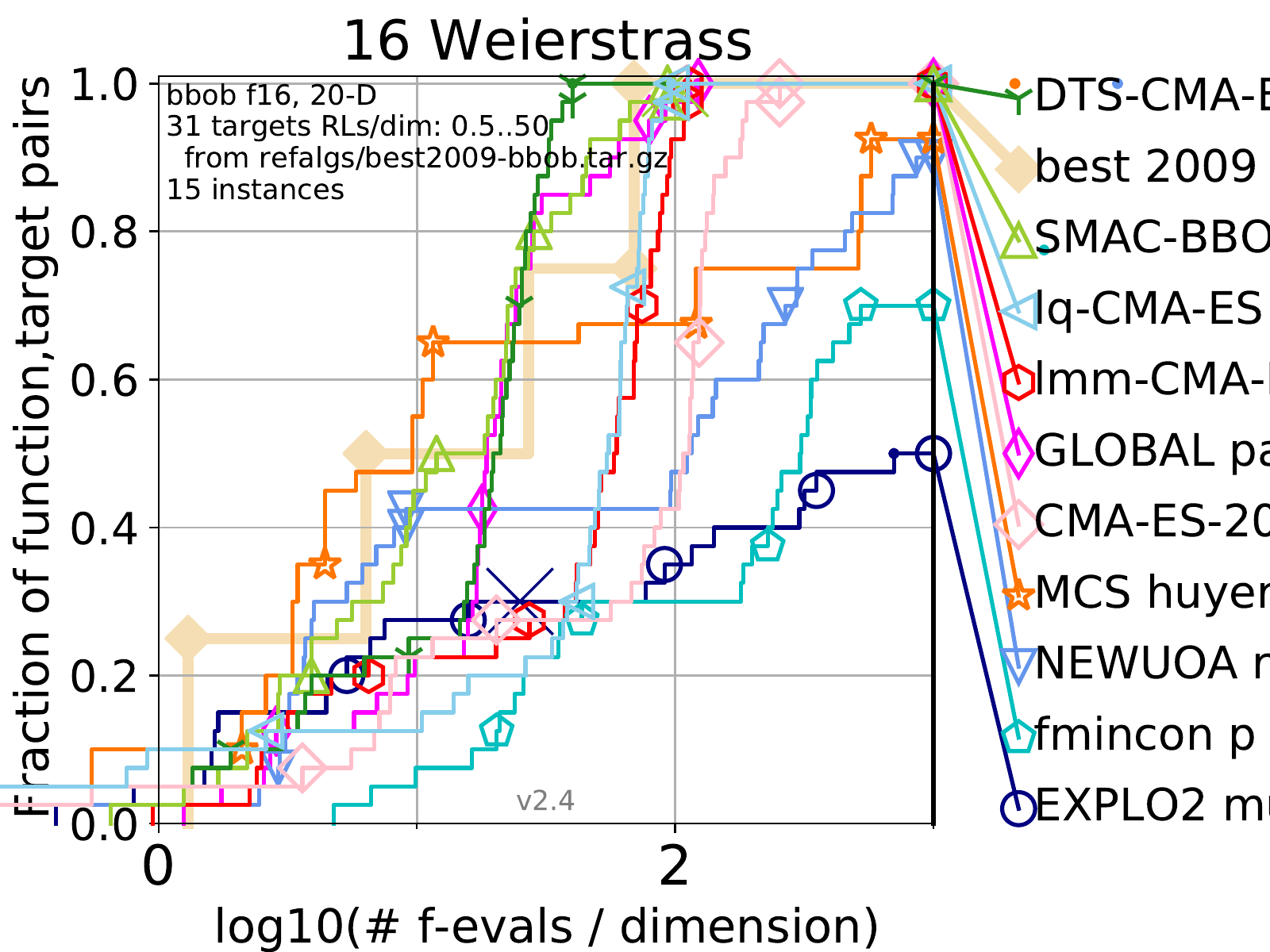}&
\includegraphics[width=0.238\textwidth]{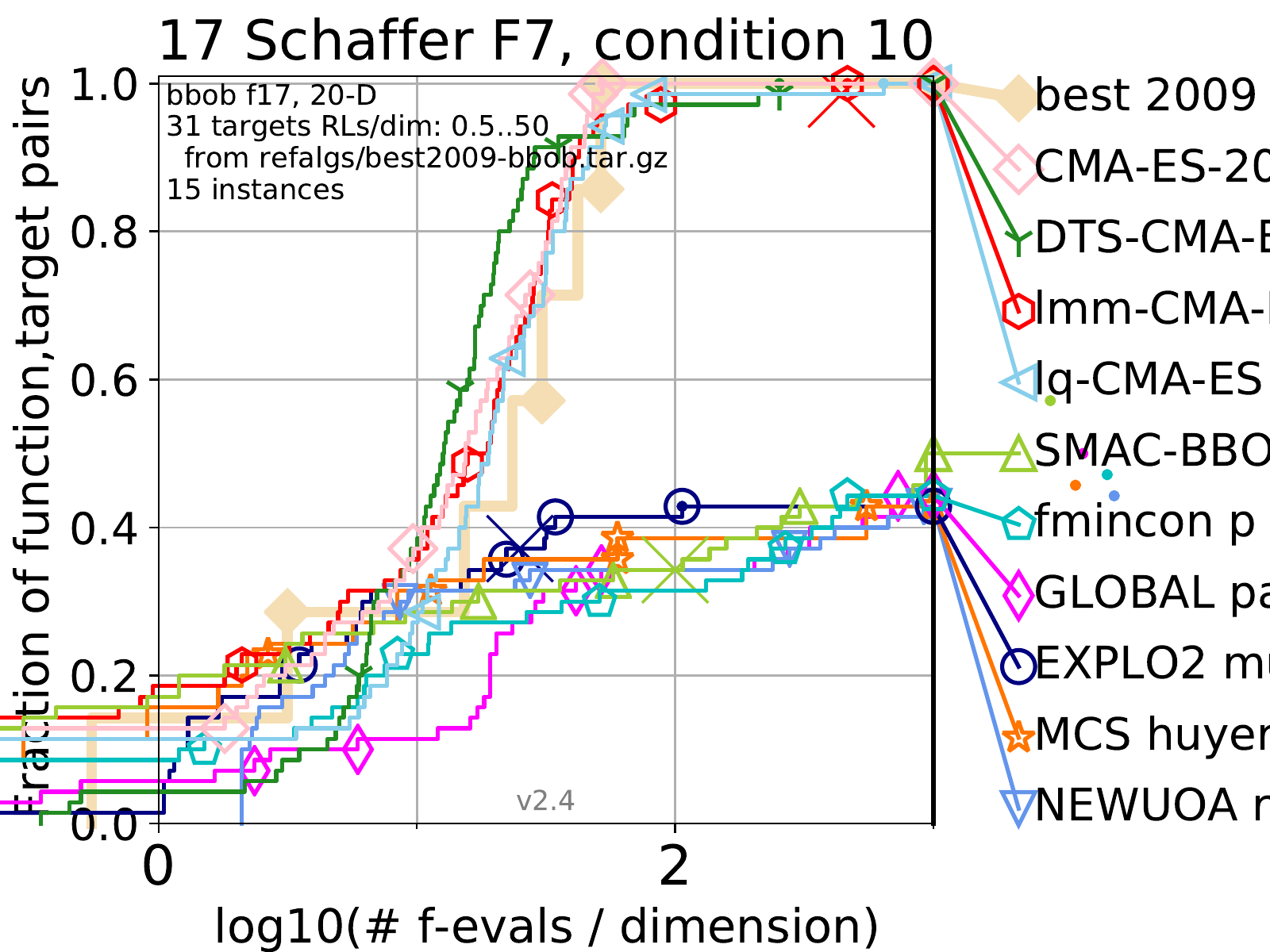}\\
\end{tabular}
\vspace*{-0.2cm}
\caption{
As in Figure \ref{fig:ECDFsingleOne20}, but for $f_{16}$ and $f_{17}$.
}
\end{figure}

\begin{figure}
\centering
\begin{tabular}{@{}c@{}c@{}c@{}c@{}}
\includegraphics[width=0.238\textwidth]{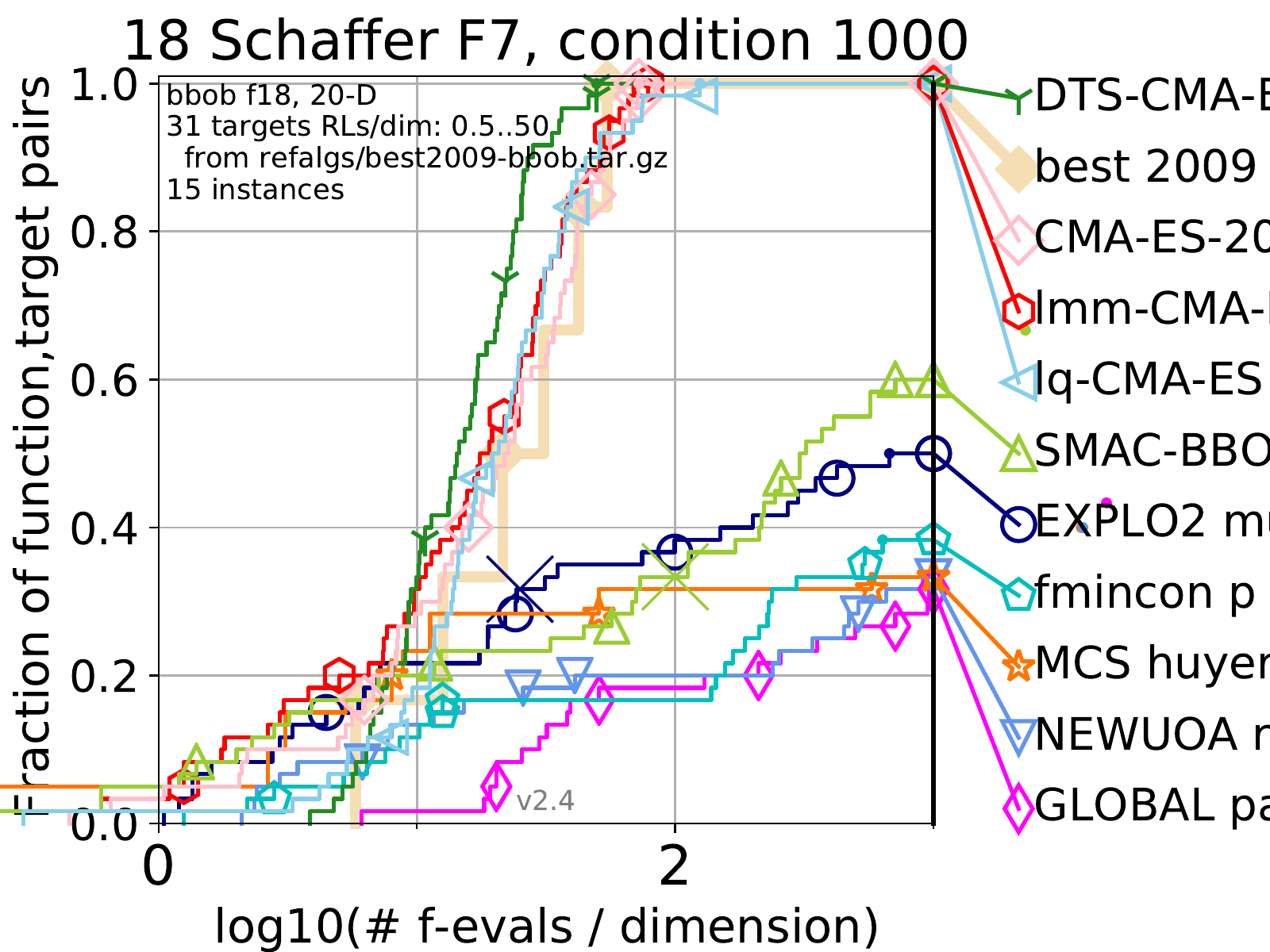}&
\includegraphics[width=0.238\textwidth]{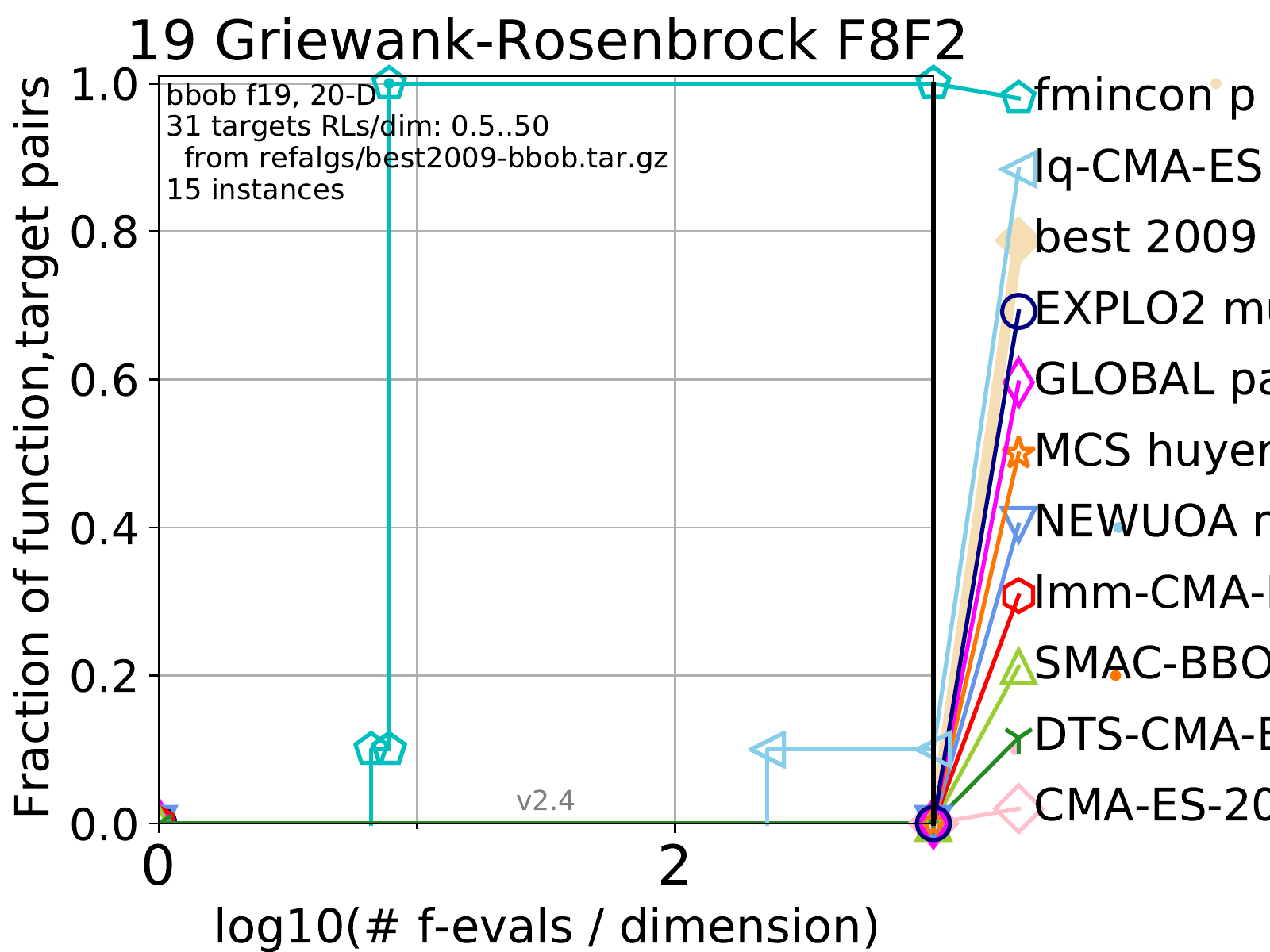}\\
\end{tabular}
\vspace*{-0.2cm}
\caption{
\label{fig:ECDFsingleOne20last}
As in Figure \ref{fig:ECDFsingleOne20}, but for $f_{18}$ and $f_{19}$. Note that Figures \ref{fig:GRF8F2VDCMAESvsEXPLO220} and \ref{fig:GRF8F2VDCMAESvsEXPLO2320} present an alternative analysis for $f_{19}$ in dimensions 20 and 320, respectively.
}
\end{figure}

\begin{figure}
\centering
\begin{tabular}{@{}c@{}c@{}c@{}c@{}}
\includegraphics[width=0.238\textwidth]{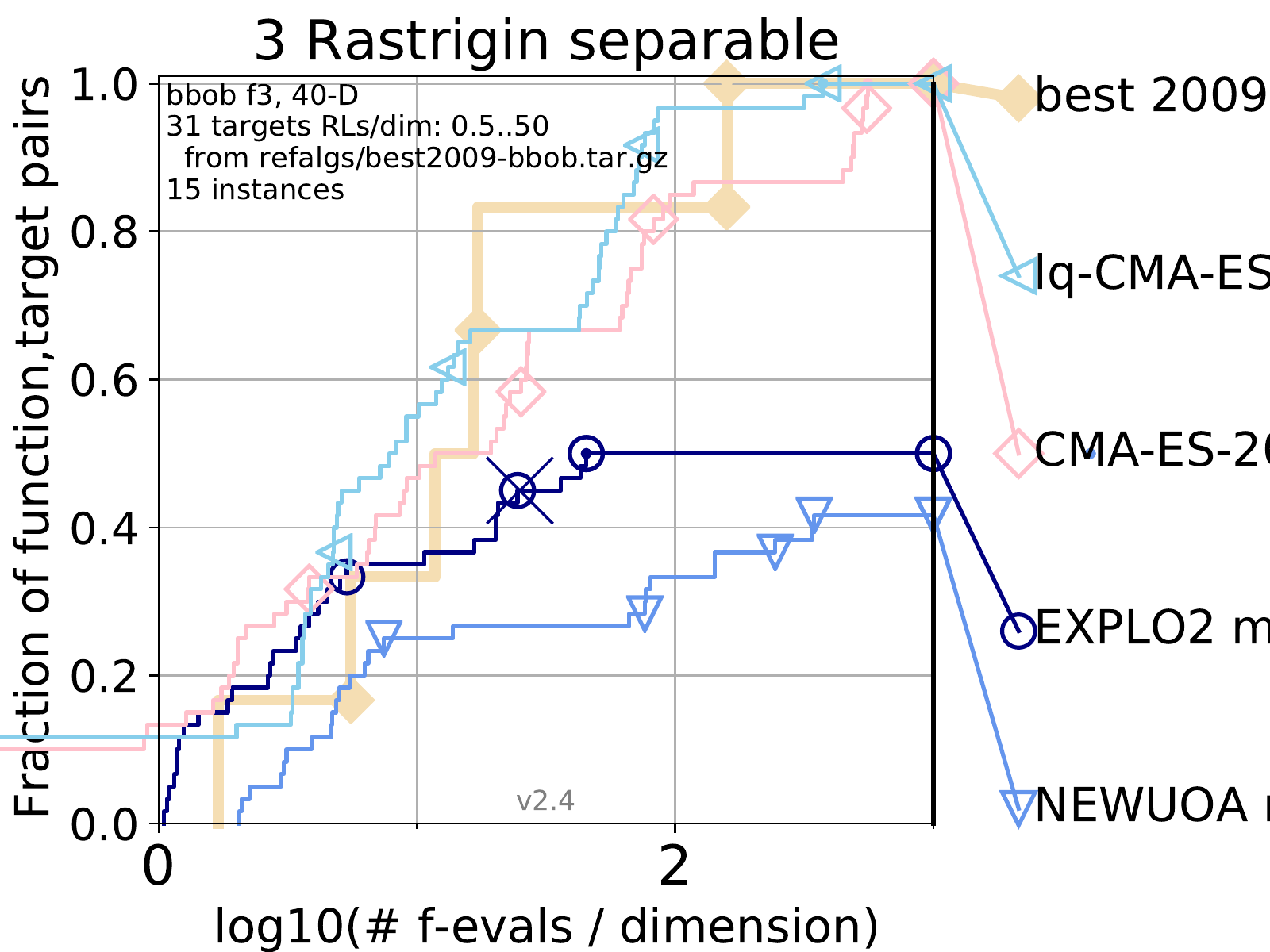}&
\includegraphics[width=0.238\textwidth]{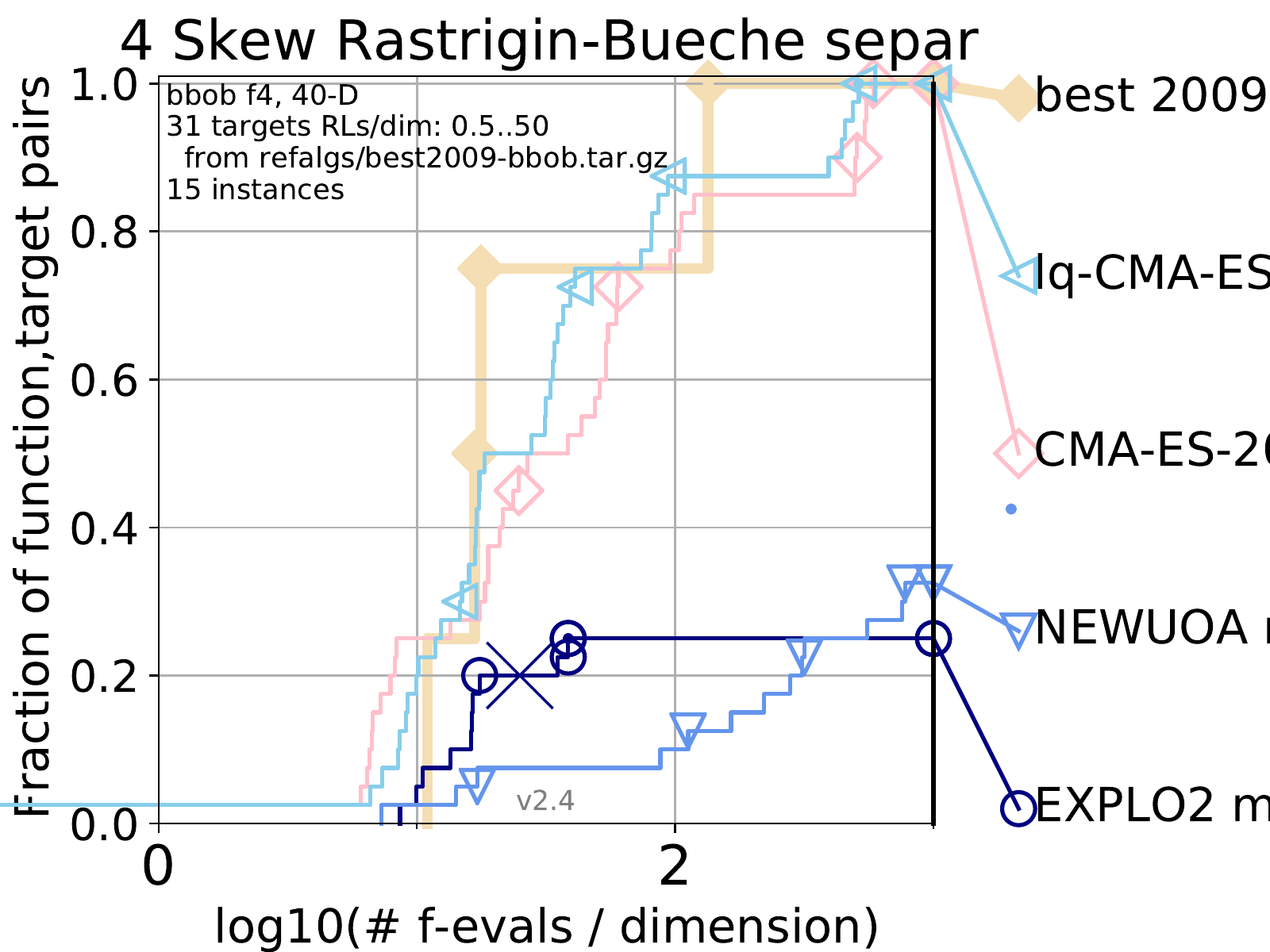}\\
\end{tabular}
\vspace*{-0.2cm}
\caption{
\label{fig:ECDFsingleOne40}
As in Figure \ref{fig:ECDFsingleOne20}, but for $D = 40$. Note that fewer algorithms have benchmark data for $D = 40$ than for $D = 20$. 
}
\end{figure}

\begin{figure}
\centering
\begin{tabular}{@{}c@{}c@{}c@{}c@{}}
\includegraphics[width=0.238\textwidth]{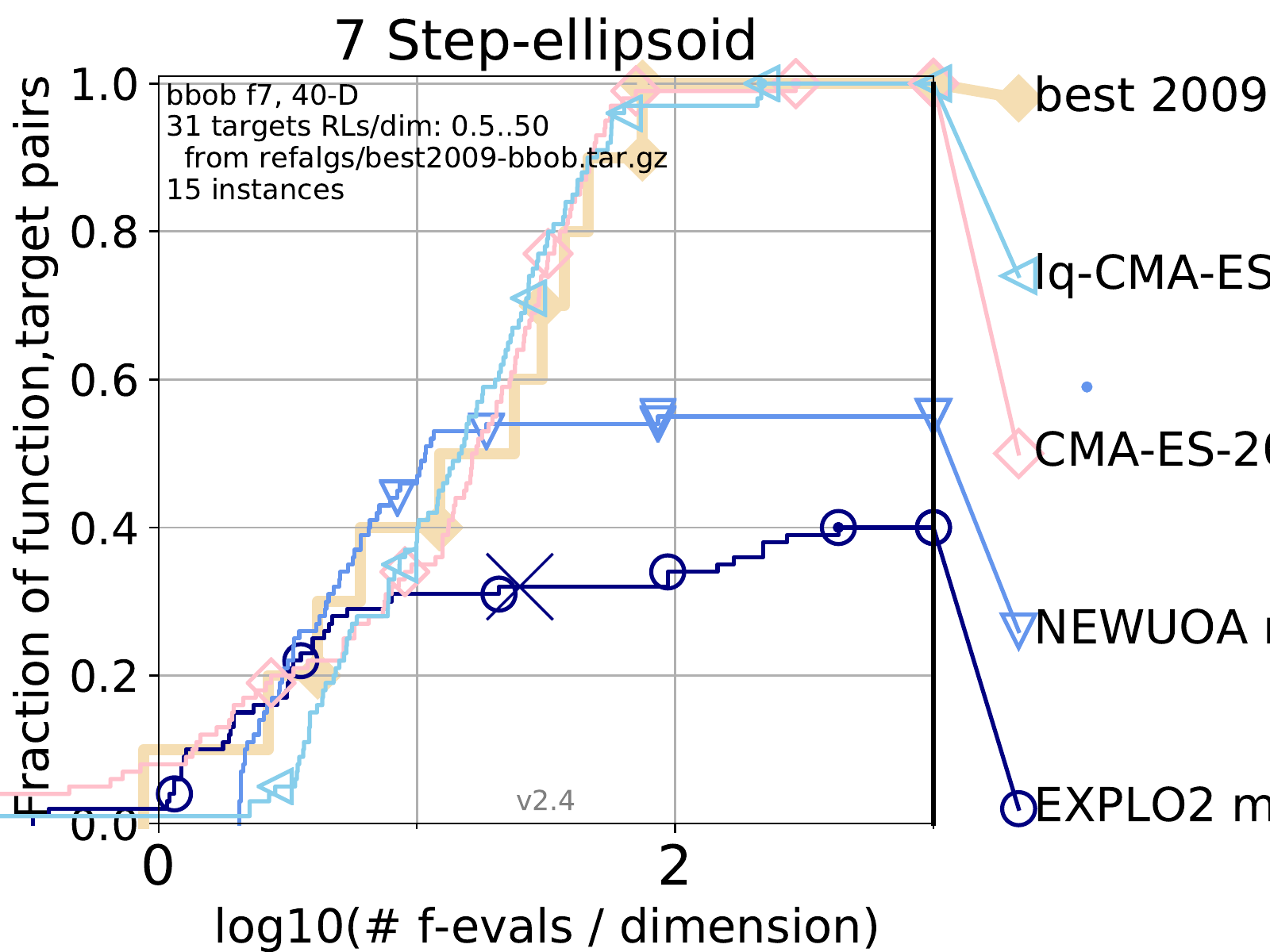}&
\includegraphics[width=0.238\textwidth]{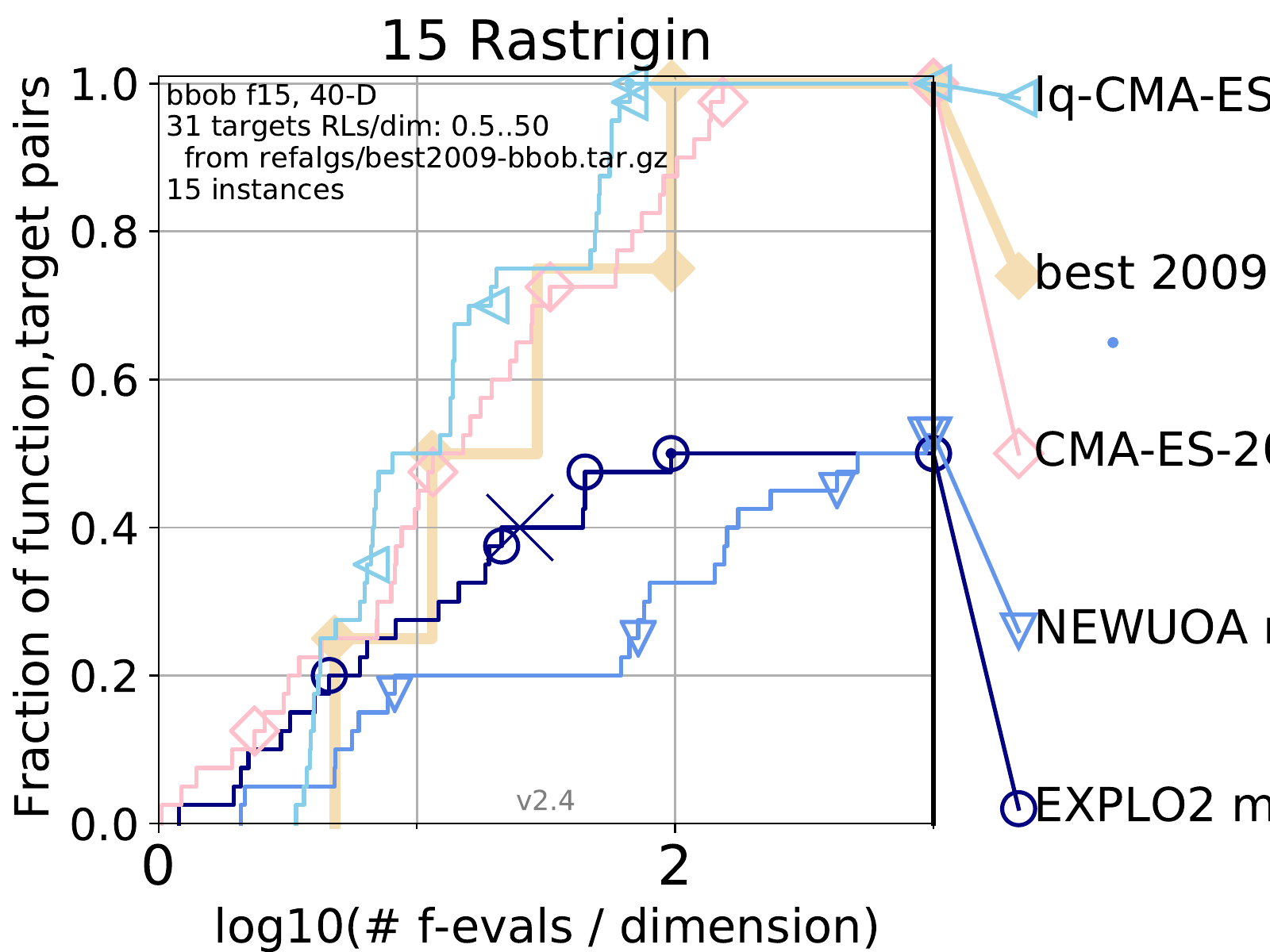}\\
\end{tabular}
\vspace*{-0.2cm}
\caption{
As in Figure \ref{fig:ECDFsingleOne20}, but for $f_{7}$ and $f_{15}$ and $D = 40$. 
}
\end{figure}

\begin{figure}
\centering
\begin{tabular}{@{}c@{}c@{}c@{}c@{}}
\includegraphics[width=0.238\textwidth]{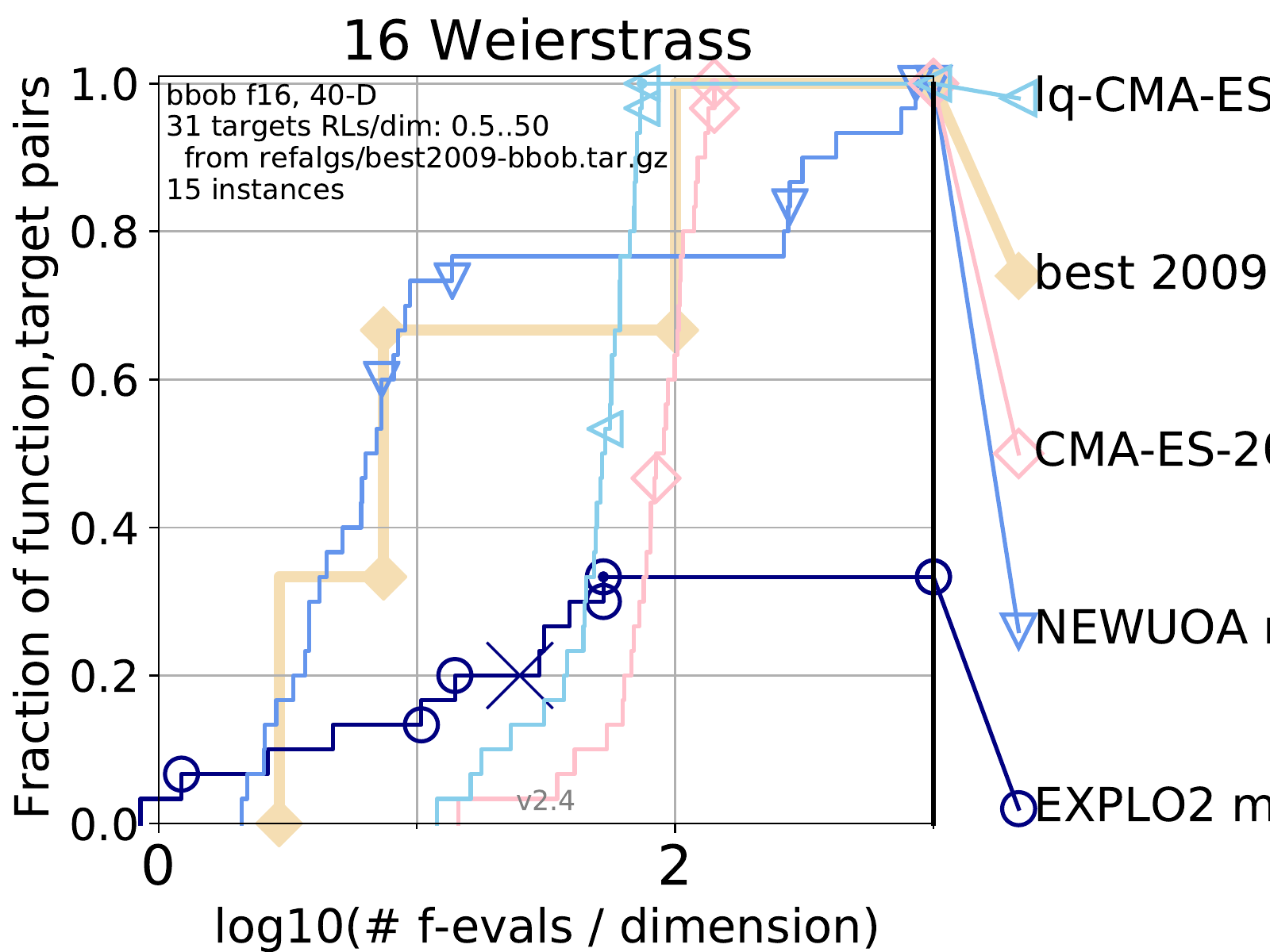}&
\includegraphics[width=0.238\textwidth]{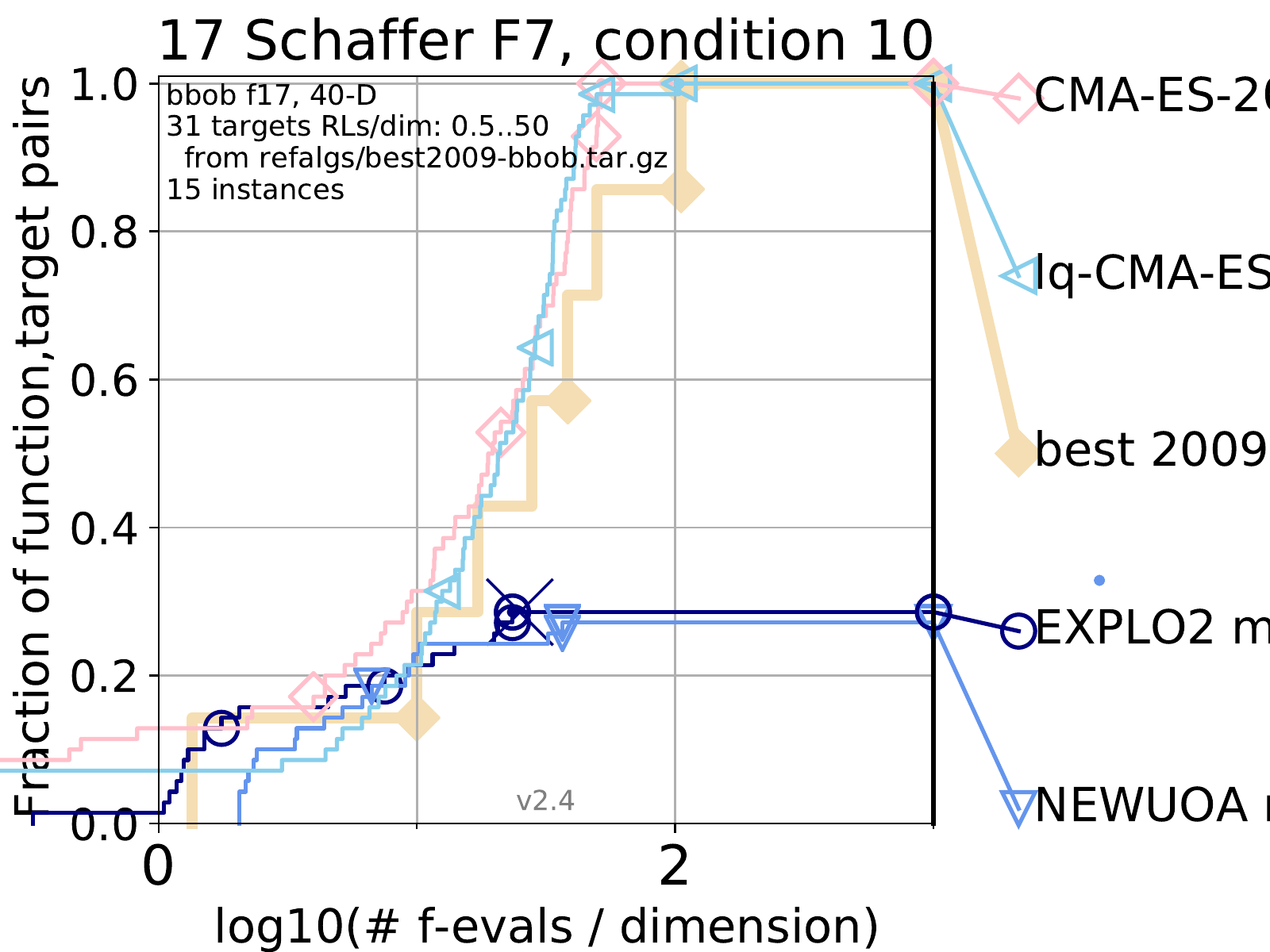}\\
\end{tabular}
\vspace*{-0.2cm}
\caption{
As in Figure \ref{fig:ECDFsingleOne20}, but for $f_{16}$ and $f_{17}$ and $D = 40$. 
}
\end{figure}

\begin{figure}
\centering
\begin{tabular}{@{}c@{}c@{}c@{}c@{}}
\includegraphics[width=0.238\textwidth]{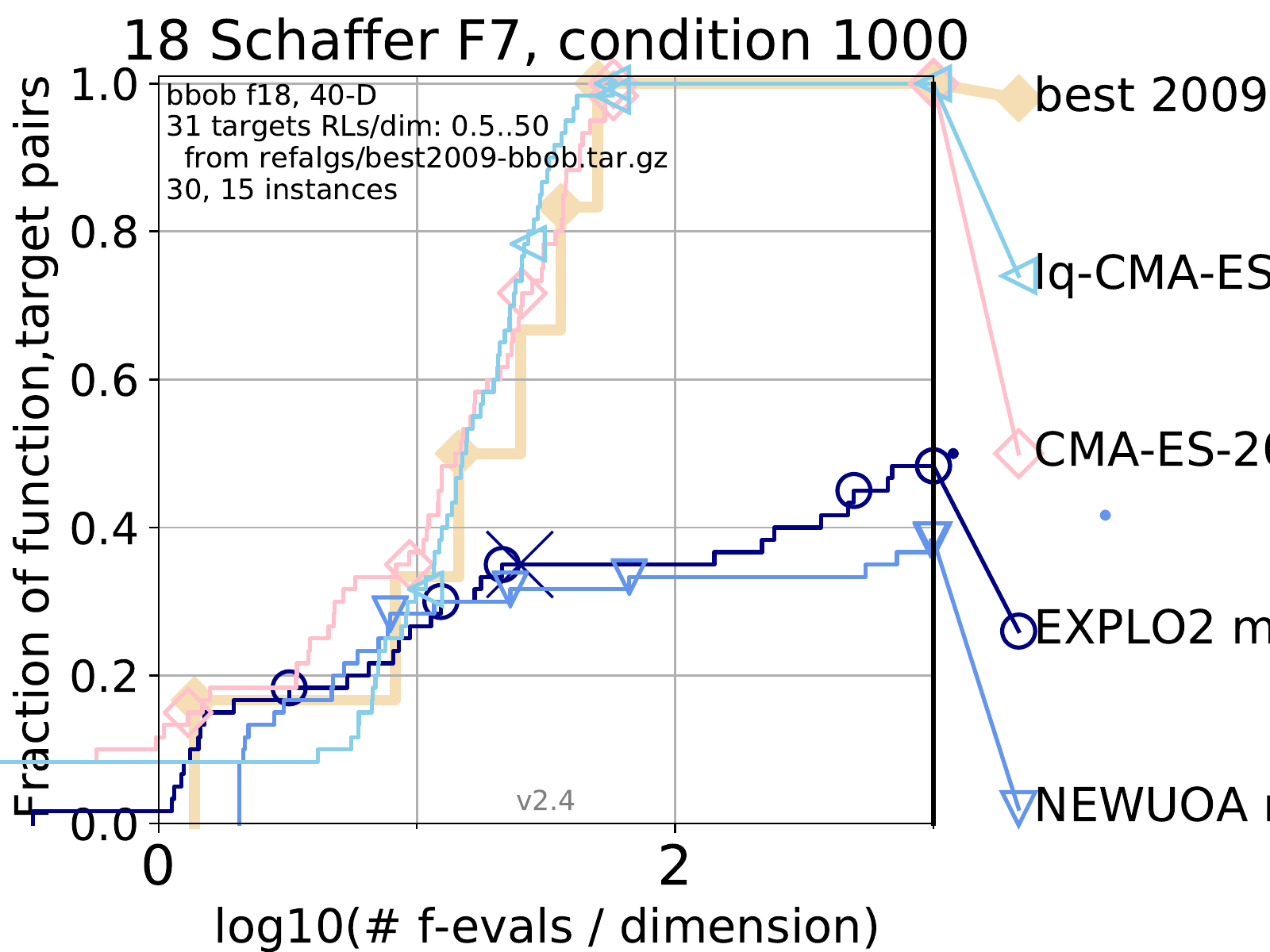}&
\includegraphics[width=0.238\textwidth]{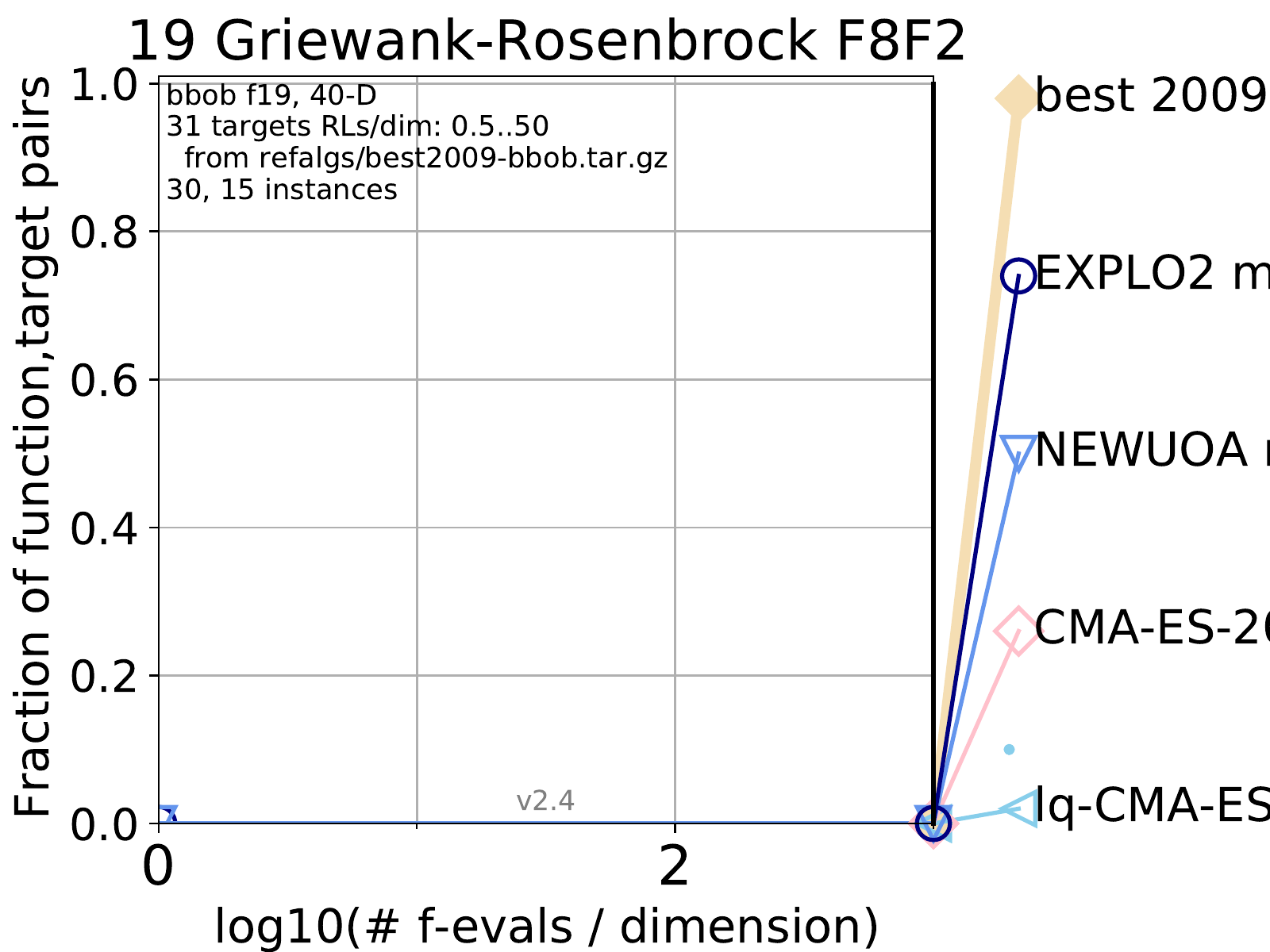}\\
\end{tabular}
\vspace*{-0.2cm}
\caption{
\label{fig:ECDFsingleOne40last}
As in Figure \ref{fig:ECDFsingleOne20}, but for $f_{18}$ and $f_{19}$ and $D = 40$. 
}
\end{figure}

Our experiments show that \texttt{EXPLO2} consistently outperforms \texttt{fmincon} on the sorts of problems for which it was designed, viz., multimodal functions with adequate global structure ($\{f_{15},\dots,f_{19}\}$), with the exception of the function $f_{19}$, which has many shallow minima. On the other hand, for most other problems in the \texttt{BBOB} suite (with the notable exceptions of $f_3$, $f_4$, and $f_7$, which respectively gauge the ability to exploit separability, not rely on symmetry, and to avoid getting trapped on plateaux), \texttt{EXPLO2} does worse than \texttt{fmincon}.

\texttt{EXPLO2} outperforms \texttt{NEWUOA} on $f_{15}$, $f_{17}$, and $f_{18}$, while the converse is true for $f_{16}$ and $f_{19}$, suggesting that \texttt{EXPLO2} (as a wrapper around \texttt{fmincon}) is best for multimodal functions with global structure that are not rugged, repetitive, or with many shallow minima. 
\footnote{\texttt{EXPLO2} also does serviceably well on multimodal functions with weak global structure: see appendices for detailed results.}
In other words, \texttt{EXPLO2} is best for problems with fairly structured landscapes. Even without accounting for the benefits of parallelism, \texttt{EXPLO2} is arguably the best available method besides \texttt{*-CMA-ES} for problems in dimension $\approx 20$ to $40$; once we do account for parallelism and/or in higher dimensions, only certain members of \texttt{*-CMA-ES} perform comparably: see \S \ref{sec:Parallel}. 

Meanwhile, in high enough dimensions, \texttt{*-CMA-ES} also ceases to be practical, with the exception of large-scale variants. \texttt{EXPLO2} also tied with or outperformed all of the large-scale algorithms benchmarked in \cite{varelas2019benchmarking} on the large-scale version of every \texttt{BBOB} function in $\{f_{15},\dots,f_{24}\}$ for dimensions 80, 160, and 320 \cite{elhara2019coco} with a budget of 2 evaluations/dimension. \texttt{EXPLO2} still performed fairly close to (and sometimes still better than) other algorithms in dimension 640 on the same budget per dimension. Finally, it is reasonable to hypothesize that a relative performance degradation in dimension 640 could be mitigated by replacing the inner \texttt{fmincon} optimization with \texttt{L-BFGS}. 

Because these experiments did not facilitate plots that conveyed meaningful information (an ``expensive plot'' option is not available in \texttt{COCO} for the large-scale \texttt{BBOB} suite), we also performed a more \emph{ad hoc} comparison of \texttt{EXPLO2} to \texttt{VD-CMA-ES} \cite{akimoto2014comparison}, the best performing large-scale algorithm for structured multimodal functions in \cite{varelas2019benchmarking}. As Figures \ref{fig:RastriginVDCMAESvsEXPLO220}-\ref{fig:GRF8F2VDCMAESvsEXPLO2320}  show, \texttt{EXPLO2} performs better, regardless of parallelism.

\begin{figure}[h]
  \centering
  \includegraphics[trim = 20mm 105mm 15mm 105mm, clip, width=\columnwidth,keepaspectratio]{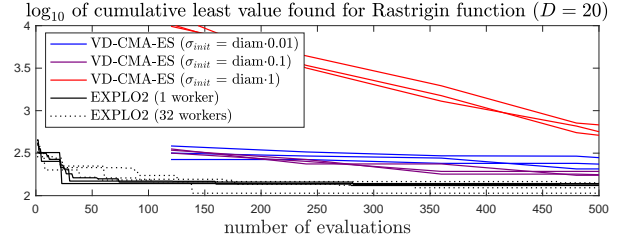}
  \caption{A comparison of \texttt{VD-CMA-ES} and \texttt{EXPLO2} on the Rastrigin function (without rotations or shifts, but otherwise corresponding to $f_{15}$ in the \texttt{BBOB} testbed) on $[-5.12,5.12]^D$ for $D = 20$ and a budget of 500 function evaluations. For \texttt{VD-CMA-ES}, various initial step sizes $\sigma_{init}$ were selected as shown and logging is intermittent; meanwhile, for \texttt{EXPLO2}, we considered $n_\parallel \in \{1,32\}$. Each algorithm/parameter pair was run three times, all shown. Note that the initialization phase of \texttt{VD-CMA-ES} is not plotted.
  }
  \label{fig:RastriginVDCMAESvsEXPLO220}
\end{figure}

\begin{figure}[h]
  \centering
  \includegraphics[trim = 20mm 105mm 15mm 105mm, clip, width=\columnwidth,keepaspectratio]{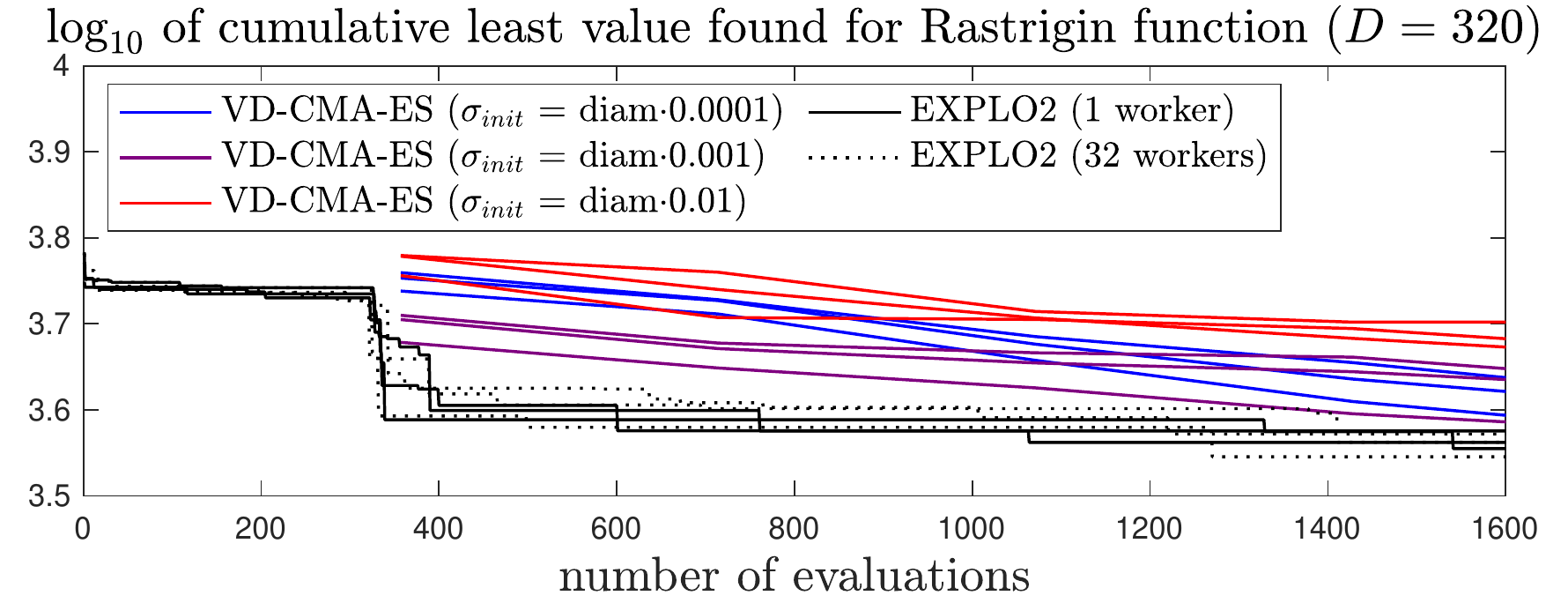}
  \caption{As in Figure \ref{fig:RastriginVDCMAESvsEXPLO220}, but for $D = 320$ and a budget of 1600. 
  }
  \label{fig:RastriginVDCMAESvsEXPLO2320}
\end{figure}


\begin{figure}[h]
  \centering
  \includegraphics[trim = 20mm 105mm 15mm 105mm, clip, width=\columnwidth,keepaspectratio]{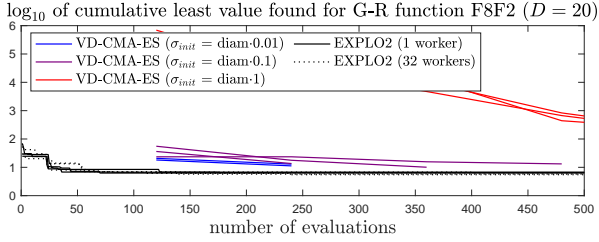}
  \caption{As in Figure \ref{fig:RastriginVDCMAESvsEXPLO220}, but for the Griewank-Rosenbrock function F8F2 (without rotations or shifts, but otherwise corresponding to $f_{19}$ in the \texttt{BBOB} testbed) on $[-5,5]^D$ for $D = 20$. Note that \texttt{VD-CMA-ES} often terminates before the budget is reached.
  }
  \label{fig:GRF8F2VDCMAESvsEXPLO220}
\end{figure}

Finally, in an experiment shown in \S \ref{sec:mixint}, we found that \texttt{EXPLO2} outperformed differential evolution and had performance almost indistinguishable from \texttt{CMA-ES} on the mixed-integer version of $f_{15}$ \cite{tuvsar2019mixed} in dimension $D = 20$ (the only function/dimension pair we tried among the mixed-integer suite \cite{tuvsar2019mixed}). 

To summarize, as the dimension of structured multimodal problems grows, \texttt{EXPLO2} outperforms \texttt{NEWUOA}, particularly when taking parallelism into account; meanwhile, the quadratic scaling with dimension of non-large-scale \texttt{*-CMA-ES} algorithms puts them at an increasing disadvantage since \texttt{EXPLO2} has essentially no runtime dependence on dimension other than through its inner solver, which here is the default large-scale interior-point \texttt{fmincon} algorithm. Finally, on structured multimodal functions, \texttt{EXPLO2} outperforms \texttt{VD-CMA-ES}, which is otherwise the best-performing large-scale algorithm that we are aware of.

\subsection{\label{sec:Parallel}Parallel performance}

While \texttt{EXPLO2} is parallelized (at marginal cost to performance), to the best of our knowledge no competing technique is except for \texttt{*-CMA-ES}. Indeed, as \cite{gao2017distributed} points out, most surrogate-based derivative-free optimizers--including \texttt{NEWUOA}--require sequential function evaluations. Though parallel algorithms exist \cite{haftka2016parallel,rehbach2018comparison,xia2020gops}, these are relatively few in number outside the context of Bayesian optimization (which is unsuitable for high-dimensional problems), and parallel techniques appropriate for high-dimensional problems have been considered in our design and/or benchmarking.
\footnote{One technique that we have not mentioned is the distributed quasi-Newton algorithm of \cite{gao2020distributed} (see also \cite{gao2021performance}), which is designed for extremely expensive situations on the order of $\ge$ 1 day/evaluation. This and related algorithms such as \texttt{SPMI} do not appear to be publically available and/or comprehensively benchmarked in the literature: indeed, \texttt{SPMI} is described in \cite{alpak2013techniques} as ``an \emph{in-house} massively parallel mixed integer and real variable optimization tool'' (emphasis added).}

With this in mind, since our experiments use $n_\parallel = 32$, 
\footnote{NB. For benchmarking, it was necessary to simulate parallelism, i.e., we replaced \texttt{parfor} loops with ordinary \texttt{for} loops.}
exceeding our per-dimension evaluation budget of 25, it is obvious that \texttt{EXPLO2} can outperform any of the competing algorithms considered here on the number of rounds of parallel function evaluations, except for \texttt{*-CMA-ES}, which becomes (depending on the variant employed) ill-suited for use or less performant than \texttt{EXPLO2} in high dimensions.

\section{\label{sec:Remarks}Remarks}

As mentioned above, it would be interesting--and likely useful--to replace \texttt{fmincon} with, e.g., \texttt{L-BFGS}. 

\begin{figure}[h]
  \centering
  \includegraphics[trim = 20mm 105mm 15mm 105mm, clip, width=\columnwidth,keepaspectratio]{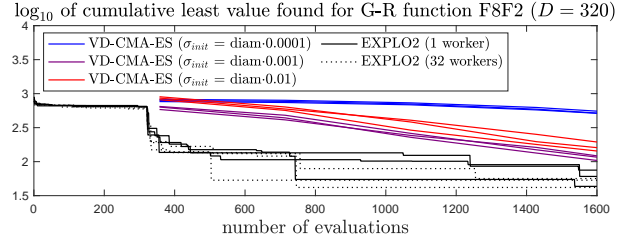}
  \caption{As in Figure \ref{fig:GRF8F2VDCMAESvsEXPLO220}, but for $D = 320$ and a budget of 1600. 
  }
  \label{fig:GRF8F2VDCMAESvsEXPLO2320}
\end{figure}

Though the runtime overhead of \texttt{EXPLO2} scales favorably with dimension, it is still high for any given function evaluation: the surrogate is complicated and even a large-scale inner optimizer takes resources. \texttt{EXPLO2} is therefore only suited for high-dimensional functions that are expensive to evaluate. While as mentioned earlier, hyperparameter optimization or simulation-defined functions are in this vein, it will generally be advisable to evaluate the suitability of \texttt{*-CMA-ES} as well in any particular application.

While \texttt{EXPLO2} seeks to balance exploration and exploitation through a regularizer, an approach in the vein of \cite{bischl2014moi} (cf. \cite{xia2020gops}) is also reasonable and lends itself to parallel execution. The idea here would be to consider the RBF interpolation and the differential magnitude as proxy objectives and evaluate the actual objective function in parallel at points on the Pareto front. This approach is likely to be useful for ``illuminating'' the search space or providing ``quality diversity'' \cite{mouret2015illuminating,pugh2016quality,fontaine2020covariance,chatzilygeroudis2021quality} in a way that focuses on diversity promotion versus the related idea \cite{vassiliades2017comparing} of niching techniques for multimodal optimization.
\footnote{
Cf. novelty search \cite{lehman2011abandoning} and its application to global optimization \cite{fister2019novelty}.
}


We have experimented with alternative choices for $t$ in Algorithm \ref{alg:EXPLO2} besides $\sqrt{\varepsilon}$, but all yielded inferior results (the exact $t=0$ limit caused trouble with \texttt{fmincon}). However, our experiments were not exhaustive or conclusive: it is possible that a more substantial modification of Algorithm \ref{alg:EXPLO2} that also lets $t$ vary would yield superior results. Nevertheless, we were not able to identify a competitive variant, let alone a clearly superior one.
\footnote{
A reasonable heuristic to try is $\exp(-t\Delta) = \varepsilon$, where machine epsilon is indicated on the right hand side and here $\Delta$ denotes the expected distance between two uniformly random points in the bounding box. While computing $\Delta$ is a notoriously hard problem without closed form solution in general and a very intricate result even for rectangles, a result of \cite{bonnet2021sharp} for a general convex body is that $\frac{3D+1}{2(D+1)(2D+1)} < \frac{\Delta}{\text{diam}} < \frac{\sqrt{\pi}}{3} \frac{\Gamma(\frac{D+1}{2})}{\Gamma(\frac{D}{2})}$, where $\text{diam}$ indicates the diameter of the body, i.e., $| u-\ell |$. However, our results with the associated lower bound on $t$ were disappointing, as were our results with data-dependent minimal values of $t$ that ensured weighting components were nonnegative.
} 

Finally, while we have already pointed out that differential magnitude is an attractive way to perform exploration for Bayesian optimization, it may also be useful for generating proposals in Markov chain Monte Carlo methods, including fast parallel algorithms such as those in \cite{huntsman2020fast,huntsman2021sampling}. More generally, magnitude-based methods hold promise for many other problems in machine learning.

\section*{Acknowledgements}

Thanks to Andy Copeland, Megan Fuller, Zac Hoffman, Rachelle Horwitz-Martin, and Jimmy Vogel for many patient questions, answers, and observations that influenced and improved this paper. This research was developed with funding from the Defense Advanced Research Projects Agency (DARPA). The views, opinions and/or findings expressed are those of the author and should not be interpreted as representing the official views or policies of the Department of Defense or the U.S. Government. DISTRIBUTION STATEMENT A. Approved for public release; distribution is unlimited.



\bibliography{singleObjOpt}
\bibliographystyle{icml2022}

\clearpage
\appendix
%
%
%


\section{\label{sec:appendix}Outline of appendices}

This appendices are organized as follows: 
\begin{itemize}
	\item \S \ref{sec:scaleZero} discusses the limit $t \downarrow 0$ on (co)weightings; 
	\item \S \ref{sec:DifferentialMagnitude} discusses the differential magnitude of a point; 
	\item \S \ref{sec:HighDimTimeSeries} discusses how to compute (co)weightings in dimension $> 10^5$ in the context of time series analysis, with an eye towards more general problems of similar scale using similarity search; 
	\item \S \ref{sec:efficientWeighting} discusses the efficient computation of weightings from the perspective of linear algebra;
	\item \S \ref{sec:iohAnalyzer} contains plots generated using \texttt{IOHAnalyzer};
	\item \S \ref{sec:ERTscaling} contains plots generated using \texttt{COCO} for scaling of runtime with dimension;
	\item \S \ref{sec:ECDFs20} contains plots generated using \texttt{COCO} for runtime distributions (ECDFs) per function for $D = 20$;
	\item \S \ref{sec:ECDFs40} contains plots generated using \texttt{COCO} for runtime distributions (ECDFs) per function for $D = 40$;
	\item \S \ref{sec:largescale} contains plots generated using \texttt{COCO} for the large-scale \texttt{BBOB} suite;
	\item \S \ref{sec:mixint} contains a plot generated using \texttt{COCO} for the mixed-integer \texttt{BBOB} suite;
	\item \S \ref{sec:source} contains MATLAB source code for \texttt{EXPLO2};
	\item \S \ref{sec:sourceBench} contains MATLAB source code for benchmarking.
\end{itemize}

\section{\label{sec:scaleZero}The scale $t = 0$}

The limit $t \uparrow \infty$ of (co)weightings and the magnitude function is uninteresting, since the (co)weightings are (co)vectors of all ones, so the limiting value of the magnitude function is just the number of points in the space under consideration. However, the limit $t \downarrow 0$ contains more detailed structural information:

\begin{lemma} 
\label{lemma:scale0} 
For $d \in GL(n,\mathbb{R})$ the solution to $\exp[-td] w(t) = 1$ has the well-defined limit $w(0) := \lim_{t \downarrow 0} w(t) = \frac{d^{-1}1}{1^T d^{-1} 1}$. Similarly, the solution to $v(t) \exp[-td] = 1$ has the limit $v(0) = \frac{1^T d^{-1}}{1^T d^{-1} 1}$.
\end{lemma}

\begin{proof}[Proof (sketch)]
We have the first-order approximation $(11^T - td) w(t) \approx 1$. Applying the Sherman-Morrison-Woodbury formula \cite{horn2012matrix} and L'H\^opital's rule yields the result.
\end{proof}

\section{\label{sec:DifferentialMagnitude}The differential magnitude of a point}

Define $M := \left ( \begin{smallmatrix} A & B \\ C & D \end{smallmatrix} \right )$. Block inversion yields the formula
\begin{equation}
\begin{pmatrix} A^{-1} + A^{-1}B(M / A)^{-1}CA^{-1} & -A^{-1}B(M / A)^{-1} \\ 
-(M / A)^{-1}CA^{-1} & (M / A)^{-1} \end{pmatrix} 
\nonumber
\end{equation}
for $M^{-1}$, where the Schur complements are given by $M / D := A-BD^{-1}C$ and $M / A := D-CA^{-1}B$, and they satisfy $(M / A)^{-1} = D^{-1} + D^{-1}C(M / D)^{-1} BD^{-1}$ and $(M / D)^{-1} = A^{-1} + A^{-1}B(M / A)^{-1}CA^{-1}$.

If $C = B^T$ and $A$, $D$, and $M$ are all positive definite, then Theorem 7.7.7 of \cite{horn2012matrix} yields that $M / A$, $M / D$, and $M$ are also all positive definite. With this in mind (and similarly to \cite{bunch2020practical}), let $Z$ be positive definite and consider a positive definite matrix of the form
\begin{equation}
\label{eq:similarityPlusOne}
Z[\zeta] := \begin{pmatrix} Z & \zeta \\ \zeta^T & 1 \end{pmatrix}.
\end{equation}
The relevant Schur complement is
\footnote{
Note that if the entries of $-\log[Z[\zeta]]$ satisfy the triangle inequality, then $\zeta \le Z \zeta \le n \zeta$ componentwise, where $n$ is the dimension of $\zeta$.
}
\begin{equation}
Z[\zeta] / Z = 1-\zeta^T Z^{-1} \zeta,
\nonumber
\end{equation}
so $Z[\zeta]^{-1}$ equals
\begin{align}
 & \ \begin{pmatrix} Z^{-1} & 0 \\ 0 & 0 \end{pmatrix} + \frac{1}{1-\zeta^T Z^{-1} \zeta} \begin{pmatrix} Z^{-1} \zeta \zeta^T Z^{-1} & -Z^{-1} \zeta \\ \zeta^T Z^{-1} & 1 \end{pmatrix} \nonumber \\
= & \ \begin{pmatrix} Z^{-1} & 0 \\ 0 & 0 \end{pmatrix} + \frac{1}{1-\zeta^T Z^{-1} \zeta} \begin{pmatrix} -Z^{-1} \zeta \\ 1 \end{pmatrix} \begin{pmatrix} -\zeta^T Z^{-1} & 1 \end{pmatrix}
\nonumber
\end{align}
so that if $w = Z^{-1}1$ and $w[\zeta] := Z[\zeta]^{-1}1$ are respectively the weightings of $Z$ and $Z[\zeta]$, then
\begin{equation}
\label{eq:weightingPlusOne}
w[\zeta] = \begin{pmatrix} w \\ 0 \end{pmatrix} + \frac{1-\zeta^T w}{1-\zeta^T Z^{-1} \zeta} \begin{pmatrix} -Z^{-1} \zeta \\ 1 \end{pmatrix}.
\end{equation}
This yields


\begin{proposition}
\end{proposition}
\begin{equation}
\label{eq:magnitudePlusOne}
\text{Mag}(Z[\zeta]) = \text{Mag}(Z) + \frac{(1-\zeta^T w)^2}{1-\zeta^T Z^{-1} \zeta}. \qed
\end{equation}

It is tempting to try to use this result to try to establish the conjecture of \cite{yevseyeva2019application} that magnitude is submodular on Euclidean point sets,
\footnote{
Recall that a function $f : 2^\Omega \rightarrow \mathbb{R}$ is submodular iff for every $X\subseteq \Omega$ and $x_1,x_2 \in \Omega \backslash X$ such that $x_{1}\neq x_{2}$ it is the case that $f(X\cup \{x_1\})+f(X\cup \{x_2\})\geq f(X\cup \{x_1,x_2\})+f(X)$.
} 
but there is a simple counterexample:

\begin{example}
Let $\Omega = \{(1,0),(0,1),(-1,0),(2,0)\}$ endowed with Euclidean distance; let $X = \{(1,0),(0,1)\}$; let $x_1 = (-1,0)$, and $x_2 = (2,0)$. Then (abusing notation) $\text{Mag}(X\cup \{x_1\})+\text{Mag}(X\cup \{x_2\}) \approx 4.1773$ while $\text{Mag}(X\cup \{x_1,x_2\})+\text{Mag}(X)\approx 4.1815$. 
\end{example}

Still, for $t \rightarrow \infty$ there is an asymptotic inclusion-exclusion formula for magnitudes of compact convex bodies, as pointed out by \cite{gimperlein2021magnitude}. In this sense the magnitude is approximately submodular. This suggests using the standard greedy approach for approximate submodular maximization \cite{nemhauser1978analysis,krause2014submodular} despite the fact that we do not have any theoretical guarantees.

On the other hand, if $Z = \exp[-td]$, $\zeta = \exp[-t\delta]$, and we write $\omega = d^{-1}1$, we have the following:

\begin{lemma}
\label{lem:mag1at0}
\begin{align}
\label{eq:mag1at0}
& \frac{(1-\zeta^T w)^2}{1-\zeta^T Z^{-1} \zeta} = o(t) + \nonumber \\
& t \left( \frac{\delta^T\omega-1}{1^T \omega-t} \right)^2 \frac{1^T \omega - t}{-1+2\delta^T \omega + \delta^T[(1^T \omega)d^{-1}-\omega \omega^T]\delta}.
\end{align}
\end{lemma}

    %
    %
    %

\begin{proof}
Using $Z \approx 11^T - td$, the Sherman-Morrison formula gives $Z^{-1} \approx -\frac{1}{t} \left (d^{-1}+\frac{\omega \omega^T}{t-1^T \omega} \right )$. A line of algebra yields $w = Z^{-1}1 \approx \frac{\omega}{1^T \omega - t}$, and similarly $1-\zeta^T w \approx t \frac{\delta^T \omega - 1}{1^T \omega - t}$. Now
\begin{align}
\zeta^T Z^{-1} \zeta & \approx -\frac{1}{t}(1^T-t\delta^T) \left ( d^{-1} + \frac{\omega \omega^T}{t-1^T \omega} \right ) (1-t\delta) \nonumber \\
& = \frac{1^T \omega}{1^T \omega - t} - 2t \frac{\delta^T \omega}{1^T \omega - t} - t \delta^T \left ( d^{-1} + \frac{\omega \omega^T}{t-1^T \omega} \right ) \delta \nonumber \\
& \approx \frac{(1^T-2t\delta^T)\omega-t \delta^T[(1^T \omega) d^{-1}-\omega \omega^T]\delta}{1^T \omega-t}, \nonumber
\end{align}
so 
\begin{equation}
1-\zeta^T Z^{-1} \zeta \approx t \frac{-1 + 2\delta^T \omega + \delta^T [(1^T \omega) d^{-1}-\omega \omega^T] \delta}{1^T \omega - t}, \nonumber
\end{equation}
from which the result follows.
\end{proof}

Another way to write \eqref{eq:mag1at0} is using the identities 
\begin{equation}
(\delta^T \omega-1)^2 = \begin{pmatrix}\delta^T & 1\end{pmatrix} \begin{pmatrix} \omega \omega^T & -\omega \\ -\omega^T & 1 \end{pmatrix} \begin{pmatrix}\delta \\ 1\end{pmatrix} \nonumber
\end{equation}
and
\begin{align}
& -1+2\delta^T \omega + \delta^T[(1^T \omega)d^{-1}-\omega \omega^T]\delta = \nonumber \\
& \begin{pmatrix}\delta^T & 1\end{pmatrix} \begin{pmatrix} (1^T \omega) d^{-1}-\omega \omega^T & \omega \\ \omega^T & -1 \end{pmatrix} \begin{pmatrix}\delta \\ 1\end{pmatrix}. \nonumber
\nonumber
\end{align}
In $\mathbb{R}^n$, we have the ``far-field'' asymptotic $\delta \sim s(\delta) 1$, where $s(\cdot)$ is any function on tuples that takes values between the minimum and maximum (e.g., an order statistic or quasi-arithmetic mean). That is, in the limit $t \rightarrow 0$ and $\delta \rightarrow \infty$, \eqref{eq:mag1at0} takes the form (with an obvious but helpful abuse of notation)
\begin{align}
& \frac{(1-\zeta^T w)^2}{1-\zeta^T Z^{-1} \zeta} \sim \nonumber \\
& \frac{t}{1^T \omega} \cdot \frac{\begin{pmatrix}s 1^T & 1\end{pmatrix} \begin{pmatrix} \omega \omega^T & -\omega \\ -\omega^T & 1 \end{pmatrix} \begin{pmatrix}s 1 \\ 1\end{pmatrix}}{\begin{pmatrix}s 1^T & 1\end{pmatrix} \begin{pmatrix} (1^T \omega) d^{-1}-\omega \omega^T & \omega \\ \omega^T & -1 \end{pmatrix} \begin{pmatrix}s 1\\ 1\end{pmatrix}} + o(t). \nonumber
\end{align}
Expanding this out, we find that the $s^2$ term in the denominator is zero, and we end up with
\begin{equation}
\label{eq:mag1at0limit}
\frac{(1-\zeta^T w)^2}{1-\zeta^T Z^{-1} \zeta} \sim \frac{s(\delta)}{2}t + o(t).
\end{equation}

That is, the change in magnitude is asymptotically linear with respect to distance. Among other things, this provides a foundation for building diverse point sets in unbounded regions of Euclidean space by maximizing the ratio of the increase of magnitude and a suitable function $s(\delta)$.

\section{\label{sec:HighDimTimeSeries}Coweightings in dimension $> 10^5$ and time series analysis}

Here we sketch how coweightings can reliably identify ``boundary'' elements of large datasets in high dimension within the restricted context of time series analysis. Efficient analogues of the construction below can likely be produced in more general settings using similarity search techniques \cite{zezula2006similarity}, provided they are available.

Consider a reasonably nice time series $\{x_j\}_{j=1}^{n+L-1}$ with $x_j \in \mathbb{R}$ 
and $n \gg L \gg 1$. By considering contiguous subsequences of length $L$, we obtain $n$ points $x^{[L]}_j := (x_j,\dots,x_{j+L-1})\in \mathbb{R}^L$. The \emph{matrix profile} \cite{yeh2016matrix} 
\footnote{
See also \cite{alaee2020matrix} and intervening papers.
}
operates on this latter representation to find maximally conserved motifs, anomalous subsequences, etc. In particular, the matrix profile efficiently yields (an anytime approximation of) a sequence of pairs $(d_j,\mathcal{I}(j))$ such that $x^{[L]}_{\mathcal{I}(j)}$ is the nearest neighbor to $x^{[L]}_j$ under $d^{(z)}$, with $d^{(z)}(x^{[L]}_j,x^{[L]}_{\mathcal{I}(j)}) = d_j$, and where $d^{(z)}$ is a variant of the $\ell^2$ distance obtained by subtracting mean values and normalizing by the standard deviations of arguments.
\footnote{
Note that generally $\mathcal{I}(\mathcal{I}(j)) \ne j$. This fact is presumably responsible for some numerical annoyances that manifest as \texttt{Inf} and/or \texttt{NaN} entries that can be swept under the rug in practice.
}
It turns out that the information provided by the matrix profile can efficiently address many if not most problems in basic time series analysis. 

In fact, the index alone $\mathcal{I}$ suffices to yield impressive information about time series when coupled with a coweighting. The idea is as follows: let $d$ be a matrix of size $n$ with all entries equal to $\infty$ except for a zero diagonal and $d_{j,\mathcal{I}(j)} := d_j$ (or any other finite number on the right hand side). Then the matrix $Z(0) := \lim_{t \downarrow 0} \exp(-td)$ is a 0-1 matrix with precisely two nonzero entries per row.
\footnote{
In principle, linear algebra can be done on this matrix in linear time: see \cite{saunders2015matrices}. In any event, MATLAB performs the requisite computations quickly. On the other hand, for $n$ large the actual distance matrix will be totally intractable to store. To the extent that any algorithmic approaches to producing weightings might work for either that situation or the case $t > 0$ in the present context, they would presumably require sophistication.
}
While this matrix is obviously singular, dividing a row vector of all ones on the left by it generally yields a reasonably nice result: i.e., the number of singular entries is reasonably small if not zero. 
\footnote{
Surprisingly, considering $Z' := \max(Z(0),Z(0)^T) = \lim_{t \downarrow 0} \exp(-t \cdot \min(d,d^T))$ breaks things in practice. We have not attempted to understand why, though the mechanism does not appear obvious.
}
Setting these to zero (or exploiting the structure of the problem to justify interpolation), the resulting approximate coweighting yields information about motifs which are presumably extremal in some sense. 

Figures \ref{fig:LDPA_5_20201112}-\ref{fig:LDPA_7_20201112} show plots of subsequences with successive highest and lowest coweightings for (slightly smoothed) coordinate data in a localization application \cite{kaluvza2010agent} with $n = 163861$ and $L = 1000$. The figures suggest that these subsequences are near each other and respectively on or ``adjacent to'' a ``boundary'' of $\{x^{[L]}_j\}_{j=1}^n \subset \mathbb{R}^L$, since their coweightings are large in absolute value and approximately cancel. Note that in some cases the pairing is not exact, but the coweighting values indicate this and can suggest alternative pairings.
\footnote{
Note that for a subsequence with a large coweighting we could consider instead the nearest subsequence with a negative coweighting, but this would be more computationally expensive and would presuppose and/or largely reproduce the qualitative results here.
}

\begin{figure}[h]
  \centering
  \includegraphics[trim = 10mm 110mm 10mm 110mm, clip, width=\columnwidth,keepaspectratio]{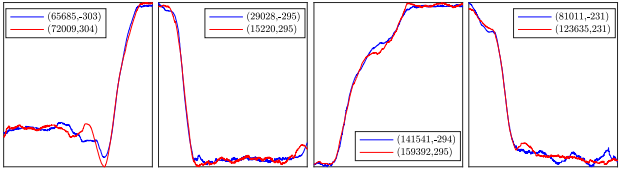}
  \caption{Plots (for the first spatial coordinate of data from \cite{kaluvza2010agent}) of $x^{[L]}_j$ after subtracting means and normalizing by standard deviations. From right to left, we show pairs of successively largest {\color{blue}negative} and {\color{red}positive} values of the (approximate) coweighting $\hat v$ in {\color{blue}blue} and {\color{red}red}, respectively. The legend entries are of the form $(j, \hat v_j)$.}
  \label{fig:LDPA_5_20201112}
\end{figure}

\begin{figure}[h]
  \centering
  \includegraphics[trim = 10mm 110mm 10mm 110mm, clip, width=\columnwidth,keepaspectratio]{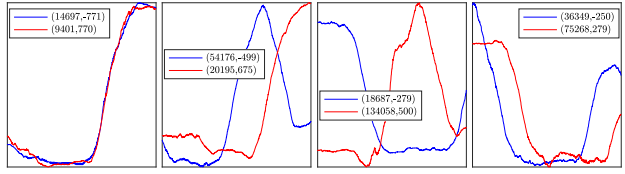}
  \caption{As in Figure \ref{fig:LDPA_5_20201112}, but for the second spatial coordinate of data from \cite{kaluvza2010agent}. Note that the pairing is not exact, but the coweighting values indicate this.}
  \label{fig:LDPA_6_20201112}
\end{figure}

\begin{figure}[h]
  \centering
  \includegraphics[trim = 10mm 110mm 10mm 110mm, clip, width=\columnwidth,keepaspectratio]{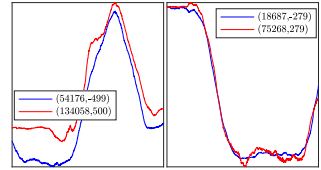}
  \caption{An improved pairing for some of the data in Figure \ref{fig:LDPA_6_20201112}. Not shown: $x^{[L]}_{20195}$ in the center left panel of Figure \ref{fig:LDPA_6_20201112} and the data in the leftmost panel of Figure \ref{fig:LDPA_6_20201112} also match decently.}
  \label{fig:LDPA_6b_20201112}
\end{figure}

\begin{figure}[h]
  \centering
  \includegraphics[trim = 10mm 110mm 10mm 110mm, clip, width=\columnwidth,keepaspectratio]{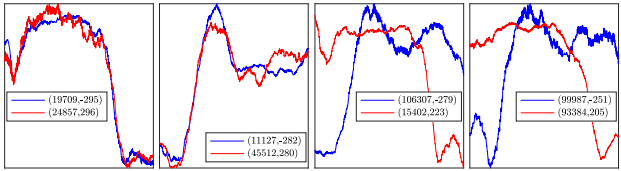}
  \caption{As in Figure \ref{fig:LDPA_5_20201112}, but for the third spatial coordinate of data from \cite{kaluvza2010agent}. Note that the pairing is not exact, but the coweighting values indicate this. It seems likely that in this particular case these points ``interfere'' with each other since there are obvious structures that are roughly conserved across the plots shown.}
  \label{fig:LDPA_7_20201112}
\end{figure}

To reinforce the geometrical characterization above, we applied the UMAP dimension reduction algorithm \cite{mcinnes2018umap} to subsets of $\{x^{[L]}_j\}_{j=1}^n$ for the first spatial coordinate of data from \cite{kaluvza2010agent} (the other spatial coordinates yield qualitatively similar results that we do omit for economy). Figure \ref{fig:UMAPunstructured} shows how this approach captures salient aspects of the geometry. Figure \ref{fig:UMAPstructured} takes a more structured approach by considering the union of $x^{[L]}_j$ with i) large (in absolute value) coweightings; ii) equispaced indices; and iii) a sample stratified by coweighting values (i.e., the lowest decile of coweights corresponds to 10\% of the [sub]sample, and so on). Figure \ref{fig:UMAPstructured} also highlights the nine points resulting from an erosion-type procedure as shown in Figure \ref{fig:SortOfErosion}. Specifically, we consider the points whose coweighting is above a threshold $\tau$, and exploit the problem structure to isolate points with the largest coweighting within an index range of $\pm L/2$. We then compute the dissimilarity at $t=0$, and check to see if any of \emph{its} coweighting components are negative. We then take the largest $\tau$ that produces a nonnegative coweighting on these isolated points: in this case $\tau \approx 119.8$. This procedure evidently retains reliable ``witnesses'' to the geometry.

\begin{figure}[h]
  \centering
  \includegraphics[trim = 30mm 80mm 25mm 70mm, clip, width=.49\textwidth,keepaspectratio]{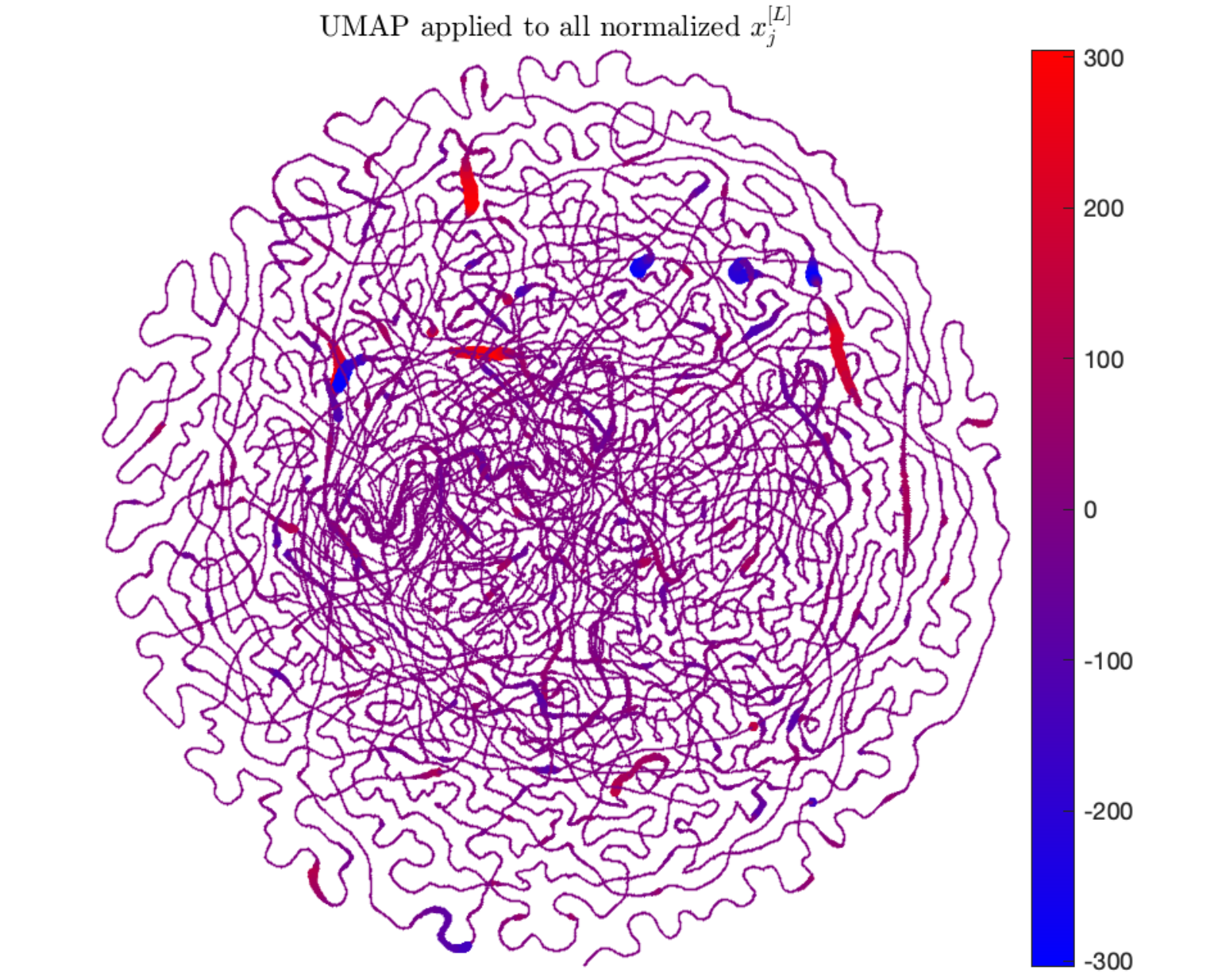}
  \includegraphics[trim = 30mm 80mm 25mm 70mm, clip, width=.49\textwidth,keepaspectratio]{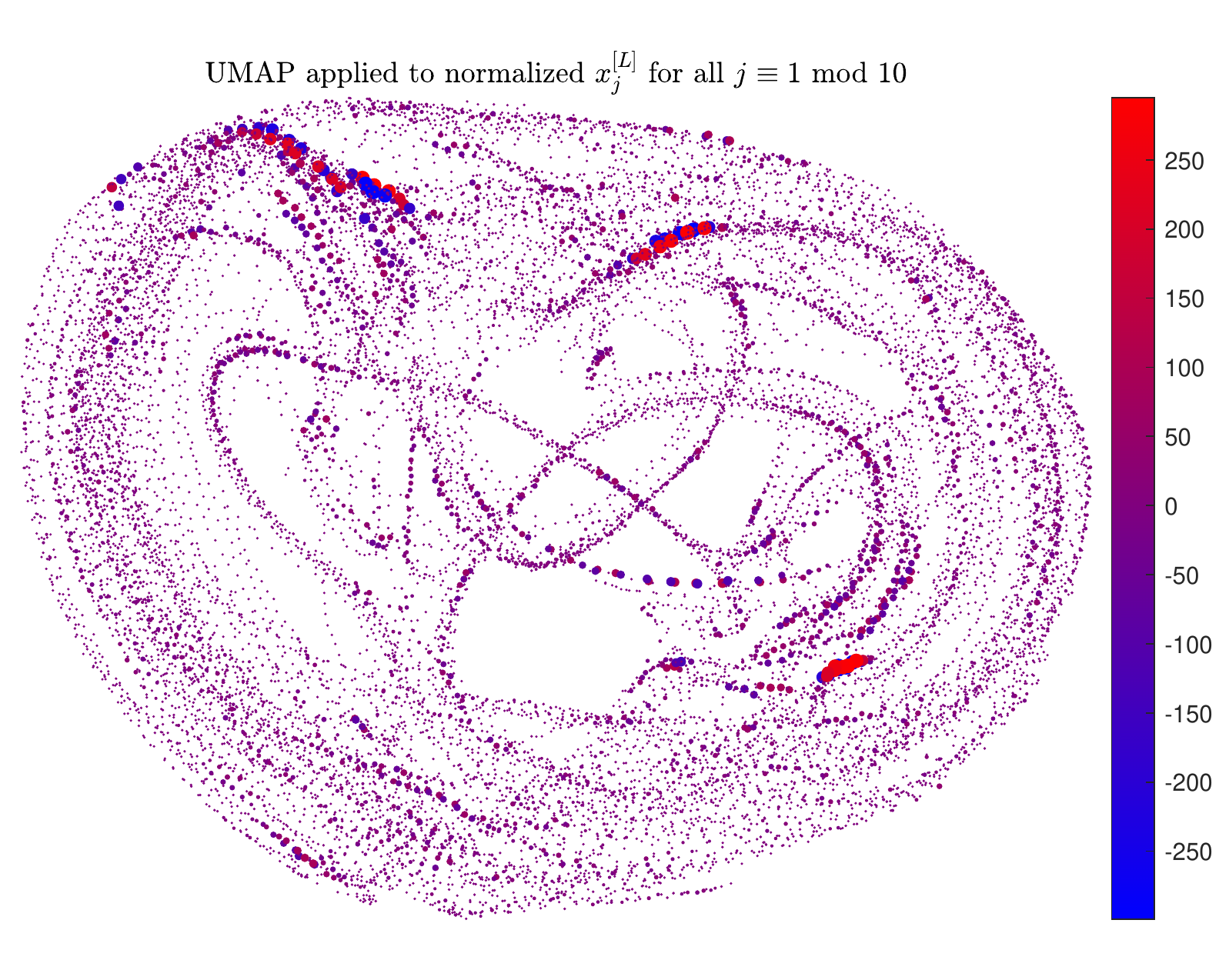}
  \caption{(L) UMAP applied to $\{x^{[L]}_j\}_{j=1}^n$ after normalization. Points are colored by (approximate) coweighting components, and sized by the absolute value of these components. While this accurately captures the intrinsic 1-dimensional geometry of the data, this is actually uninformative from the point of view of bulk geometry. (R) UMAP applied to $\{x^{[L]}_j\}_{j \equiv 1 \text{ mod } 10}$ after normalization. Here the colocation of positively and negatively coweighted points is visually apparent, as is their distribution along ``ridges'' in the data.}
  \label{fig:UMAPunstructured}
\end{figure}

\begin{figure}[h]
  \centering
  \includegraphics[trim = 30mm 80mm 30mm 65mm, clip, width=.5\textwidth,keepaspectratio]{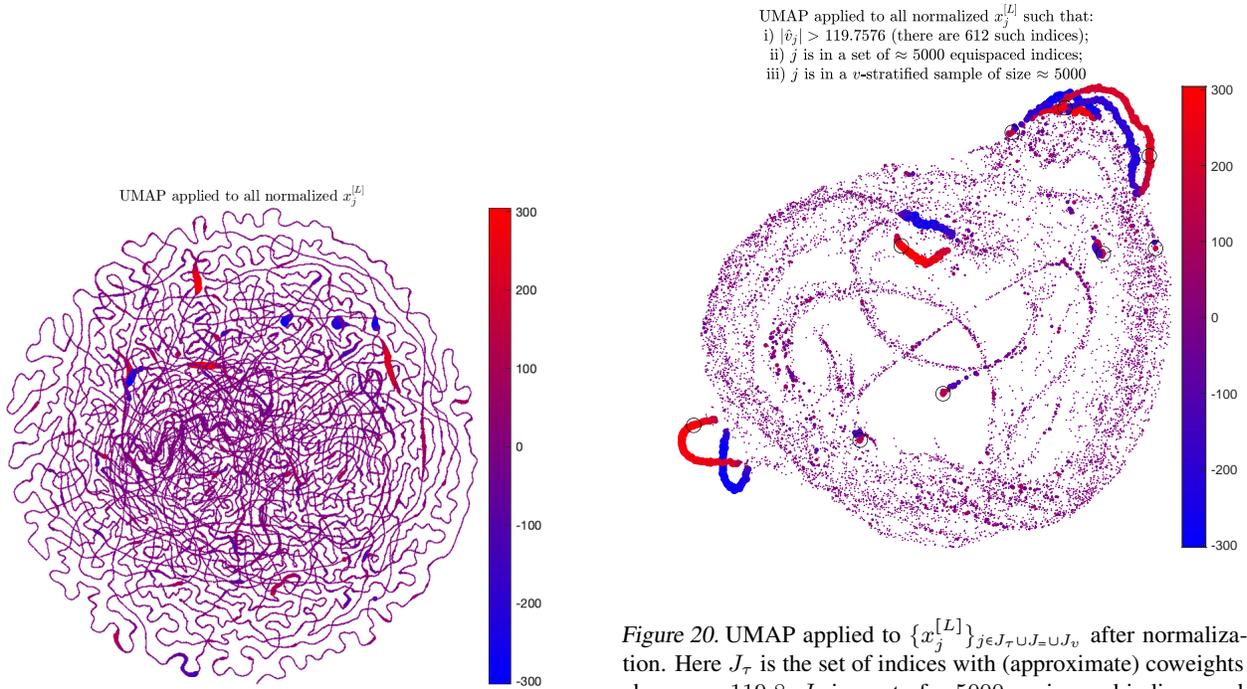}
  \caption{UMAP applied to $\{x^{[L]}_j\}_{j \in J_\tau \cup J_= \cup J_v}$ after normalization. Here $J_\tau$ is the set of indices with (approximate) coweights above $\tau \approx 119.8$, $J_=$ is a set of $\approx 5000$ equispaced indices, and $J_v$ is a set of $\approx 5000$ indices corresponding to an approximately uniformly stratified random sample according to the coweighting. Here the ``boundaryness'' of points with large (approximate; in absolute value) coweighting is apparent. Black circles indicate the nine points depicted in the right panel of Figure \ref{fig:SortOfErosion}.}
  \label{fig:UMAPstructured}
\end{figure}

\begin{figure}[h]
  \centering
  \includegraphics[trim = 30mm 100mm 20mm 100mm, clip, width=.49\textwidth,keepaspectratio]{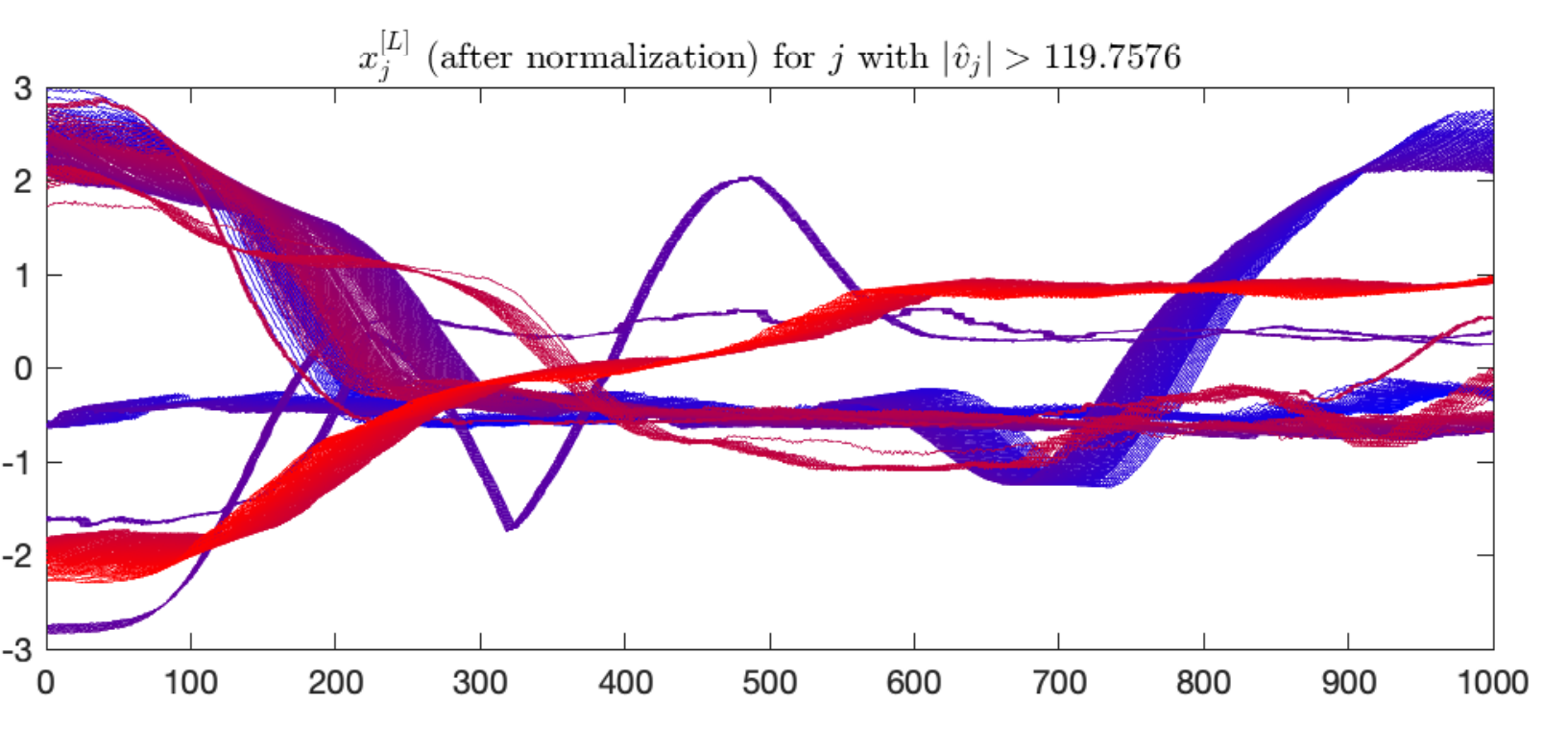}
  \includegraphics[trim = 30mm 100mm 20mm 100mm, clip, width=.49\textwidth,keepaspectratio]{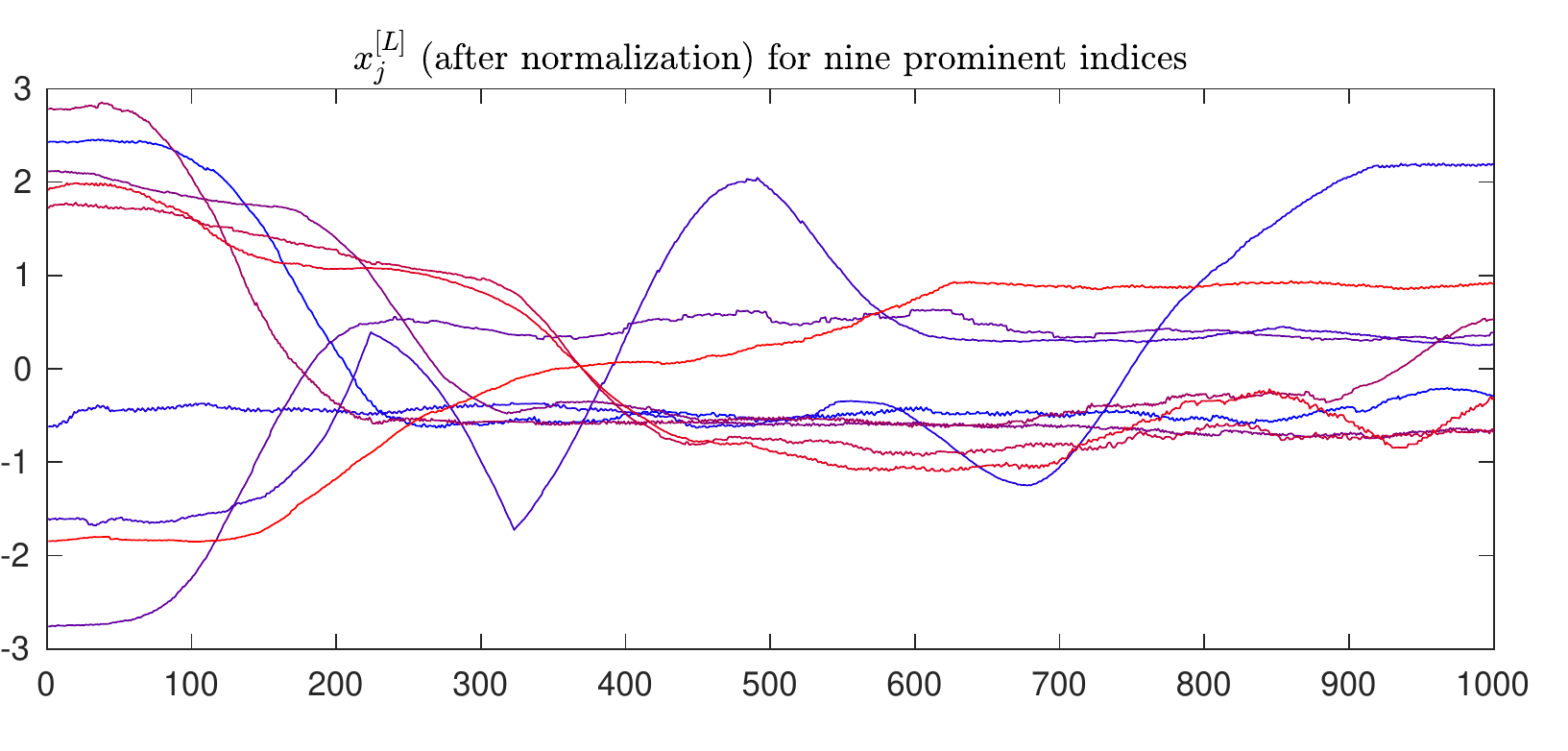}
  \caption{(L) Plots of normalized $x^{[L]}_j$ for $j$ such that $|\hat v_j| > \tau$. (R) Plots of the nine $x^{[L]}_j$ resulting from an ``erosion-like'' procedure that repeatedly discards points with negative (co)weighting components. These remaining points evidently capture the bulk behavior of the larger set in the left-hand panel.}
  \label{fig:SortOfErosion}
\end{figure}

\section{\label{sec:efficientWeighting}Efficiently computing weightings}

Besides a clever use of inclusion/exclusion to accelerate weighting computations in the cases where a few points are added and/or perturbed (for which see \S \ref{sec:DifferentialMagnitude} and \S 2 of \cite{bunch2020practical}), we briefly discuss methods for efficiently solving the weighting equation $Zw = 1$.

For the usual metric on $\mathbb{R}^D$, the similarity matrix $Z$ is positive definite, so we get a reproducing kernel Hilbert space to which Mercer's theorem applies. In practice this amounts to the existence of a Cholesky decomposition $Z = LL^T$ and a positive square root.

However, the ideal method to solve $Zw = 1$ at scale is to use efficient solvers that exploit kernel structure in a more detailed way. As \cite{rebrova2018study} points out, ``kernel matrices are good candidates for hierarchical low-rank solvers'' like \cite{rouet2016distributed,chenhan2017n} that offer subquadratic performance.
\footnote{
In particular, see \url{https://padas.oden.utexas.edu/libaskit/}.
}

Still, using a special-purpose solver may not be worthwhile at intermediate scales. As an alternative, the system $Zw = 1$ can be solved without explicitly forming the entire matrix $Z$ via the the (preconditioned) conjugate gradient method, which has computational complexity $O(\text{nnz} \cdot \sqrt{\kappa})$, where $\text{nnz}$ and $\kappa$ respectively denote the number of nonzero entries and the condition number of the matrix \cite{shewchuk1994introduction}. Moreover, we can improve on this complexity with a good initial guess for a weighting: this is likely to be of particular relevance for iterative algorithms.

%
%

%

The conjugate gradient approach is not really worthwhile unless we can set small entries of $Z$ to zero (so that $\text{nnz} \ll n^2$) while maintaining positive definiteness. Although ``hard thresholding'' generic positive definite matrices in this way generally does not preserve the property of positive definiteness \cite{guillot2012retaining}, it does for sufficiently large $t$ (e.g., above the minimal value that ensures a weighting is positive; cf. Example 6.3.27 of \cite{leinster2021entropy}), and probably more generally.

In any event, we can hard-threshold $Z$ and efficiently solve the resulting system, exploiting positive definiteness if we have it by applying the conjugate gradient method. One way to produce such a sparse version of $Z$ \emph{de novo} without forming the full matrix is to follow the approach of \S \ref{sec:HighDimTimeSeries} and only compute entries for nearby points. Ultimately, hard thresholding requires an efficient similiarity search \cite{zezula2006similarity,leskovec2020mining} to achieve subquadratic runtimes. Fortunately, a wide variety of nearest neighbor techniques are available for Euclidean point clouds \cite{datar2004locality,andoni2018approximate,li2020approximate}. 
\footnote{
Outside the setting of Euclidean space and at general scales, we are (probably) reduced to using a generic sparse solver on a similarity matrix with nonzero entries obtained by approximate similarity searches based on locality sensitive (or preserving) hashing. For example, Jaccard distance and its variants lead to the efficient MinHash family of algorithms \cite{shrivastava2014defense,wu2020review}.
}

\clearpage

\section{\label{sec:iohAnalyzer}\texttt{IOHAnalyzer} plots}

Figures \ref{fig:IOHanalyzer/EE_mul25par32lamLin/EE_mul25par32lamLin_f03d20}-\ref{fig:IOHanalyzer/EE_mul25par32lamLin/EE_mul25par32lamLin_f19d40vScipyECDF} were produced by \texttt{IOHanalyzer} \cite{doerr2018iohprofiler} at \url{https://iohanalyzer.liacs.nl/} to supplement Figures \ref{fig:ECDFsingleOne20}-\ref{fig:ECDFsingleOne40last} (see footnote \ref{foot:plots}). {\bf NB. In these plots, \texttt{EXPLO2} is indicated by \texttt{EE}.}


\begin{figure}[h]
  \centering
  \includegraphics[trim = 0mm 0mm 0mm 0mm, clip, width=\columnwidth,keepaspectratio]{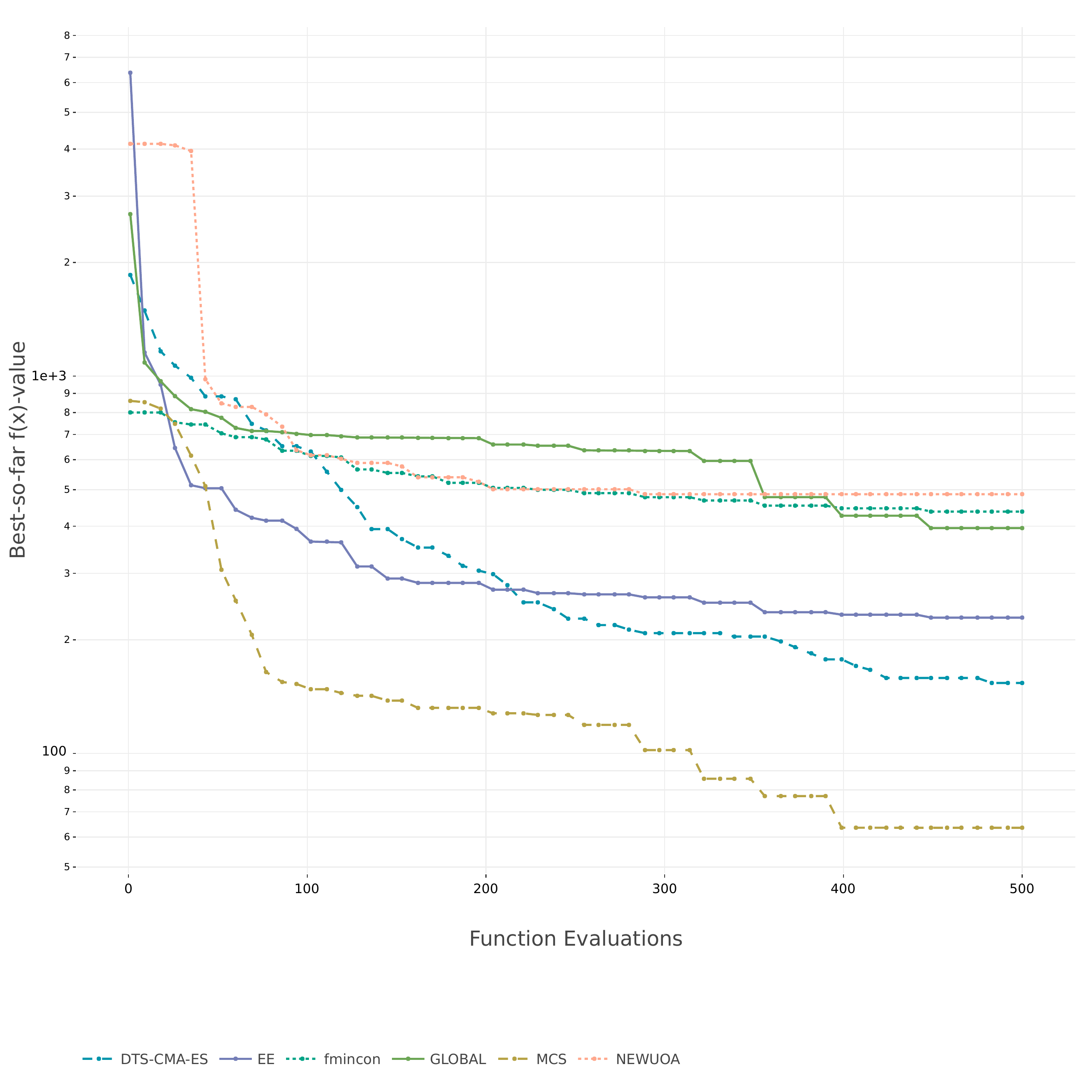}
  \caption{$f_{3}$, $D = 20$. Here \texttt{EXPLO2} is indicated by \texttt{EE}.}
  \label{fig:IOHanalyzer/EE_mul25par32lamLin/EE_mul25par32lamLin_f03d20}
\end{figure}

\begin{figure}[h]
  \centering
  \includegraphics[trim = 0mm 0mm 0mm 0mm, clip, width=\columnwidth,keepaspectratio]{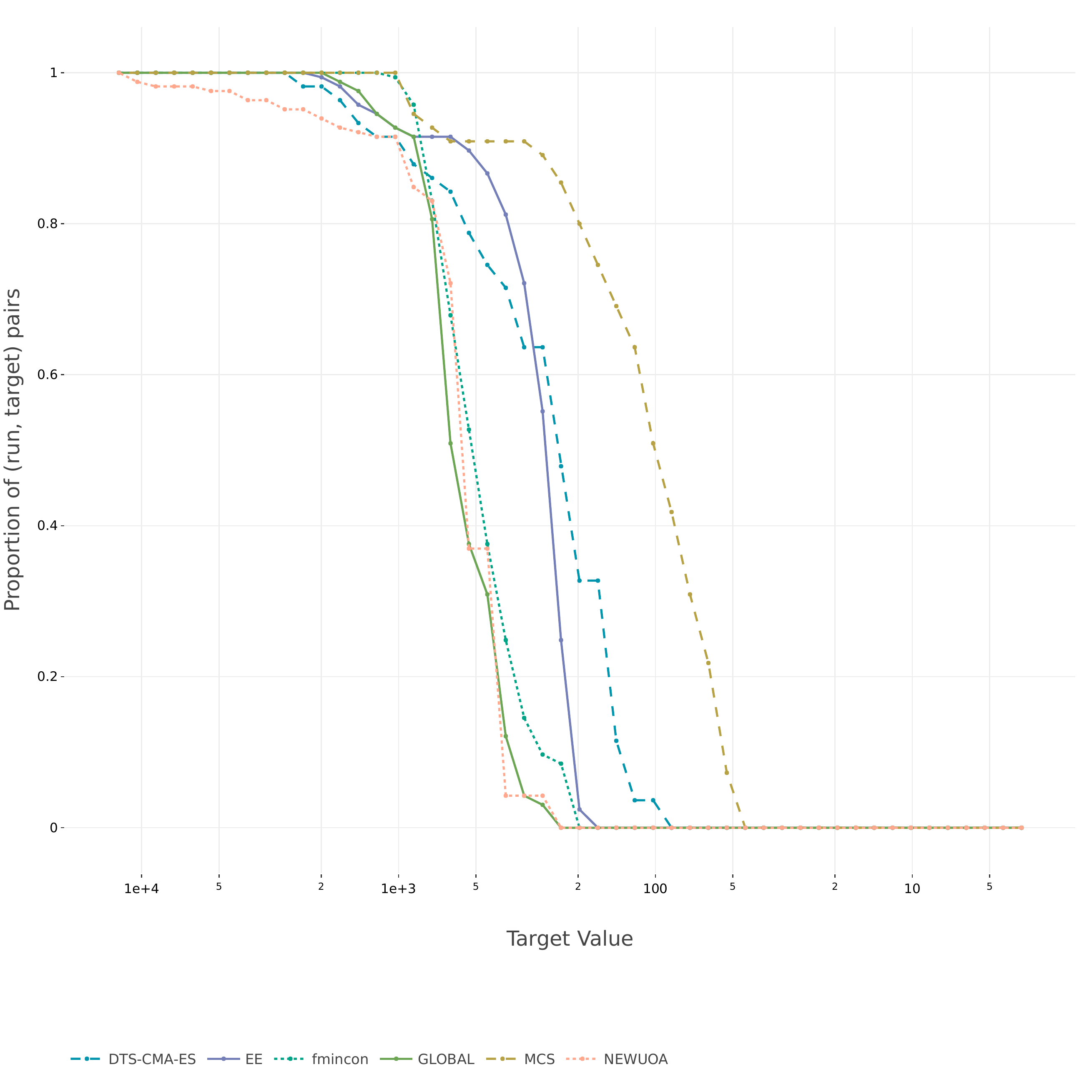}
  \caption{$f_{3}$, $D = 20$. Here \texttt{EXPLO2} is indicated by \texttt{EE}.}
  \label{fig:IOHanalyzer/EE_mul25par32lamLin/EE_mul25par32lamLin_f03d20ECDF}
\end{figure}

\begin{figure}[h]
  \centering
  \includegraphics[trim = 0mm 0mm 0mm 0mm, clip, width=\columnwidth,keepaspectratio]{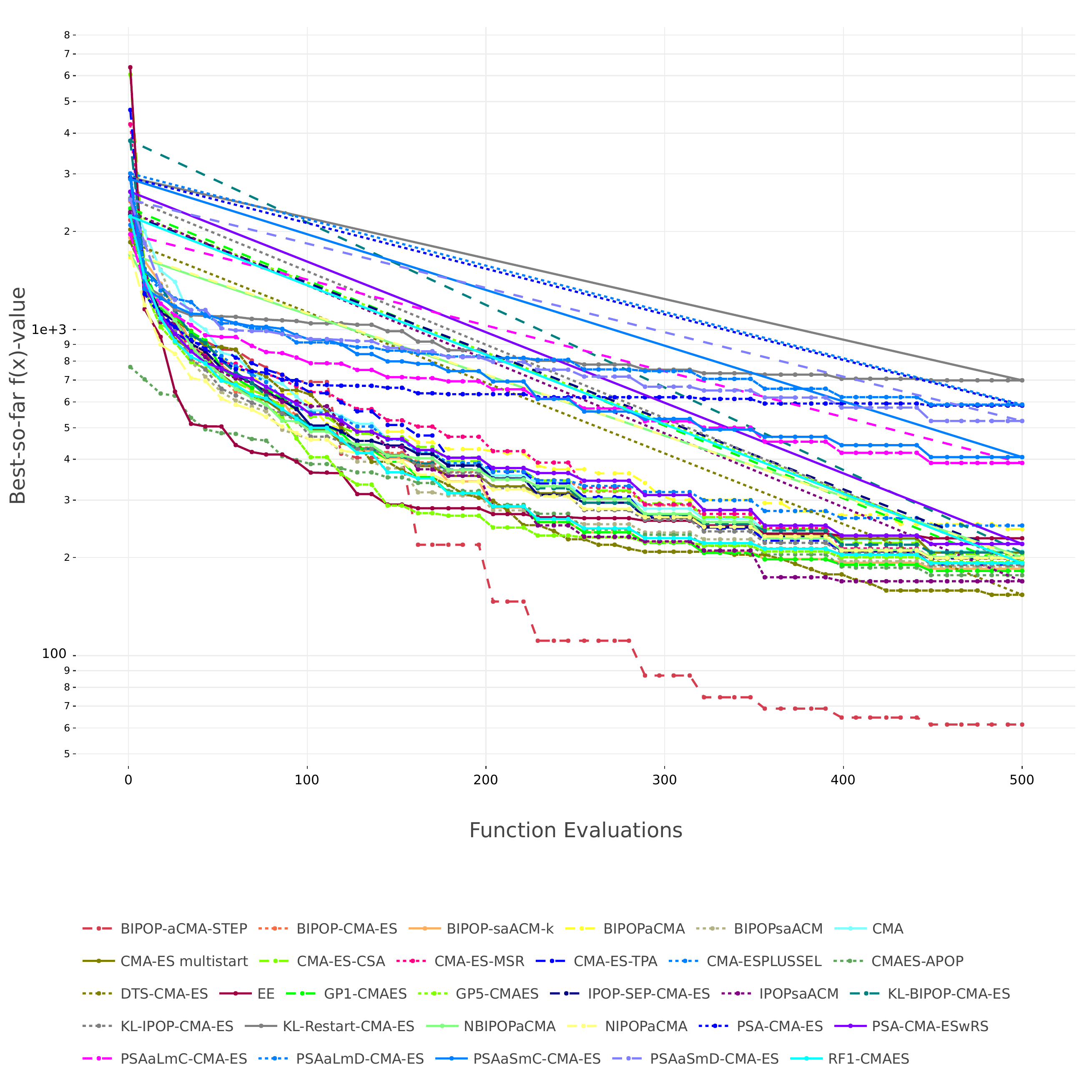}
  \caption{$f_{3}$, $D = 20$. Here \texttt{EXPLO2} is indicated by \texttt{EE}.}
  \label{fig:IOHanalyzer/EE_mul25par32lamLin/EE_mul25par32lamLin_f03d20allCMA}
\end{figure}

\begin{figure}[h]
  \centering
  \includegraphics[trim = 0mm 0mm 0mm 0mm, clip, width=\columnwidth,keepaspectratio]{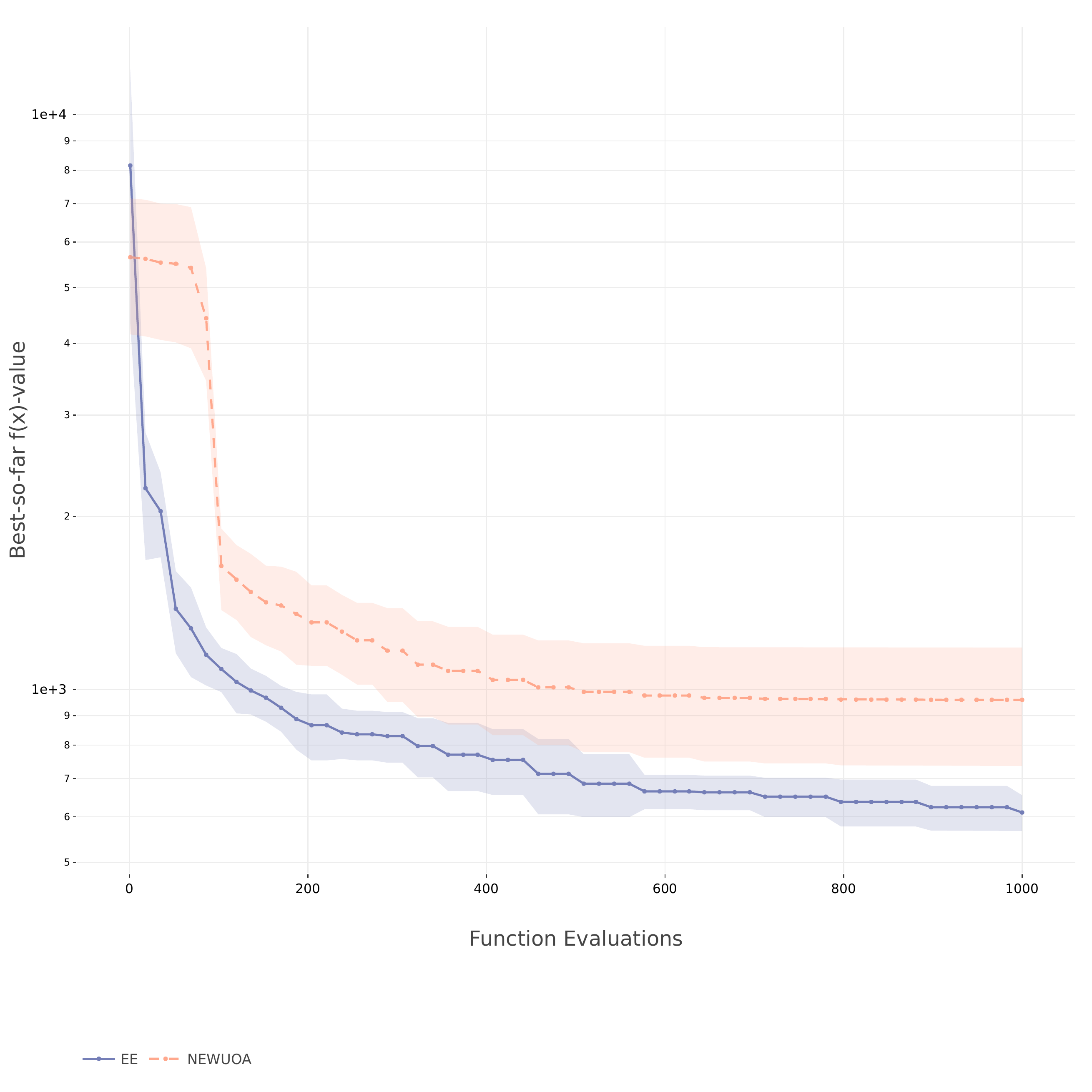}
  \caption{$f_{3}$, $D = 40$. Here \texttt{EXPLO2} is indicated by \texttt{EE}.}
  \label{fig:IOHanalyzer/EE_mul25par32lamLin/EE_mul25par32lamLin_f03d40}
\end{figure}

\begin{figure}[h]
  \centering
  \includegraphics[trim = 0mm 0mm 0mm 0mm, clip, width=\columnwidth,keepaspectratio]{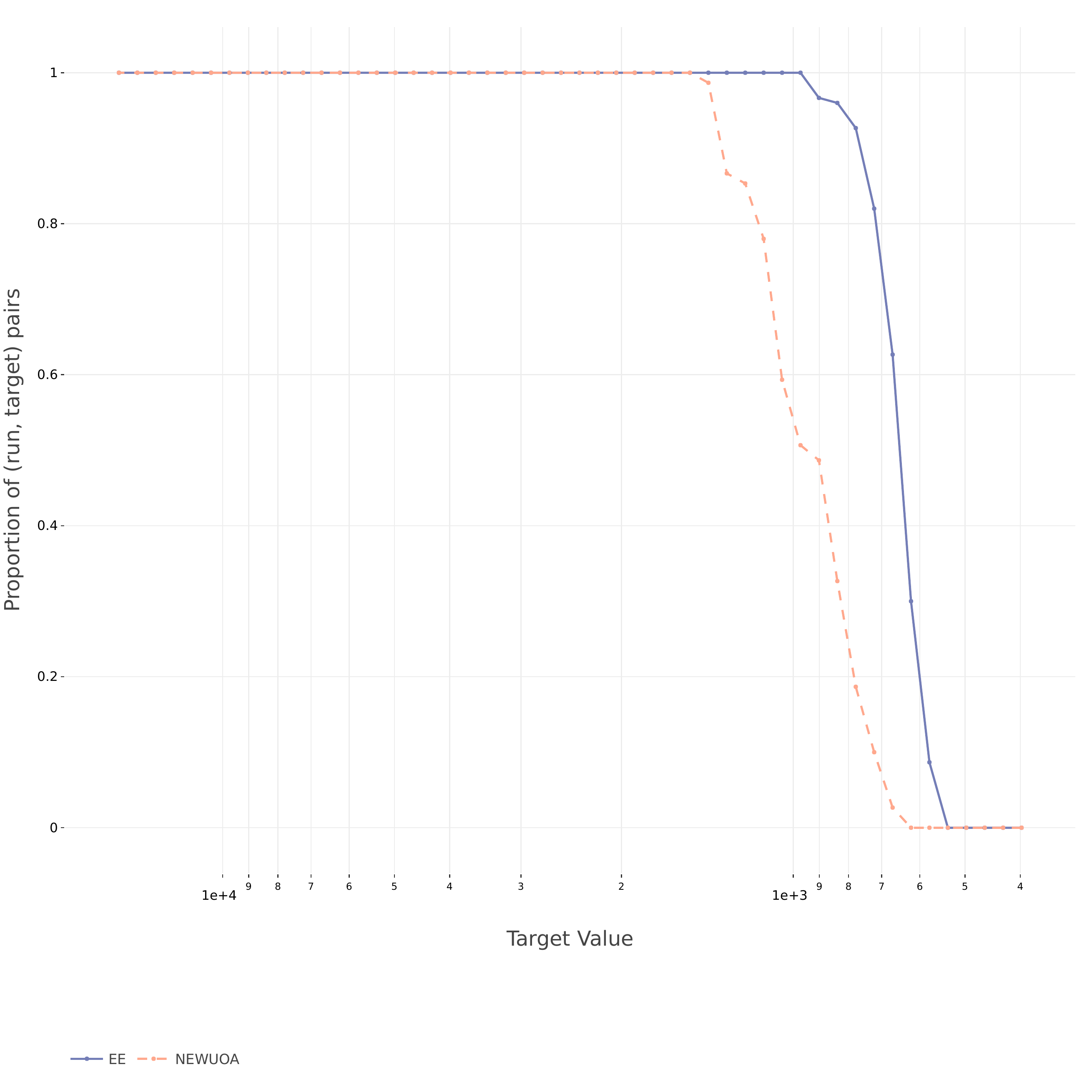}
  \caption{$f_{3}$, $D = 40$. Here \texttt{EXPLO2} is indicated by \texttt{EE}.}
  \label{fig:IOHanalyzer/EE_mul25par32lamLin/EE_mul25par32lamLin_f03d40ECDF}
\end{figure}

\begin{figure}[h]
  \centering
  \includegraphics[trim = 0mm 0mm 0mm 0mm, clip, width=\columnwidth,keepaspectratio]{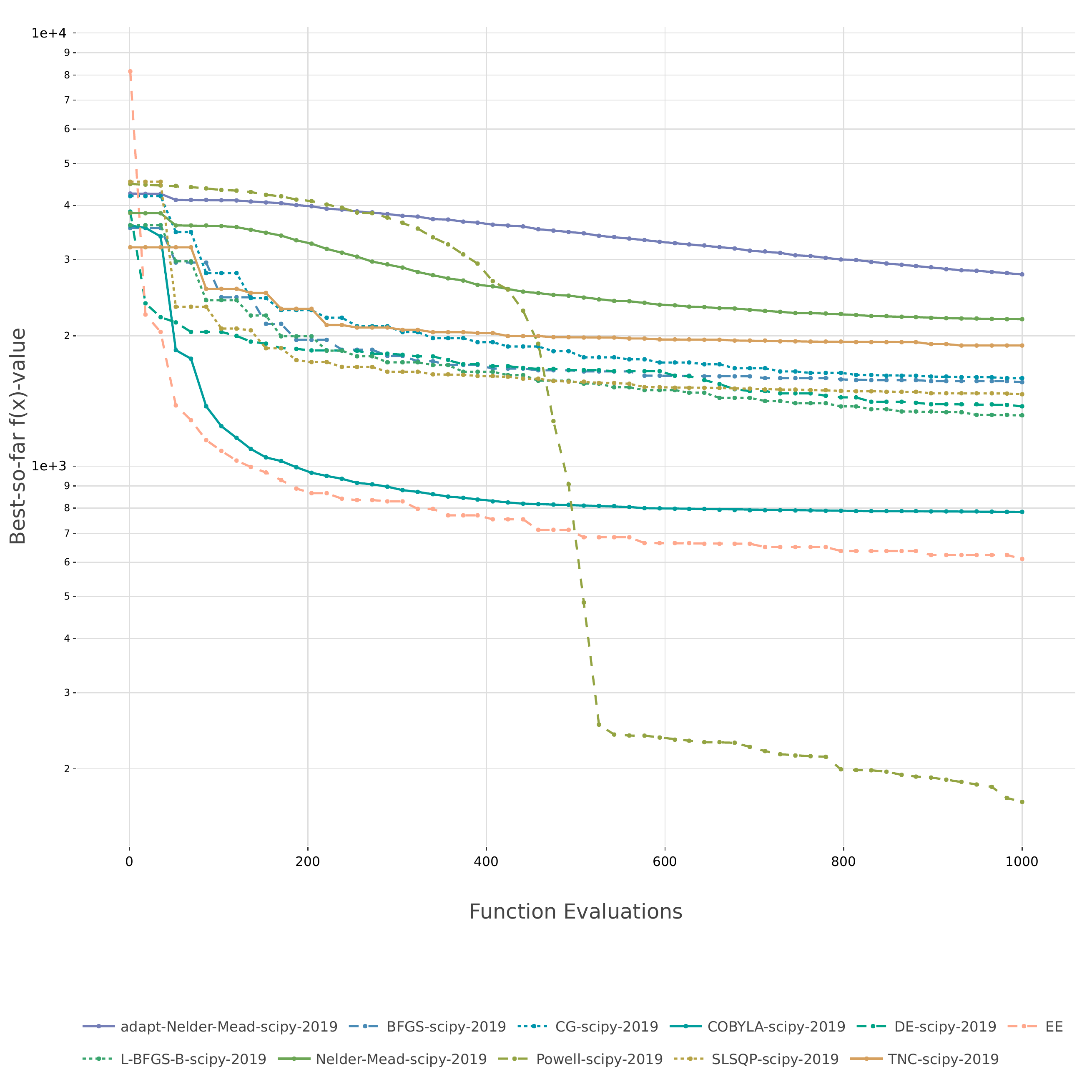}
  \caption{$f_{3}$, $D = 40$. Here \texttt{EXPLO2} is indicated by \texttt{EE}.}
  \label{fig:IOHanalyzer/EE_mul25par32lamLin/EE_mul25par32lamLin_f03d40vScipy}
\end{figure}

\begin{figure}[h]
  \centering
  \includegraphics[trim = 0mm 0mm 0mm 0mm, clip, width=\columnwidth,keepaspectratio]{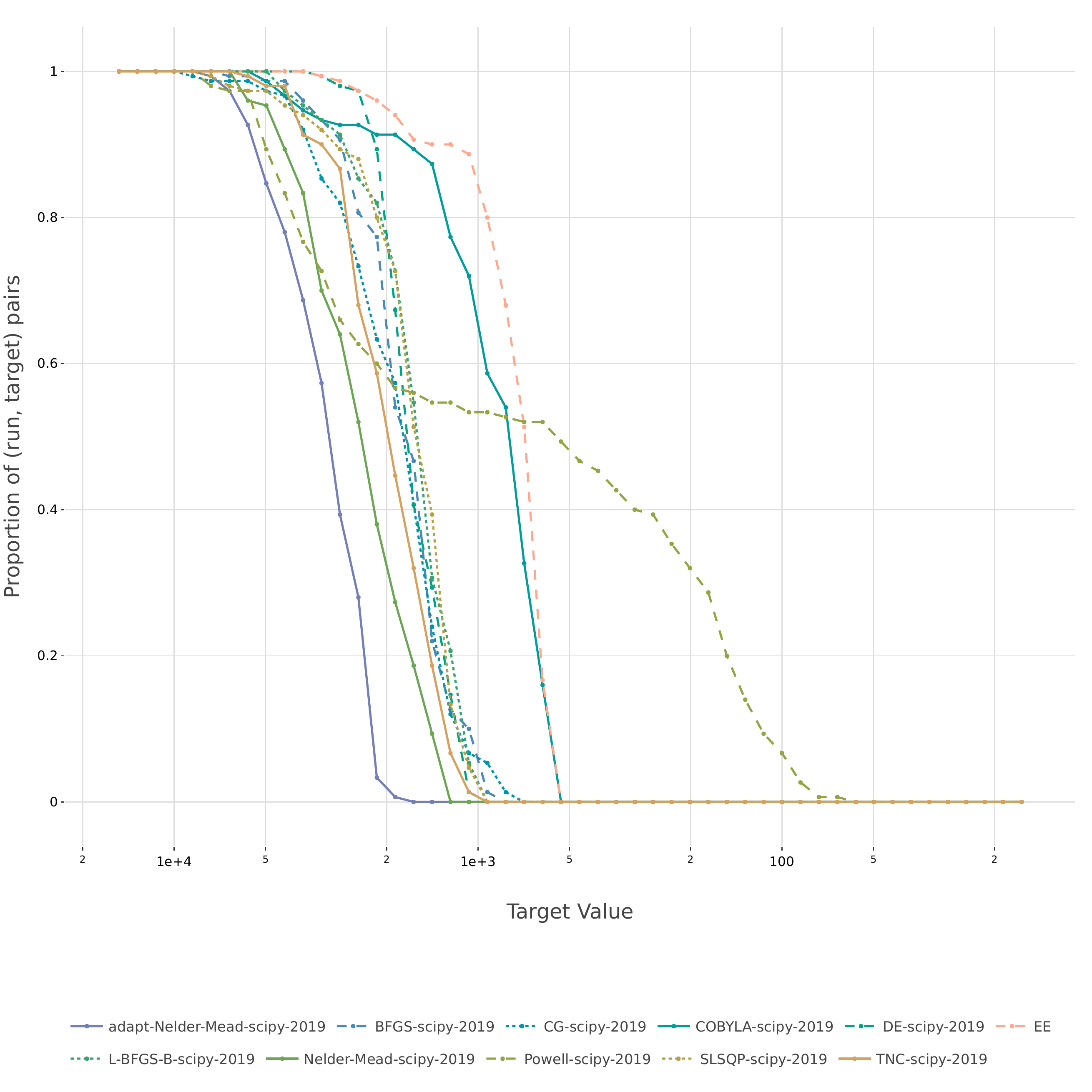}
  \caption{$f_{3}$, $D = 40$. Here \texttt{EXPLO2} is indicated by \texttt{EE}.}
  \label{fig:IOHanalyzer/EE_mul25par32lamLin/EE_mul25par32lamLin_f03d40vScipyECDF}
\end{figure}

\clearpage


\begin{figure}[h]
  \centering
  \includegraphics[trim = 0mm 0mm 0mm 0mm, clip, width=\columnwidth,keepaspectratio]{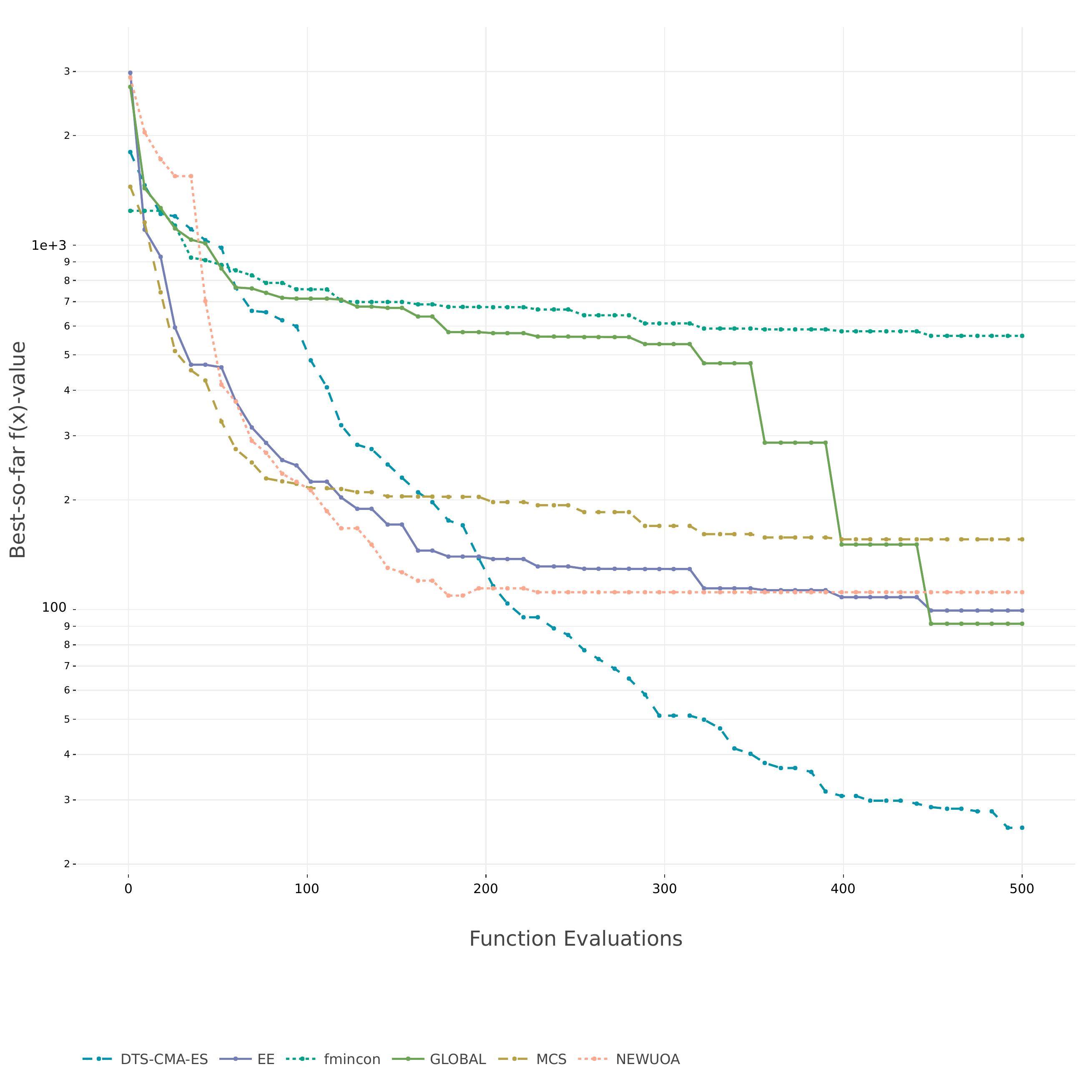}
  \caption{$f_{7}$, $D = 20$. Here \texttt{EXPLO2} is indicated by \texttt{EE}.}
  \label{fig:IOHanalyzer/EE_mul25par32lamLin/EE_mul25par32lamLin_f07d20}
\end{figure}

\begin{figure}[h]
  \centering
  \includegraphics[trim = 0mm 0mm 0mm 0mm, clip, width=\columnwidth,keepaspectratio]{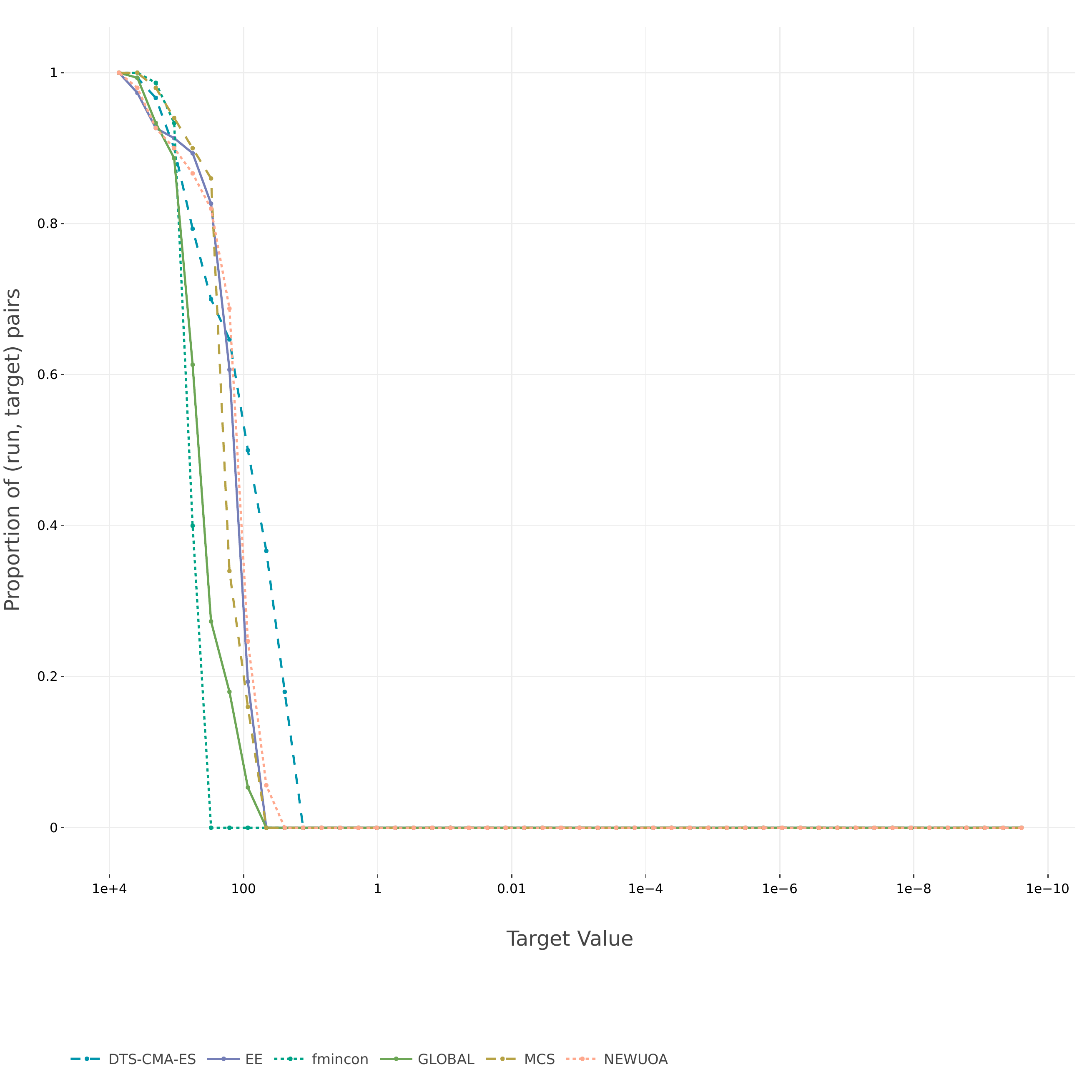}
  \caption{$f_{7}$, $D = 20$. Here \texttt{EXPLO2} is indicated by \texttt{EE}.}
  \label{fig:IOHanalyzer/EE_mul25par32lamLin/EE_mul25par32lamLin_f07d20ECDF}
\end{figure}

\begin{figure}[h]
  \centering
  \includegraphics[trim = 0mm 0mm 0mm 0mm, clip, width=\columnwidth,keepaspectratio]{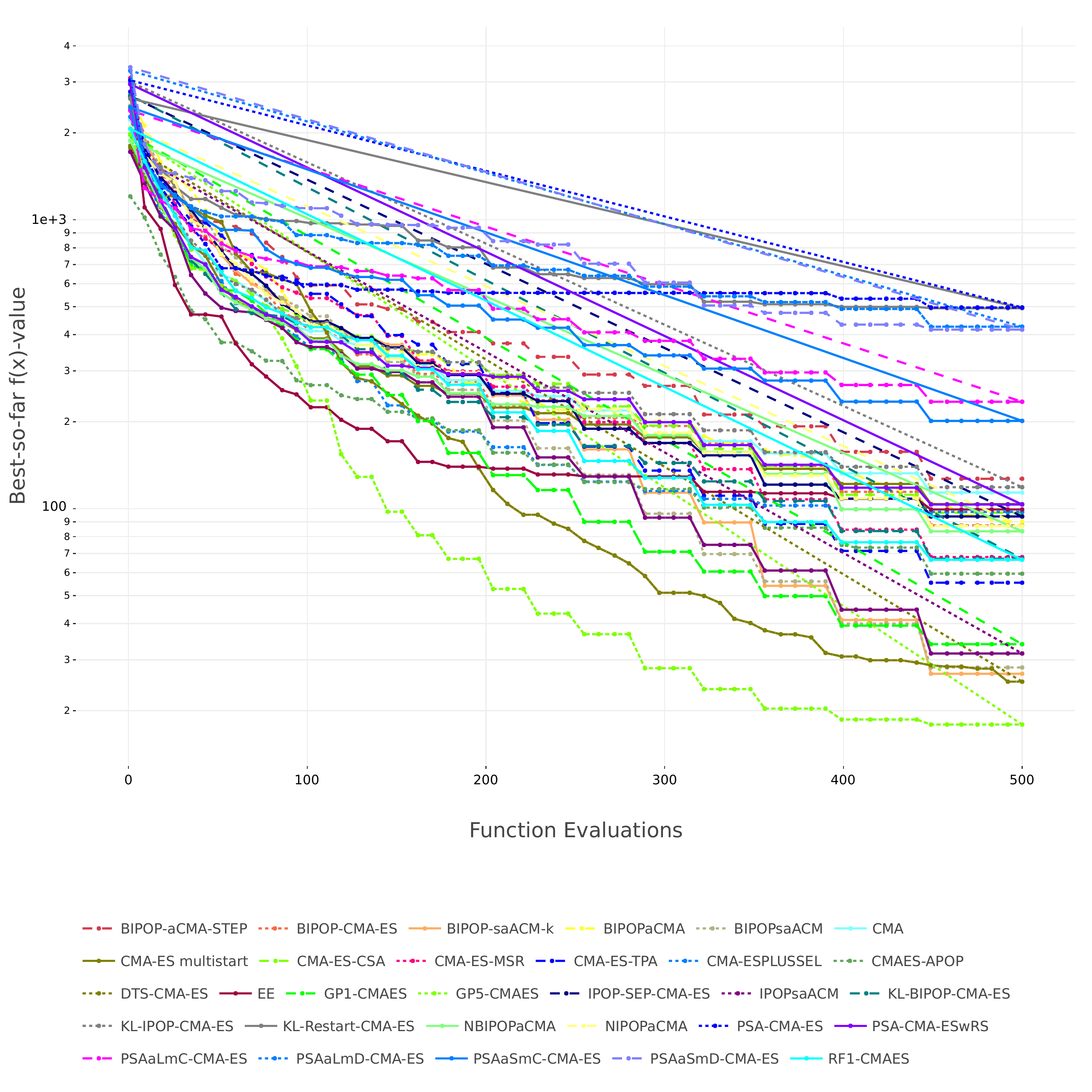}
  \caption{$f_{7}$, $D = 20$. Here \texttt{EXPLO2} is indicated by \texttt{EE}.}
  \label{fig:IOHanalyzer/EE_mul25par32lamLin/EE_mul25par32lamLin_f07d20allCMA}
\end{figure}

\begin{figure}[h]
  \centering
  \includegraphics[trim = 0mm 0mm 0mm 0mm, clip, width=\columnwidth,keepaspectratio]{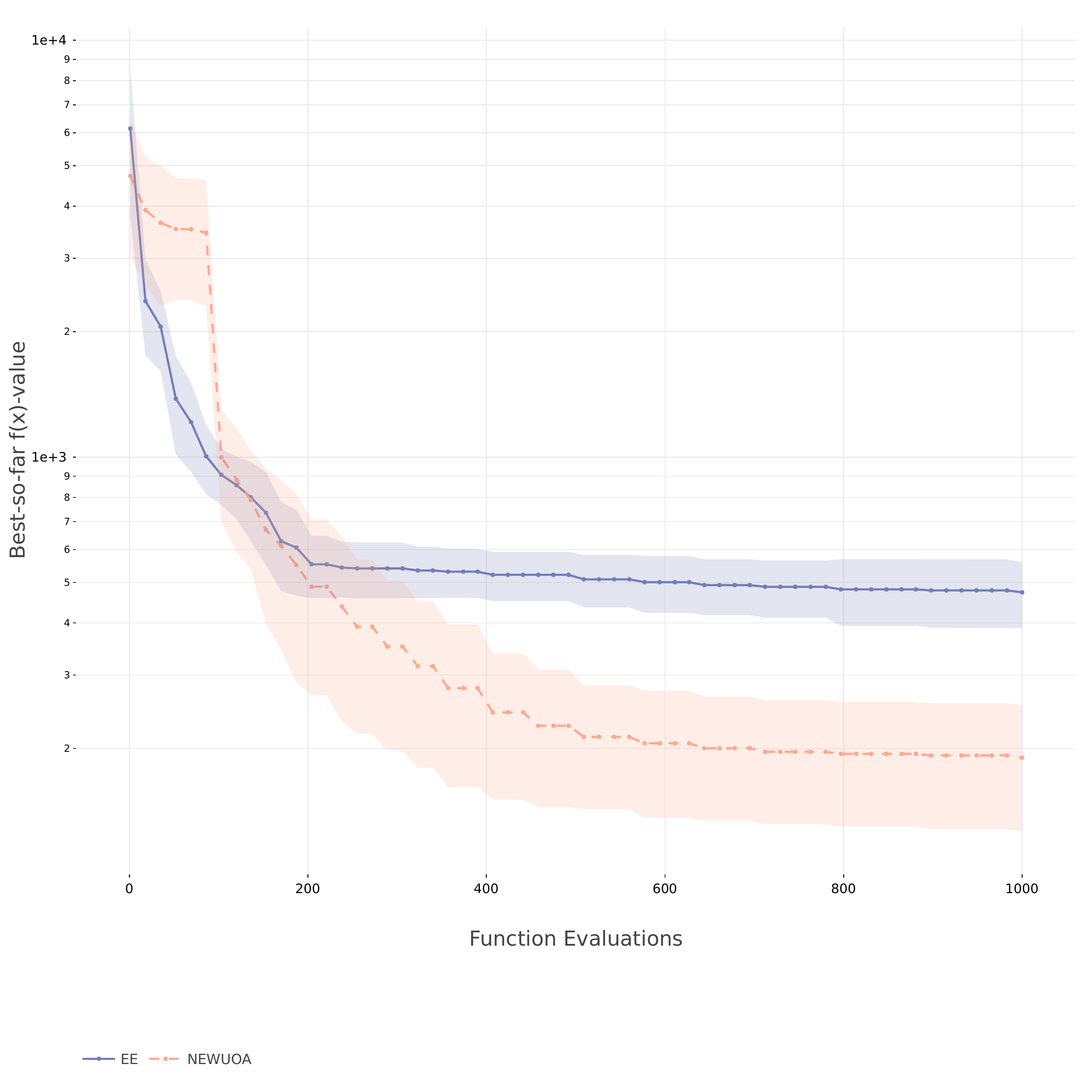}
  \caption{$f_{7}$, $D = 40$. Here \texttt{EXPLO2} is indicated by \texttt{EE}.}
  \label{fig:IOHanalyzer/EE_mul25par32lamLin/EE_mul25par32lamLin_f07d40}
\end{figure}

\begin{figure}[h]
  \centering
  \includegraphics[trim = 0mm 0mm 0mm 0mm, clip, width=\columnwidth,keepaspectratio]{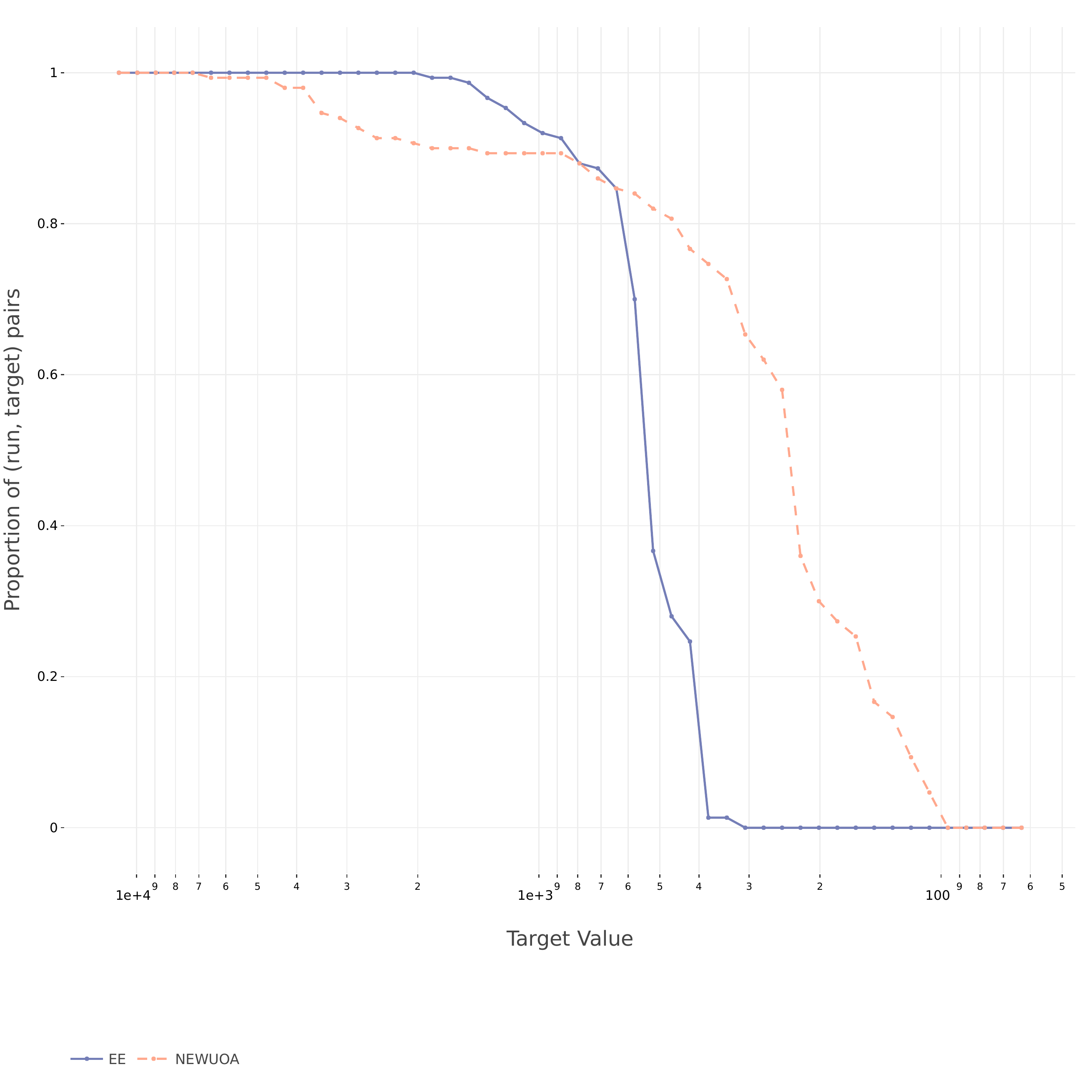}
  \caption{$f_{7}$, $D = 40$. Here \texttt{EXPLO2} is indicated by \texttt{EE}.}
  \label{fig:IOHanalyzer/EE_mul25par32lamLin/EE_mul25par32lamLin_f07d40ECDF}
\end{figure}

\begin{figure}[h]
  \centering
  \includegraphics[trim = 0mm 0mm 0mm 0mm, clip, width=\columnwidth,keepaspectratio]{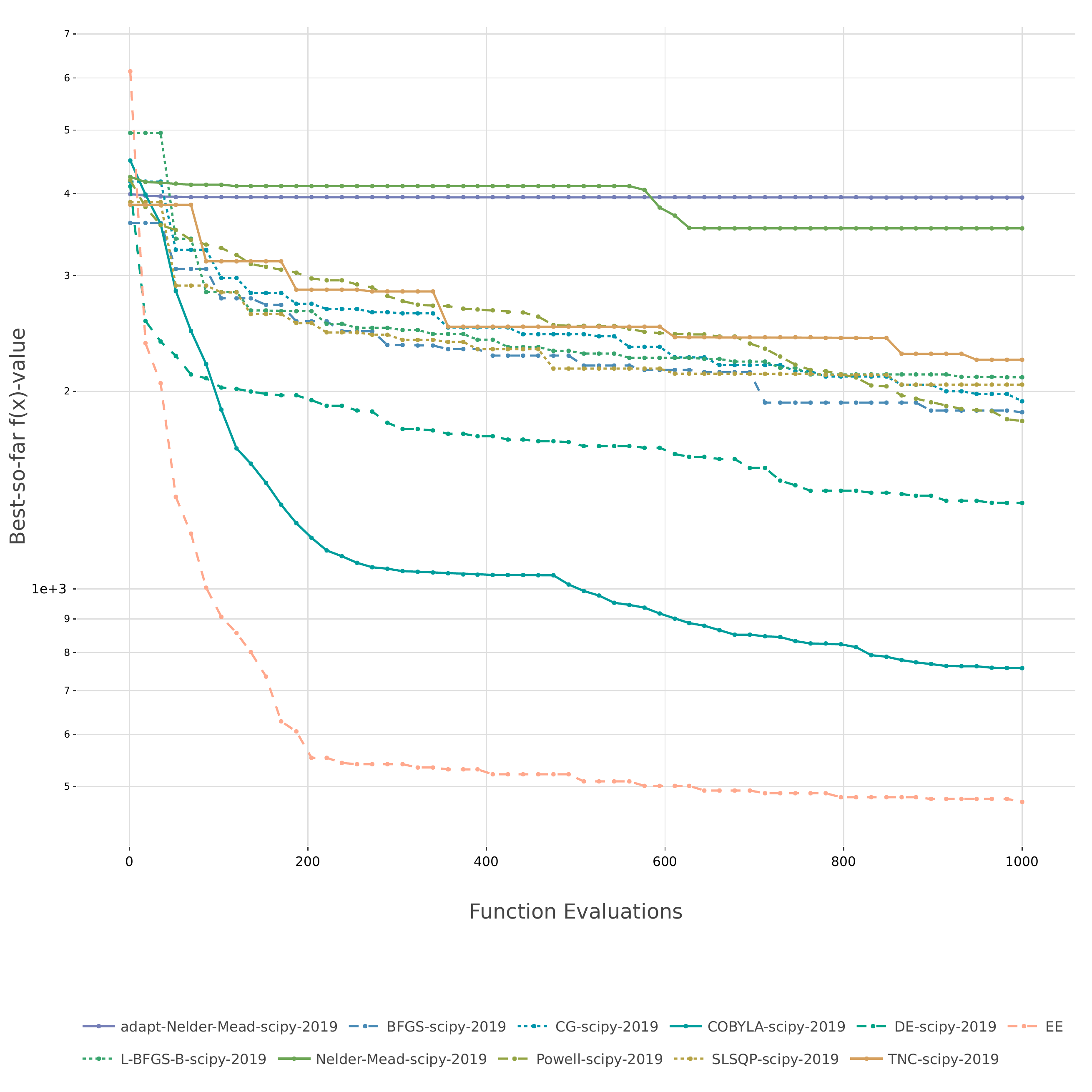}
  \caption{$f_{7}$, $D = 40$. Here \texttt{EXPLO2} is indicated by \texttt{EE}.}
  \label{fig:IOHanalyzer/EE_mul25par32lamLin/EE_mul25par32lamLin_f07d40vScipy}
\end{figure}

\begin{figure}[h]
  \centering
  \includegraphics[trim = 0mm 0mm 0mm 0mm, clip, width=\columnwidth,keepaspectratio]{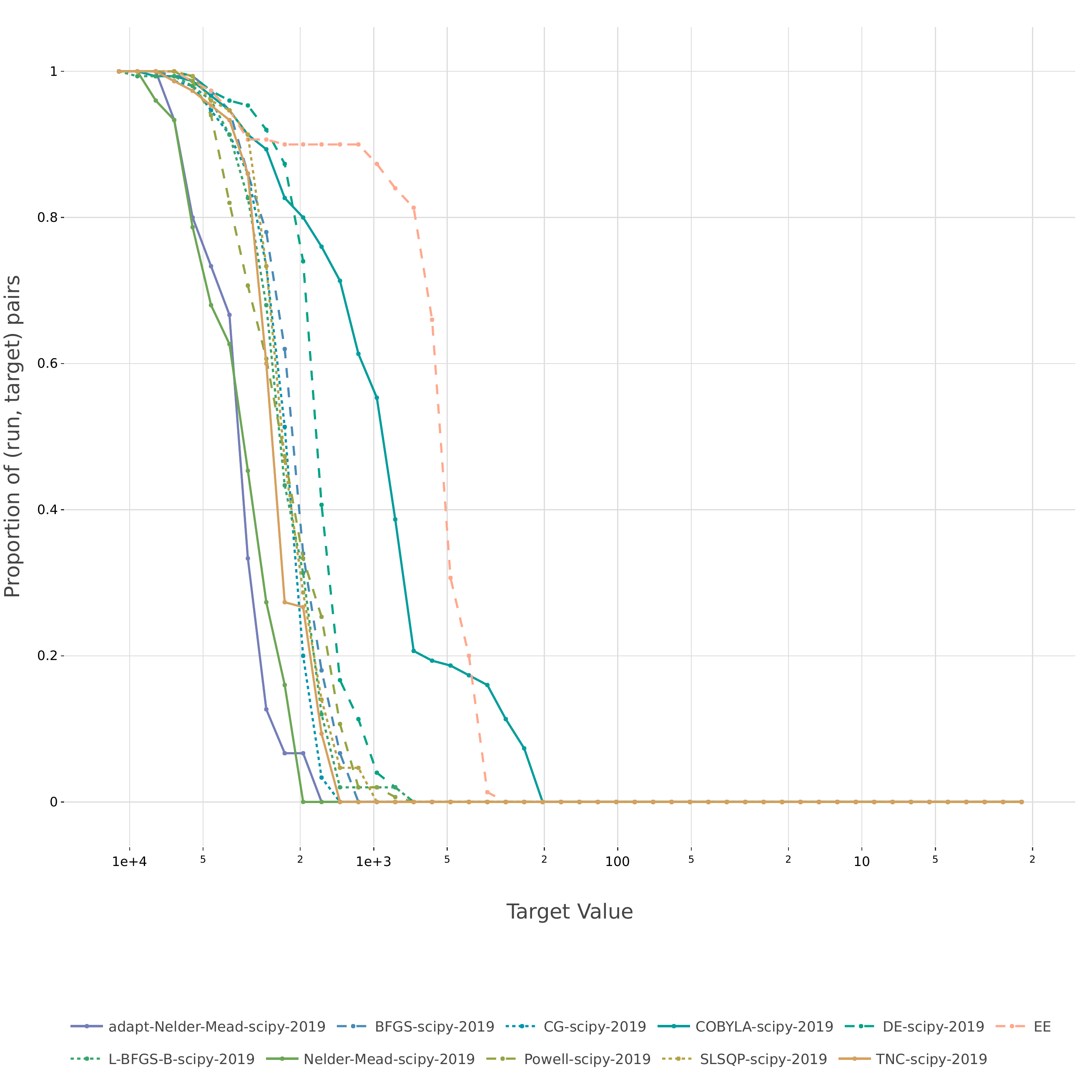}
  \caption{$f_{7}$, $D = 40$. Here \texttt{EXPLO2} is indicated by \texttt{EE}.}
  \label{fig:IOHanalyzer/EE_mul25par32lamLin/EE_mul25par32lamLin_f07d40vScipyECDF}
\end{figure}

\clearpage


\begin{figure}[h]
  \centering
  \includegraphics[trim = 0mm 0mm 0mm 0mm, clip, width=\columnwidth,keepaspectratio]{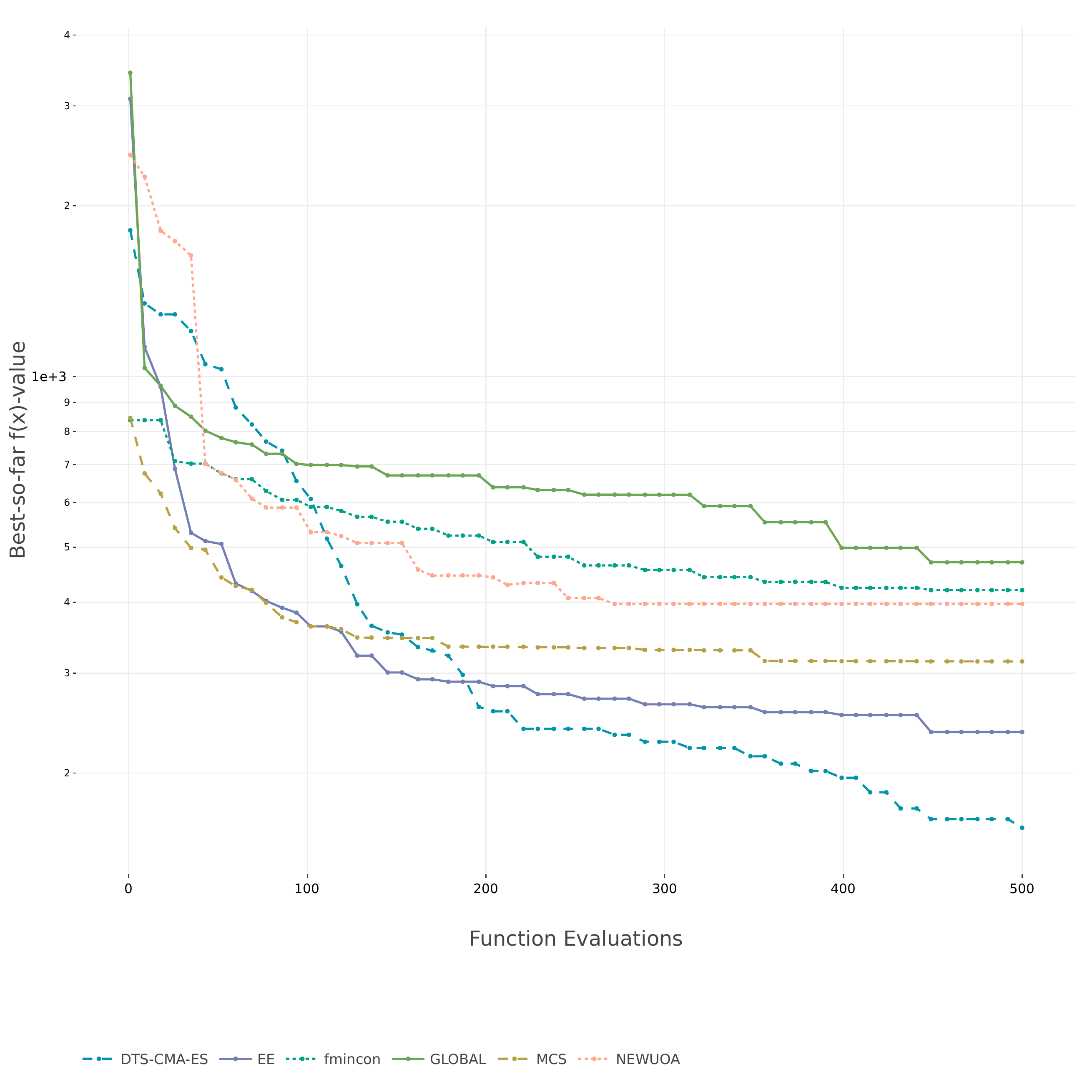}
  \caption{$f_{15}$, $D = 20$. Here \texttt{EXPLO2} is indicated by \texttt{EE}.}
  \label{fig:IOHanalyzer/EE_mul25par32lamLin/EE_mul25par32lamLin_f15d20}
\end{figure}

\begin{figure}[h]
  \centering
  \includegraphics[trim = 0mm 0mm 0mm 0mm, clip, width=\columnwidth,keepaspectratio]{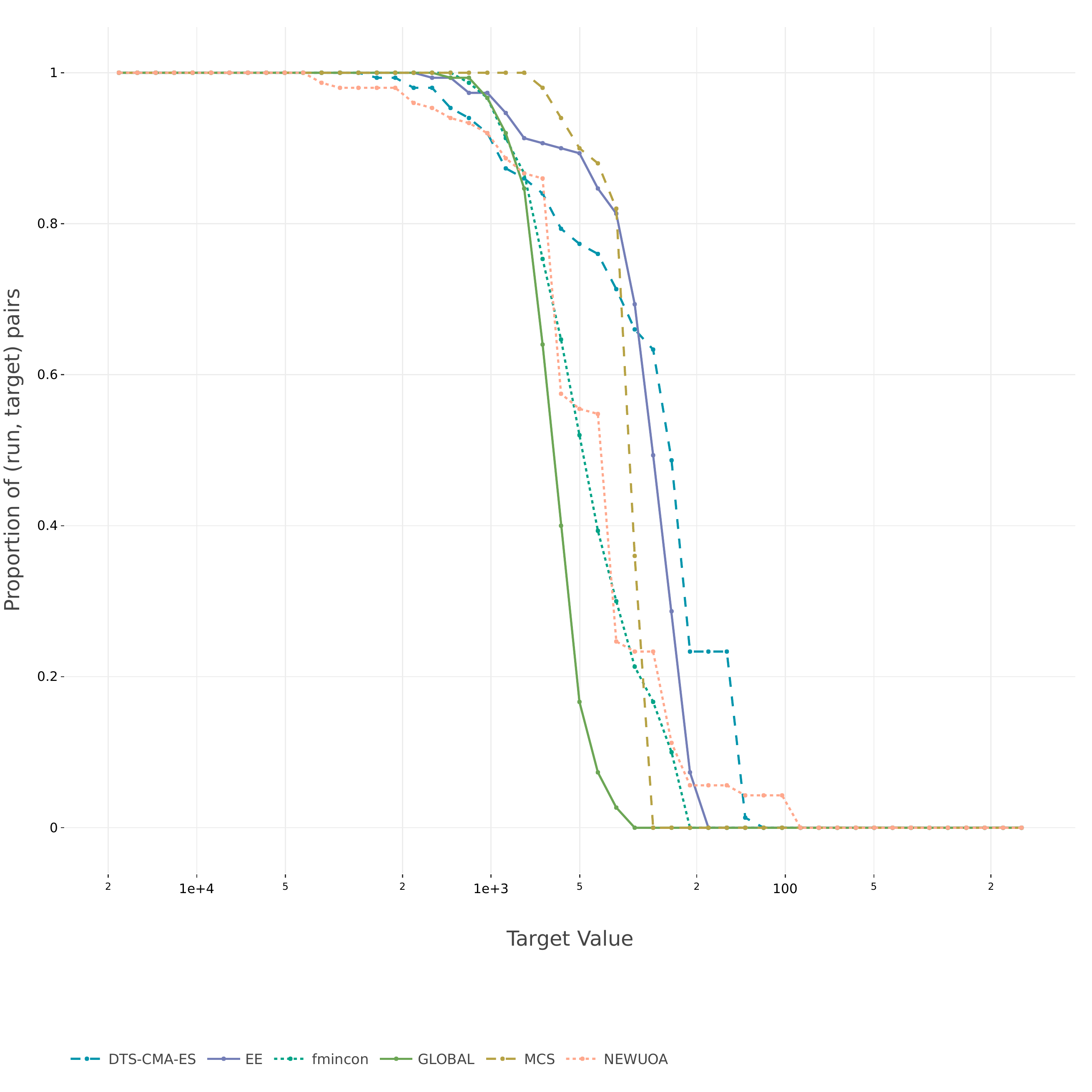}
  \caption{$f_{15}$, $D = 20$. Here \texttt{EXPLO2} is indicated by \texttt{EE}.}
  \label{fig:IOHanalyzer/EE_mul25par32lamLin/EE_mul25par32lamLin_f15d20ECDF}
\end{figure}

\begin{figure}[h]
  \centering
  \includegraphics[trim = 0mm 0mm 0mm 0mm, clip, width=\columnwidth,keepaspectratio]{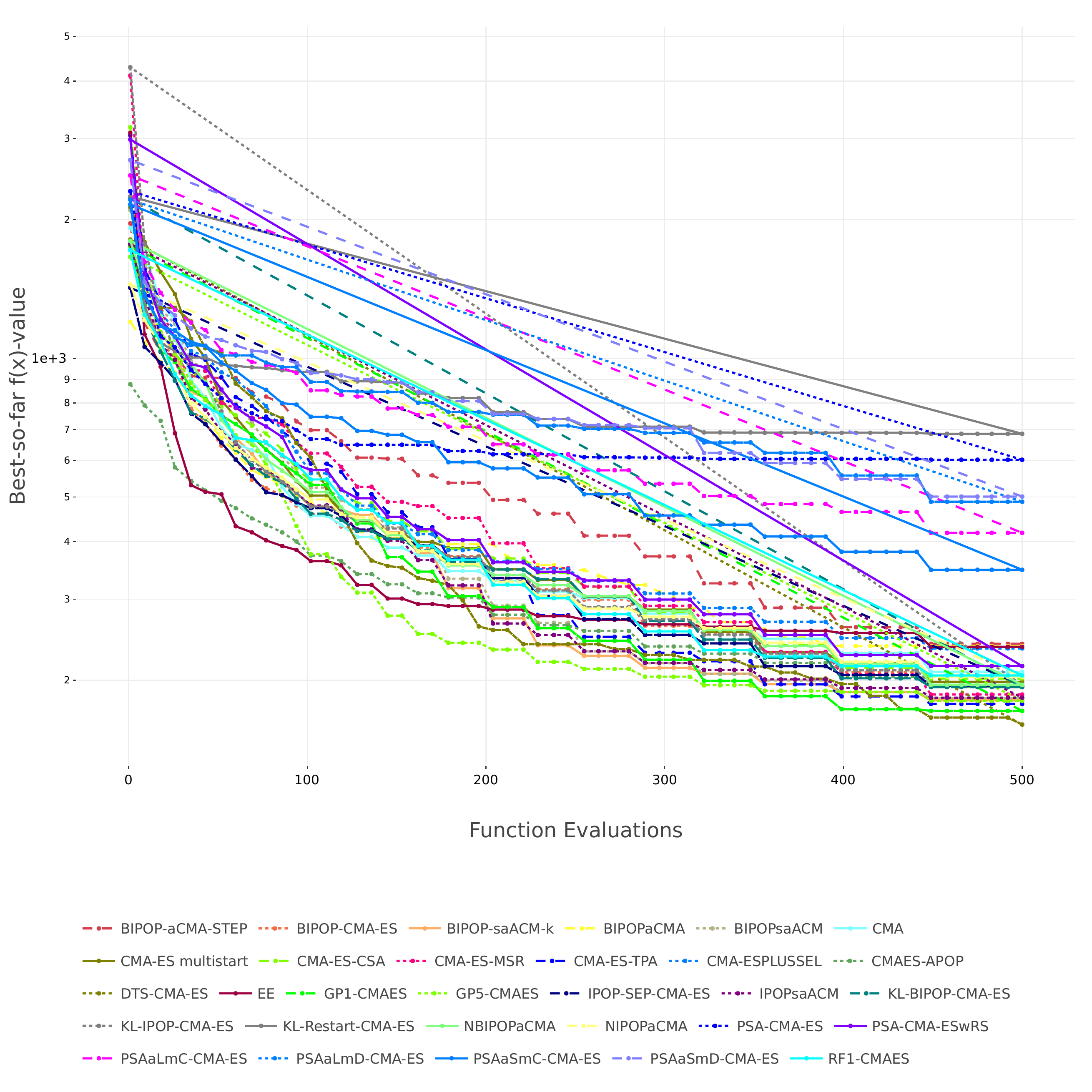}
  \caption{$f_{15}$, $D = 20$. Here \texttt{EXPLO2} is indicated by \texttt{EE}.}
  \label{fig:IOHanalyzer/EE_mul25par32lamLin/EE_mul25par32lamLin_f15d20allCMA}
\end{figure}

\begin{figure}[h]
  \centering
  \includegraphics[trim = 0mm 0mm 0mm 0mm, clip, width=\columnwidth,keepaspectratio]{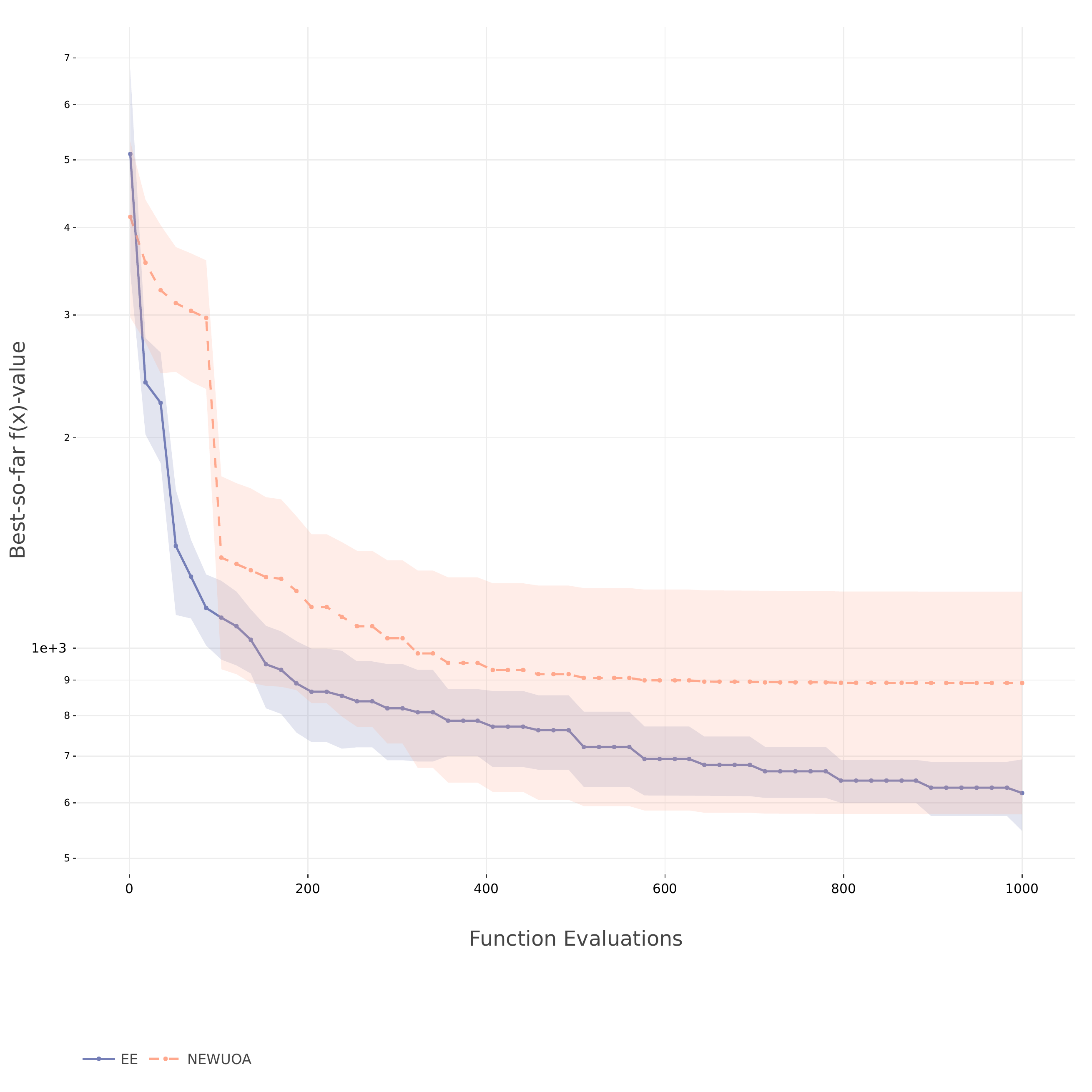}
  \caption{$f_{15}$, $D = 40$. Here \texttt{EXPLO2} is indicated by \texttt{EE}.}
  \label{fig:IOHanalyzer/EE_mul25par32lamLin/EE_mul25par32lamLin_f15d40}
\end{figure}

\begin{figure}[h]
  \centering
  \includegraphics[trim = 0mm 0mm 0mm 0mm, clip, width=\columnwidth,keepaspectratio]{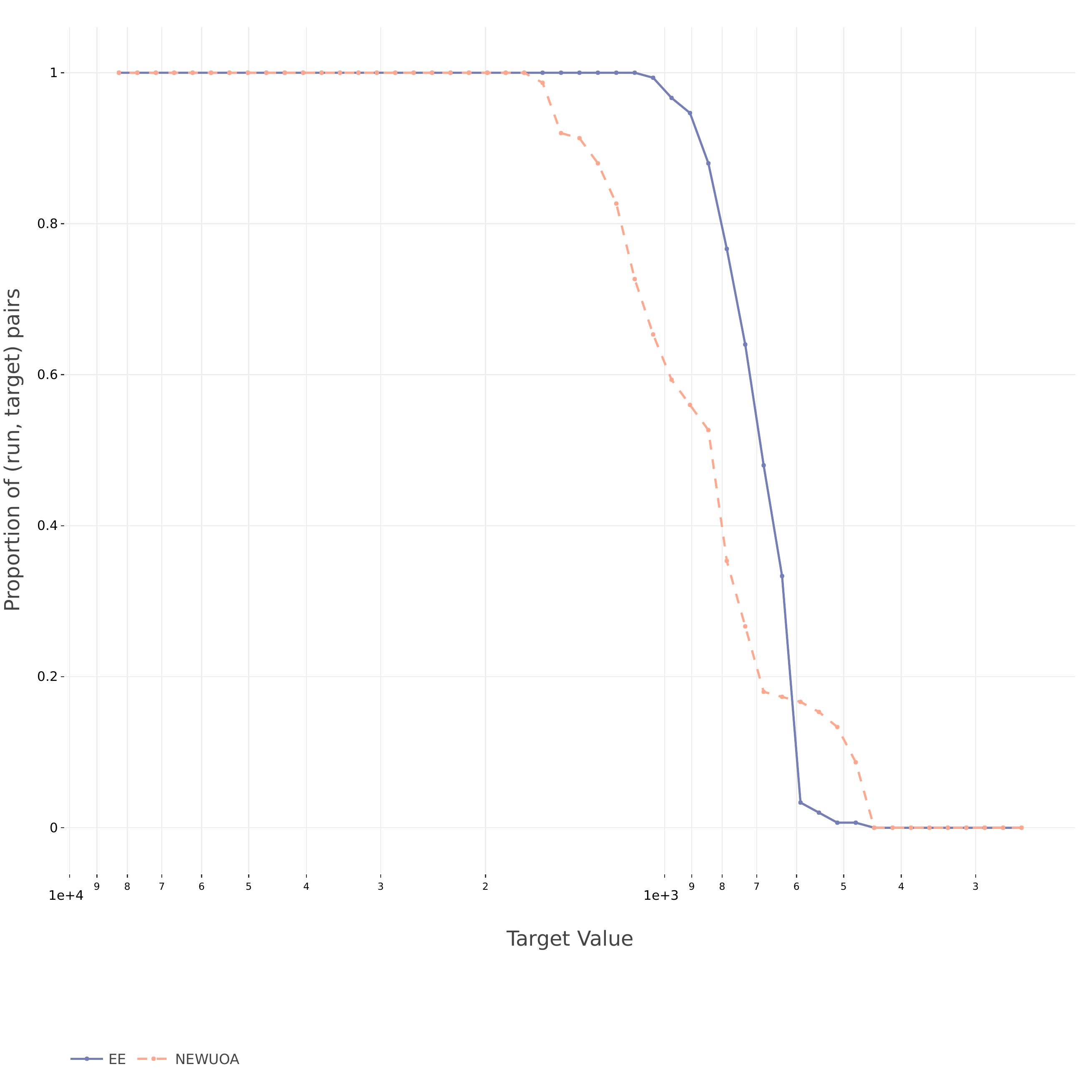}
  \caption{$f_{15}$, $D = 40$. Here \texttt{EXPLO2} is indicated by \texttt{EE}.}
  \label{fig:IOHanalyzer/EE_mul25par32lamLin/EE_mul25par32lamLin_f15d40ECDF}
\end{figure}

\begin{figure}[h]
  \centering
  \includegraphics[trim = 0mm 0mm 0mm 0mm, clip, width=\columnwidth,keepaspectratio]{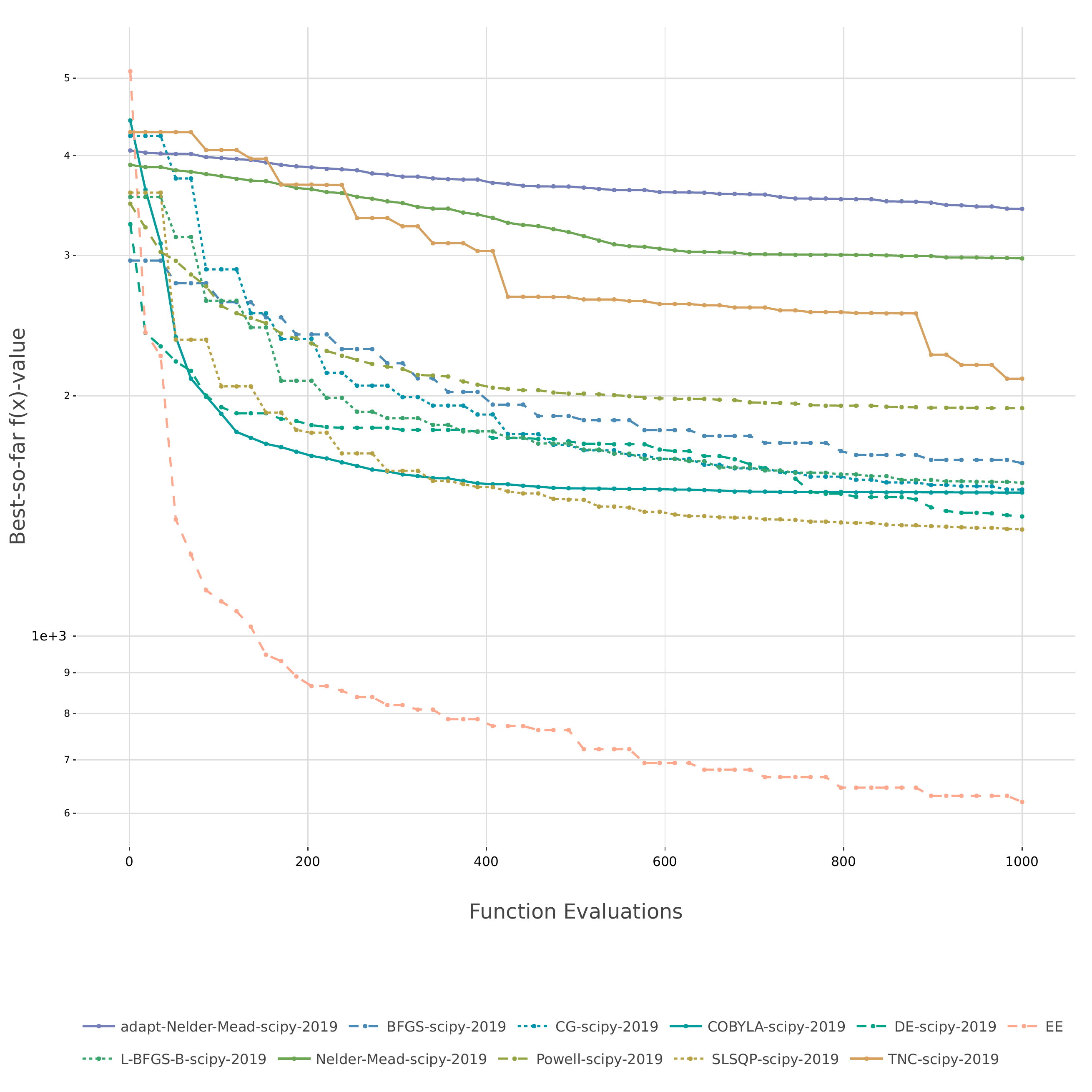}
  \caption{$f_{15}$, $D = 40$. Here \texttt{EXPLO2} is indicated by \texttt{EE}.}
  \label{fig:IOHanalyzer/EE_mul25par32lamLin/EE_mul25par32lamLin_f15d40vScipy}
\end{figure}

\begin{figure}[h]
  \centering
  \includegraphics[trim = 0mm 0mm 0mm 0mm, clip, width=\columnwidth,keepaspectratio]{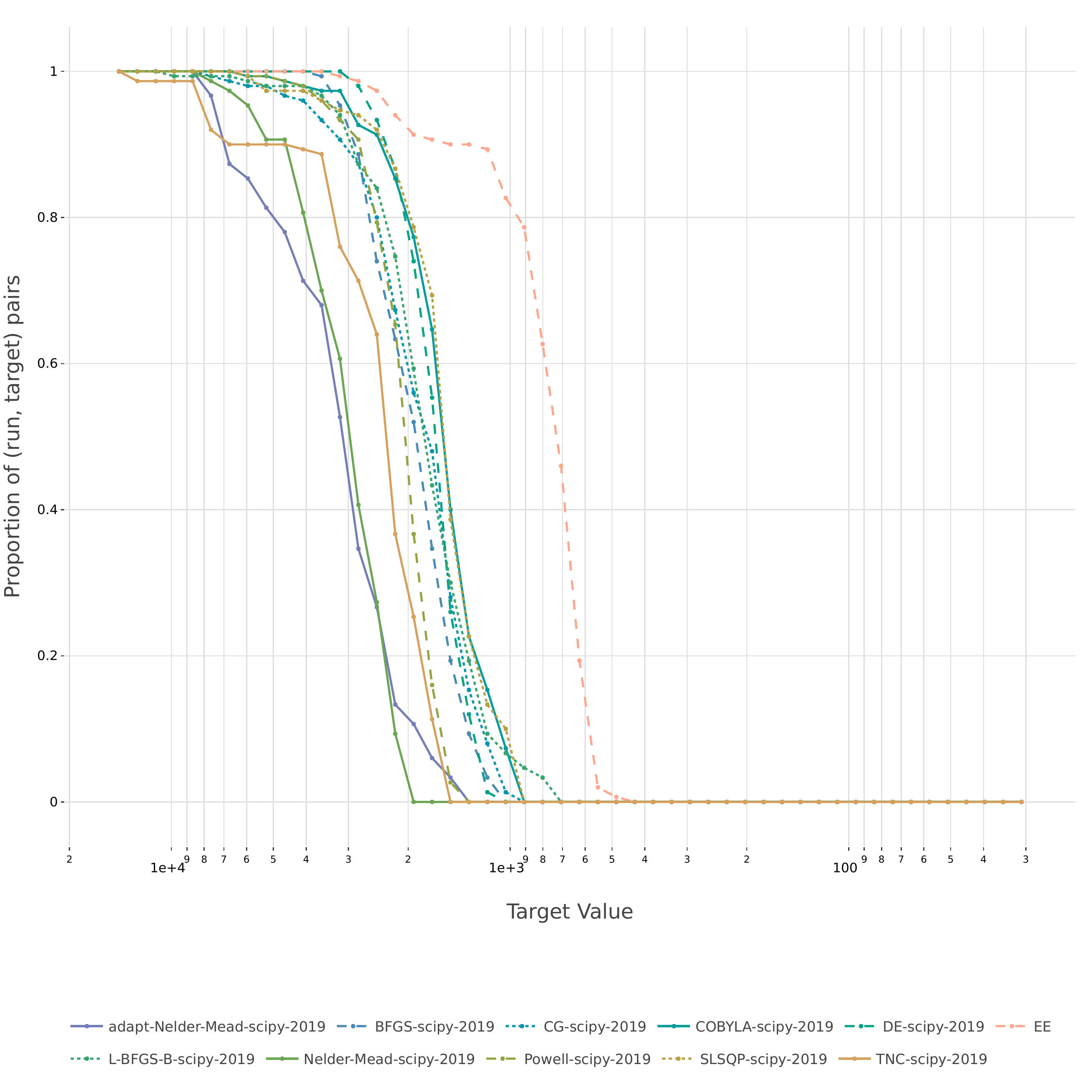}
  \caption{$f_{15}$, $D = 40$. Here \texttt{EXPLO2} is indicated by \texttt{EE}.}
  \label{fig:IOHanalyzer/EE_mul25par32lamLin/EE_mul25par32lamLin_f15d40vScipyECDF}
\end{figure}

\clearpage


\begin{figure}[h]
  \centering
  \includegraphics[trim = 0mm 0mm 0mm 0mm, clip, width=\columnwidth,keepaspectratio]{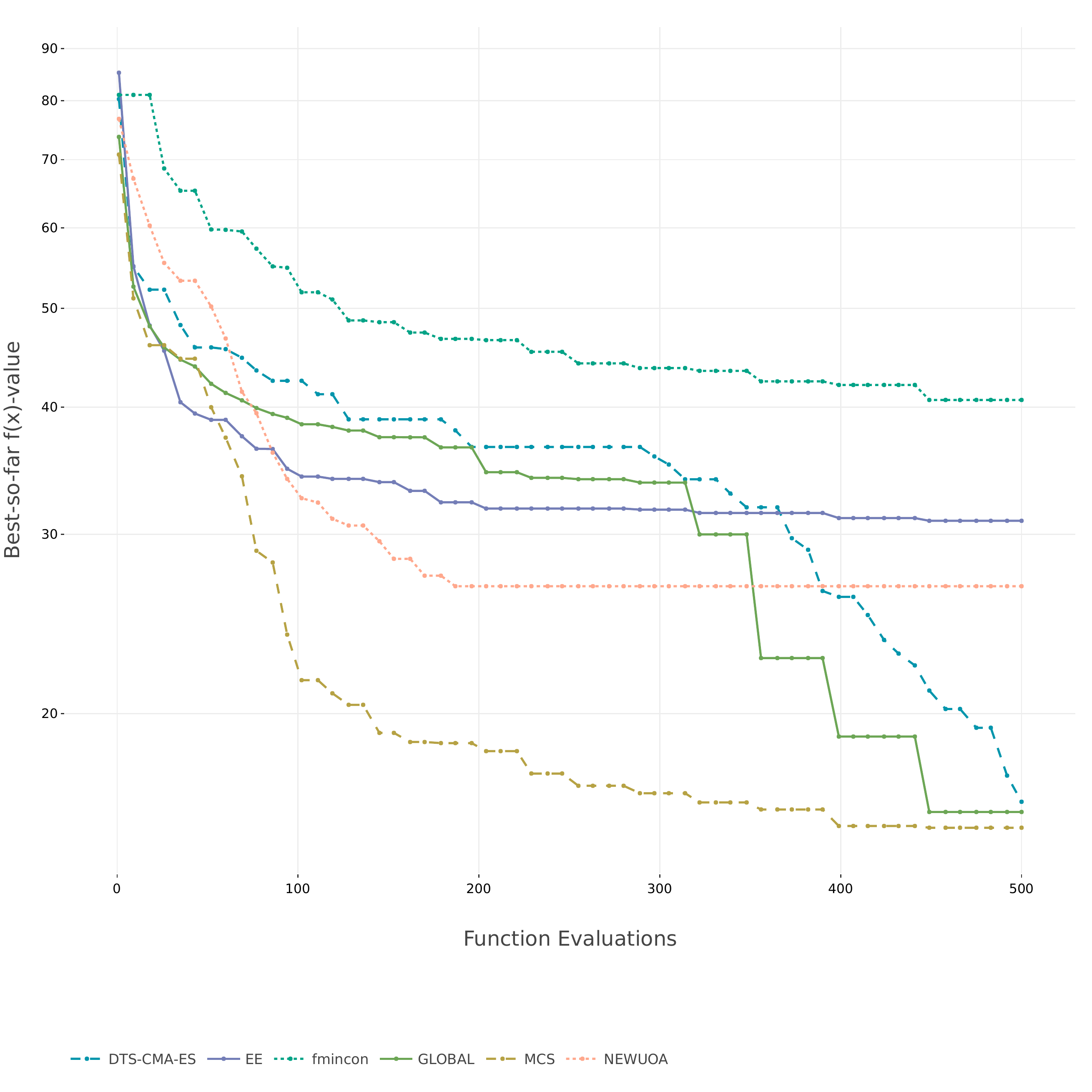}
  \caption{$f_{16}$, $D = 20$. Here \texttt{EXPLO2} is indicated by \texttt{EE}.}
  \label{fig:IOHanalyzer/EE_mul25par32lamLin/EE_mul25par32lamLin_f16d20}
\end{figure}

\begin{figure}[h]
  \centering
  \includegraphics[trim = 0mm 0mm 0mm 0mm, clip, width=\columnwidth,keepaspectratio]{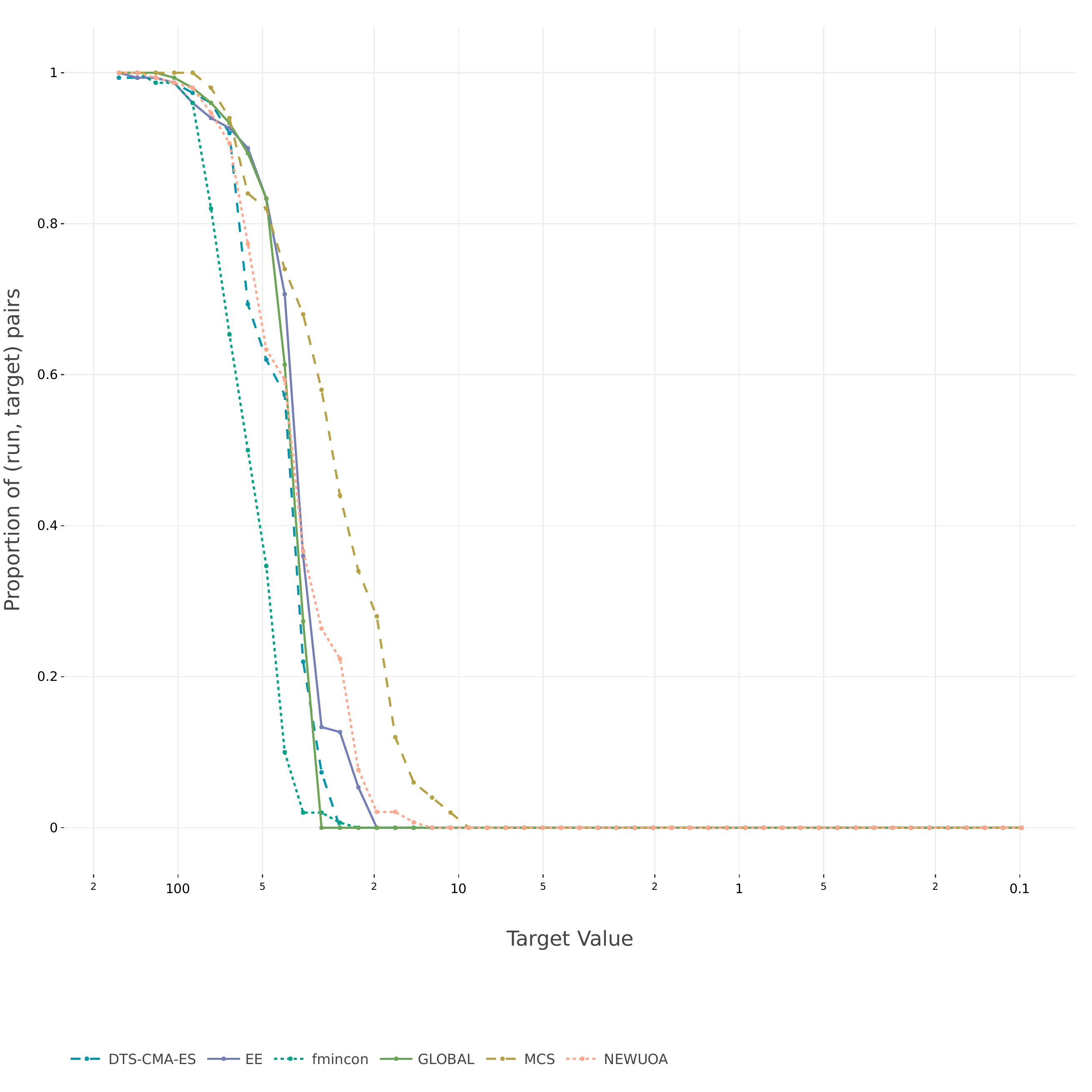}
  \caption{$f_{16}$, $D = 20$. Here \texttt{EXPLO2} is indicated by \texttt{EE}.}
  \label{fig:IOHanalyzer/EE_mul25par32lamLin/EE_mul25par32lamLin_f16d20ECDF}
\end{figure}

\begin{figure}[h]
  \centering
  \includegraphics[trim = 0mm 0mm 0mm 0mm, clip, width=\columnwidth,keepaspectratio]{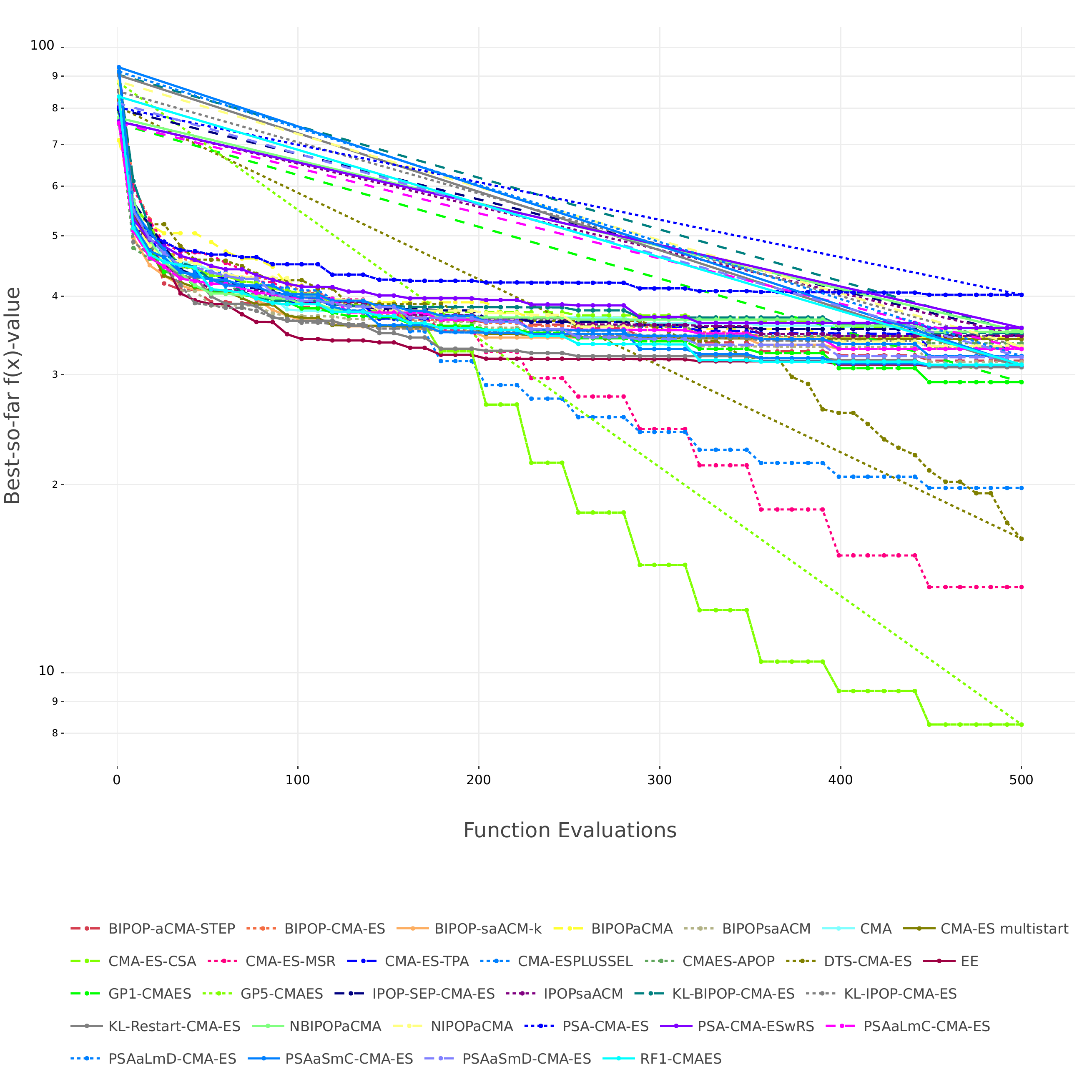}
  \caption{$f_{16}$, $D = 20$. Here \texttt{EXPLO2} is indicated by \texttt{EE}.}
  \label{fig:IOHanalyzer/EE_mul25par32lamLin/EE_mul25par32lamLin_f16d20allCMA}
\end{figure}

\begin{figure}[h]
  \centering
  \includegraphics[trim = 0mm 0mm 0mm 0mm, clip, width=\columnwidth,keepaspectratio]{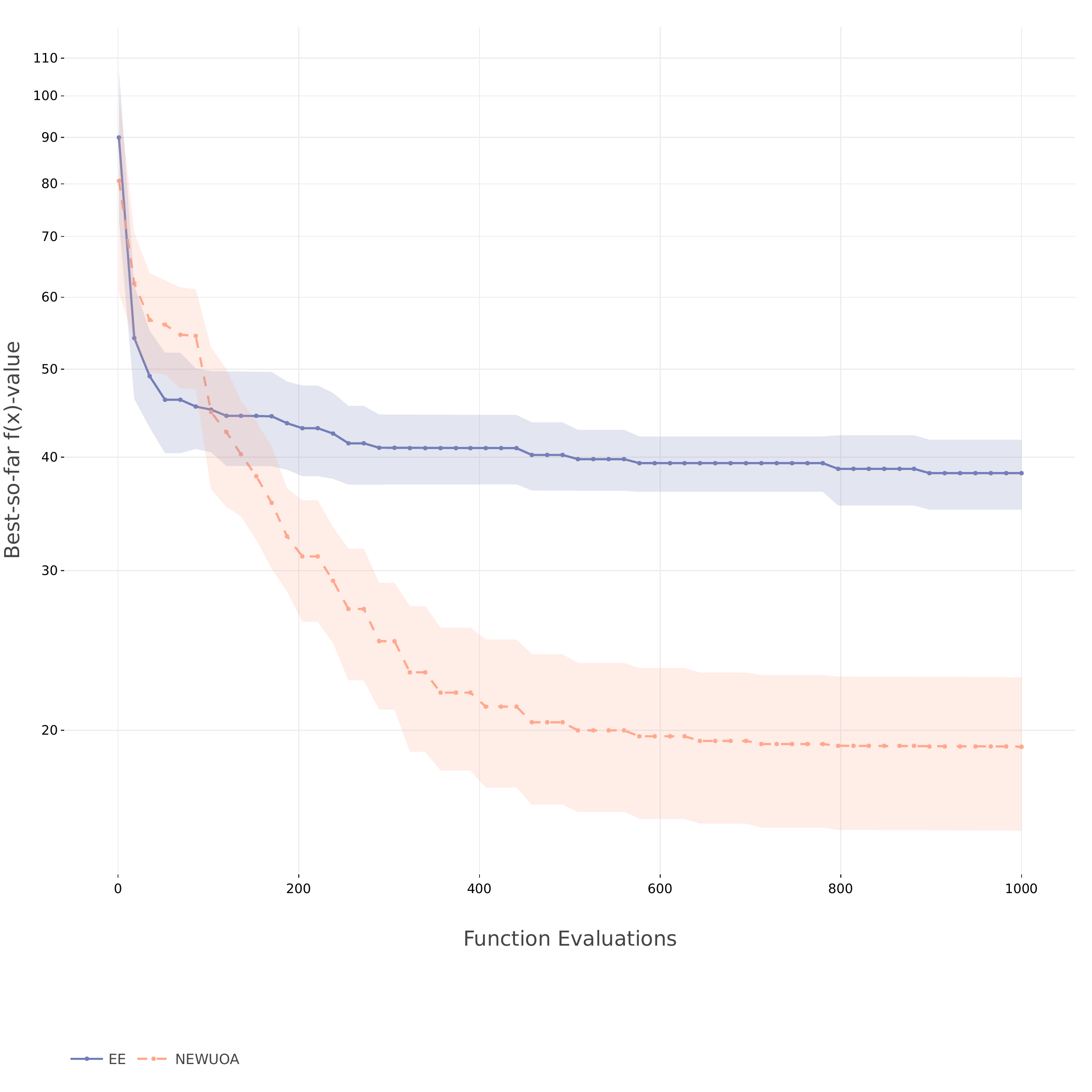}
  \caption{$f_{16}$, $D = 40$. Here \texttt{EXPLO2} is indicated by \texttt{EE}.}
  \label{fig:IOHanalyzer/EE_mul25par32lamLin/EE_mul25par32lamLin_f16d40}
\end{figure}

\begin{figure}[h]
  \centering
  \includegraphics[trim = 0mm 0mm 0mm 0mm, clip, width=\columnwidth,keepaspectratio]{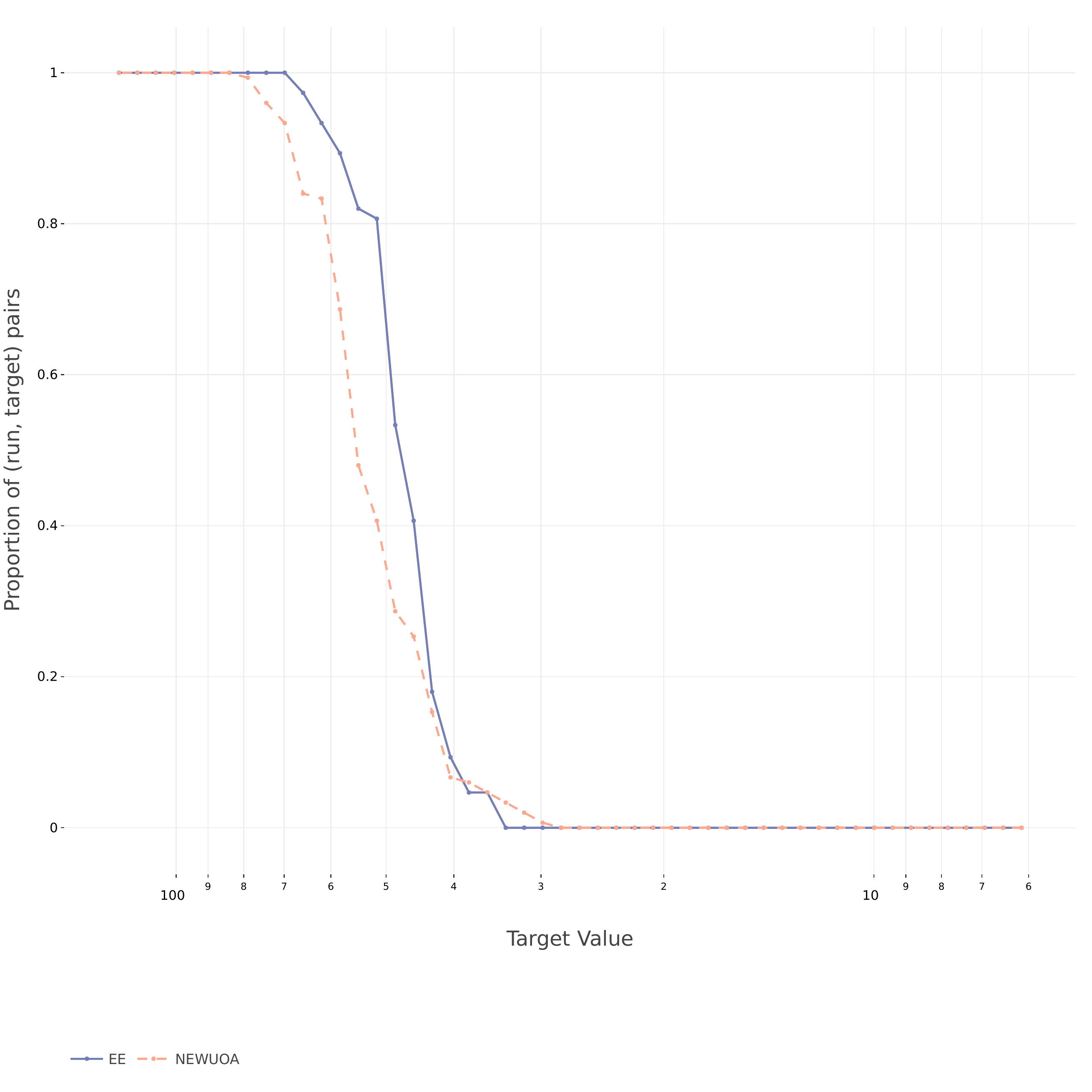}
  \caption{$f_{16}$, $D = 40$. Here \texttt{EXPLO2} is indicated by \texttt{EE}.}
  \label{fig:IOHanalyzer/EE_mul25par32lamLin/EE_mul25par32lamLin_f16d40ECDF}
\end{figure}

\begin{figure}[h]
  \centering
  \includegraphics[trim = 0mm 0mm 0mm 0mm, clip, width=\columnwidth,keepaspectratio]{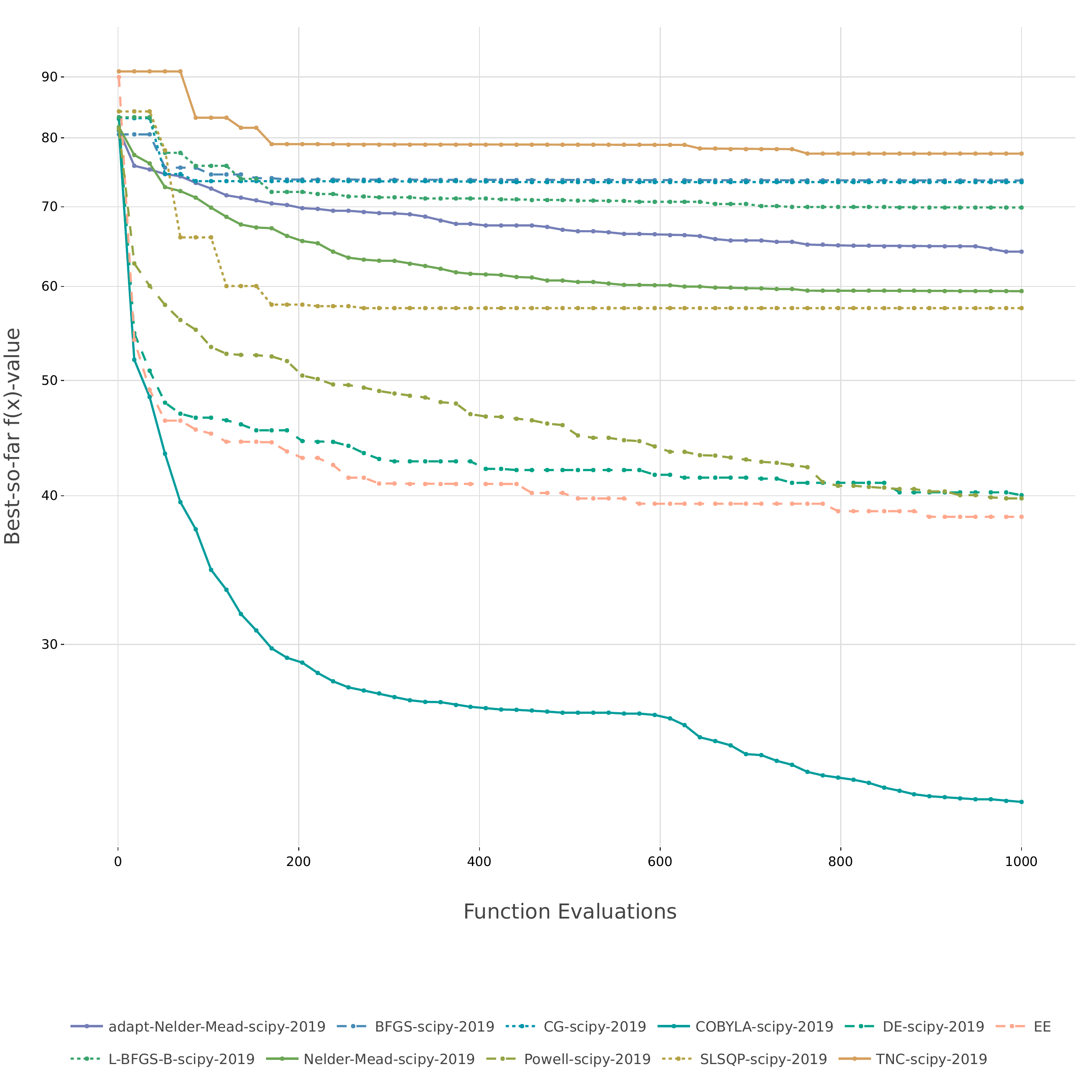}
  \caption{$f_{16}$, $D = 40$. Here \texttt{EXPLO2} is indicated by \texttt{EE}.}
  \label{fig:IOHanalyzer/EE_mul25par32lamLin/EE_mul25par32lamLin_f16d40vScipy}
\end{figure}

\begin{figure}[h]
  \centering
  \includegraphics[trim = 0mm 0mm 0mm 0mm, clip, width=\columnwidth,keepaspectratio]{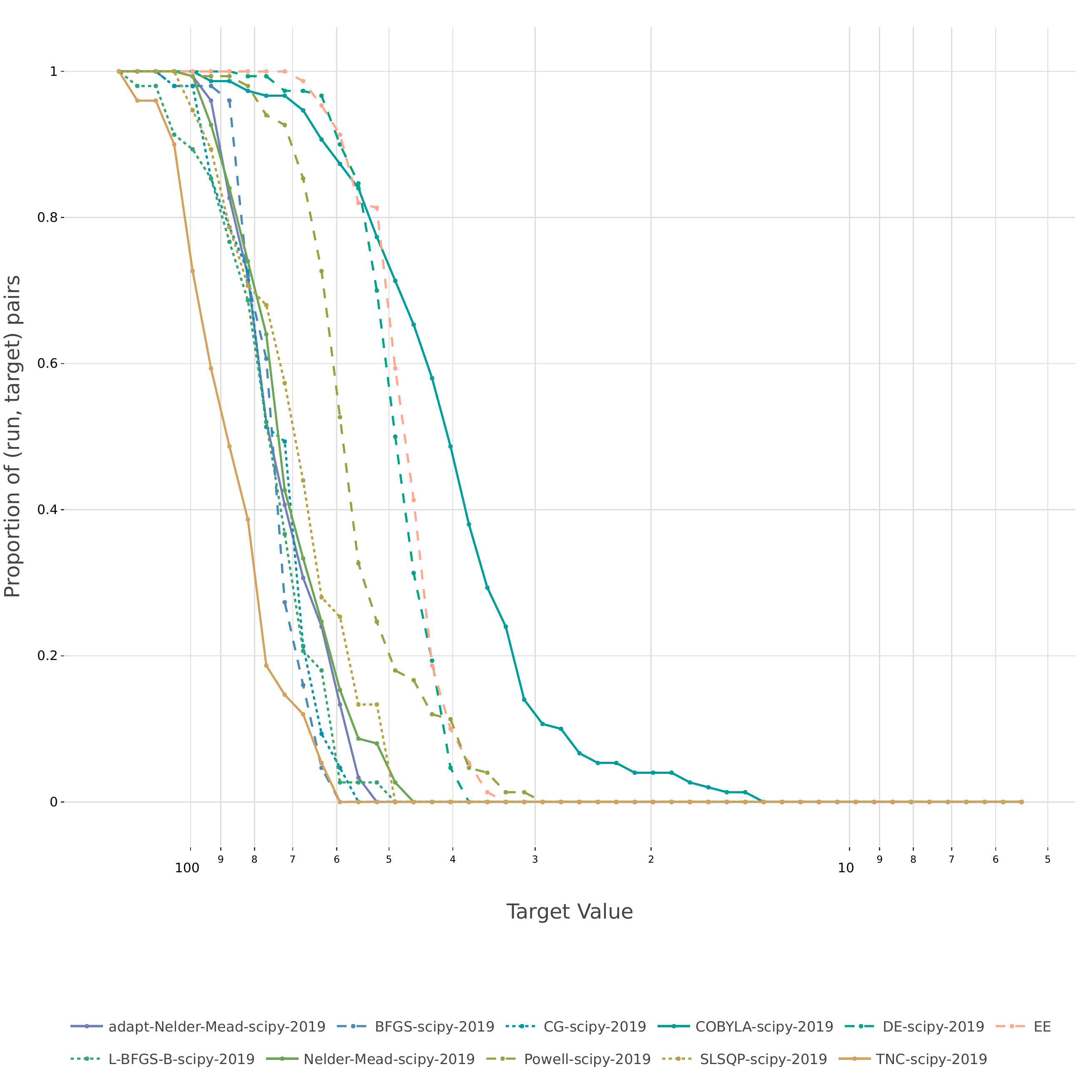}
  \caption{$f_{16}$, $D = 40$. Here \texttt{EXPLO2} is indicated by \texttt{EE}.}
  \label{fig:IOHanalyzer/EE_mul25par32lamLin/EE_mul25par32lamLin_f16d40vScipyECDF}
\end{figure}

\clearpage


\begin{figure}[h]
  \centering
  \includegraphics[trim = 0mm 0mm 0mm 0mm, clip, width=\columnwidth,keepaspectratio]{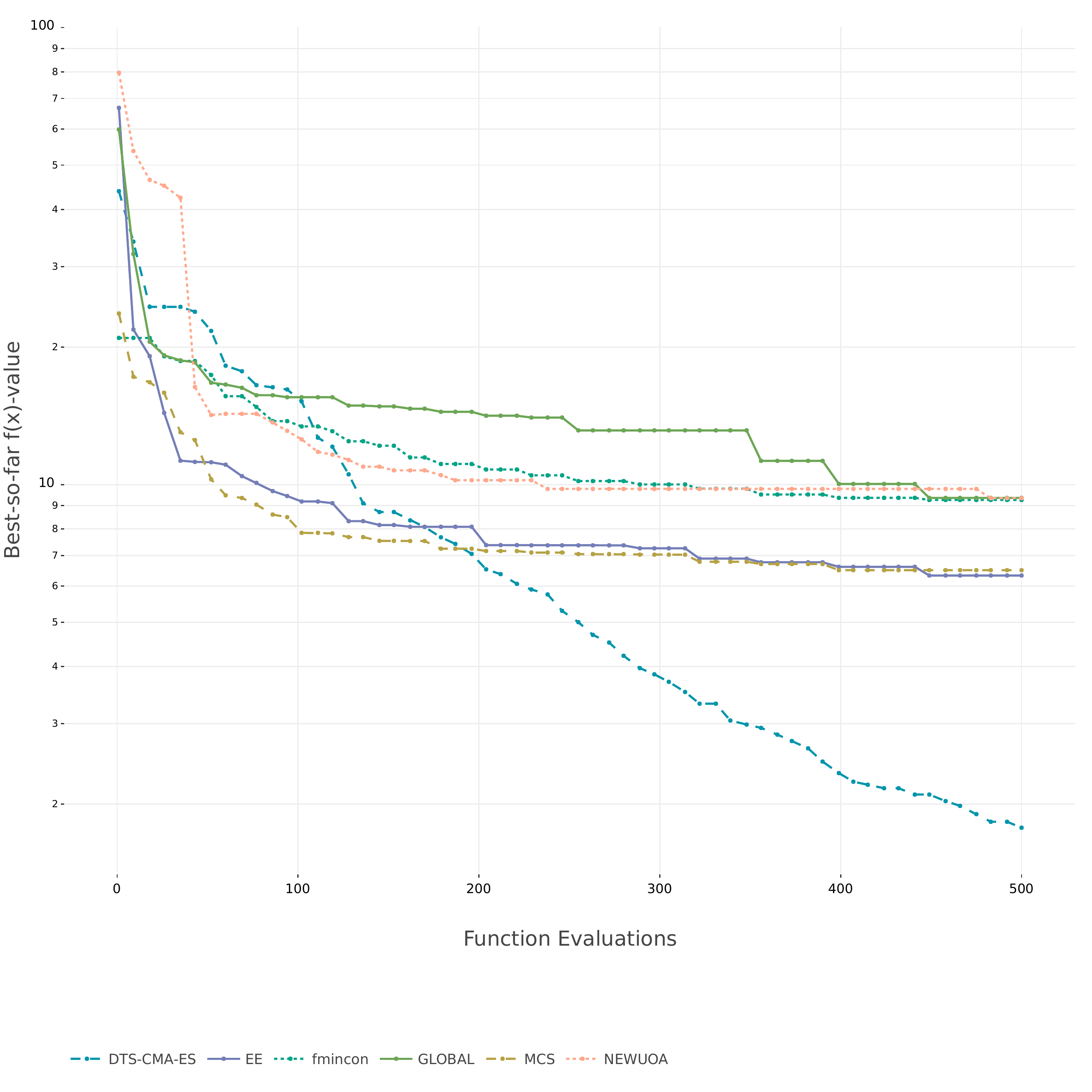}
  \caption{$f_{17}$, $D = 20$. Here \texttt{EXPLO2} is indicated by \texttt{EE}.}
  \label{fig:IOHanalyzer/EE_mul25par32lamLin/EE_mul25par32lamLin_f17d20}
\end{figure}

\begin{figure}[h]
  \centering
  \includegraphics[trim = 0mm 0mm 0mm 0mm, clip, width=\columnwidth,keepaspectratio]{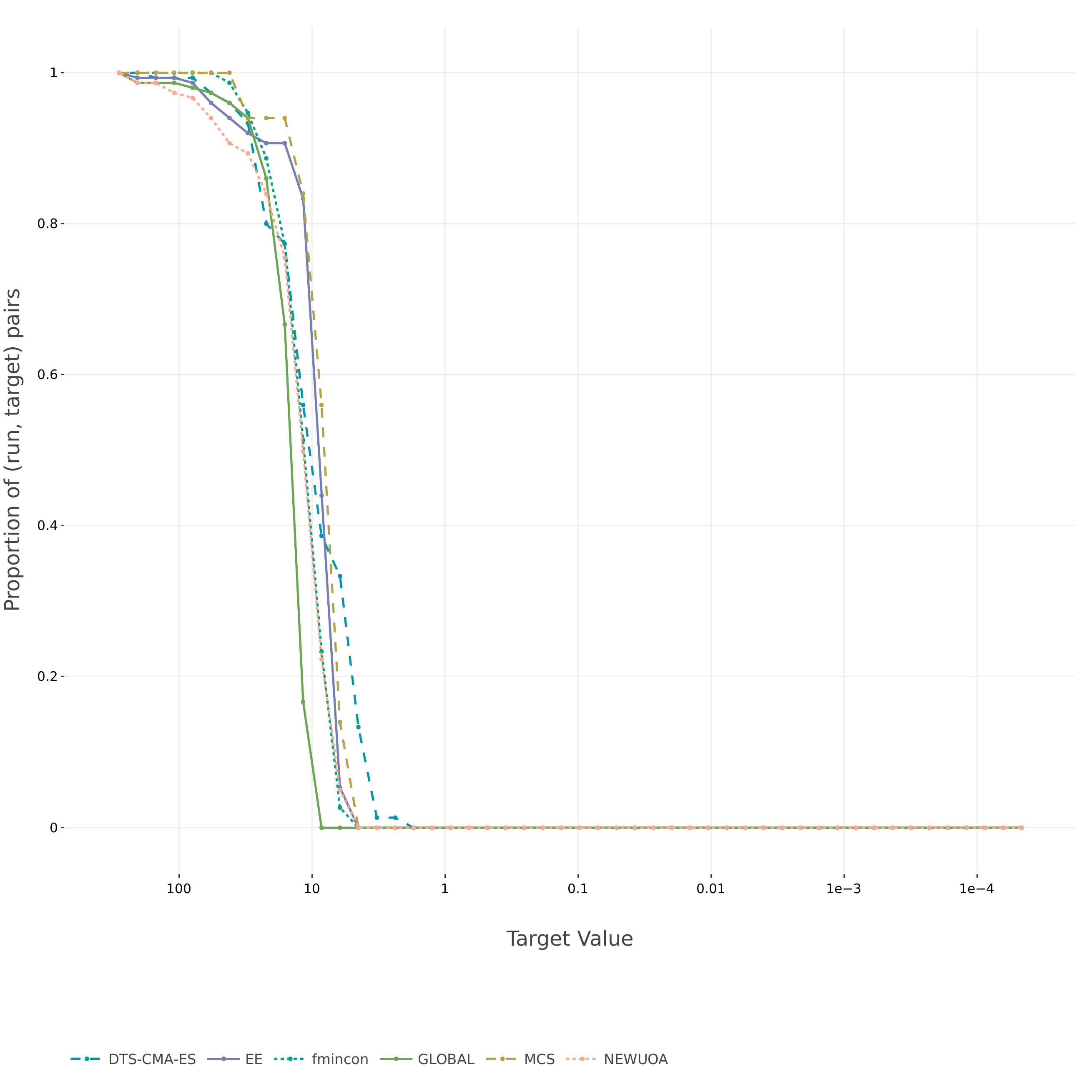}
  \caption{$f_{17}$, $D = 20$. Here \texttt{EXPLO2} is indicated by \texttt{EE}.}
  \label{fig:IOHanalyzer/EE_mul25par32lamLin/EE_mul25par32lamLin_f17d20ECDF}
\end{figure}

\begin{figure}[h]
  \centering
  \includegraphics[trim = 0mm 0mm 0mm 0mm, clip, width=\columnwidth,keepaspectratio]{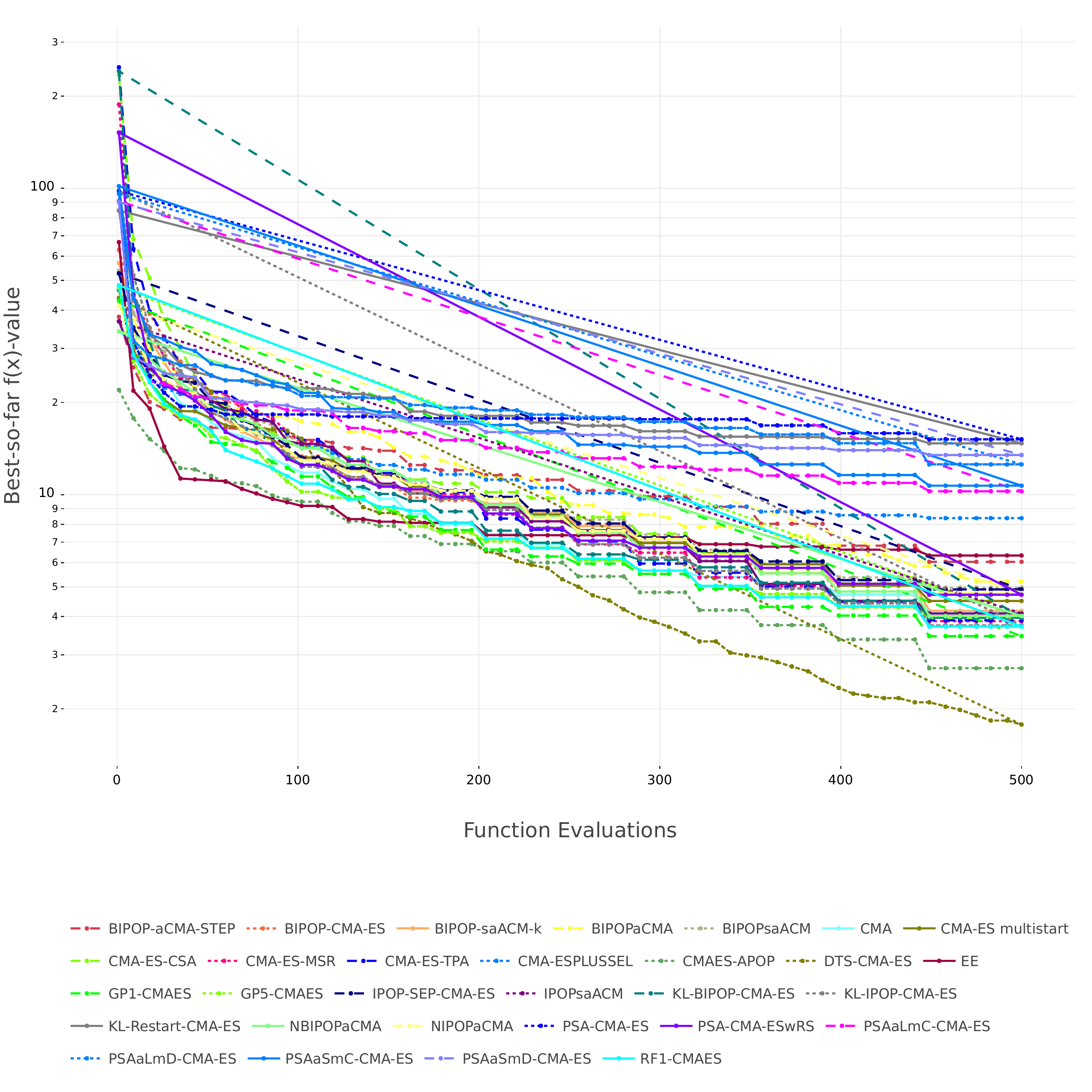}
  \caption{$f_{17}$, $D = 20$. Here \texttt{EXPLO2} is indicated by \texttt{EE}.}
  \label{fig:IOHanalyzer/EE_mul25par32lamLin/EE_mul25par32lamLin_f17d20allCMA}
\end{figure}

\begin{figure}[h]
  \centering
  \includegraphics[trim = 0mm 0mm 0mm 0mm, clip, width=\columnwidth,keepaspectratio]{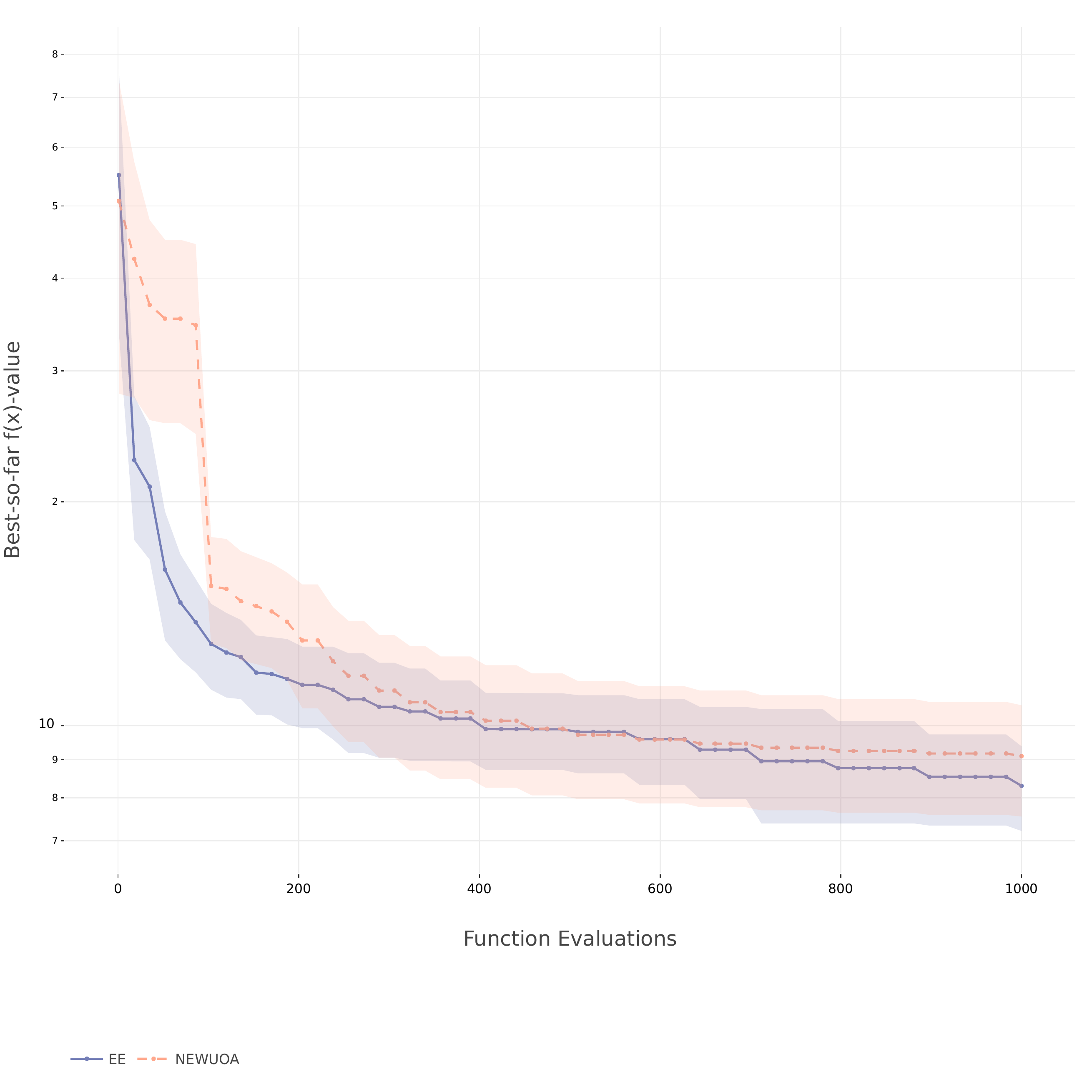}
  \caption{$f_{17}$, $D = 40$. Here \texttt{EXPLO2} is indicated by \texttt{EE}.}
  \label{fig:IOHanalyzer/EE_mul25par32lamLin/EE_mul25par32lamLin_f17d40}
\end{figure}

\begin{figure}[h]
  \centering
  \includegraphics[trim = 0mm 0mm 0mm 0mm, clip, width=\columnwidth,keepaspectratio]{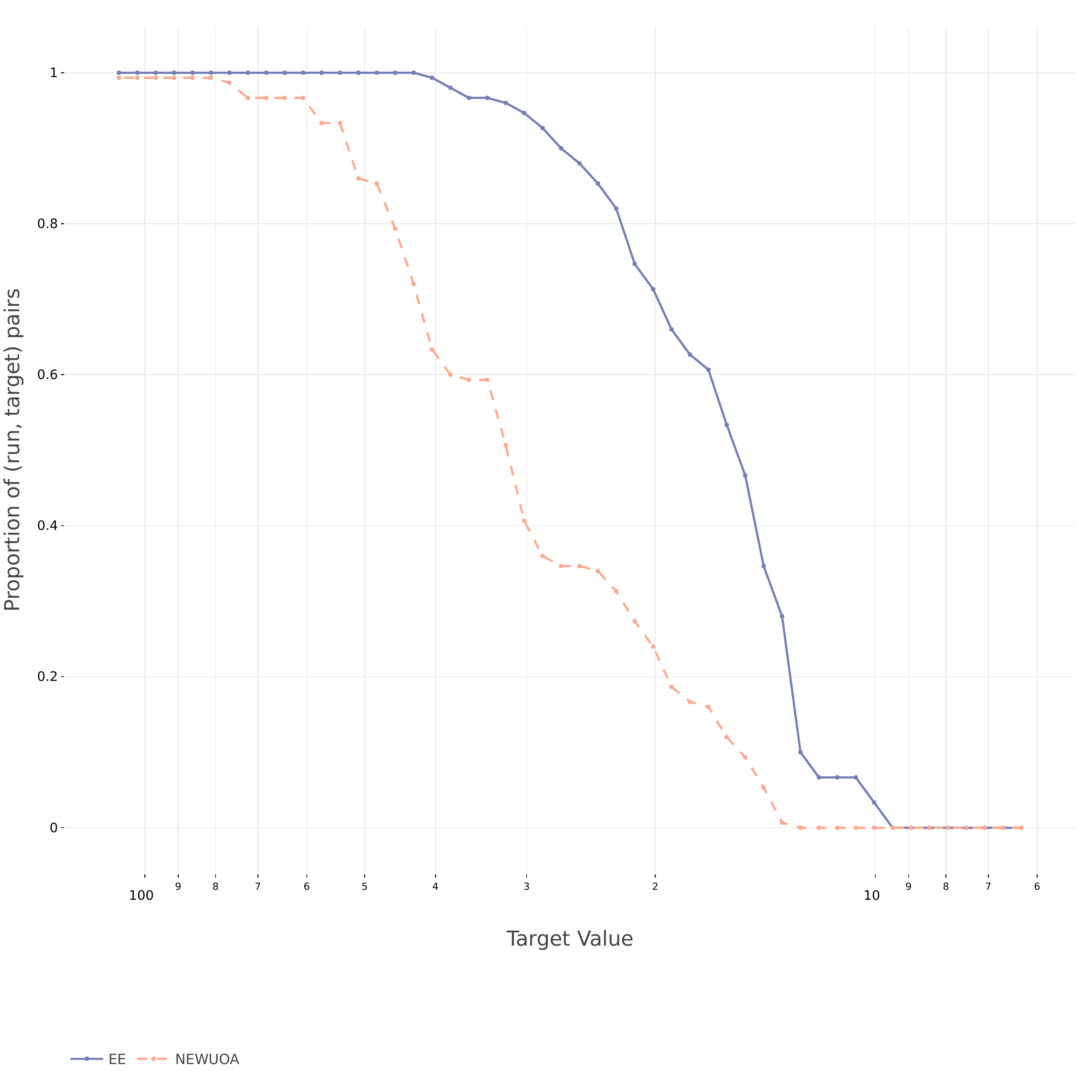}
  \caption{$f_{17}$, $D = 40$. Here \texttt{EXPLO2} is indicated by \texttt{EE}.}
  \label{fig:IOHanalyzer/EE_mul25par32lamLin/EE_mul25par32lamLin_f17d40ECDF}
\end{figure}

\begin{figure}[h]
  \centering
  \includegraphics[trim = 0mm 0mm 0mm 0mm, clip, width=\columnwidth,keepaspectratio]{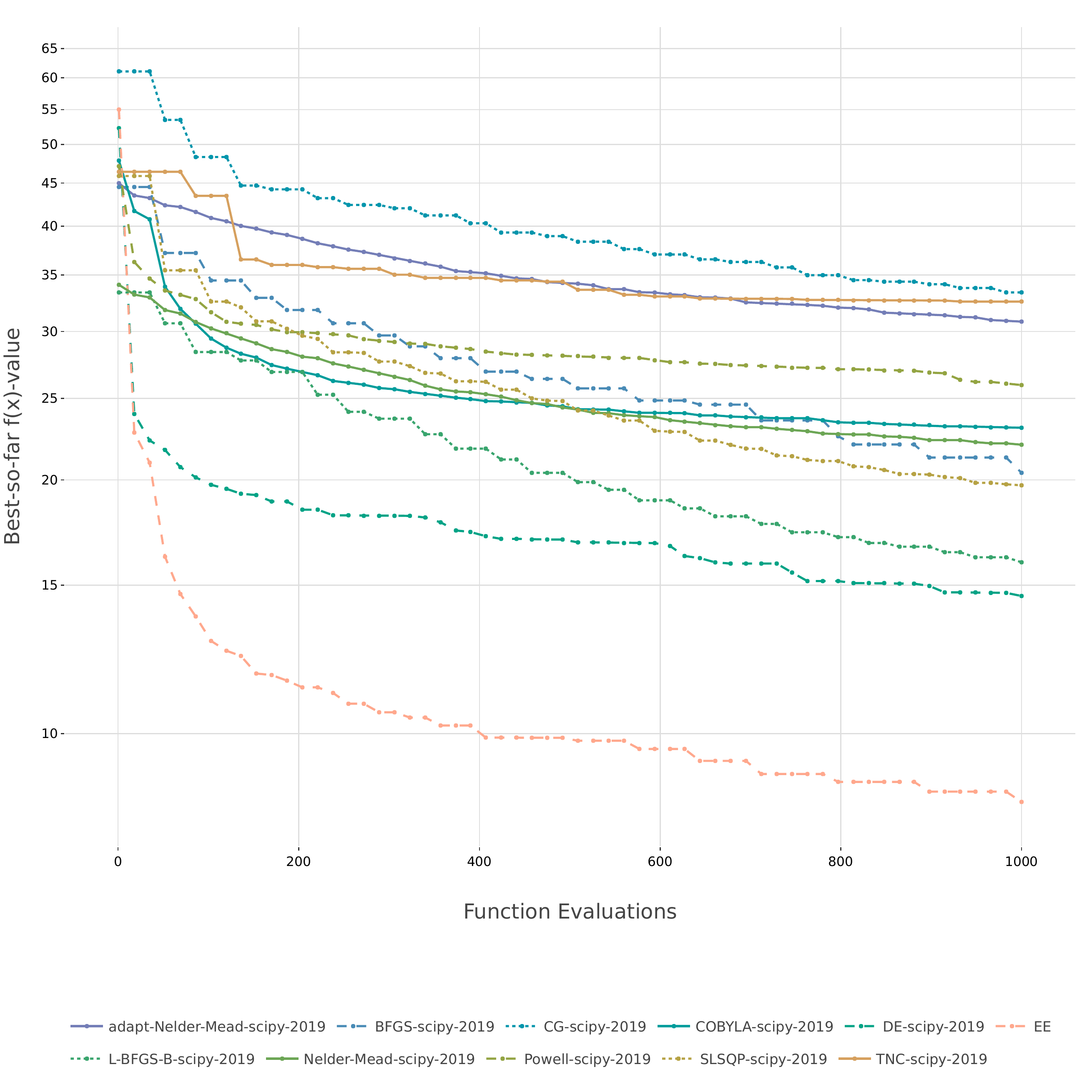}
  \caption{$f_{17}$, $D = 40$. Here \texttt{EXPLO2} is indicated by \texttt{EE}.}
  \label{fig:IOHanalyzer/EE_mul25par32lamLin/EE_mul25par32lamLin_f17d40vScipy}
\end{figure}

\begin{figure}[h]
  \centering
  \includegraphics[trim = 0mm 0mm 0mm 0mm, clip, width=\columnwidth,keepaspectratio]{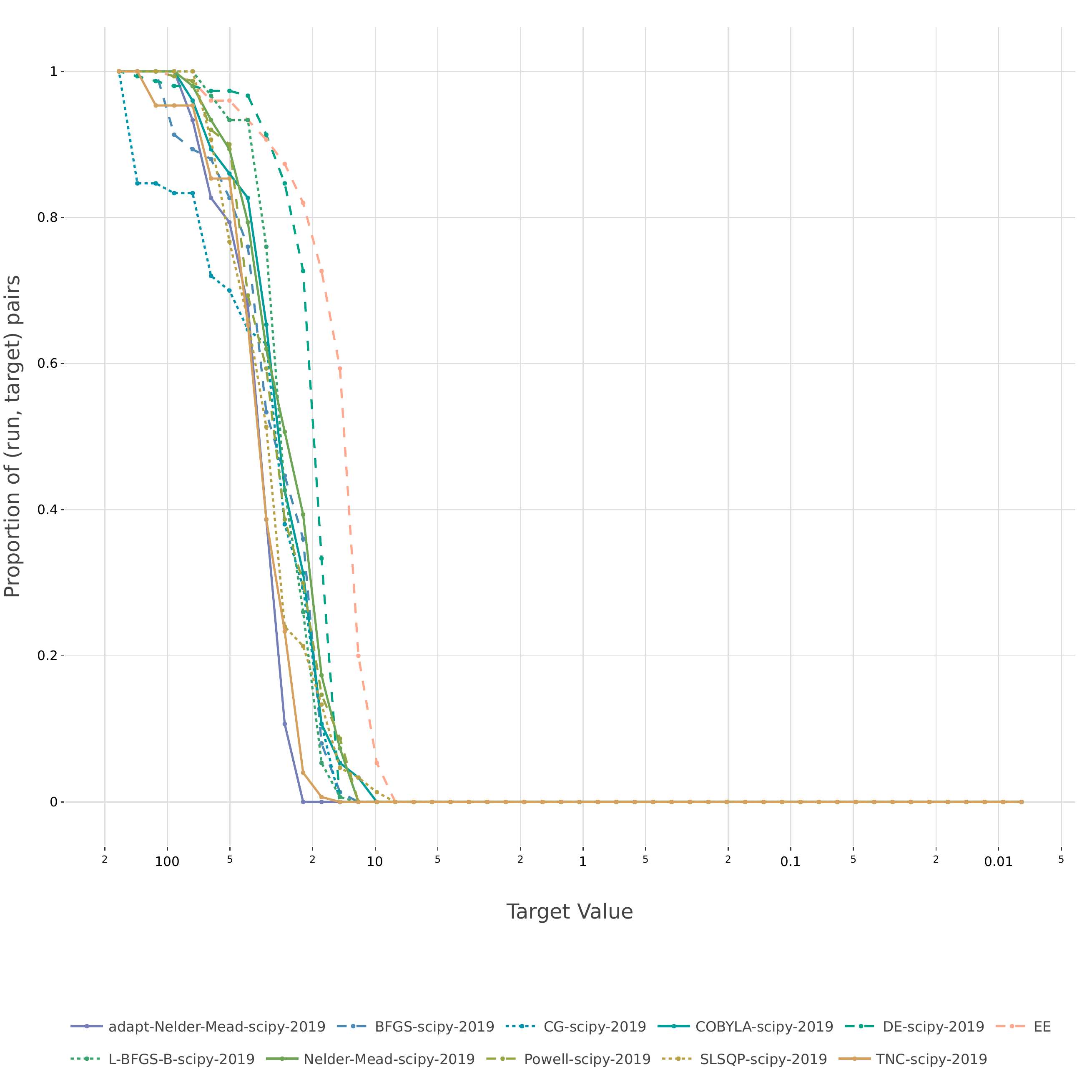}
  \caption{$f_{17}$, $D = 40$. Here \texttt{EXPLO2} is indicated by \texttt{EE}.}
  \label{fig:IOHanalyzer/EE_mul25par32lamLin/EE_mul25par32lamLin_f17d40vScipyECDF}
\end{figure}

\clearpage


\begin{figure}[h]
  \centering
  \includegraphics[trim = 0mm 0mm 0mm 0mm, clip, width=\columnwidth,keepaspectratio]{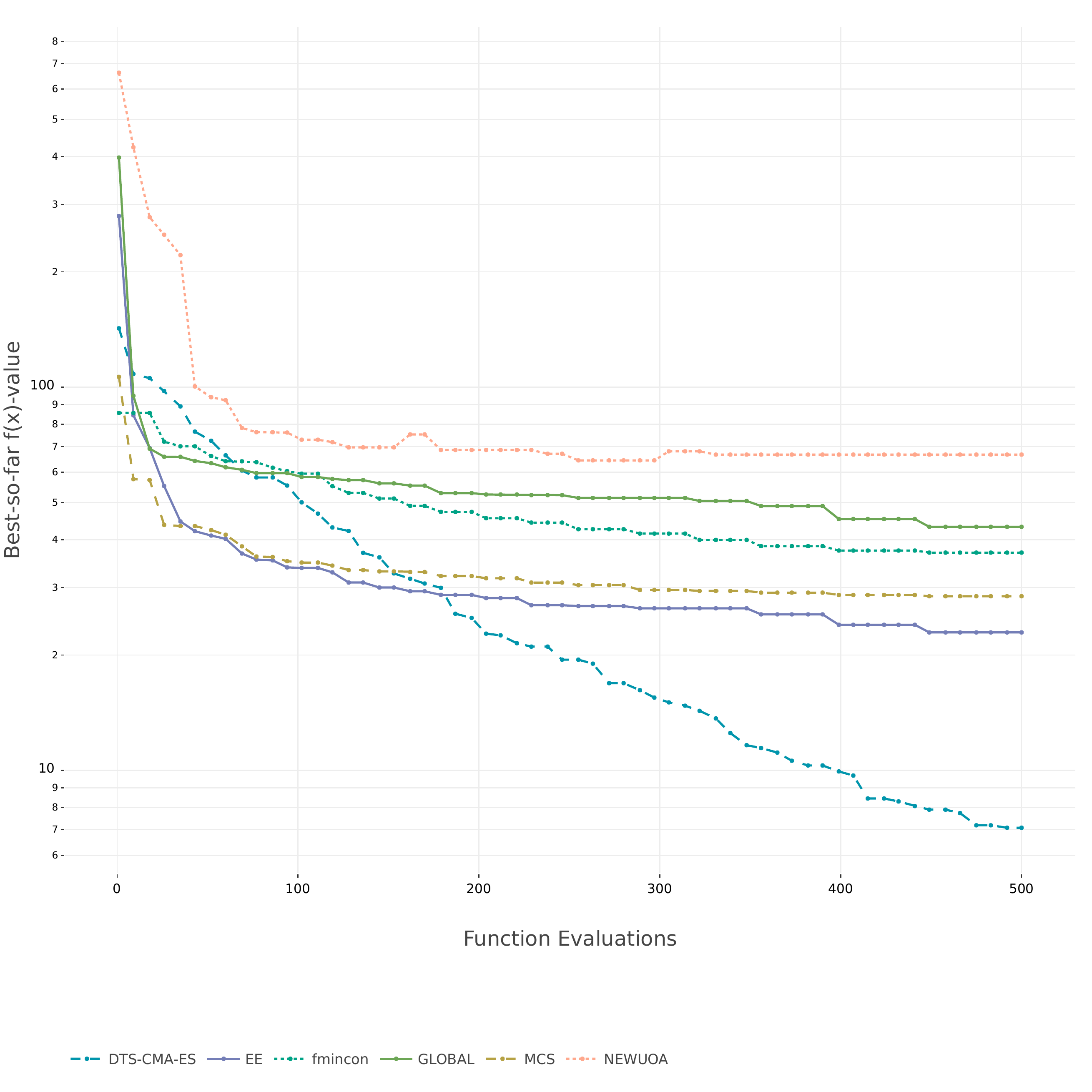}
  \caption{$f_{18}$, $D = 20$. Here \texttt{EXPLO2} is indicated by \texttt{EE}.}
  \label{fig:IOHanalyzer/EE_mul25par32lamLin/EE_mul25par32lamLin_f18d20}
\end{figure}

\begin{figure}[h]
  \centering
  \includegraphics[trim = 0mm 0mm 0mm 0mm, clip, width=\columnwidth,keepaspectratio]{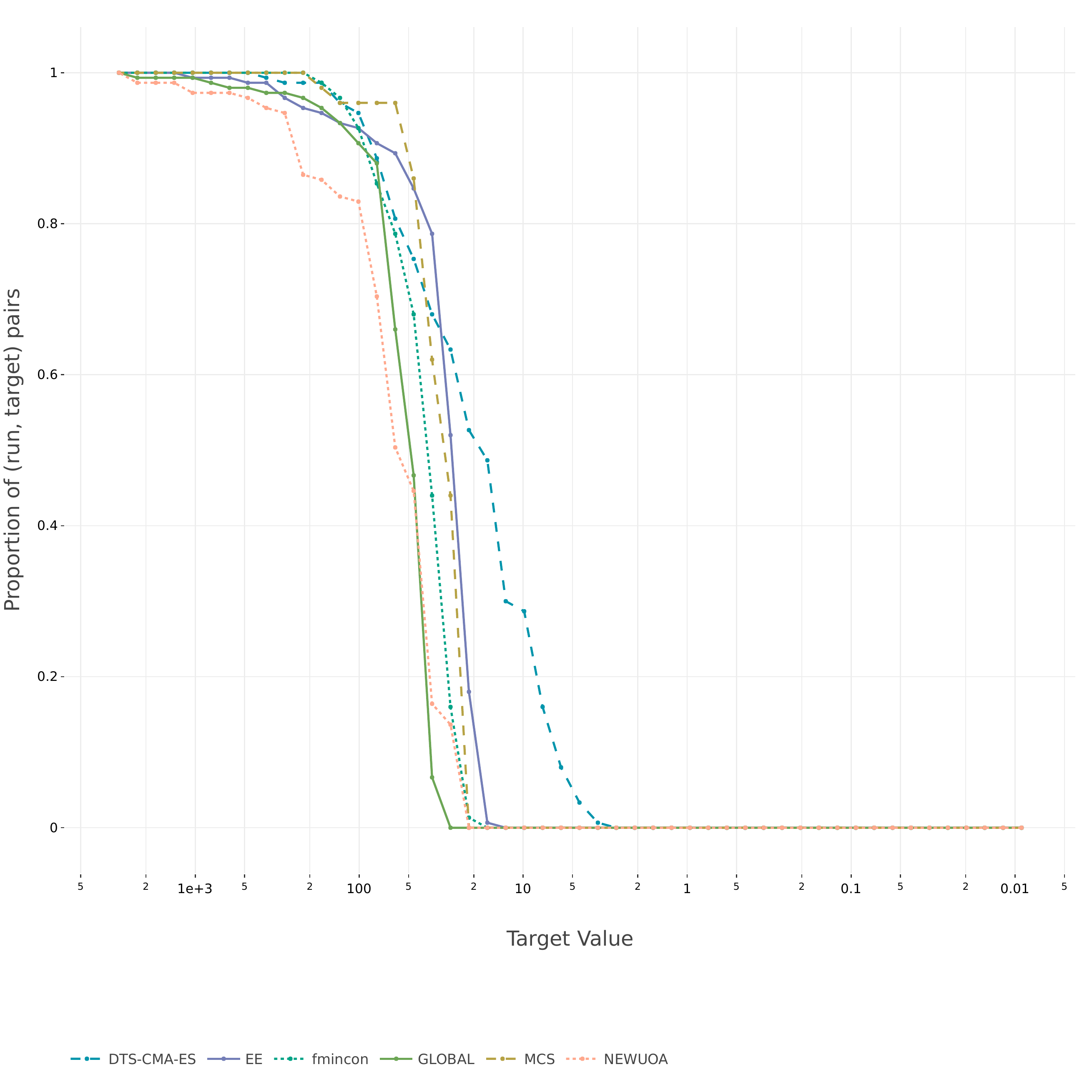}
  \caption{$f_{18}$, $D = 20$. Here \texttt{EXPLO2} is indicated by \texttt{EE}.}
  \label{fig:IOHanalyzer/EE_mul25par32lamLin/EE_mul25par32lamLin_f18d20ECDF}
\end{figure}

\begin{figure}[h]
  \centering
  \includegraphics[trim = 0mm 0mm 0mm 0mm, clip, width=\columnwidth,keepaspectratio]{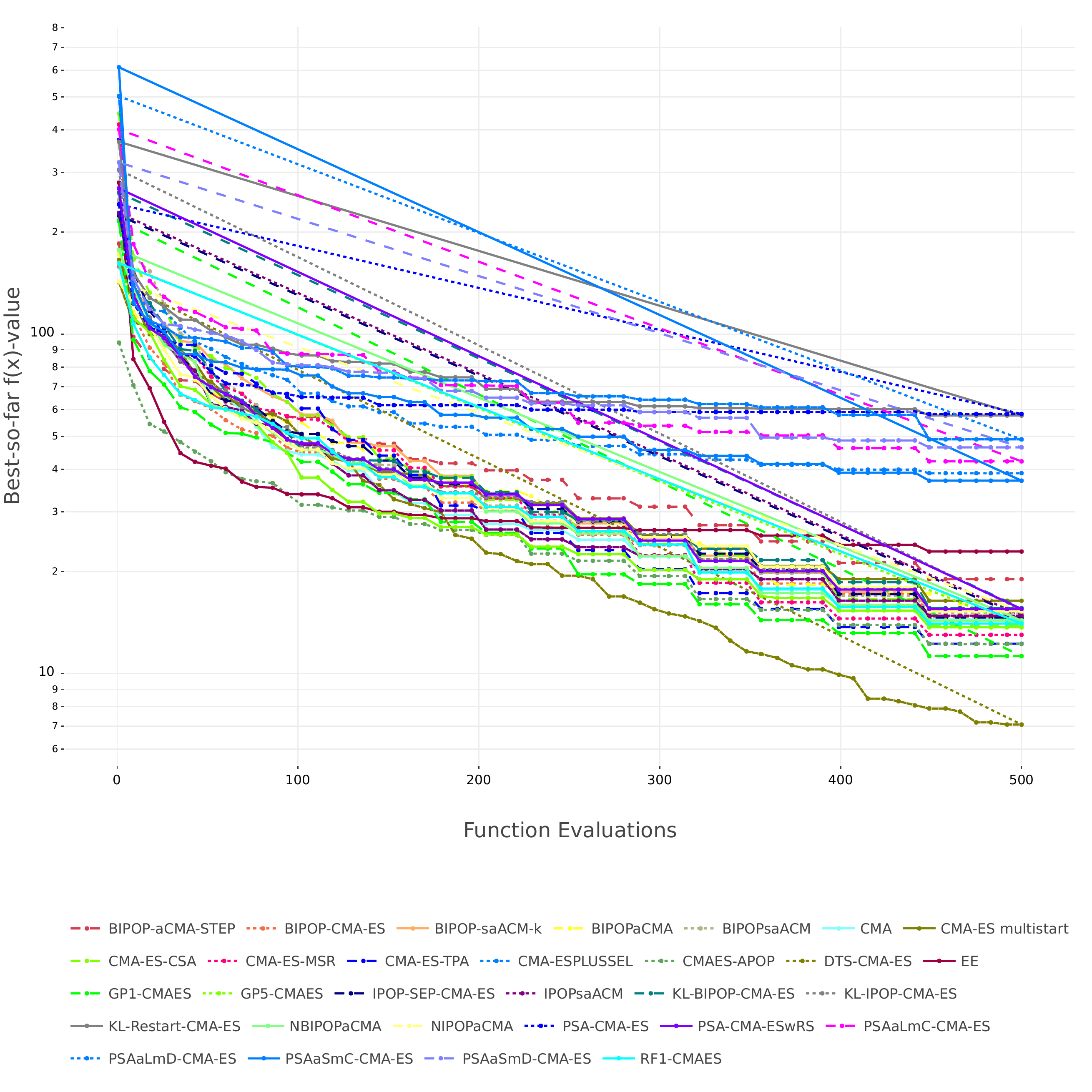}
  \caption{$f_{18}$, $D = 20$. Here \texttt{EXPLO2} is indicated by \texttt{EE}.}
  \label{fig:IOHanalyzer/EE_mul25par32lamLin/EE_mul25par32lamLin_f18d20allCMA}
\end{figure}

\begin{figure}[h]
  \centering
  \includegraphics[trim = 0mm 0mm 0mm 0mm, clip, width=\columnwidth,keepaspectratio]{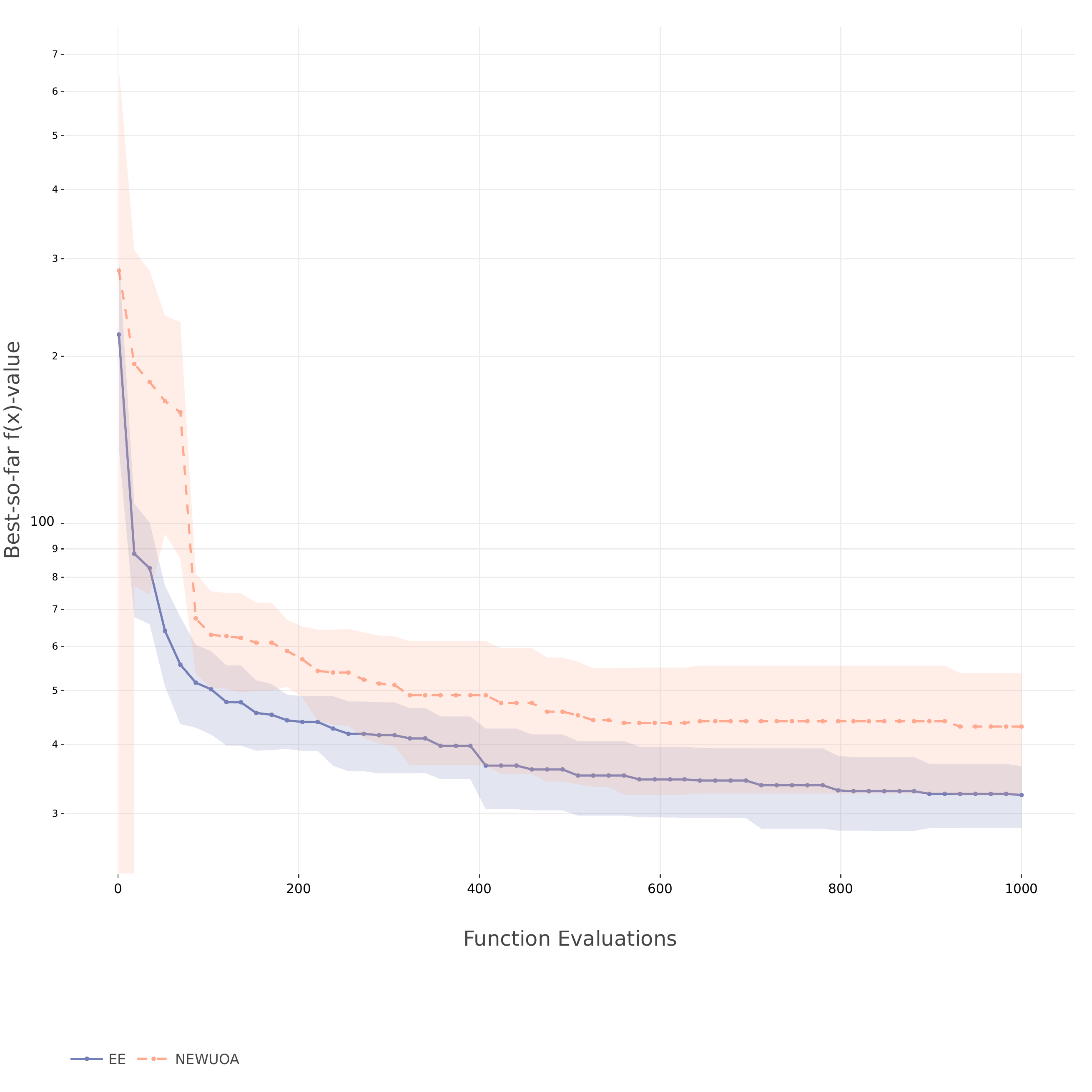}
  \caption{$f_{18}$, $D = 40$. Here \texttt{EXPLO2} is indicated by \texttt{EE}.}
  \label{fig:IOHanalyzer/EE_mul25par32lamLin/EE_mul25par32lamLin_f18d40}
\end{figure}

\begin{figure}[h]
  \centering
  \includegraphics[trim = 0mm 0mm 0mm 0mm, clip, width=\columnwidth,keepaspectratio]{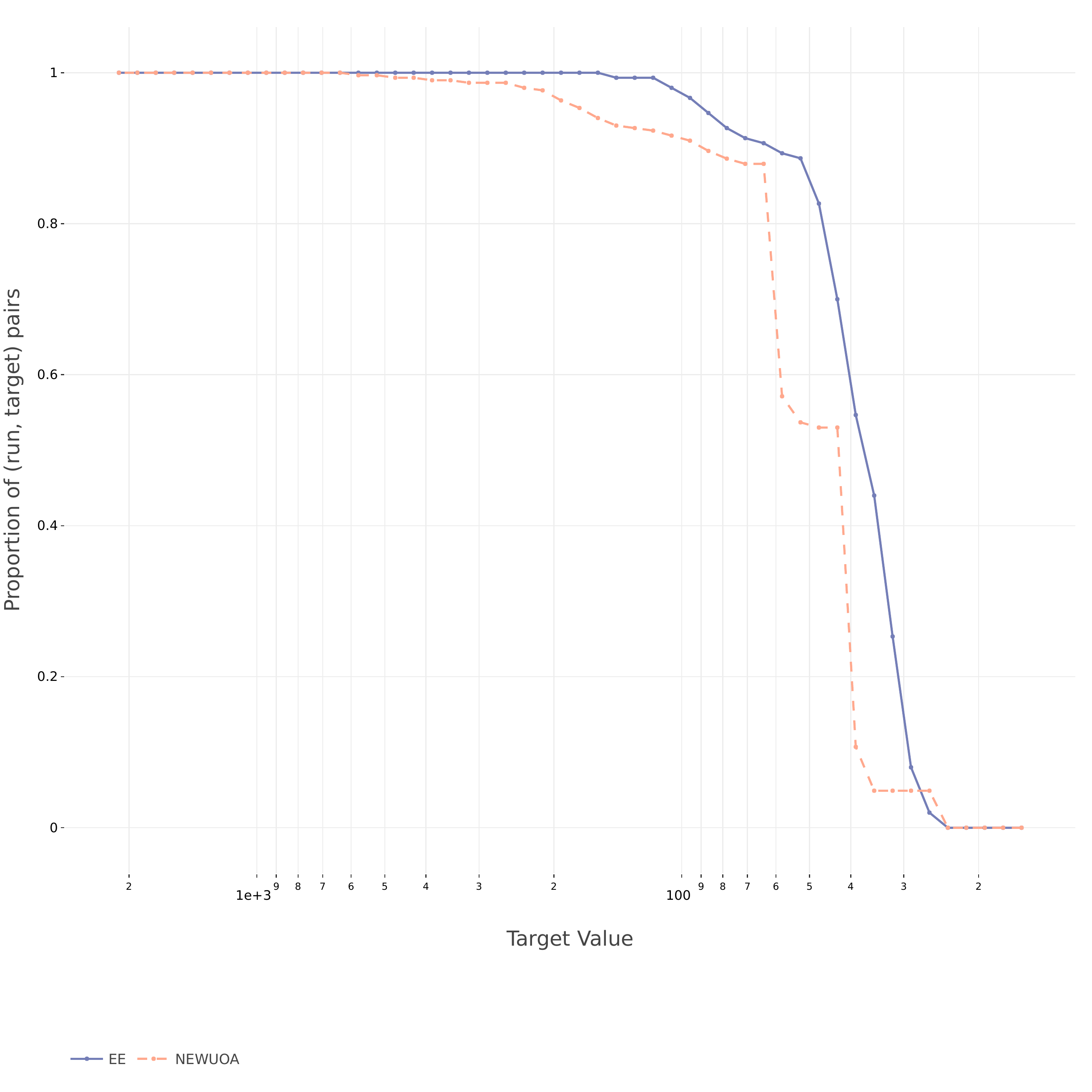}
  \caption{$f_{18}$, $D = 40$. Here \texttt{EXPLO2} is indicated by \texttt{EE}.}
  \label{fig:IOHanalyzer/EE_mul25par32lamLin/EE_mul25par32lamLin_f18d40ECDF}
\end{figure}

\begin{figure}[h]
  \centering
  \includegraphics[trim = 0mm 0mm 0mm 0mm, clip, width=\columnwidth,keepaspectratio]{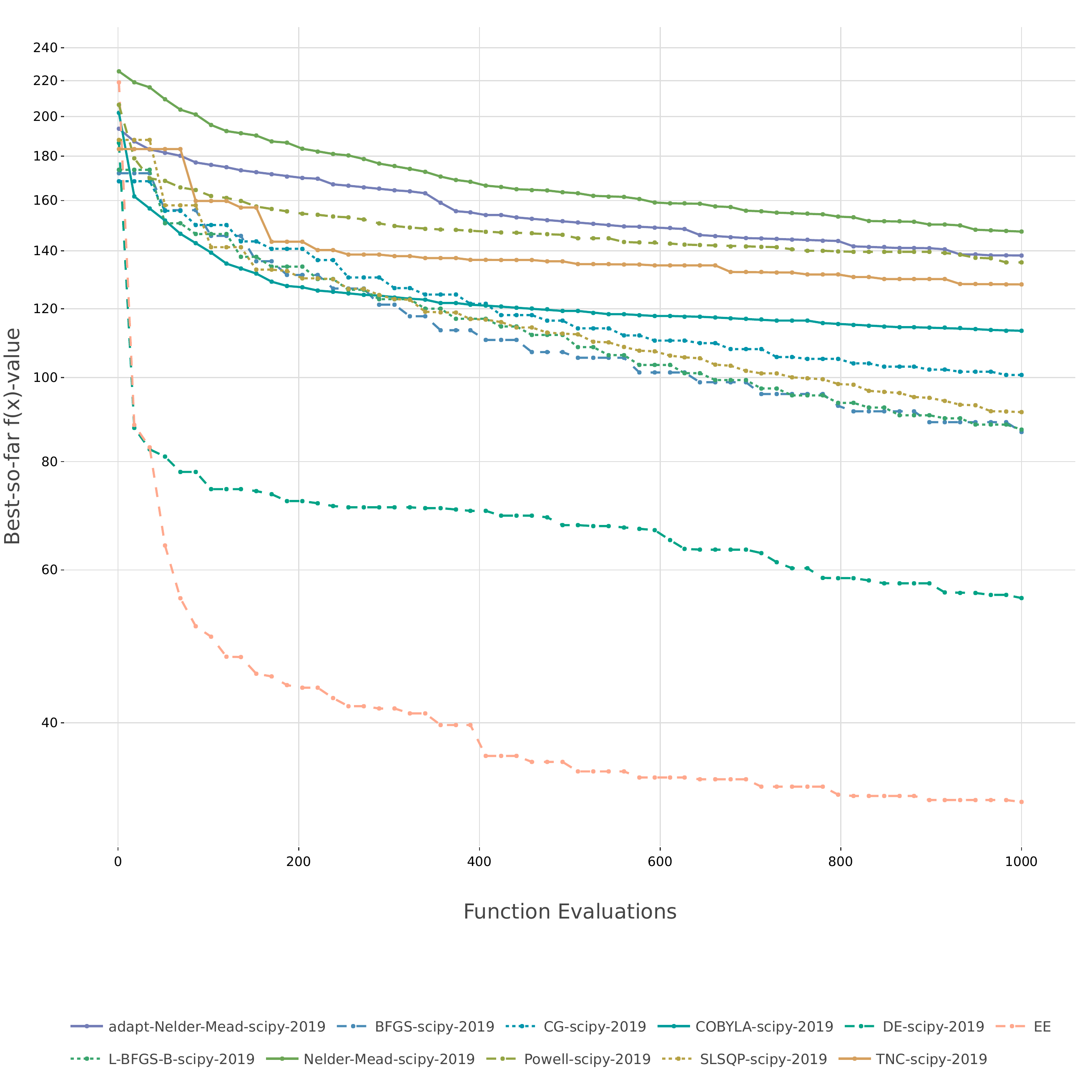}
  \caption{$f_{18}$, $D = 40$. Here \texttt{EXPLO2} is indicated by \texttt{EE}.}
  \label{fig:IOHanalyzer/EE_mul25par32lamLin/EE_mul25par32lamLin_f18d40vScipy}
\end{figure}

\begin{figure}[h]
  \centering
  \includegraphics[trim = 0mm 0mm 0mm 0mm, clip, width=\columnwidth,keepaspectratio]{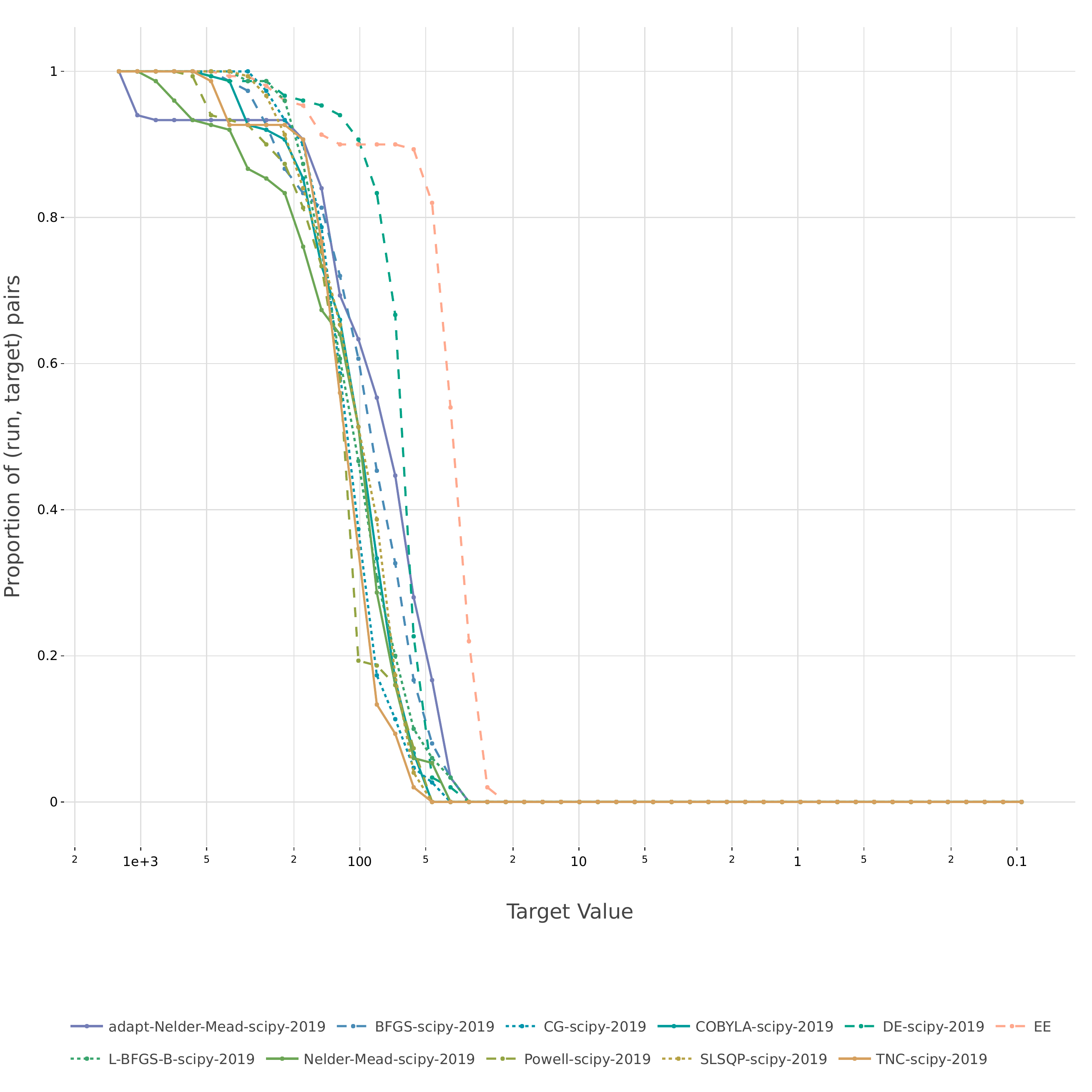}
  \caption{$f_{18}$, $D = 40$. Here \texttt{EXPLO2} is indicated by \texttt{EE}.}
  \label{fig:IOHanalyzer/EE_mul25par32lamLin/EE_mul25par32lamLin_f18d40vScipyECDF}
\end{figure}

\clearpage


\begin{figure}[h]
  \centering
  \includegraphics[trim = 0mm 0mm 0mm 0mm, clip, width=\columnwidth,keepaspectratio]{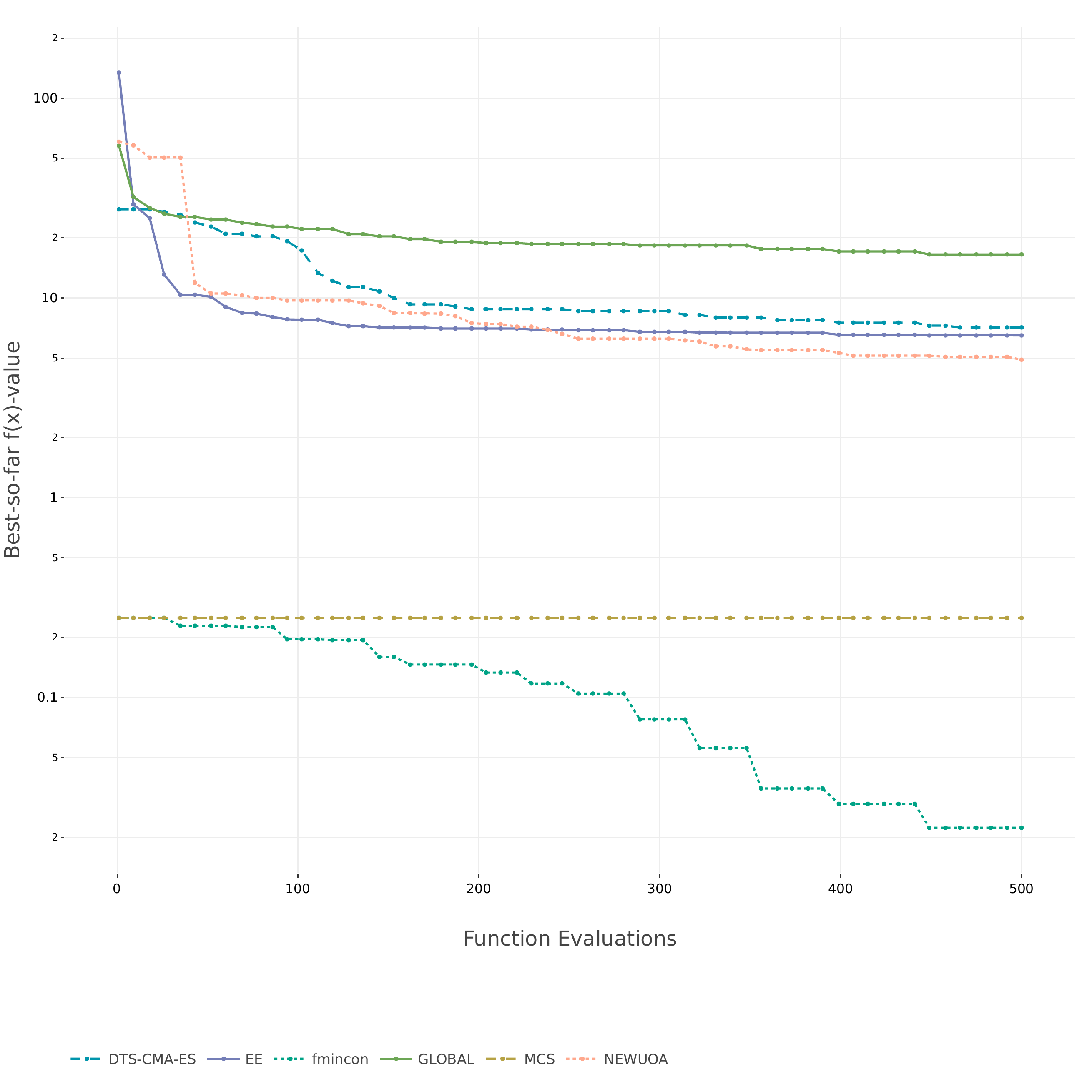}
  \caption{$f_{19}$, $D = 20$. Here \texttt{EXPLO2} is indicated by \texttt{EE}.}
  \label{fig:IOHanalyzer/EE_mul25par32lamLin/EE_mul25par32lamLin_f19d20}
\end{figure}

\begin{figure}[h]
  \centering
  \includegraphics[trim = 0mm 0mm 0mm 0mm, clip, width=\columnwidth,keepaspectratio]{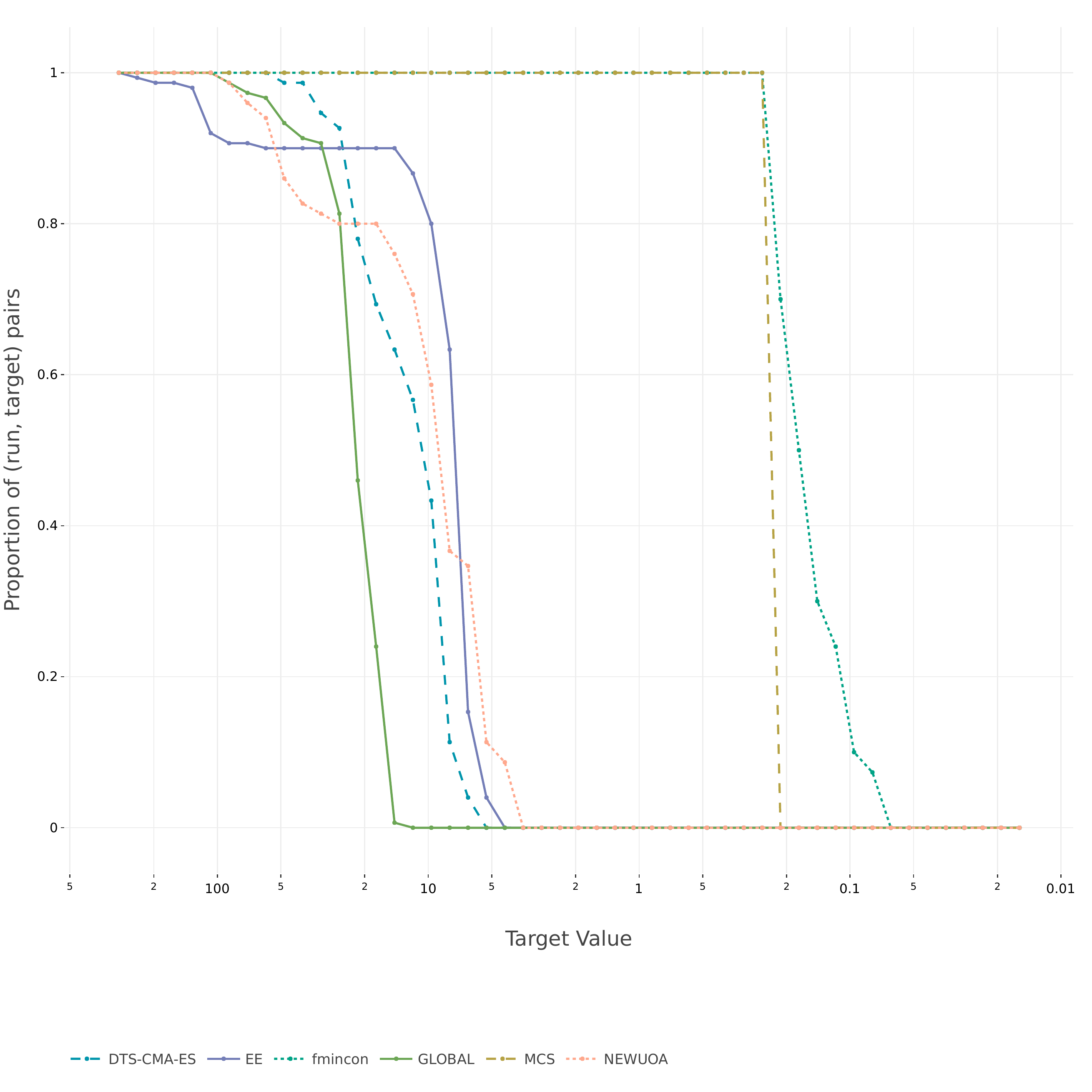}
  \caption{$f_{19}$, $D = 20$. Here \texttt{EXPLO2} is indicated by \texttt{EE}.}
  \label{fig:IOHanalyzer/EE_mul25par32lamLin/EE_mul25par32lamLin_f19d20ECDF}
\end{figure}

\begin{figure}[h]
  \centering
  \includegraphics[trim = 0mm 0mm 0mm 0mm, clip, width=\columnwidth,keepaspectratio]{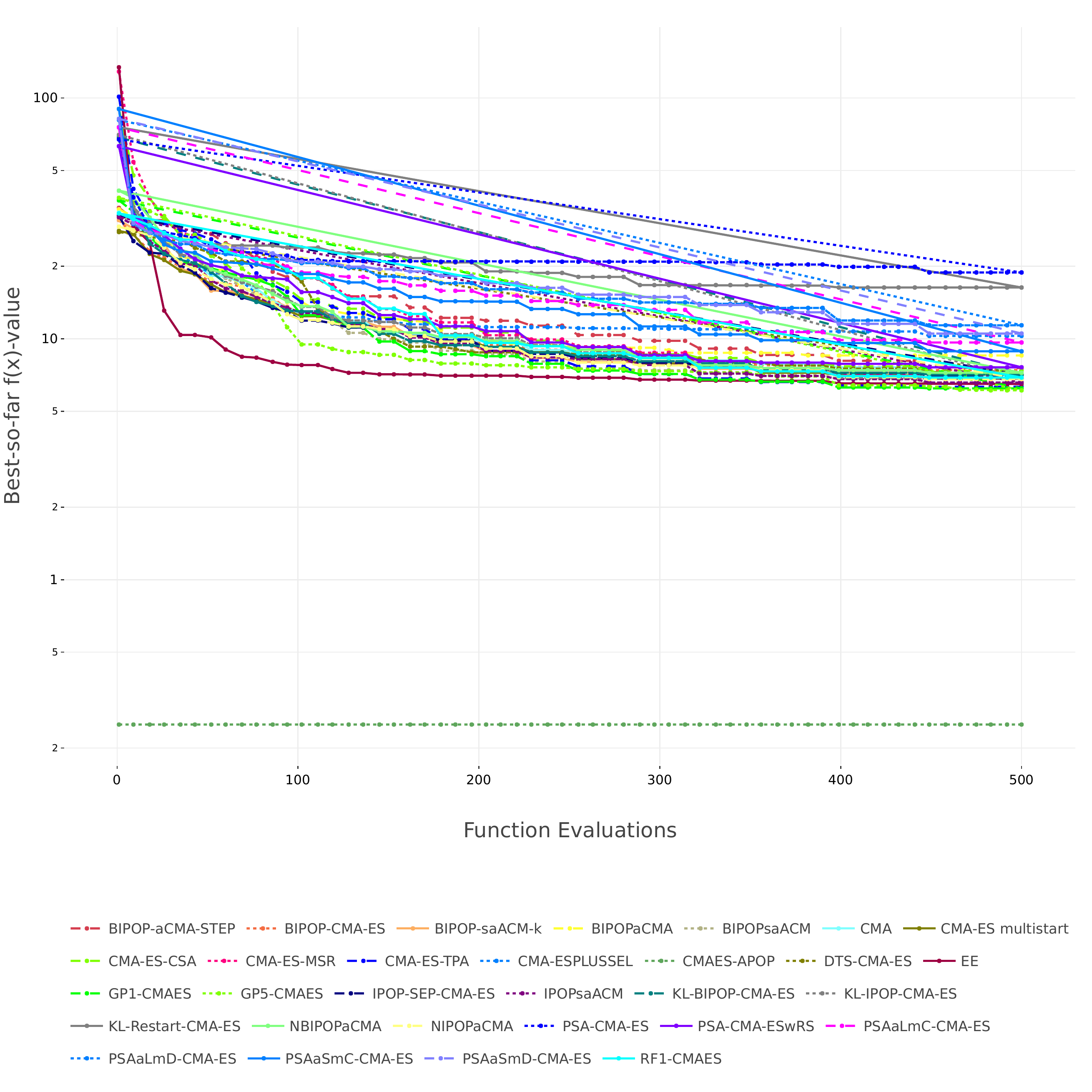}
  \caption{$f_{19}$, $D = 20$. Here \texttt{EXPLO2} is indicated by \texttt{EE}.}
  \label{fig:IOHanalyzer/EE_mul25par32lamLin/EE_mul25par32lamLin_f19d20allCMA}
\end{figure}

\begin{figure}[h]
  \centering
  \includegraphics[trim = 0mm 0mm 0mm 0mm, clip, width=\columnwidth,keepaspectratio]{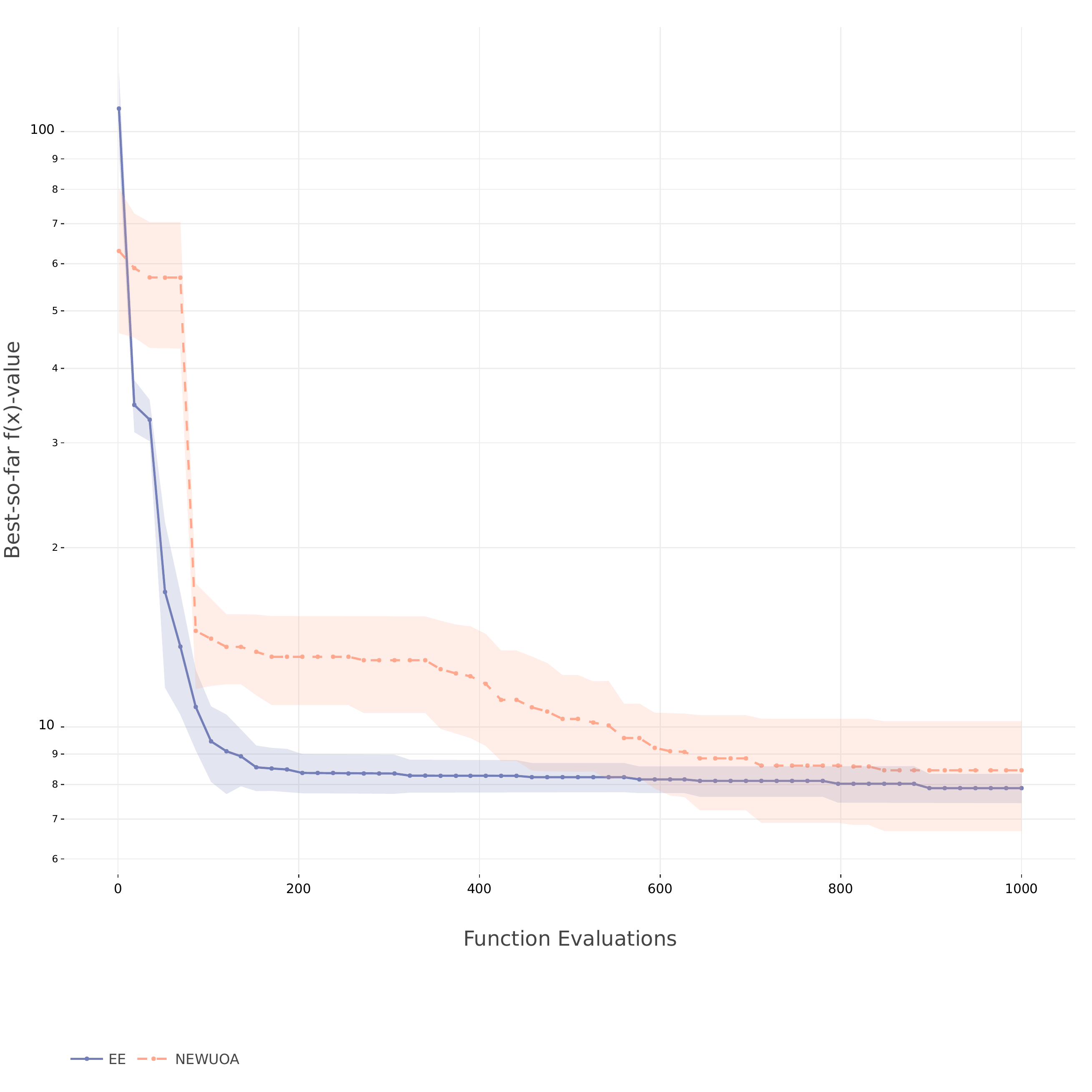}
  \caption{$f_{19}$, $D = 40$}
  \label{fig:IOHanalyzer/EE_mul25par32lamLin/EE_mul25par32lamLin_f19d40}
\end{figure}

\begin{figure}[h]
  \centering
  \includegraphics[trim = 0mm 0mm 0mm 0mm, clip, width=\columnwidth,keepaspectratio]{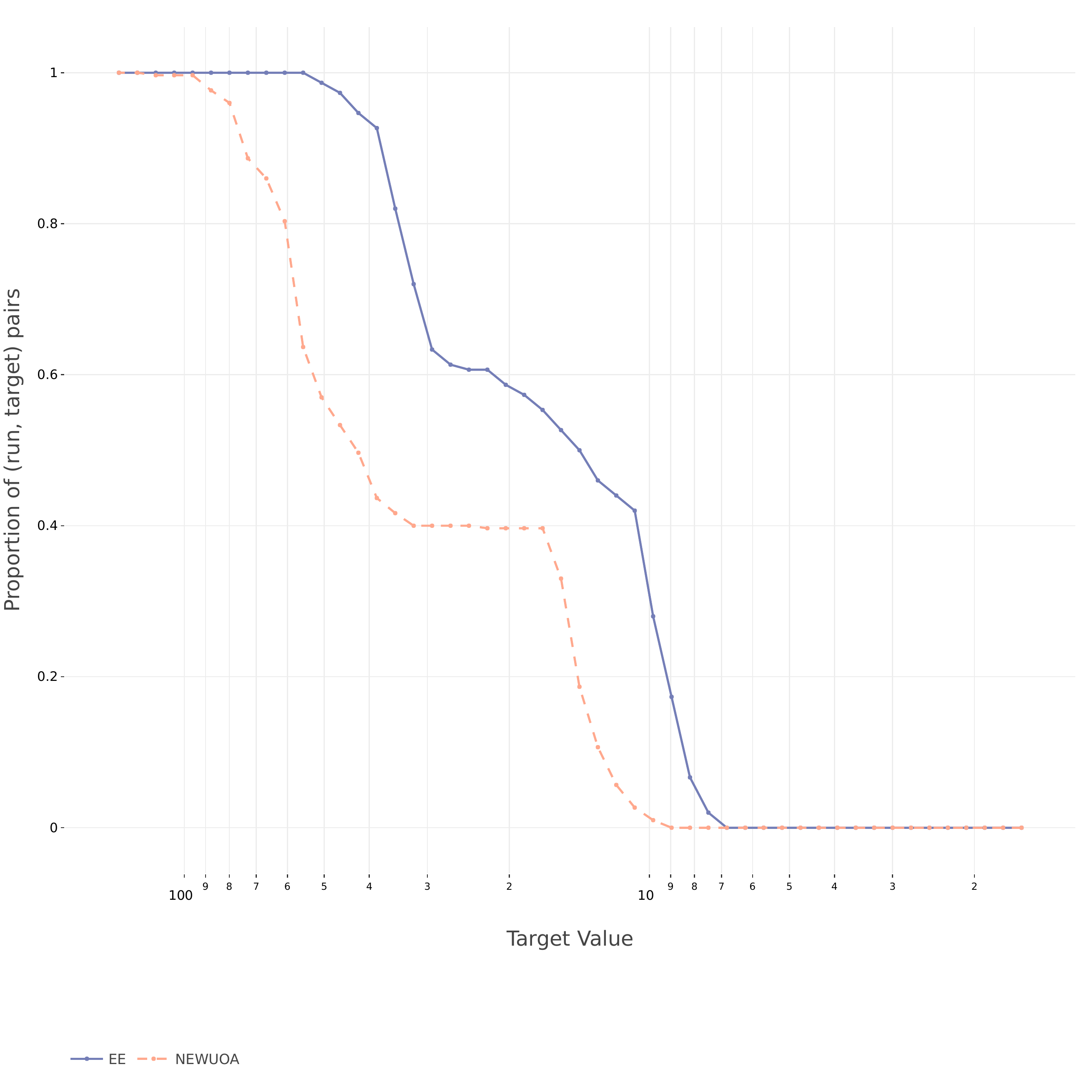}
  \caption{$f_{19}$, $D = 40$. Here \texttt{EXPLO2} is indicated by \texttt{EE}.}
  \label{fig:IOHanalyzer/EE_mul25par32lamLin/EE_mul25par32lamLin_f19d40ECDF}
\end{figure}

\begin{figure}[h]
  \centering
  \includegraphics[trim = 0mm 0mm 0mm 0mm, clip, width=\columnwidth,keepaspectratio]{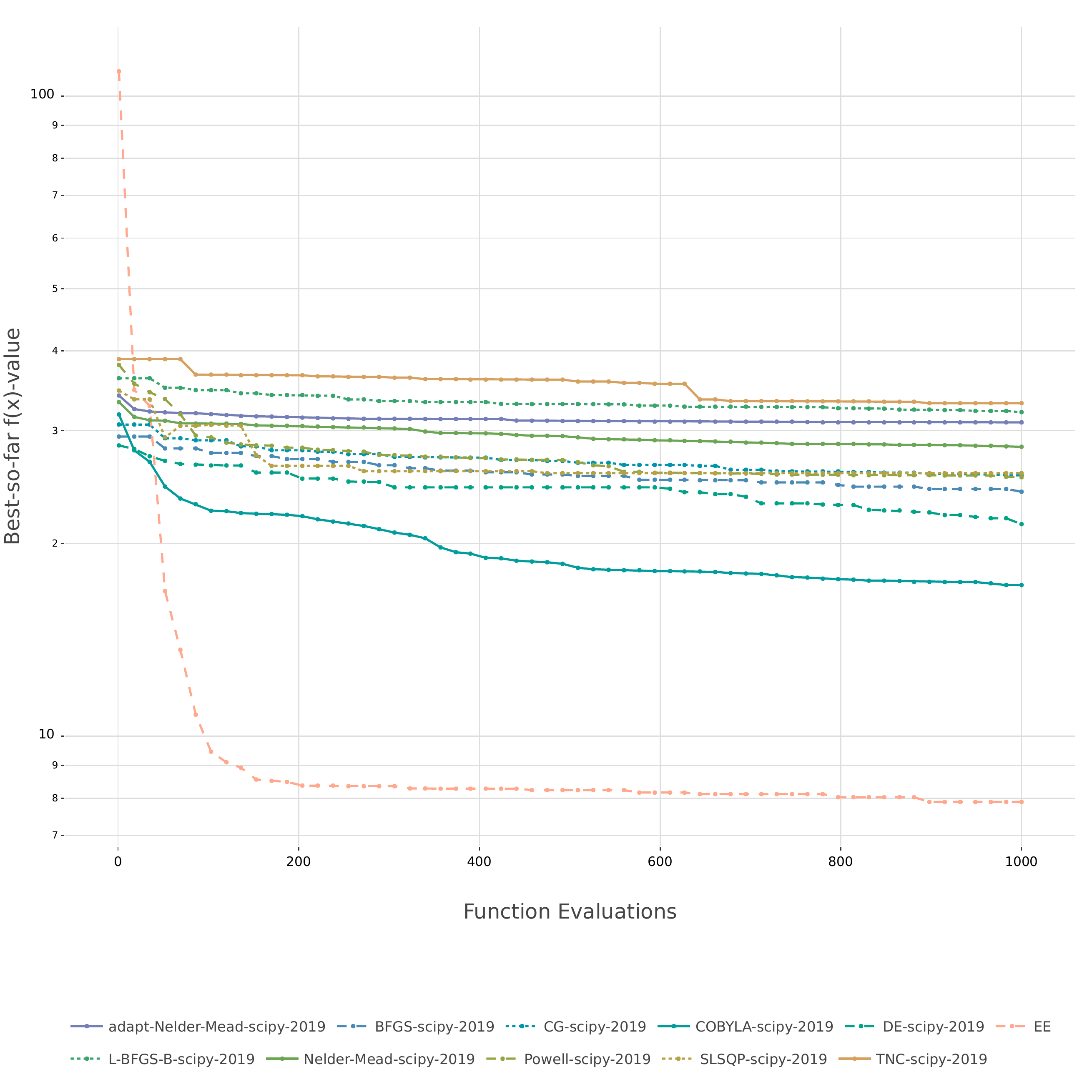}
  \caption{$f_{19}$, $D = 40$. Here \texttt{EXPLO2} is indicated by \texttt{EE}.}
  \label{fig:IOHanalyzer/EE_mul25par32lamLin/EE_mul25par32lamLin_f19d40vScipy}
\end{figure}

\begin{figure}[h]
  \centering
  \includegraphics[trim = 0mm 0mm 0mm 0mm, clip, width=\columnwidth,keepaspectratio]{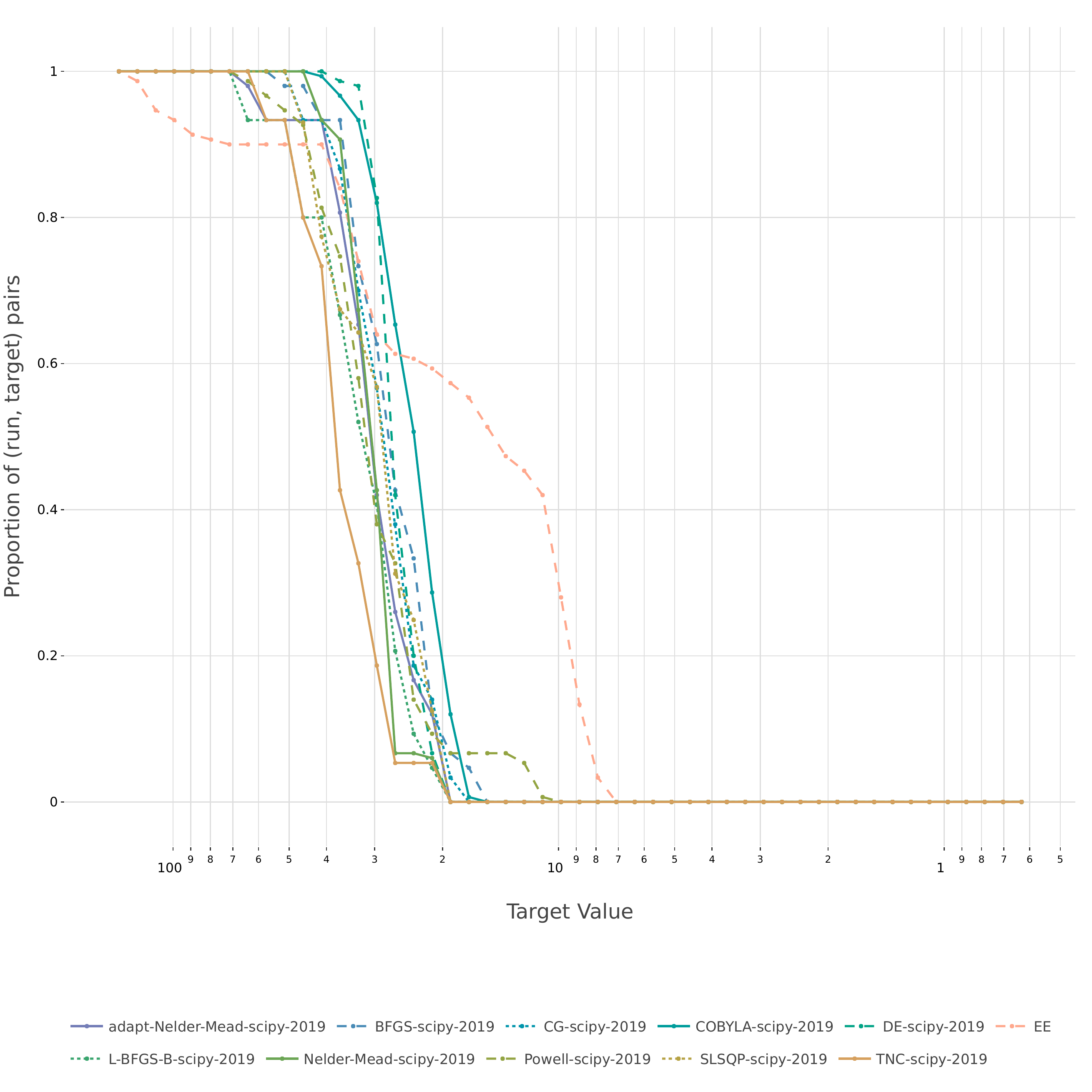}
  \caption{$f_{19}$, $D = 40$. Here \texttt{EXPLO2} is indicated by \texttt{EE}.}
  \label{fig:IOHanalyzer/EE_mul25par32lamLin/EE_mul25par32lamLin_f19d40vScipyECDF}
\end{figure}

\clearpage

\section{\label{sec:ERTscaling}Scaling of runtime with dimension}

The \emph{expected runtime} (ERT), used in the figures and tables,
depends on a given target function value, $\ftarget=\fopt+\Df$, and is
computed over all relevant trials as the number of function
evaluations executed during each trial while the best function value
did not reach \ftarget, summed over all trials and divided by the
number of trials that actually reached \ftarget\
\cite{hansen2012exp,price1997dev}. Statistical significance
is tested with the rank-sum test for a given target $\Delta\ftarget$
using, for each trial,
either the number of needed function evaluations to reach
$\Delta\ftarget$ (inverted and multiplied by $-1$), or, if the target
was not reached, the best $\Df$-value achieved, measured only up to
the smallest number of overall function evaluations for any
unsuccessful trial under consideration.



\begin{figure*}
\centering
\begin{tabular}{@{}c@{}c@{}c@{}c@{}}
\includegraphics[width=0.238\textwidth]{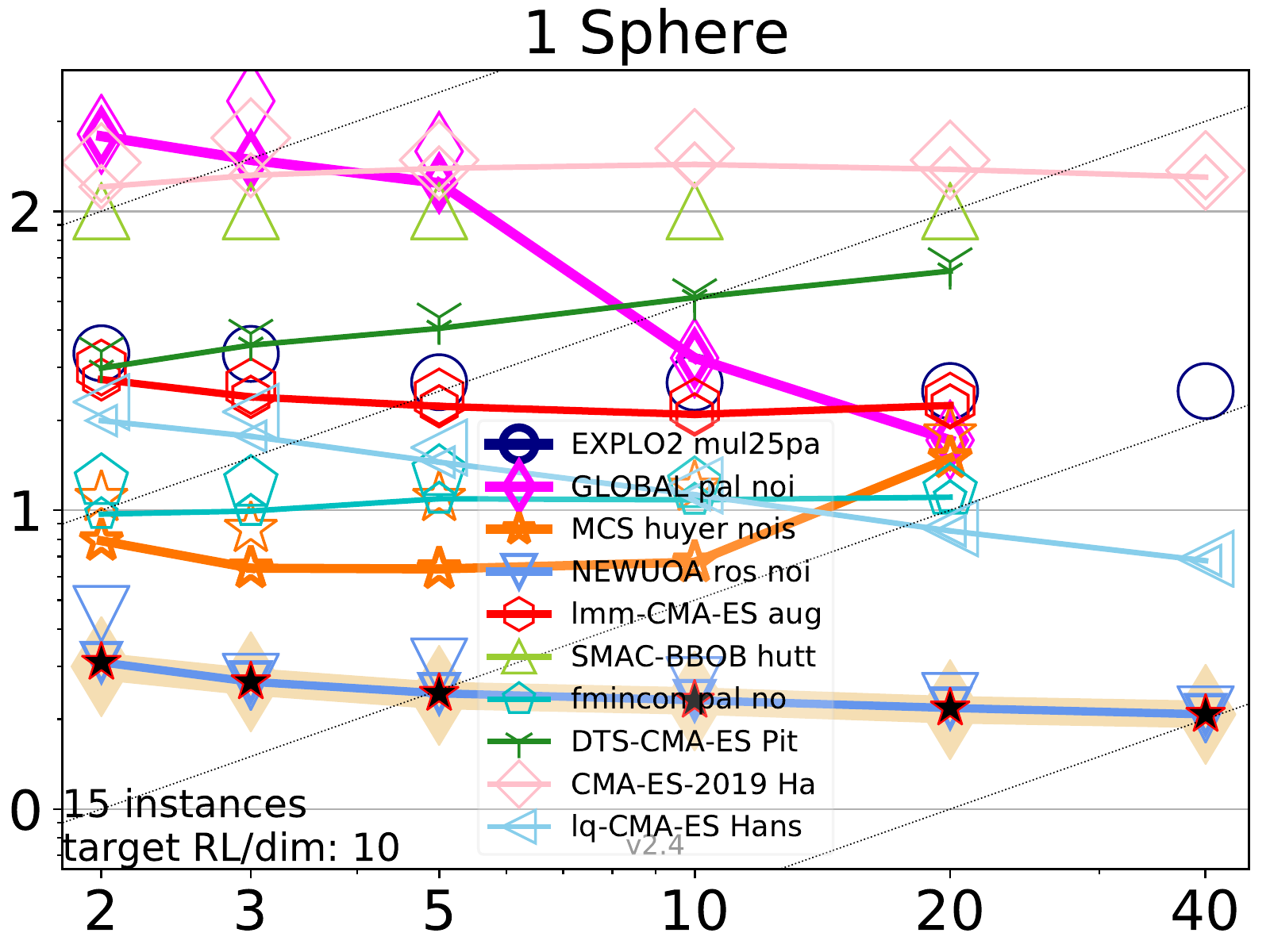}&
\includegraphics[width=0.238\textwidth]{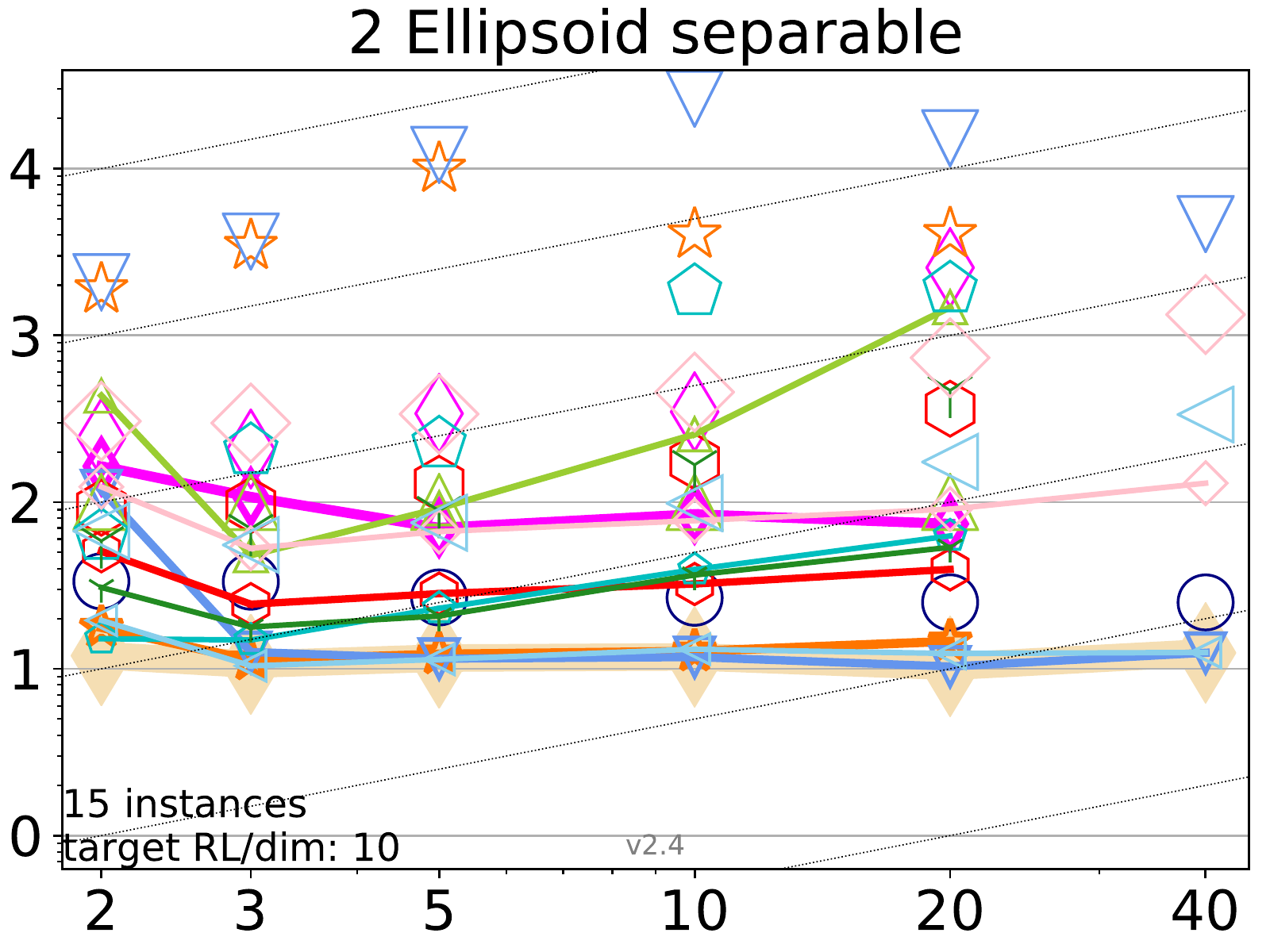}&
\includegraphics[width=0.238\textwidth]{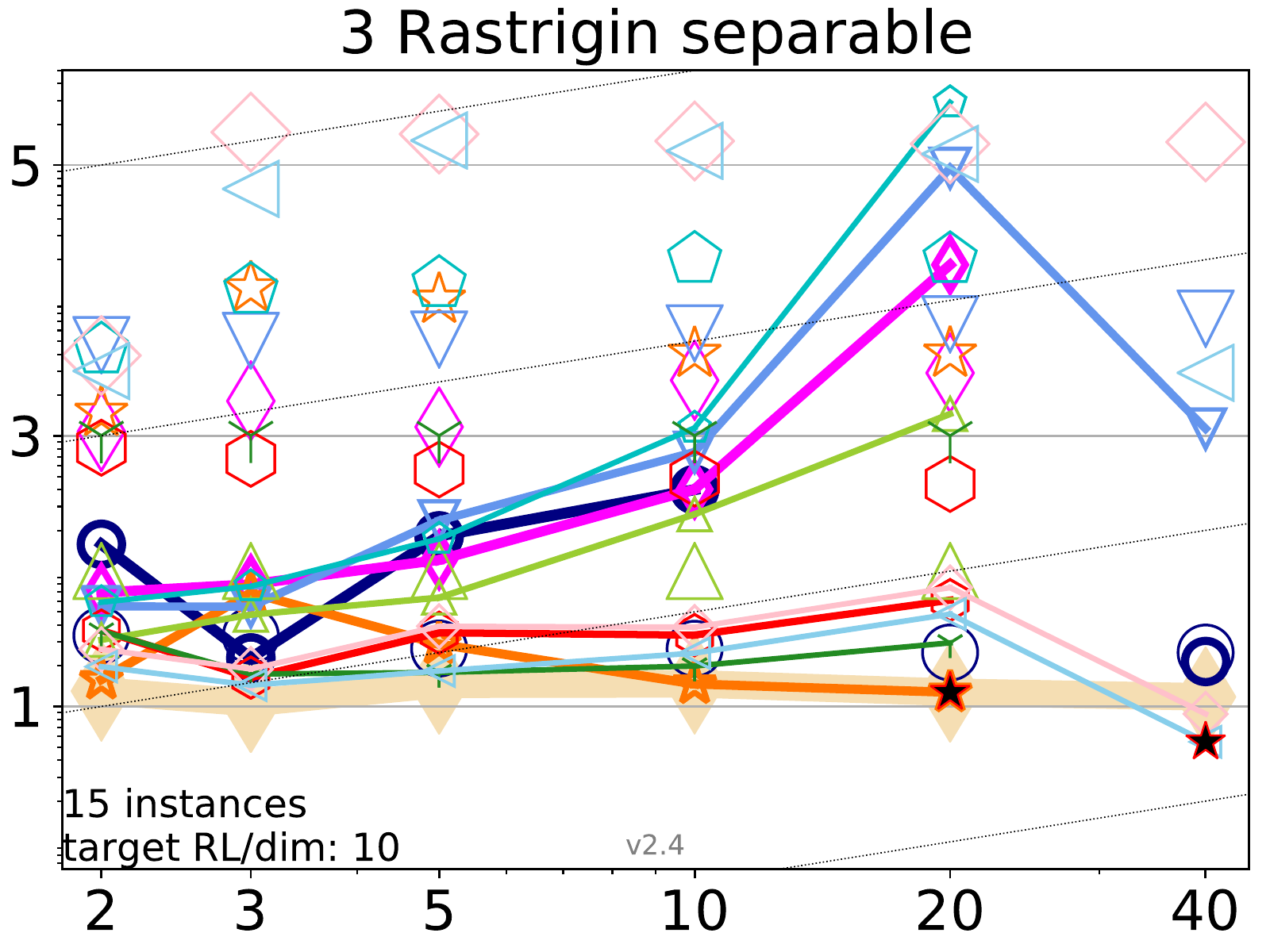}&
\includegraphics[width=0.238\textwidth]{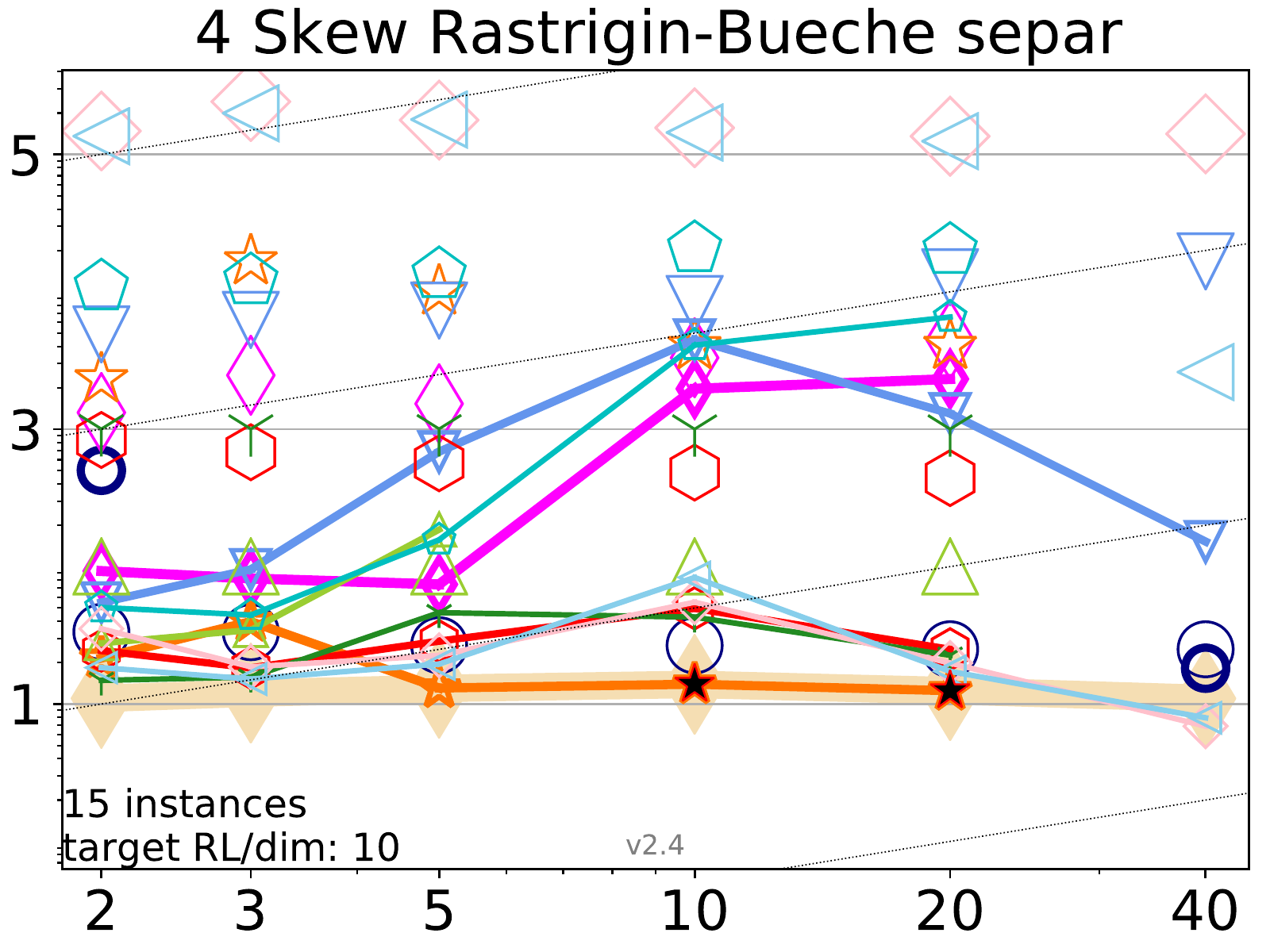}\\
\includegraphics[width=0.238\textwidth]{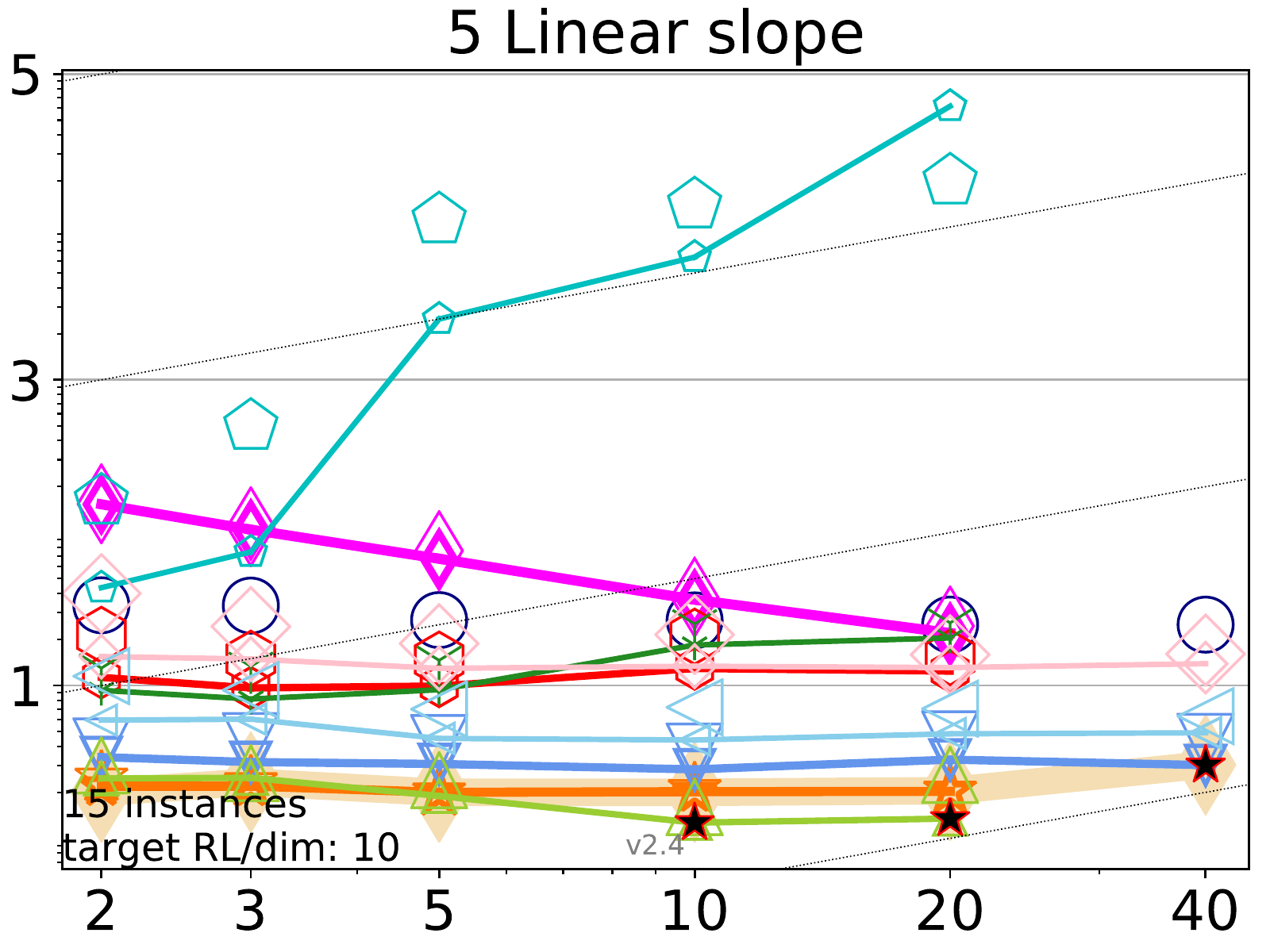}&
\includegraphics[width=0.238\textwidth]{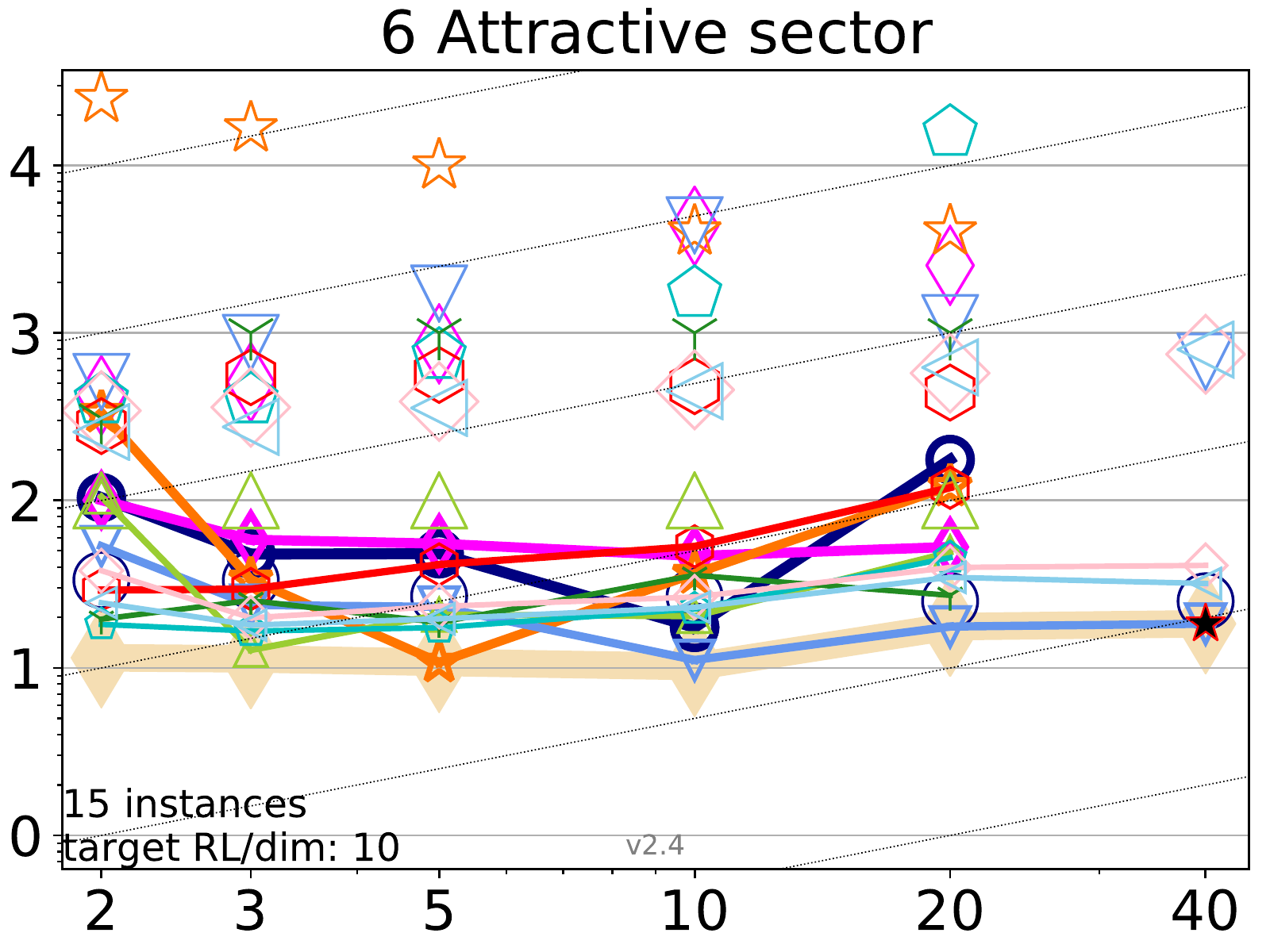}&
\includegraphics[width=0.238\textwidth]{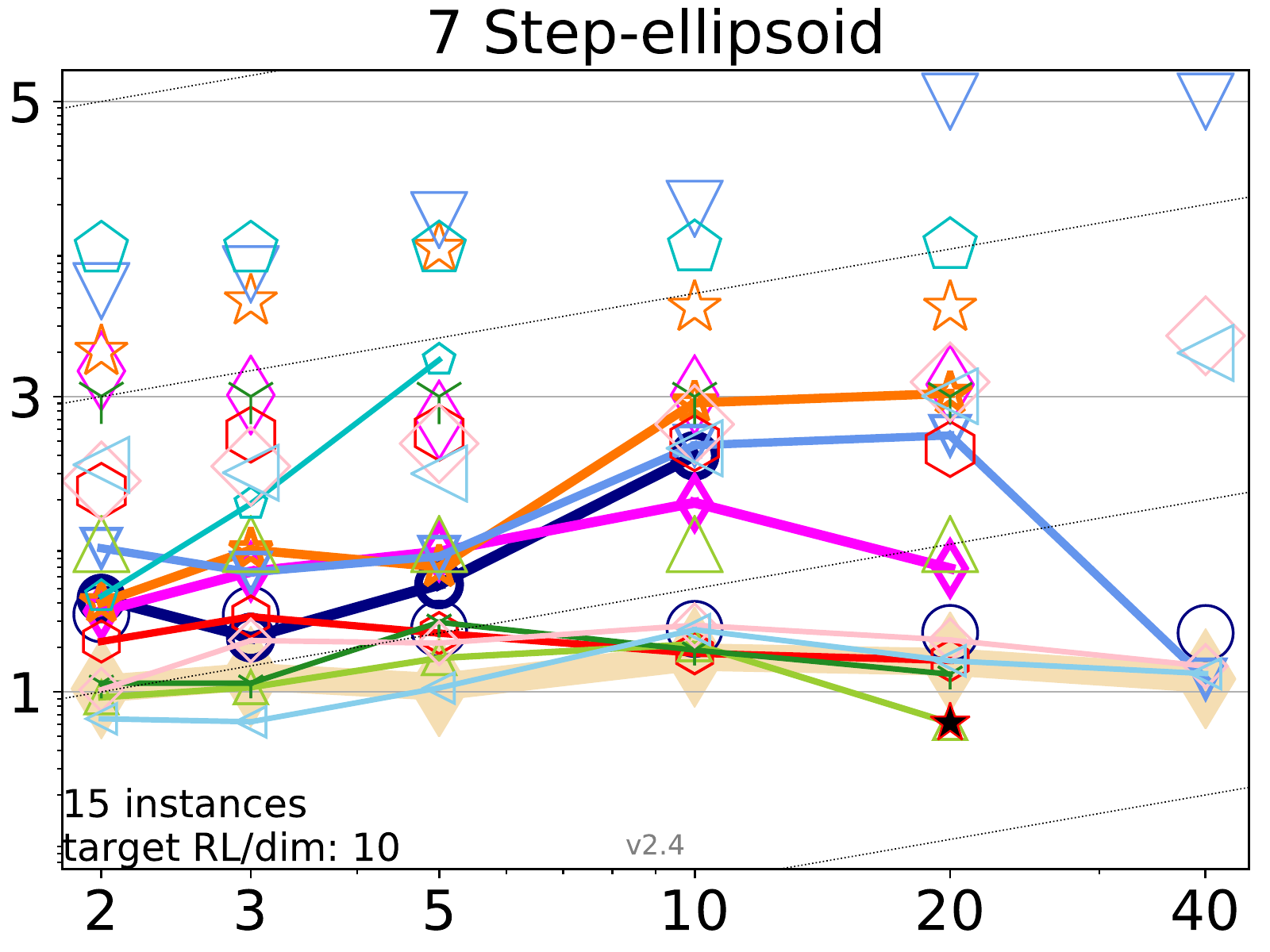}&
\includegraphics[width=0.238\textwidth]{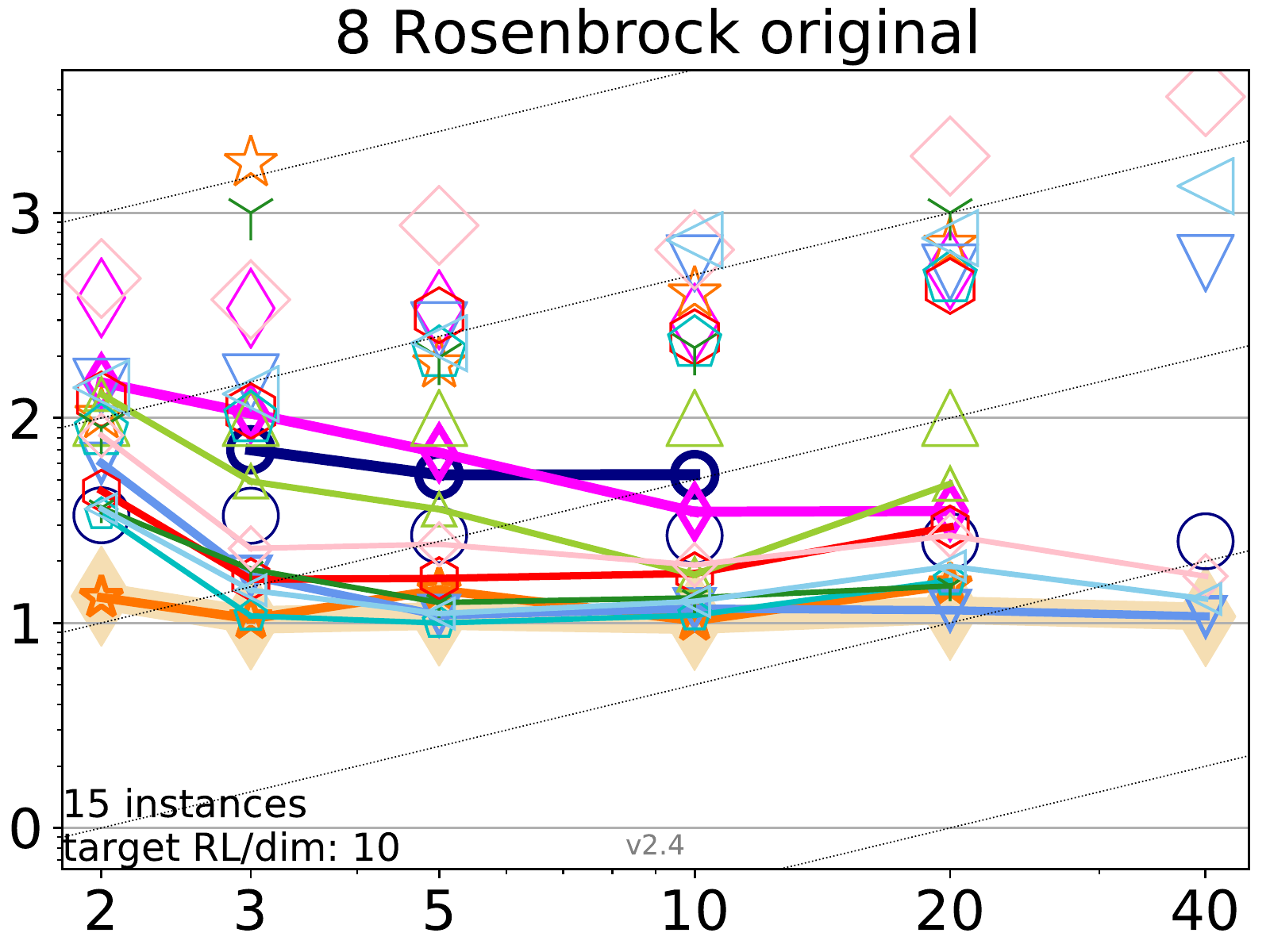}\\
\includegraphics[width=0.238\textwidth]{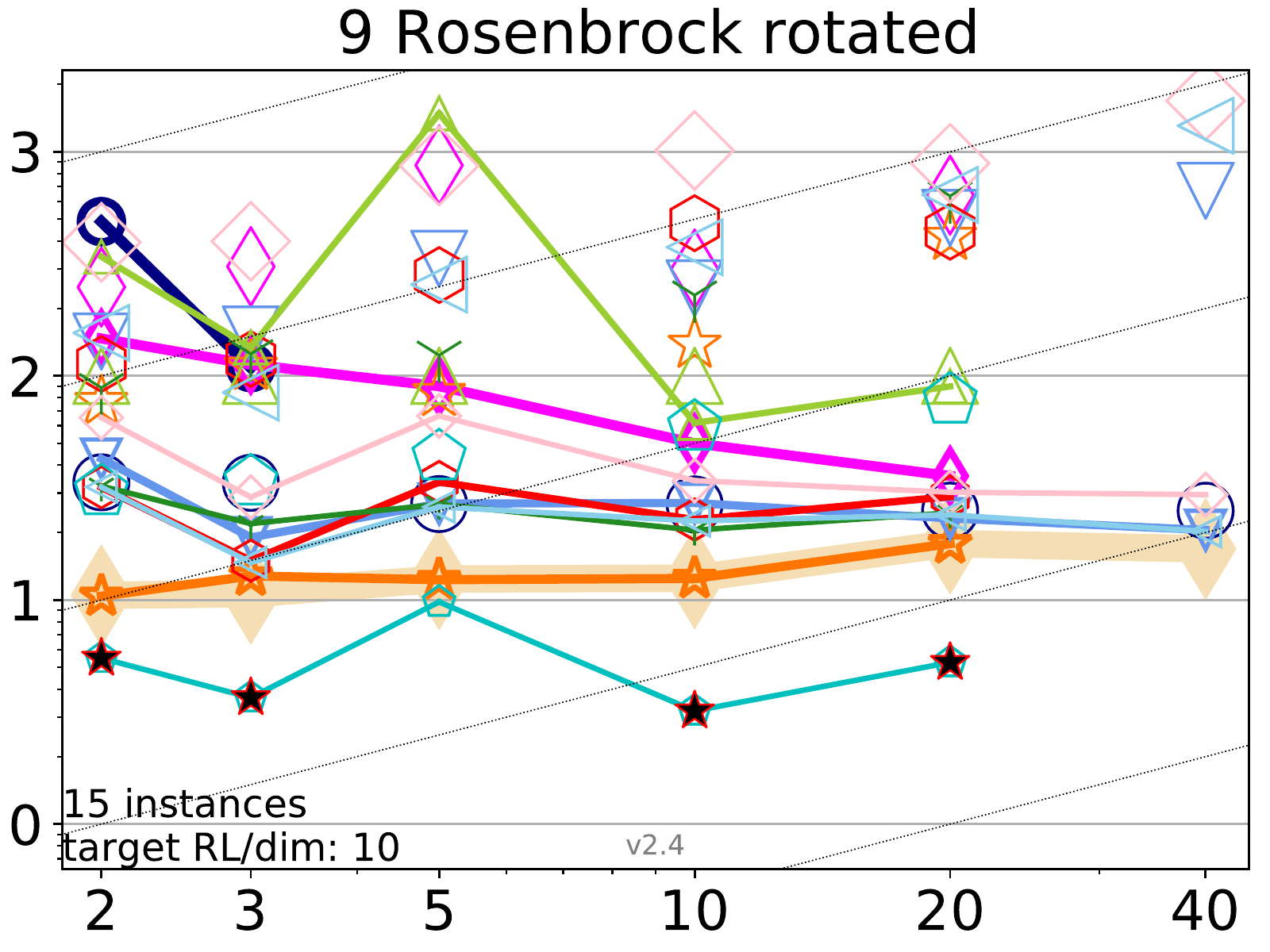}&
\includegraphics[width=0.238\textwidth]{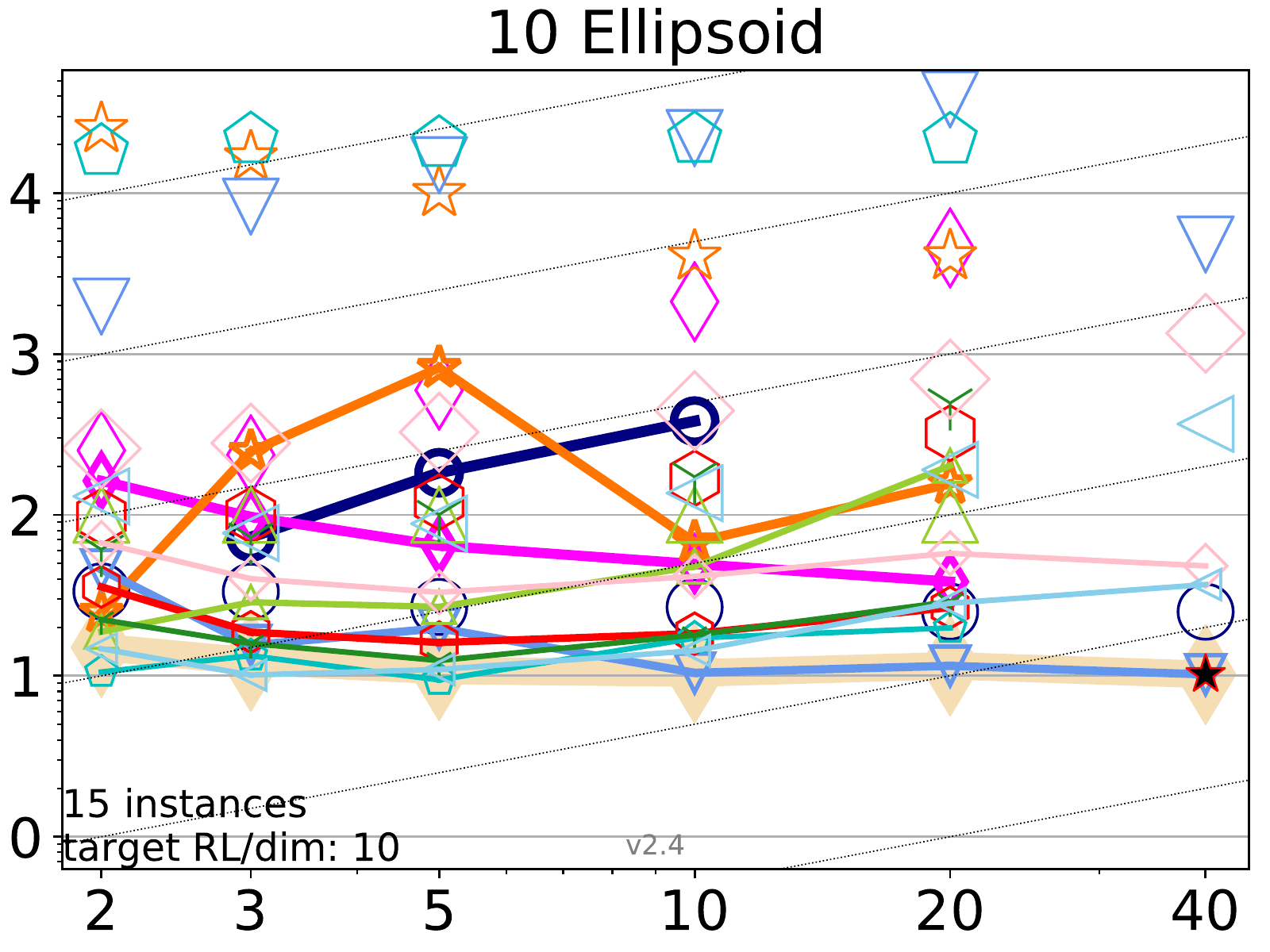}&
\includegraphics[width=0.238\textwidth]{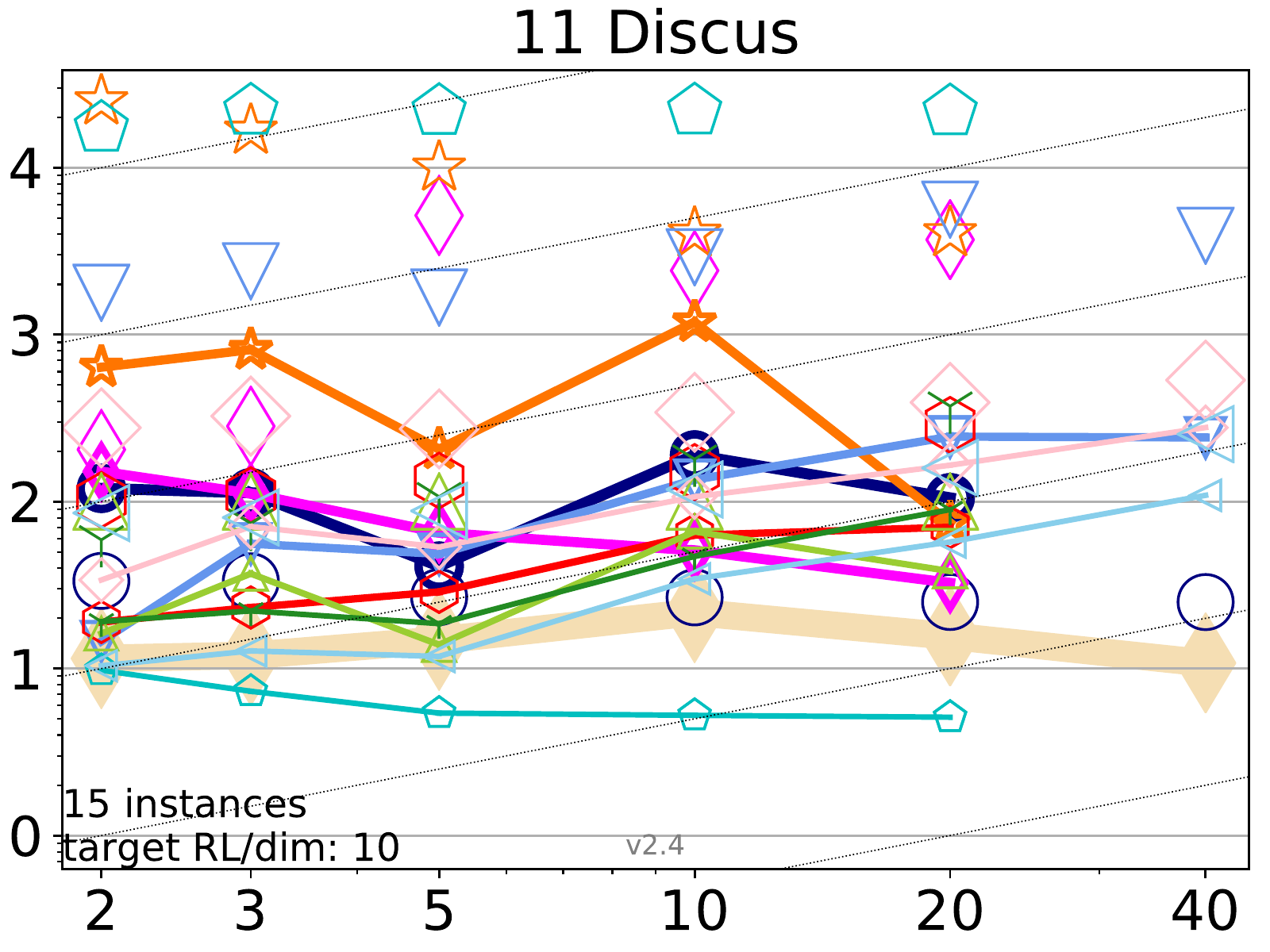}&
\includegraphics[width=0.238\textwidth]{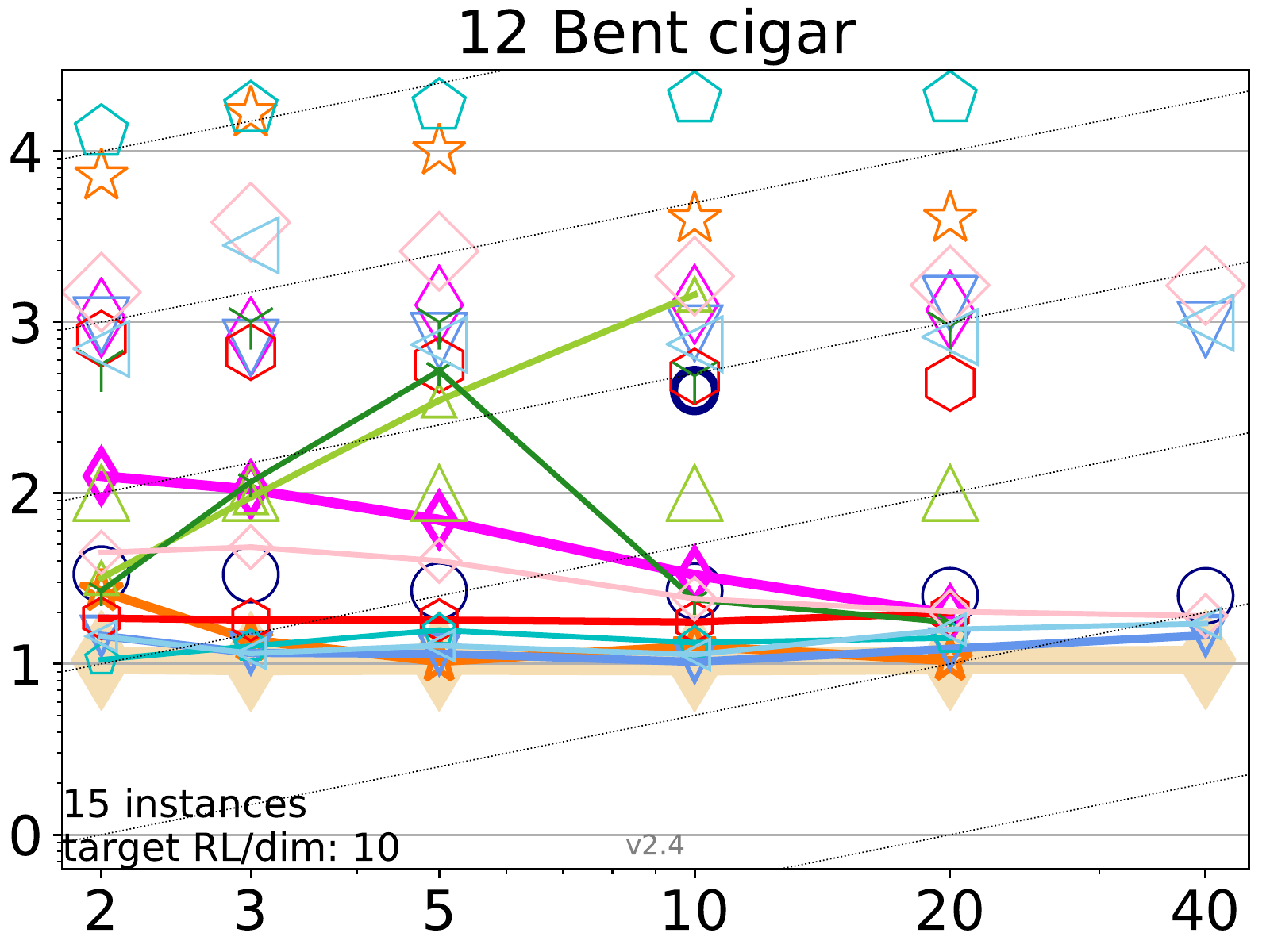}\\
\includegraphics[width=0.238\textwidth]{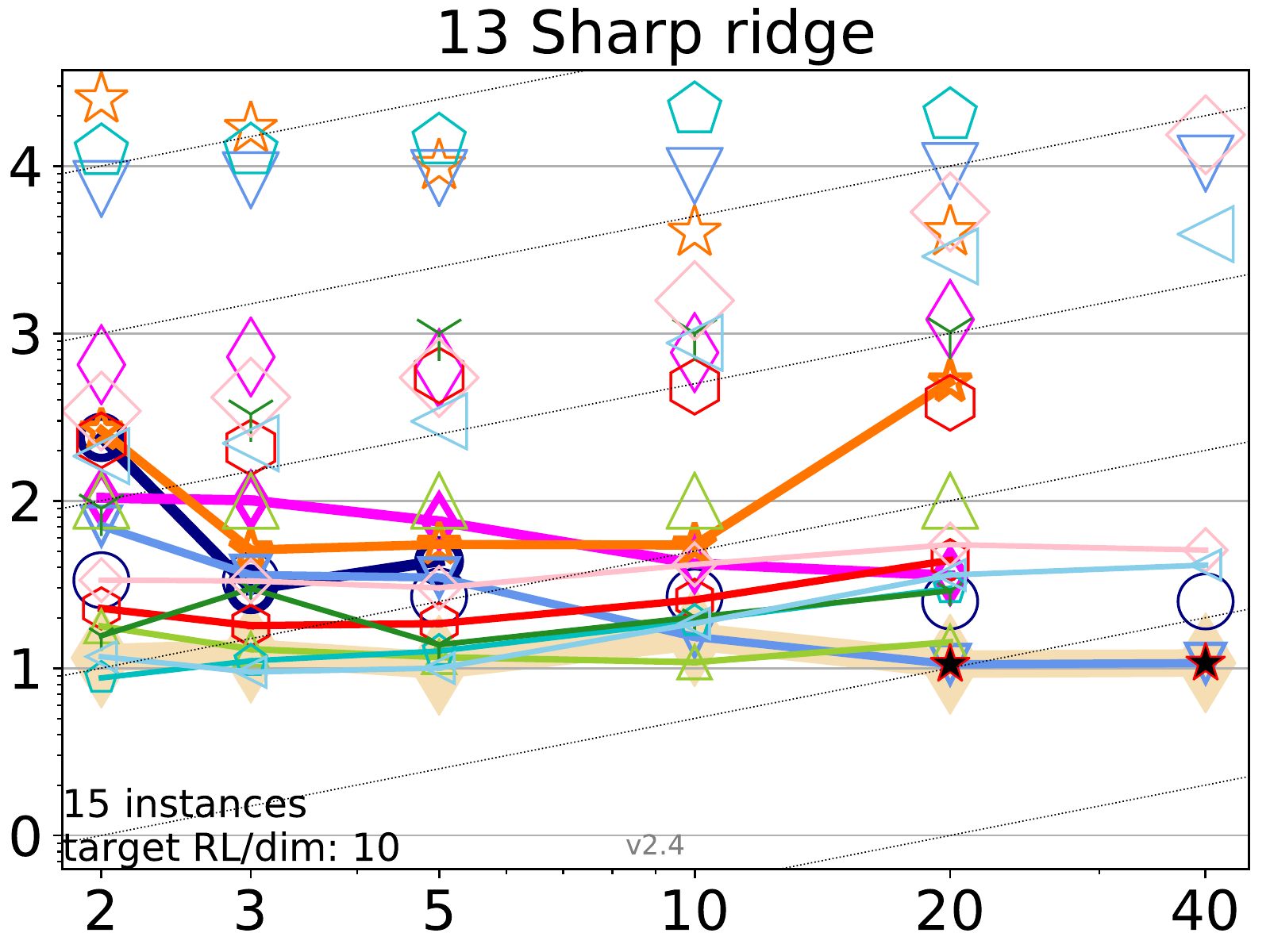}&
\includegraphics[width=0.238\textwidth]{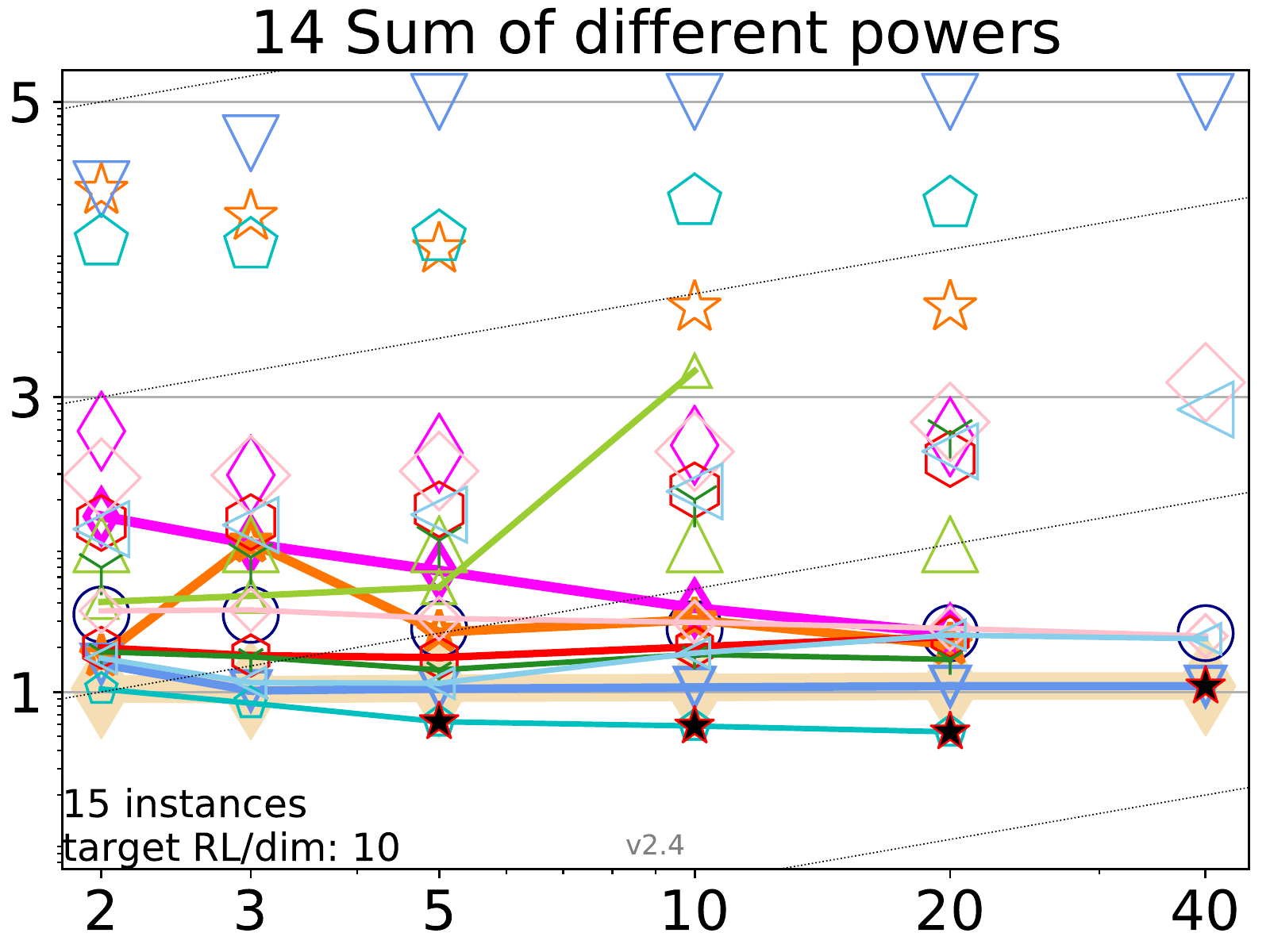}&
\includegraphics[width=0.238\textwidth]{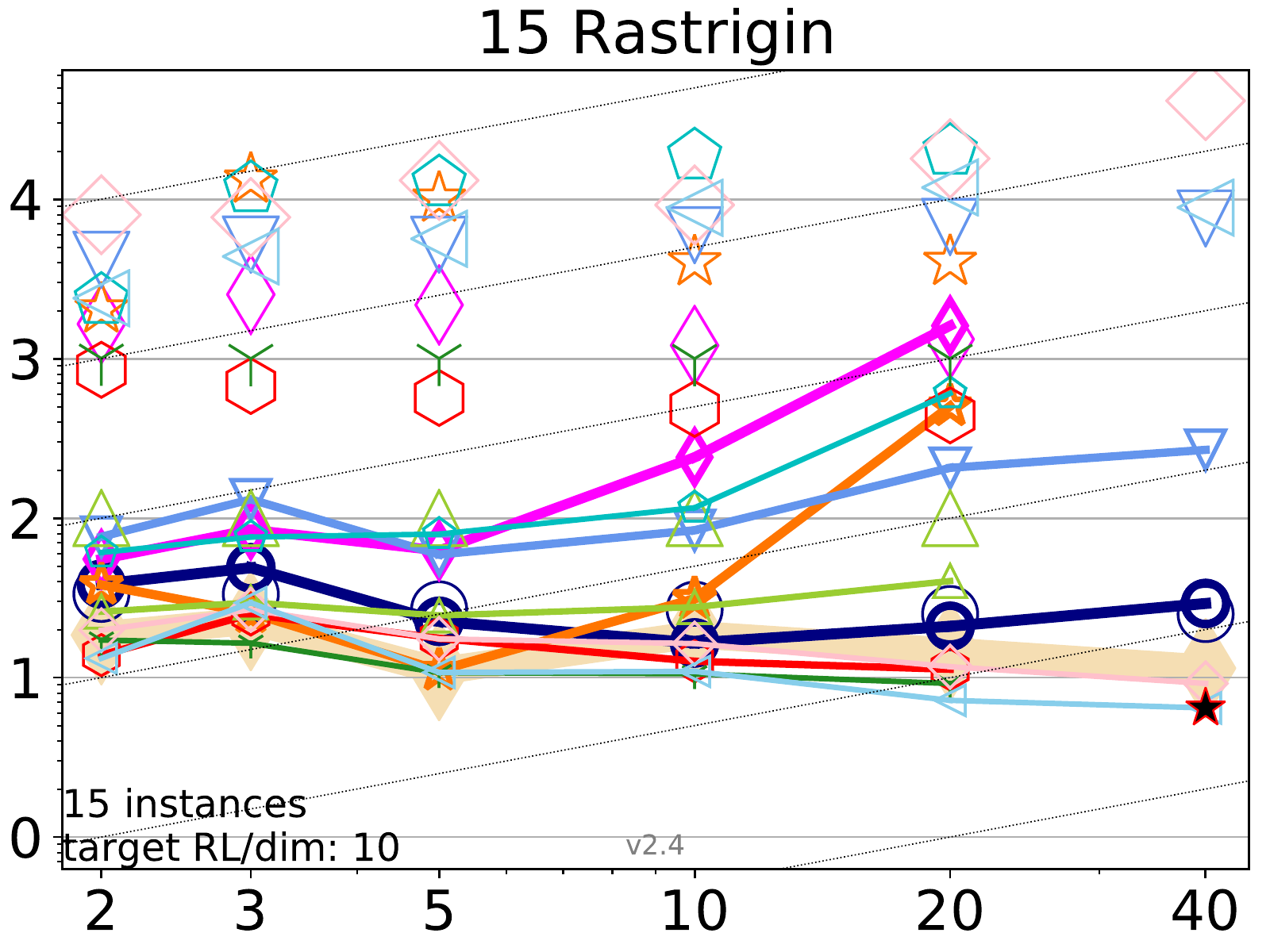}&
\includegraphics[width=0.238\textwidth]{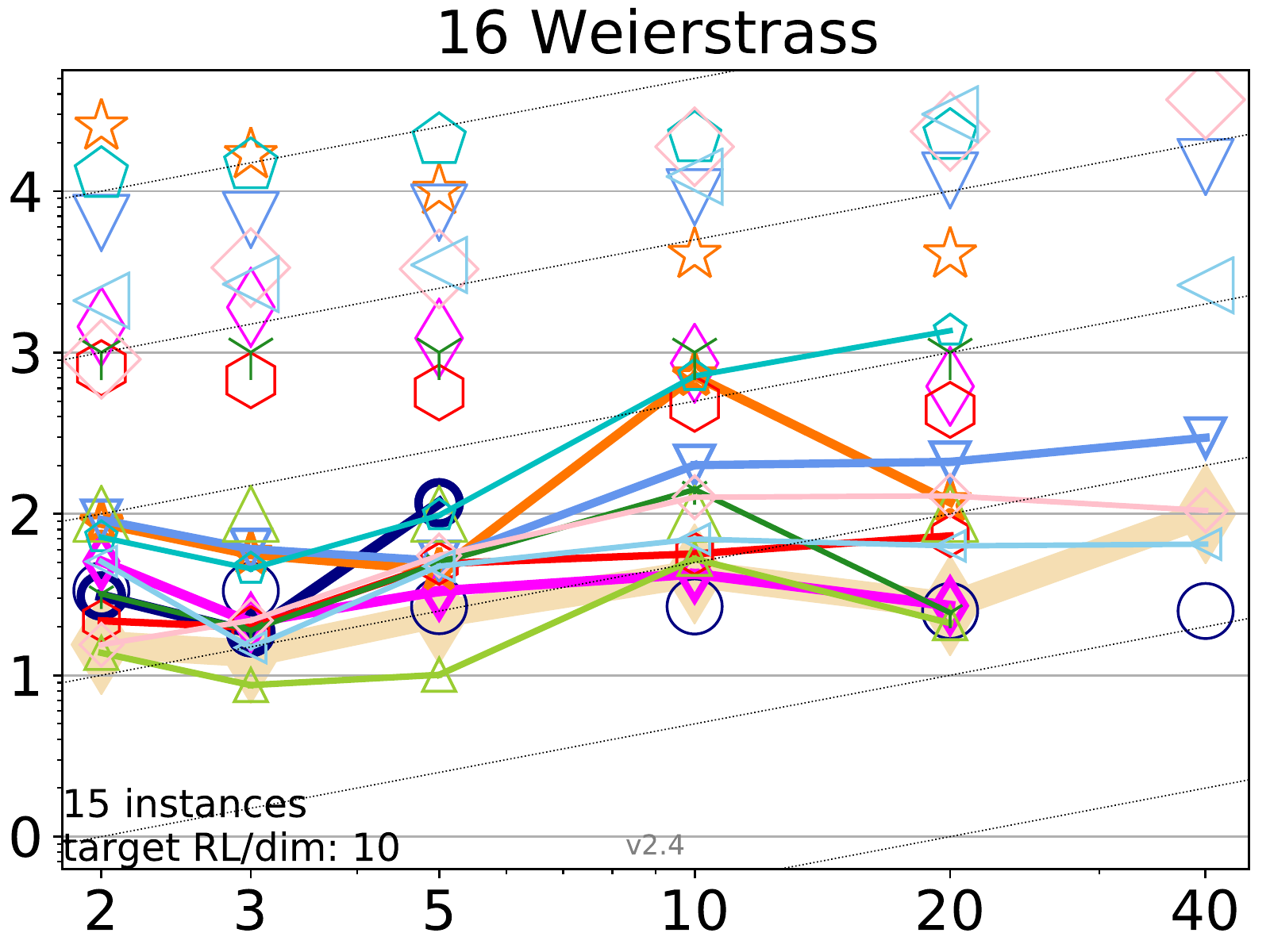}\\
\includegraphics[width=0.238\textwidth]{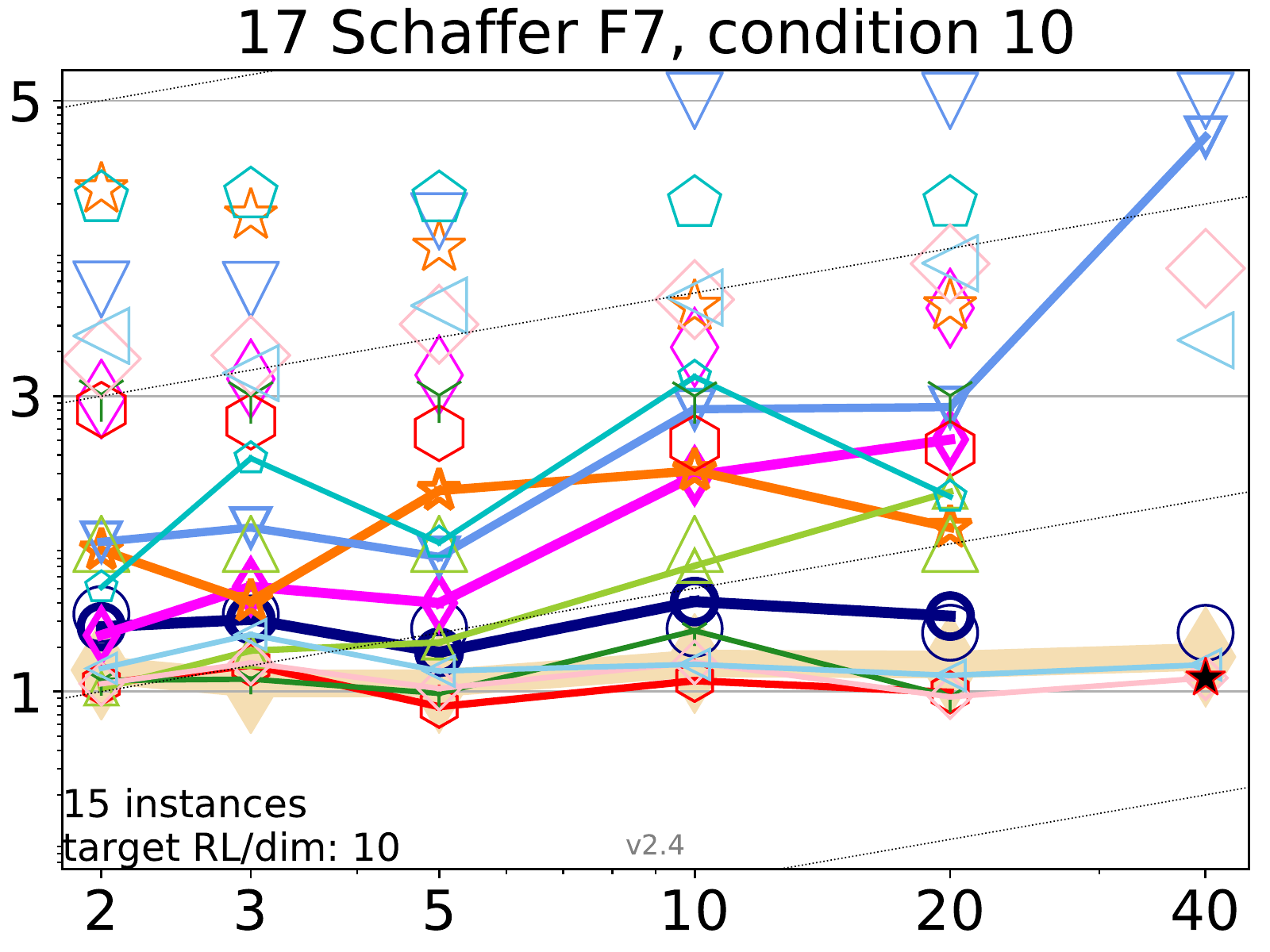}&
\includegraphics[width=0.238\textwidth]{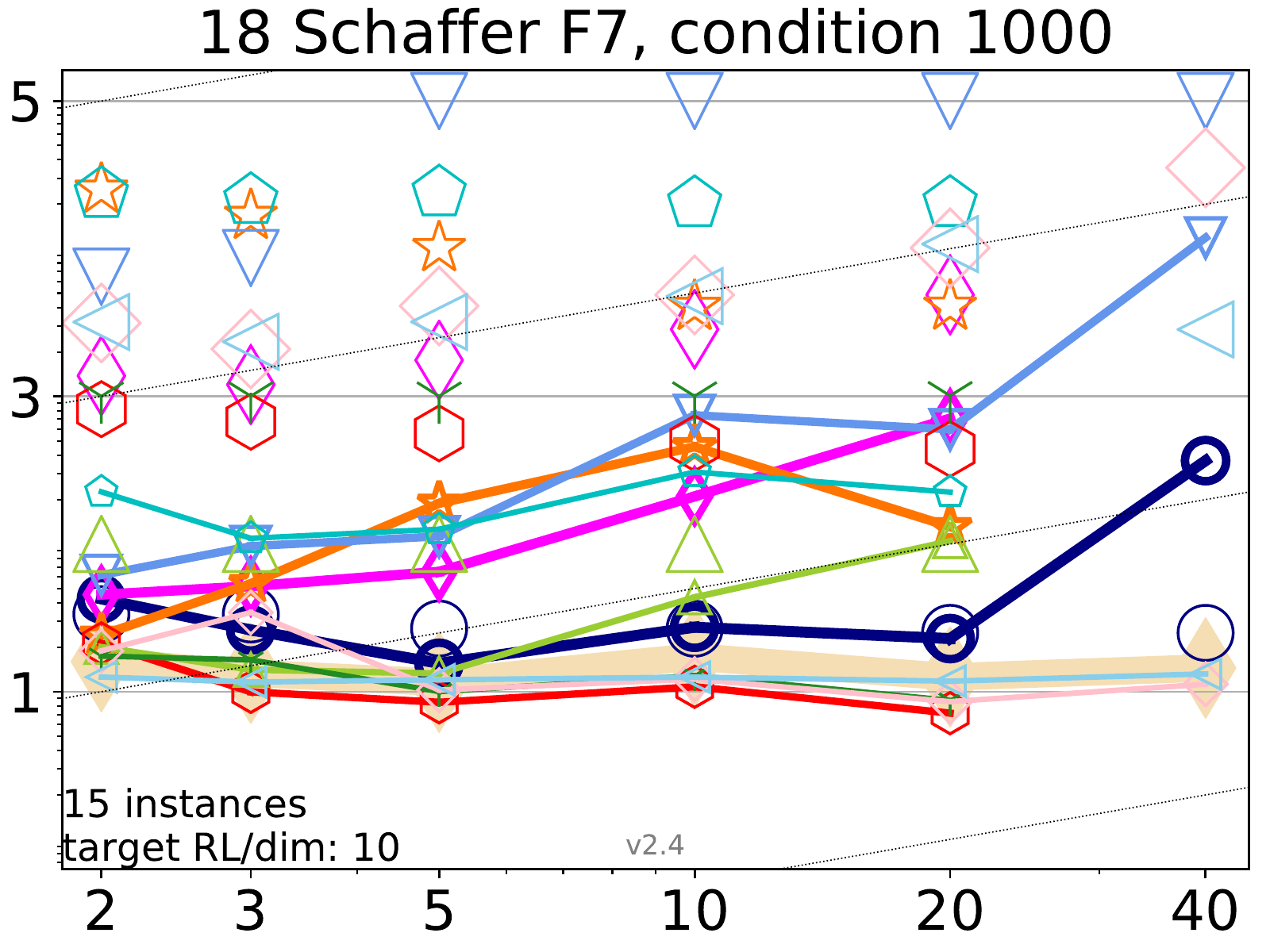}&
\includegraphics[width=0.238\textwidth]{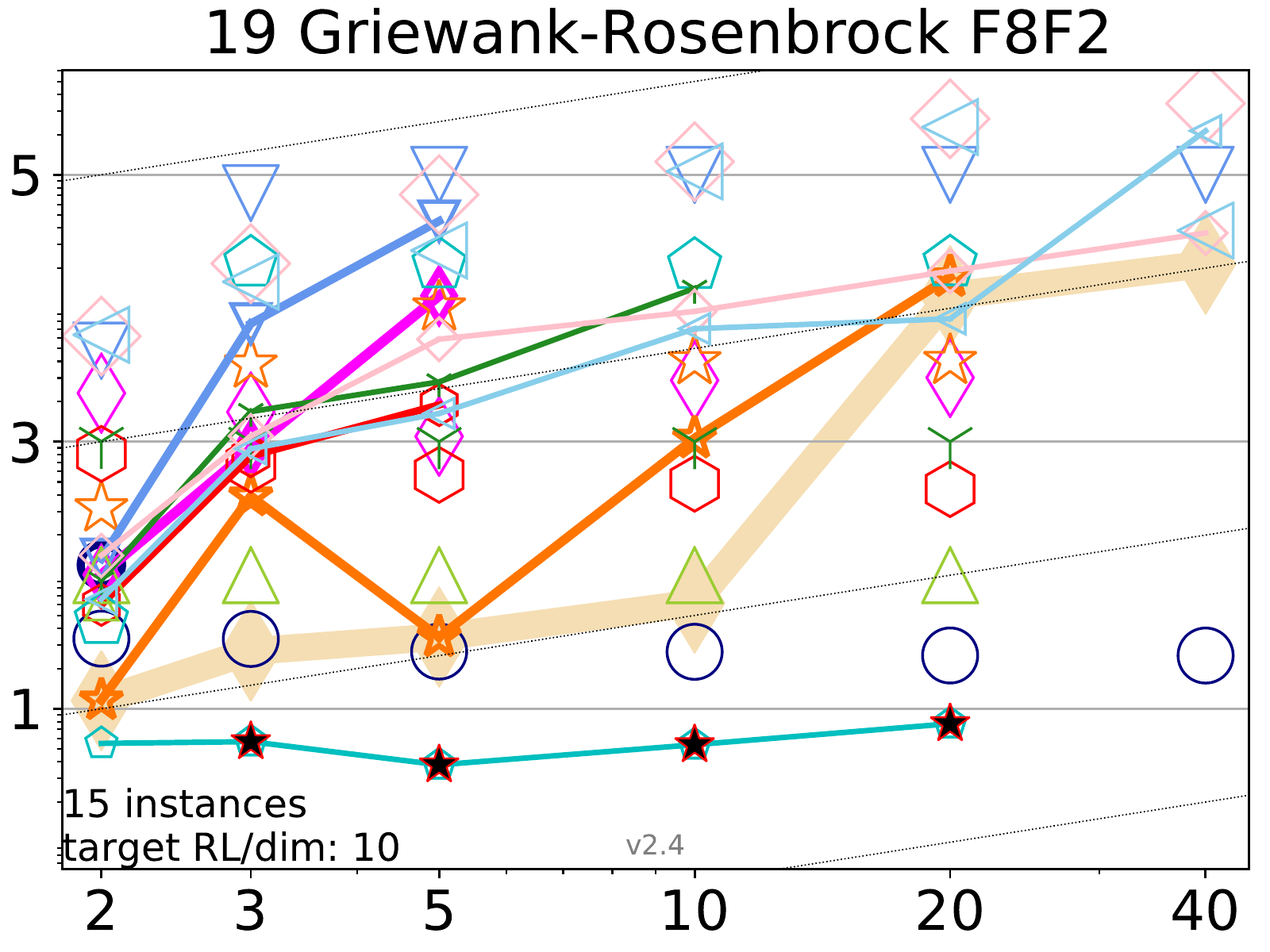}&
\includegraphics[width=0.238\textwidth]{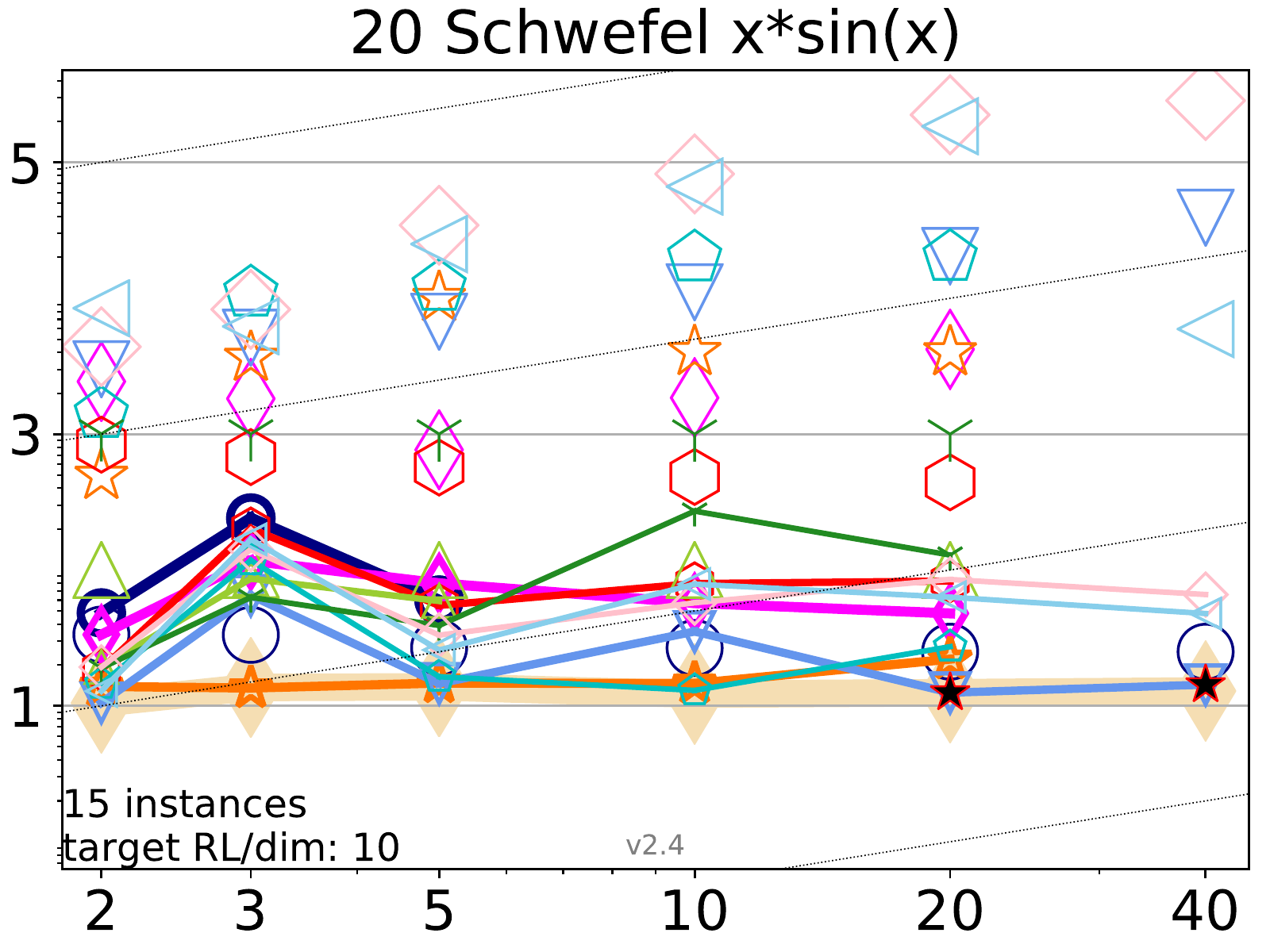}\\
\includegraphics[width=0.238\textwidth]{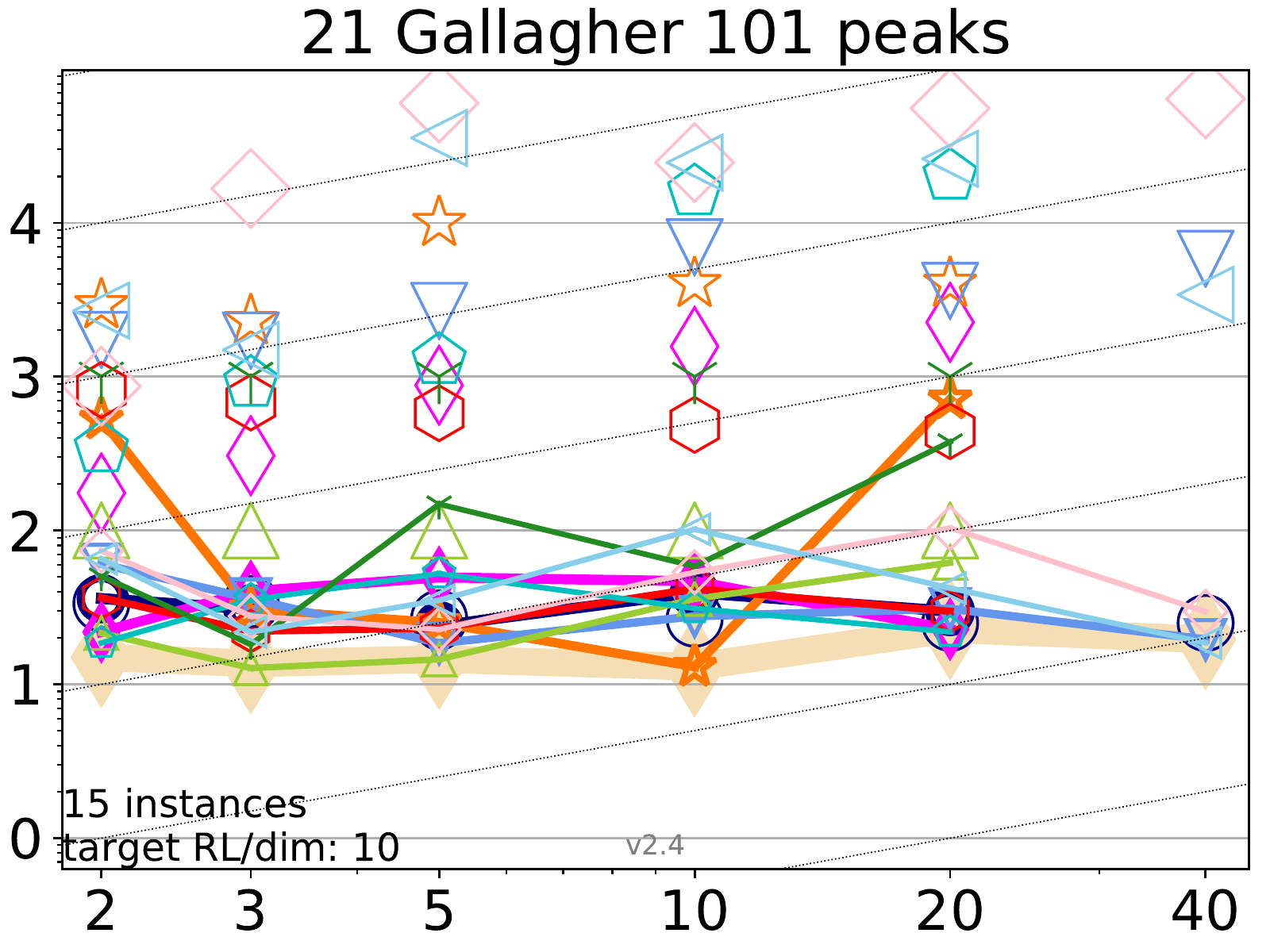}&
\includegraphics[width=0.238\textwidth]{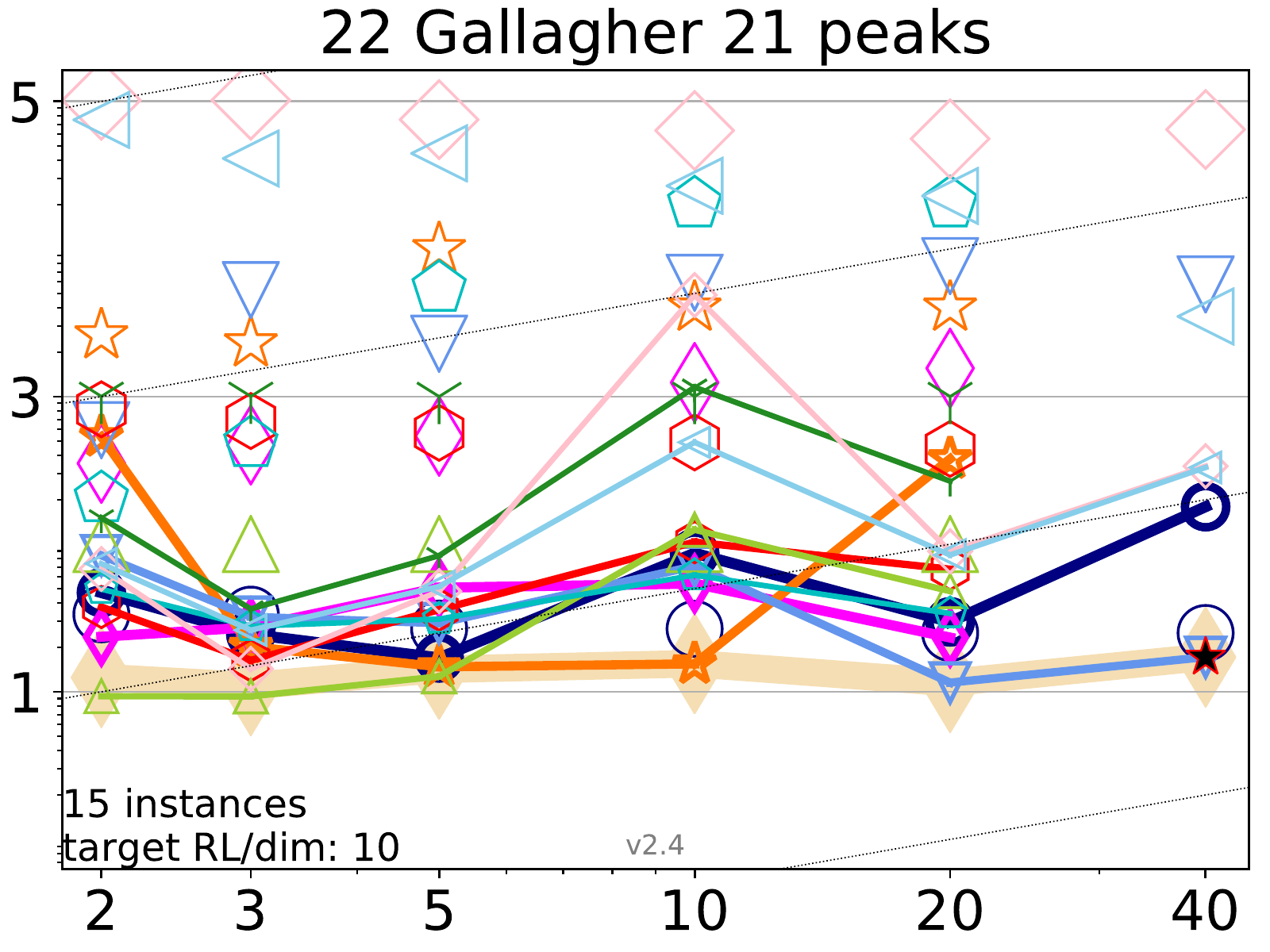}&
\includegraphics[width=0.238\textwidth]{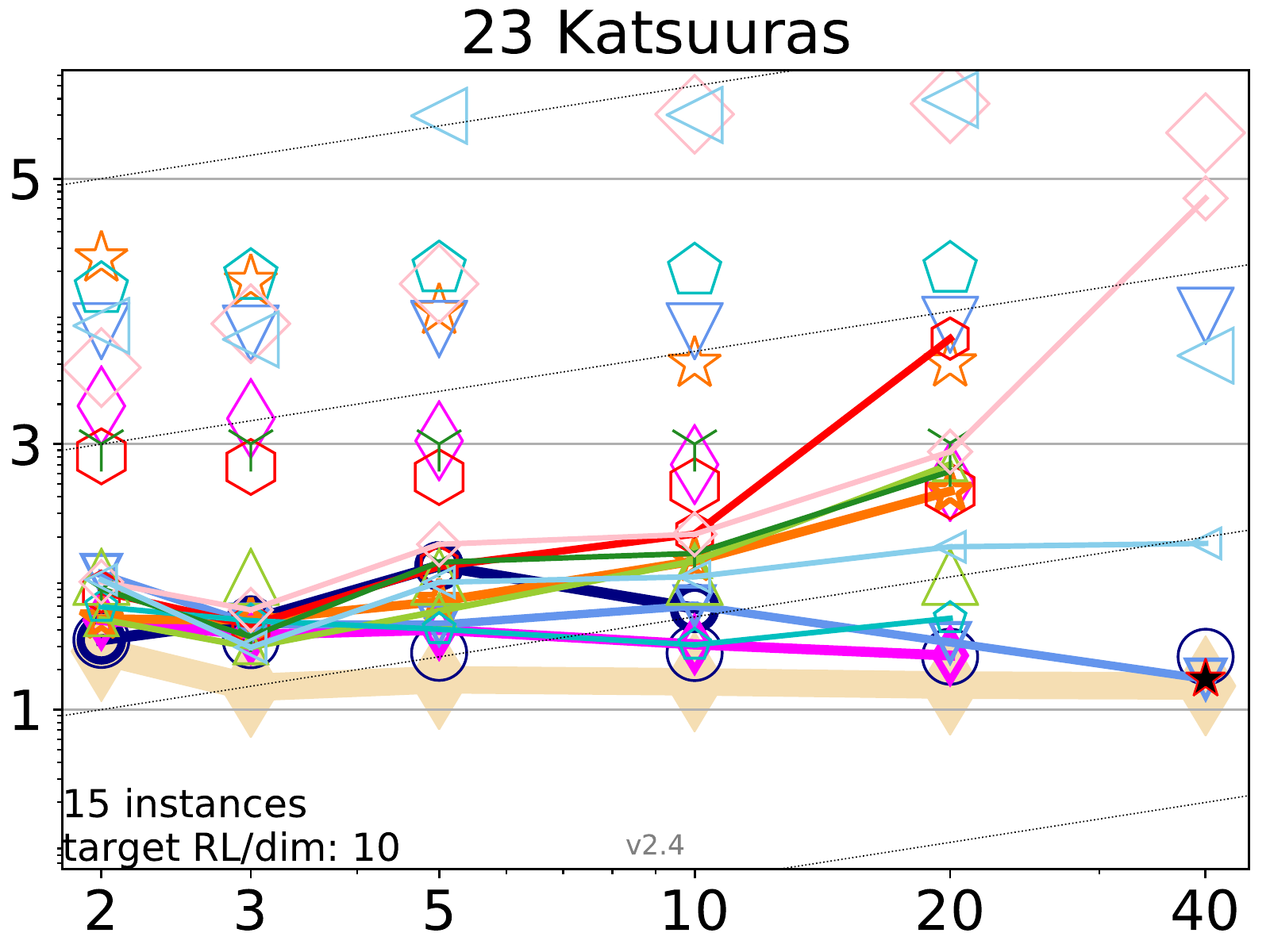}&
\includegraphics[width=0.238\textwidth]{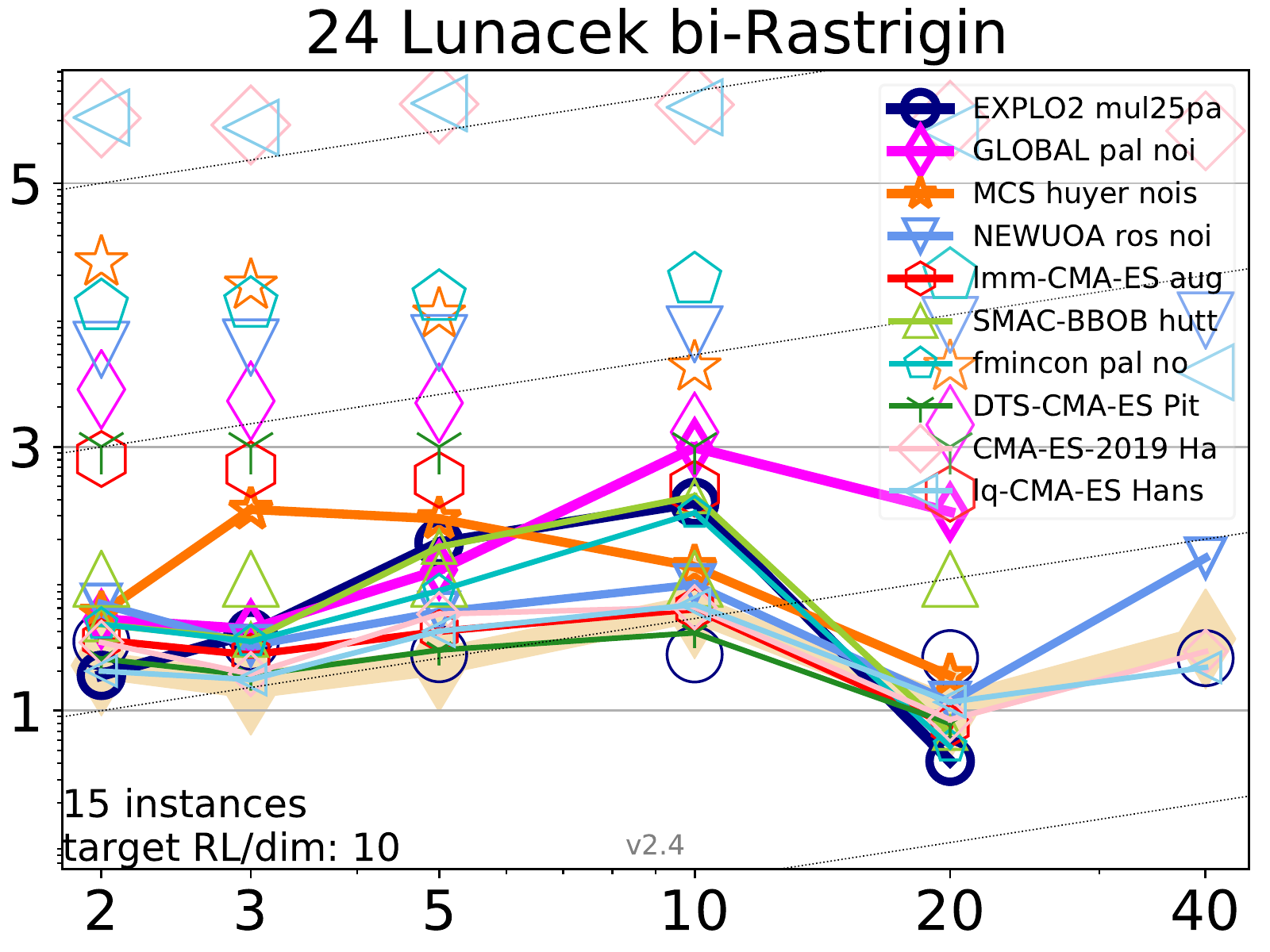}\\
\end{tabular}
\vspace*{-0.2cm}
\caption[Expected running time divided by dimension
versus dimension]{
\label{fig:scaling}
\bbobppfigslegend{$f_1$ and $f_{24}$}  
}
\end{figure*}

\clearpage

\section{\label{sec:ECDFs20}Runtime distributions (ECDFs) per function: $D = 20$}

\begin{figure*}
\centering
\begin{tabular}{@{}c@{}c@{}c@{}c@{}}
\includegraphics[width=0.238\textwidth]{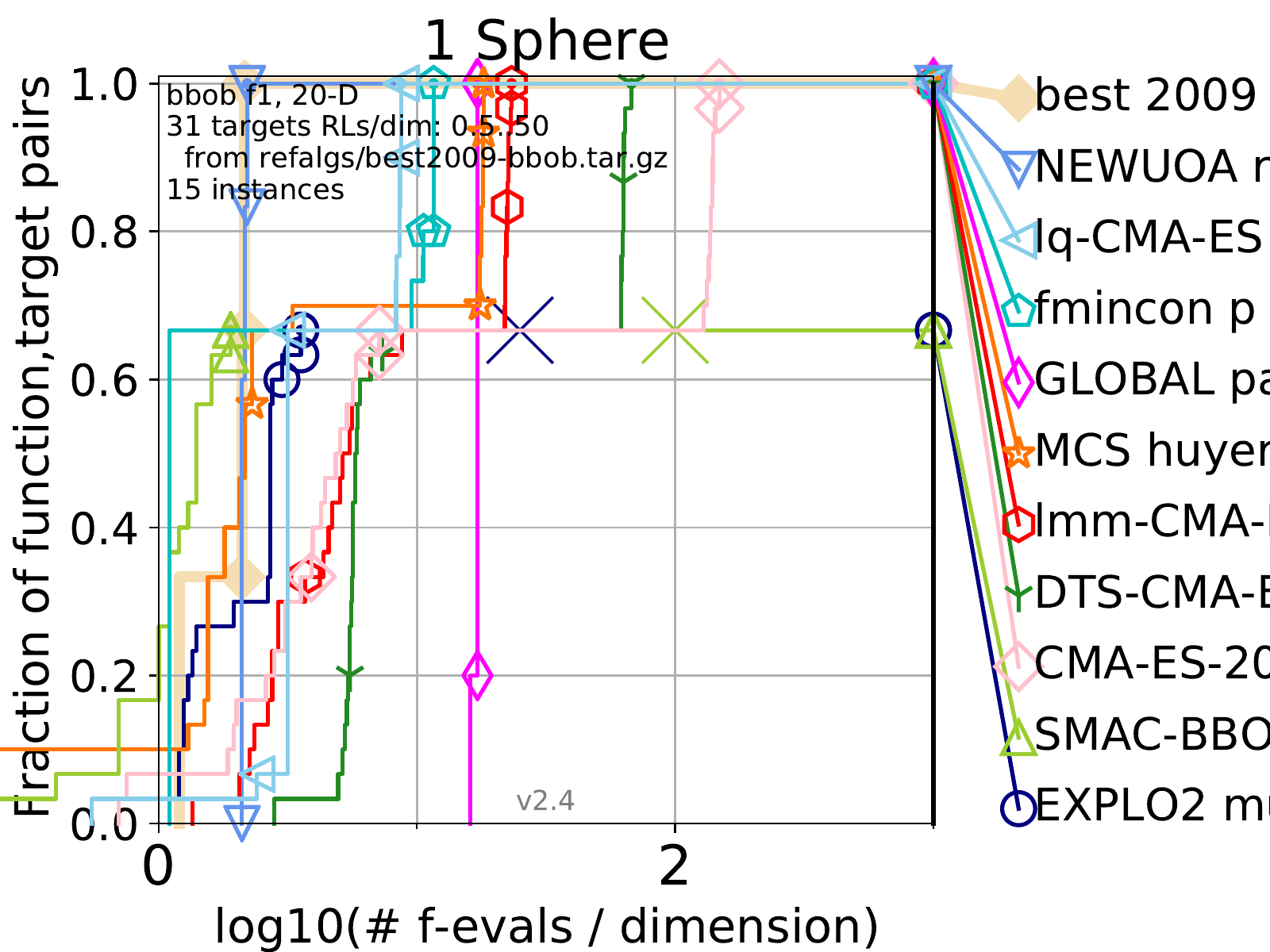}&
\includegraphics[width=0.238\textwidth]{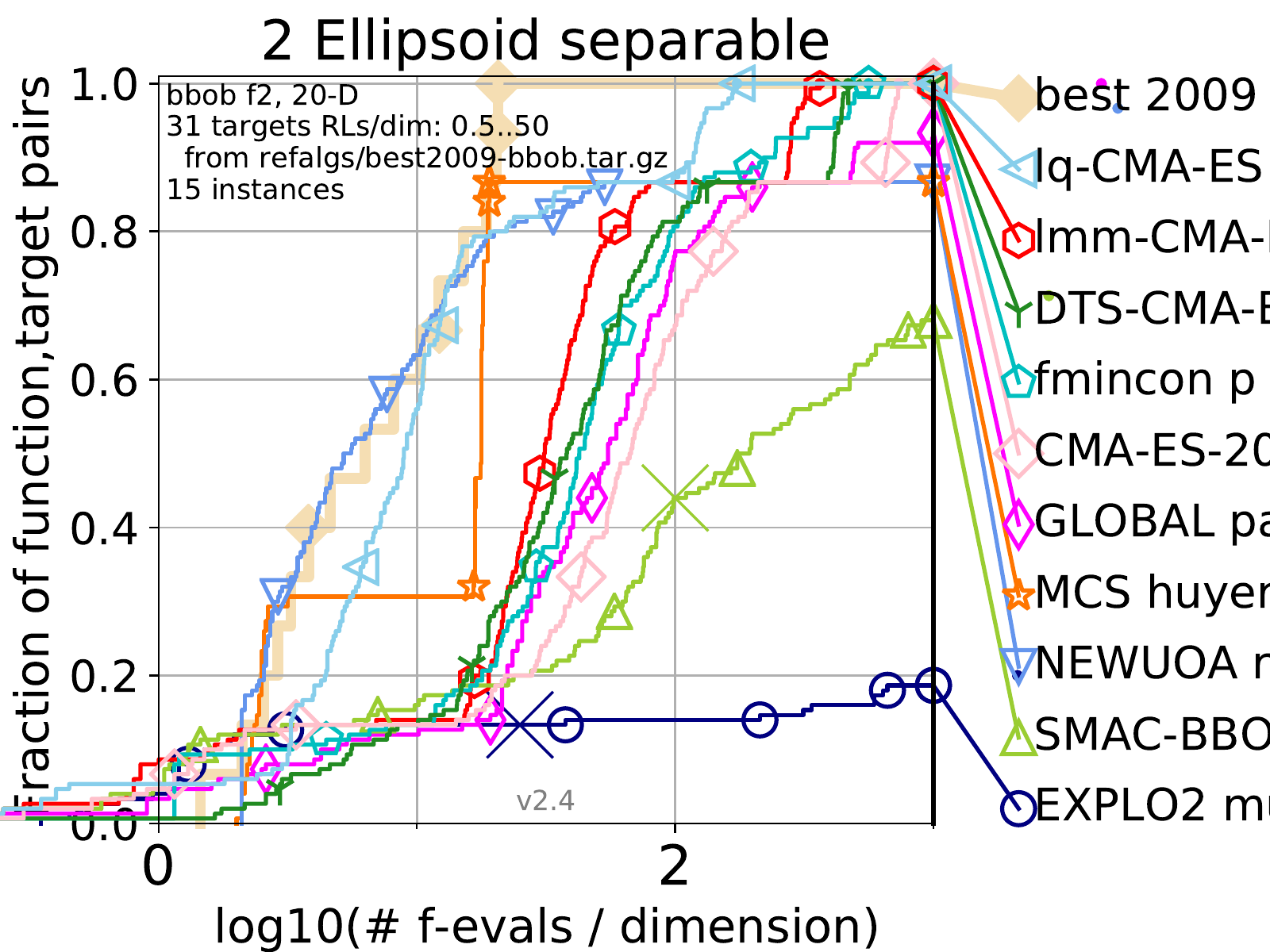}&
\includegraphics[width=0.238\textwidth]{EXPLO_GLOBA_MCS_h_NEWUO_lmm-C_SMAC-_fminc_et_al/pprldmany-single-functions/pprldmany_f003_20D}&
\includegraphics[width=0.238\textwidth]{EXPLO_GLOBA_MCS_h_NEWUO_lmm-C_SMAC-_fminc_et_al/pprldmany-single-functions/pprldmany_f004_20D}\\
\includegraphics[width=0.238\textwidth]{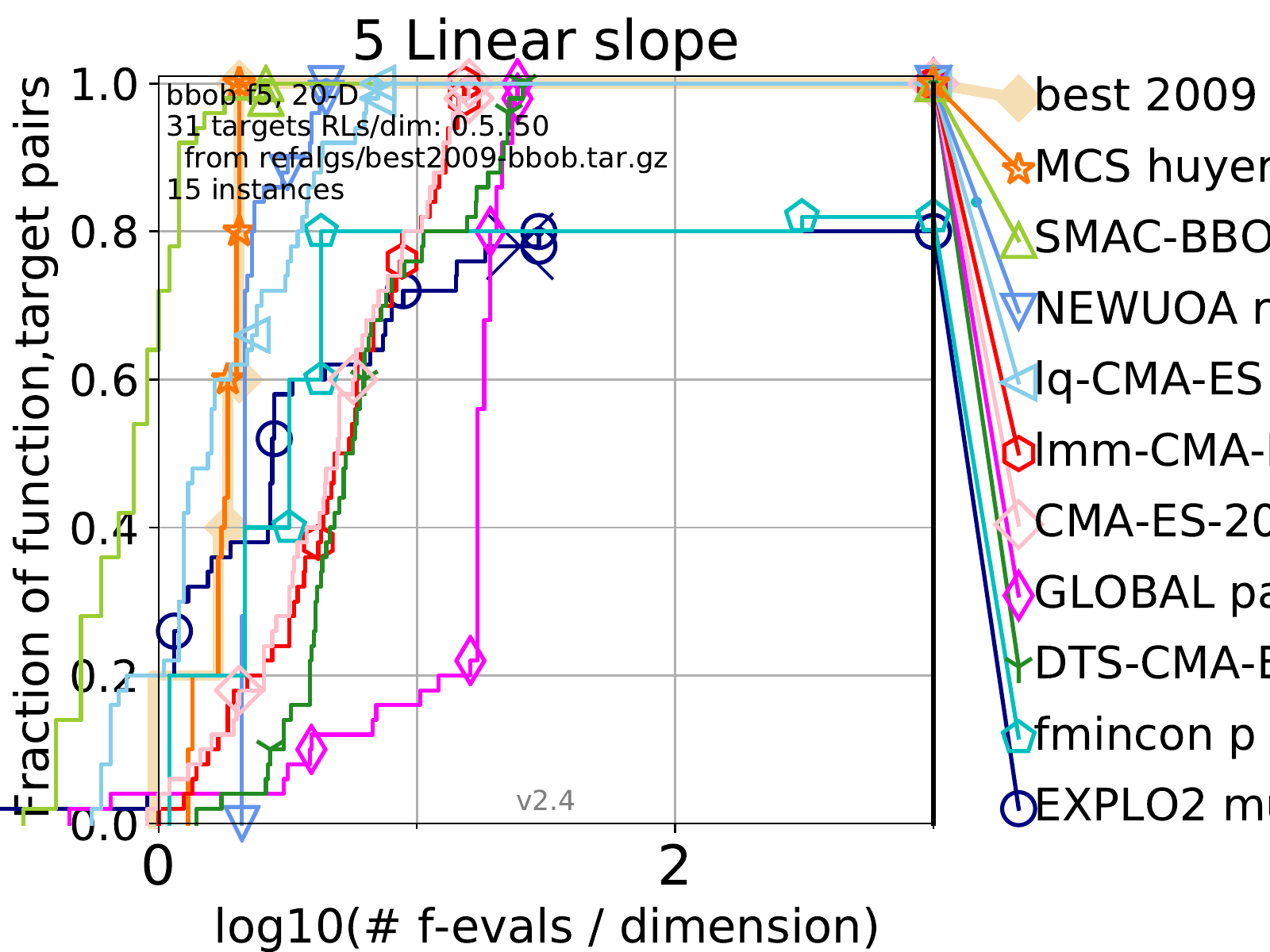}&
\includegraphics[width=0.238\textwidth]{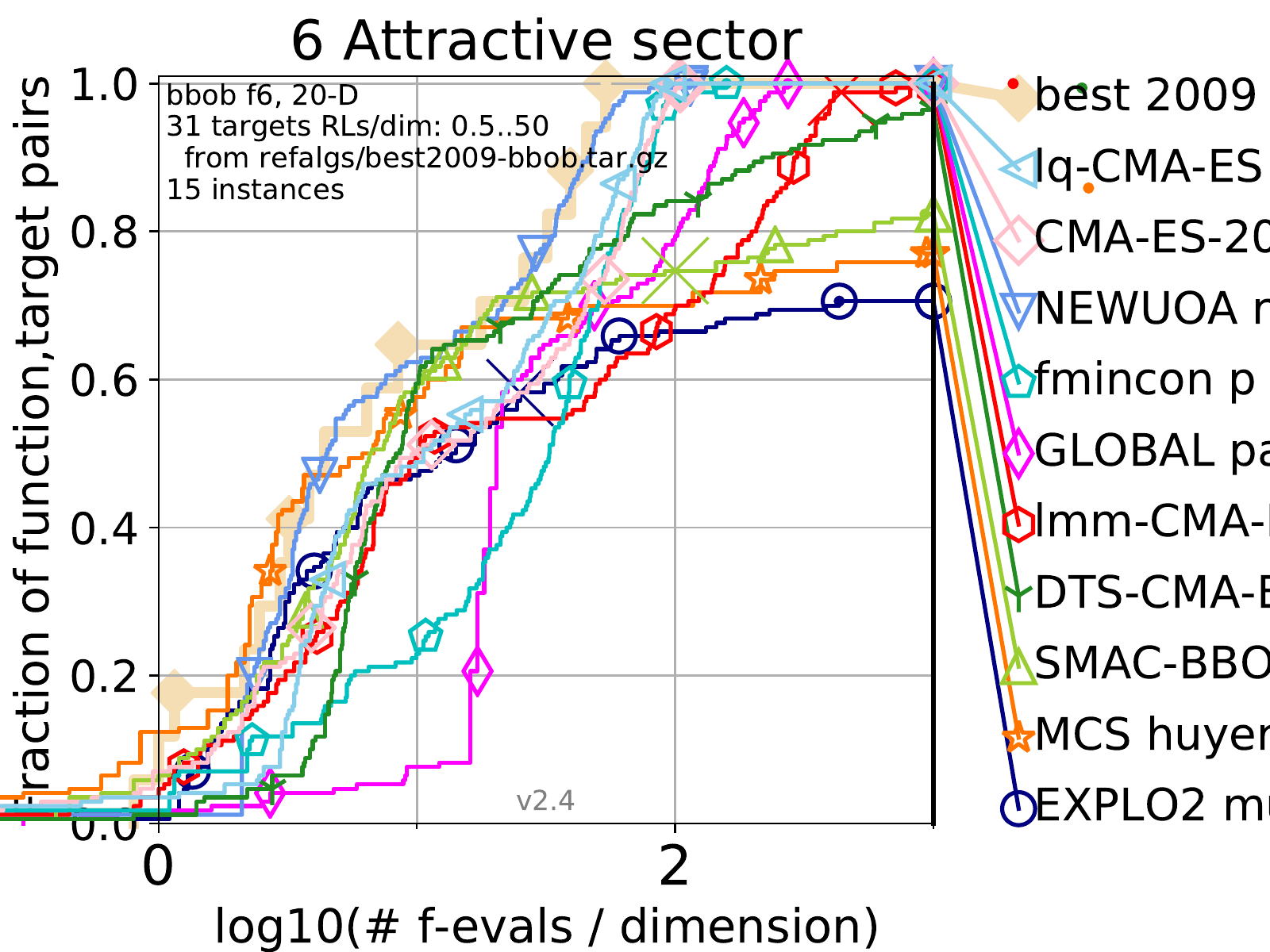}&
\includegraphics[width=0.238\textwidth]{EXPLO_GLOBA_MCS_h_NEWUO_lmm-C_SMAC-_fminc_et_al/pprldmany-single-functions/pprldmany_f007_20D}&
\includegraphics[width=0.238\textwidth]{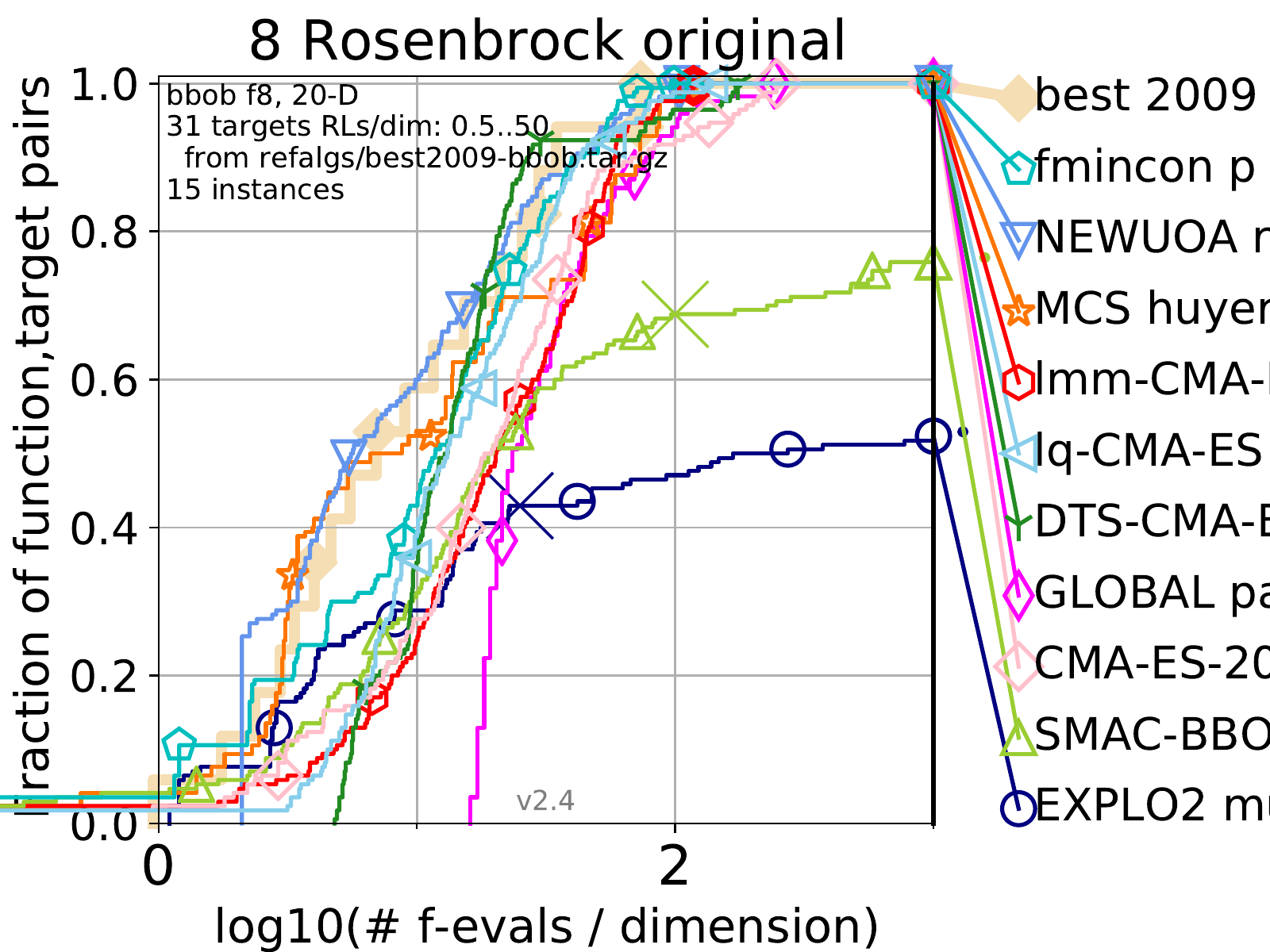}\\
\includegraphics[width=0.238\textwidth]{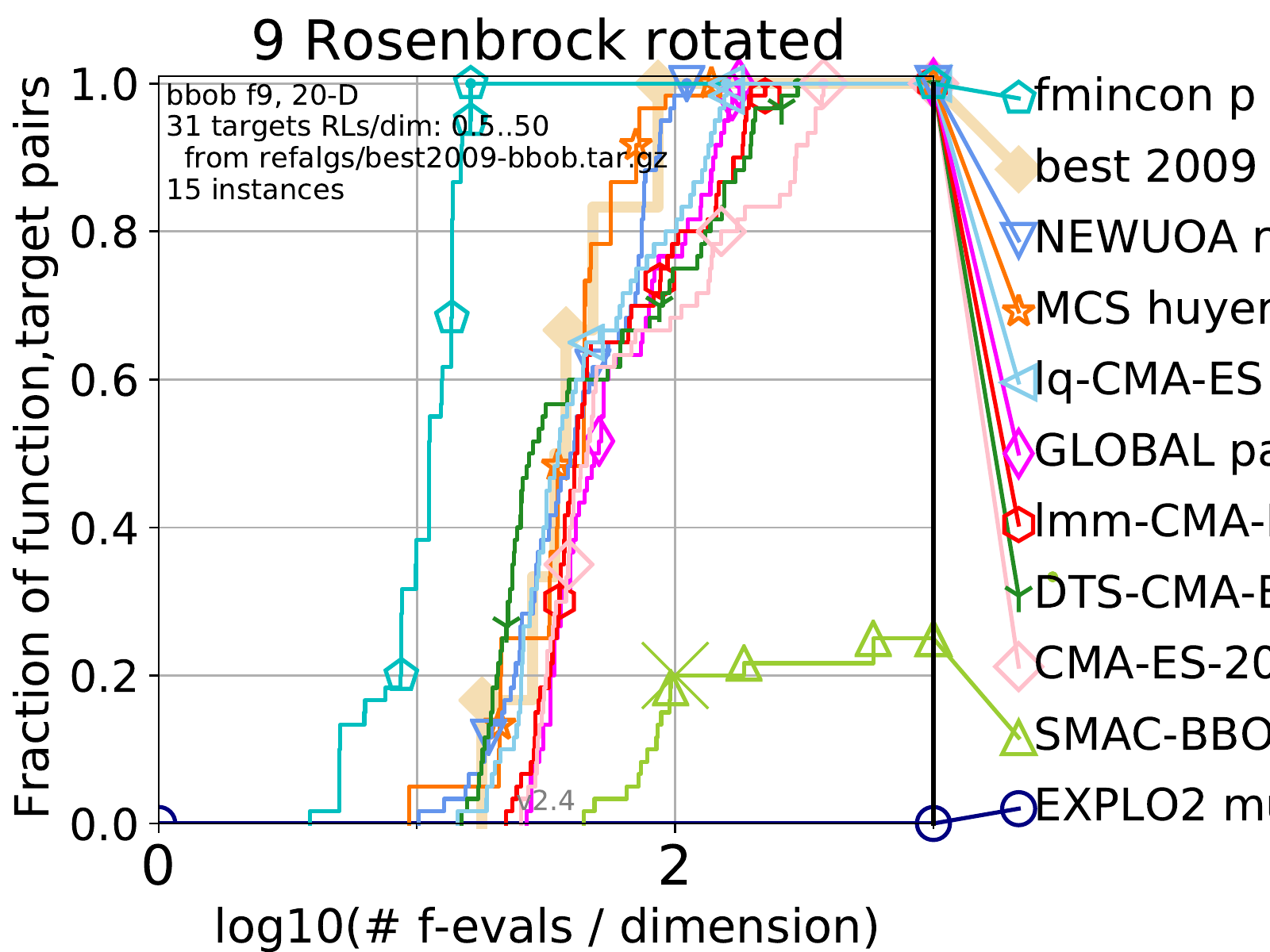}&
\includegraphics[width=0.238\textwidth]{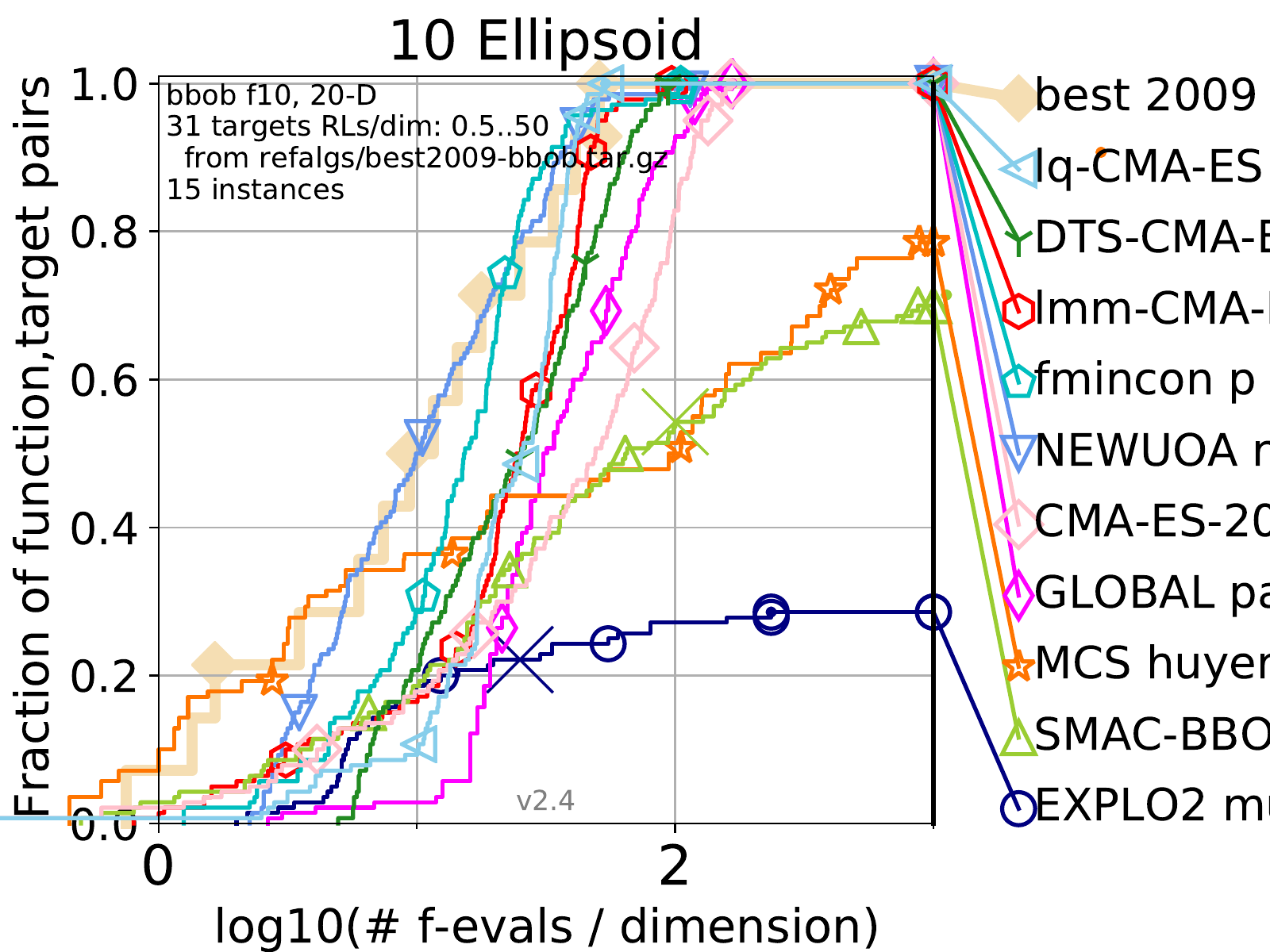}&
\includegraphics[width=0.238\textwidth]{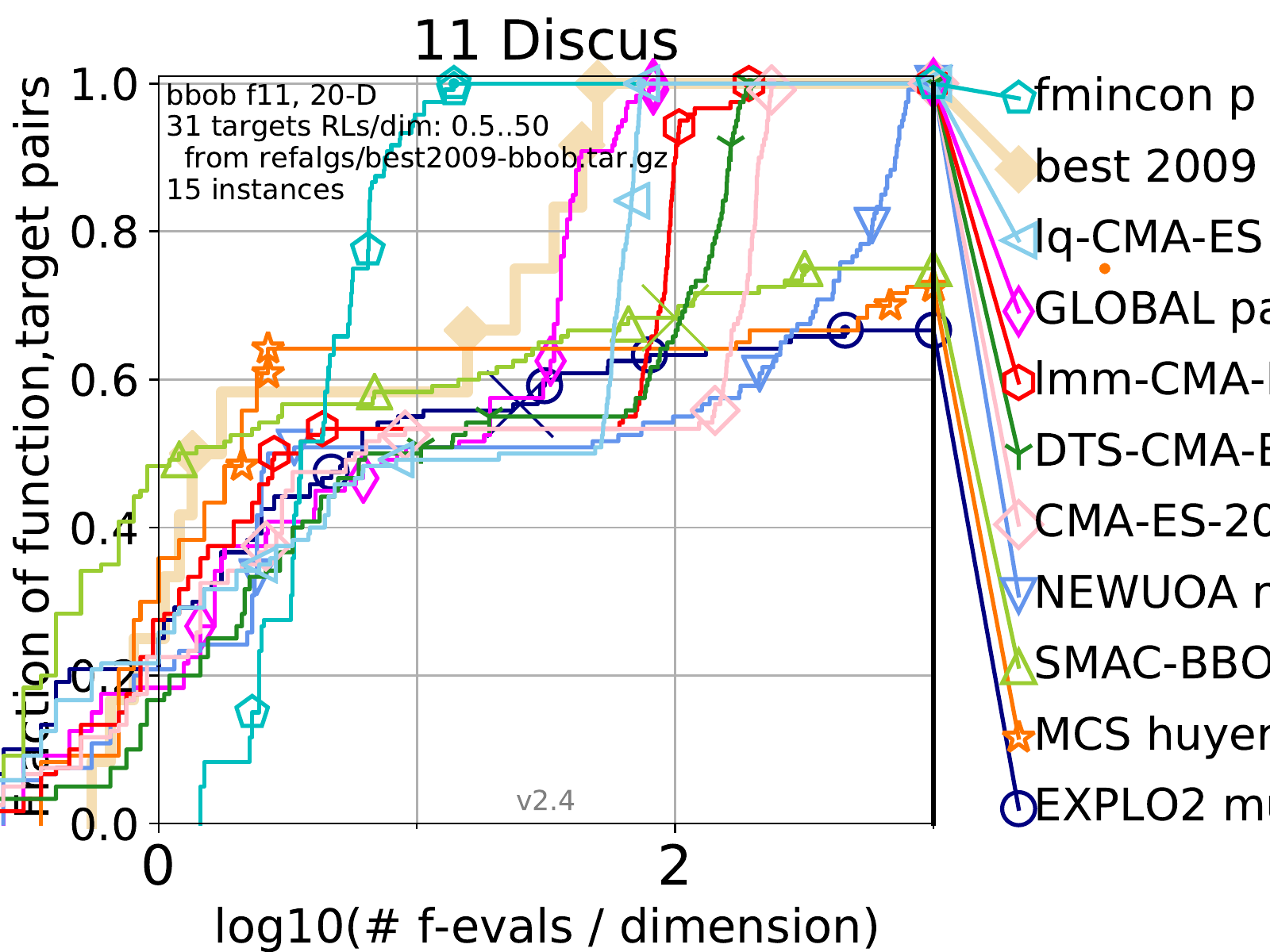}&
\includegraphics[width=0.238\textwidth]{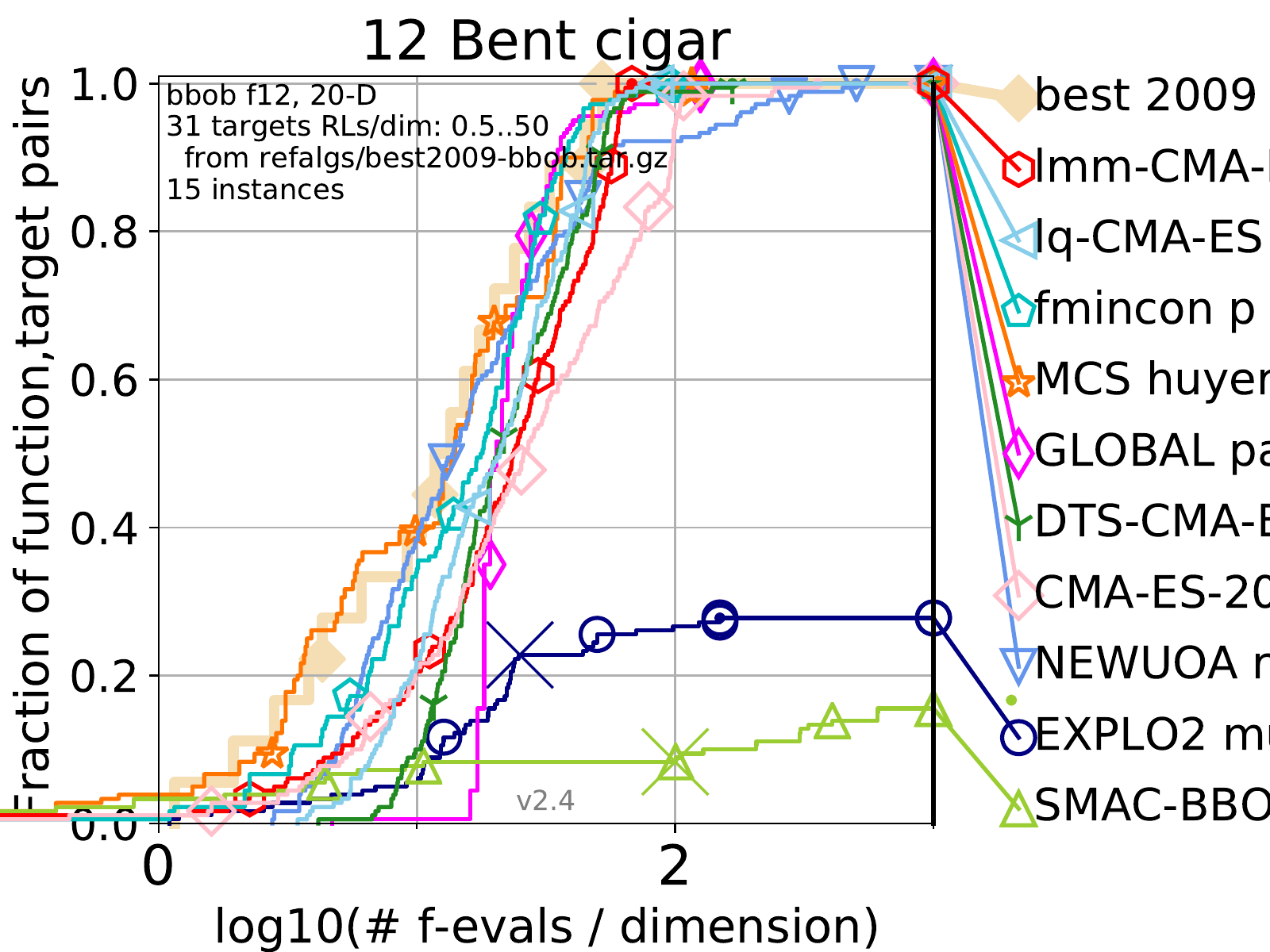}\\
\includegraphics[width=0.238\textwidth]{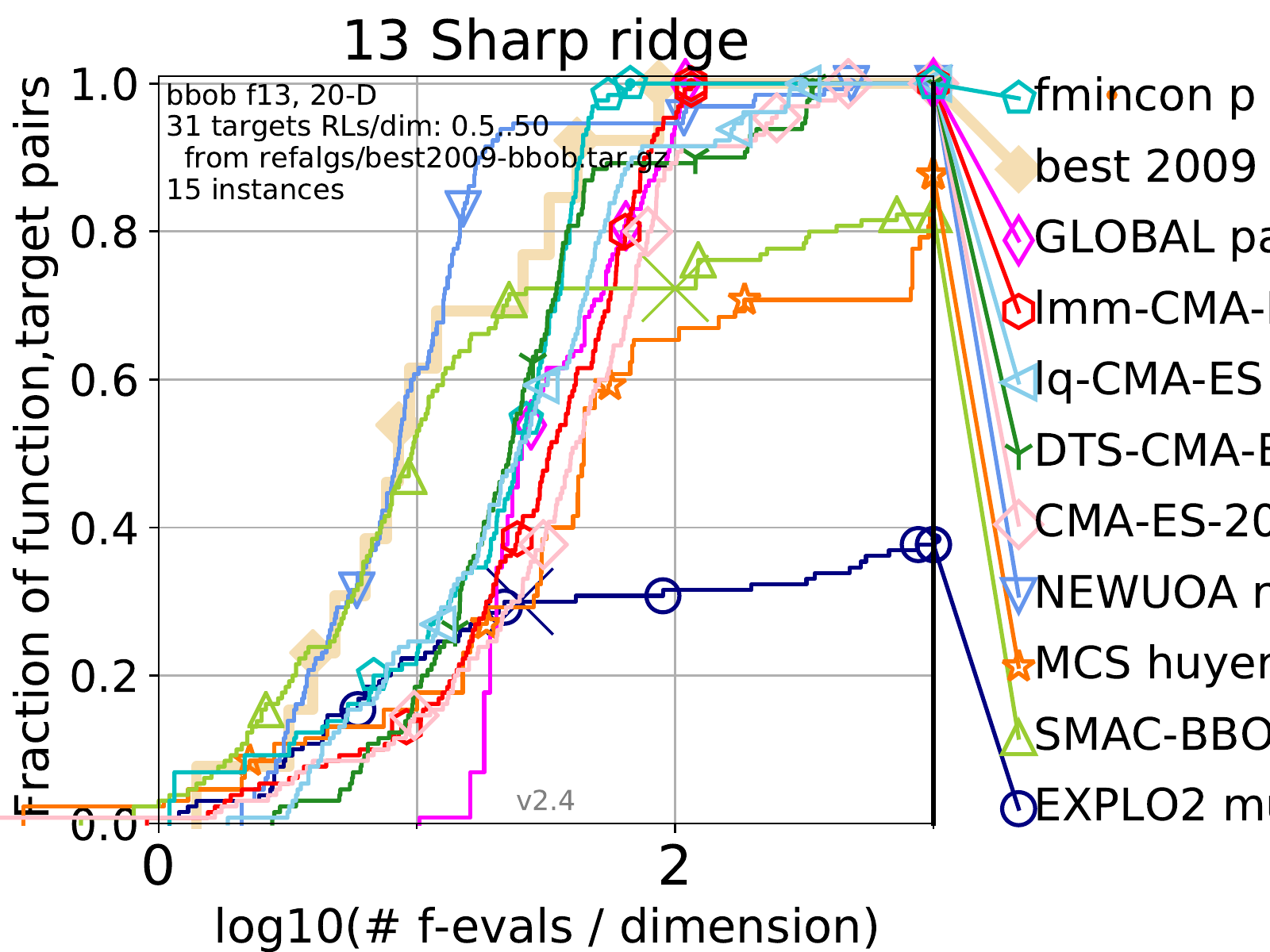}&
\includegraphics[width=0.238\textwidth]{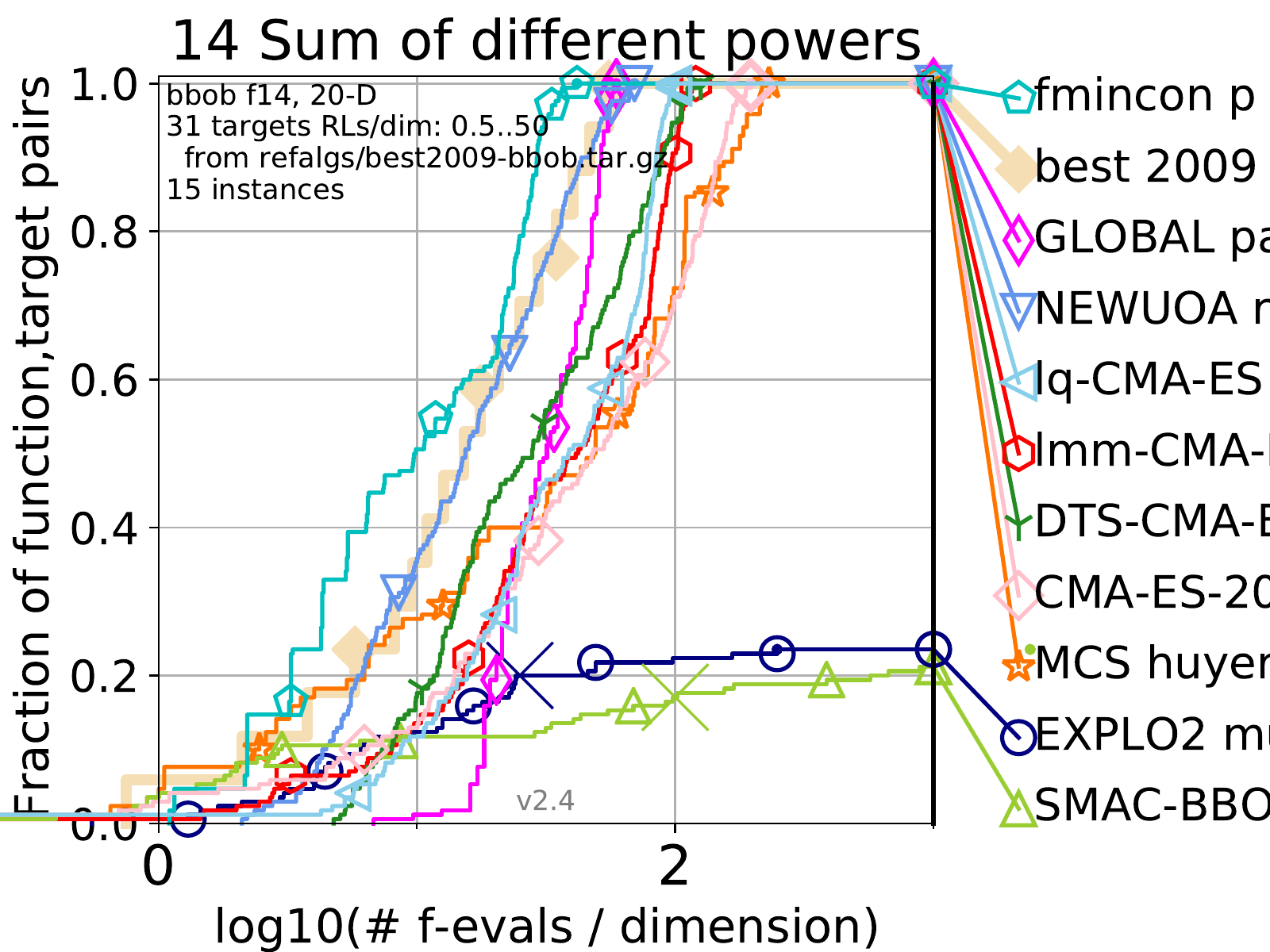}&
\includegraphics[width=0.238\textwidth]{EXPLO_GLOBA_MCS_h_NEWUO_lmm-C_SMAC-_fminc_et_al/pprldmany-single-functions/pprldmany_f015_20D}&
\includegraphics[width=0.238\textwidth]{EXPLO_GLOBA_MCS_h_NEWUO_lmm-C_SMAC-_fminc_et_al/pprldmany-single-functions/pprldmany_f016_20D}\\
\includegraphics[width=0.238\textwidth]{EXPLO_GLOBA_MCS_h_NEWUO_lmm-C_SMAC-_fminc_et_al/pprldmany-single-functions/pprldmany_f017_20D}&
\includegraphics[width=0.238\textwidth]{EXPLO_GLOBA_MCS_h_NEWUO_lmm-C_SMAC-_fminc_et_al/pprldmany-single-functions/pprldmany_f018_20D}&
\includegraphics[width=0.238\textwidth]{EXPLO_GLOBA_MCS_h_NEWUO_lmm-C_SMAC-_fminc_et_al/pprldmany-single-functions/pprldmany_f019_20D}&
\includegraphics[width=0.238\textwidth]{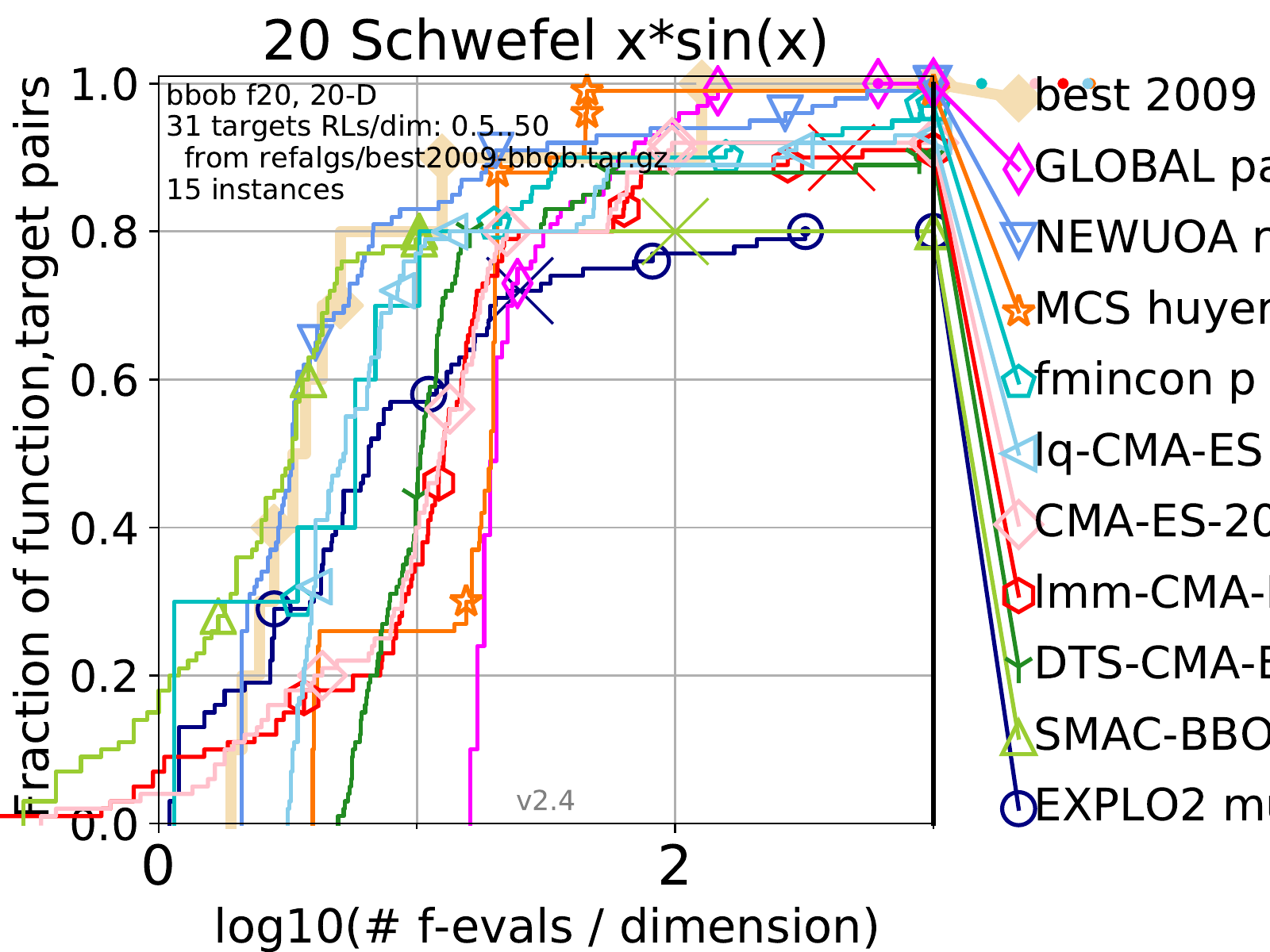}\\
\includegraphics[width=0.238\textwidth]{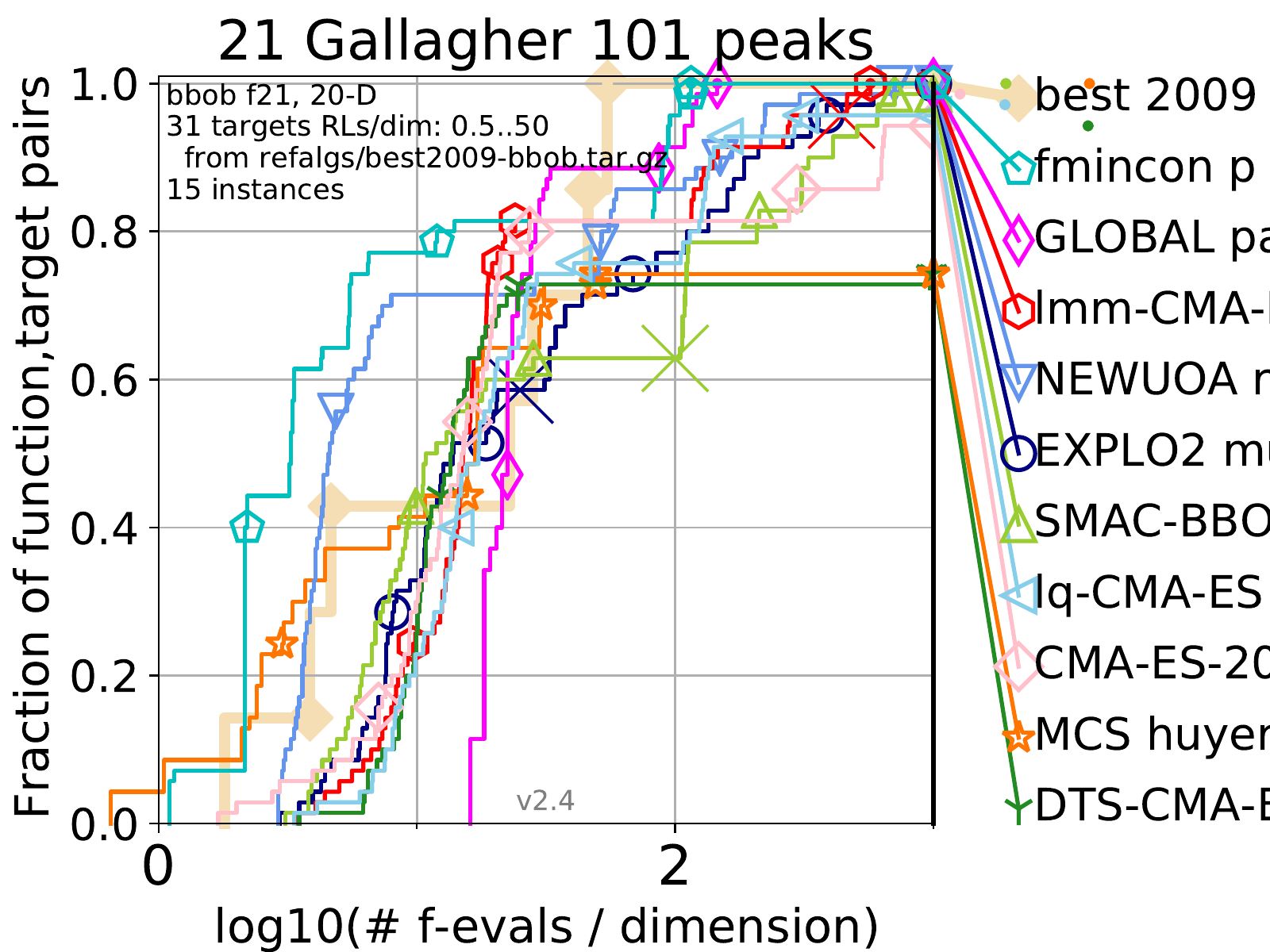}&
\includegraphics[width=0.238\textwidth]{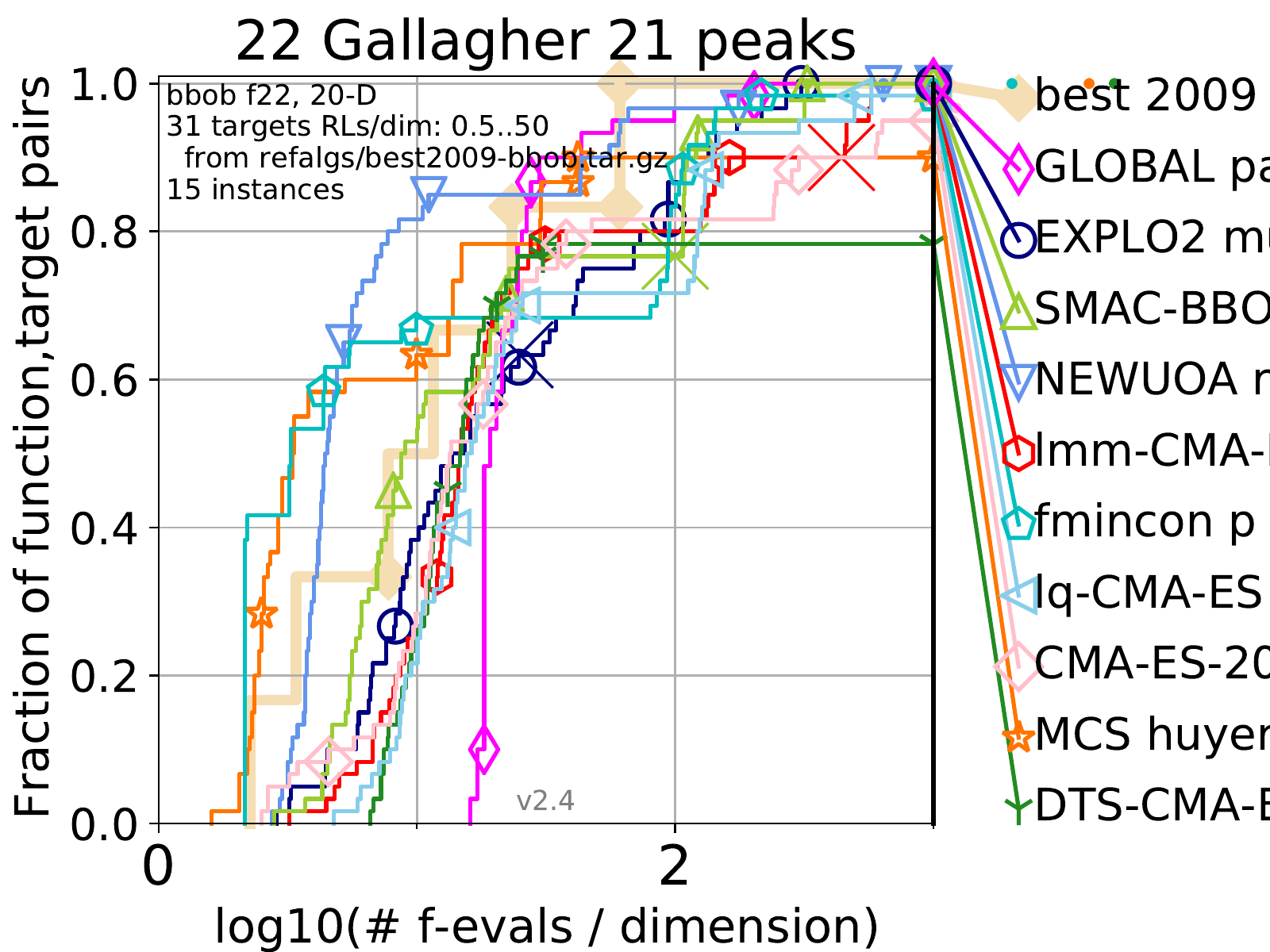}&
\includegraphics[width=0.238\textwidth]{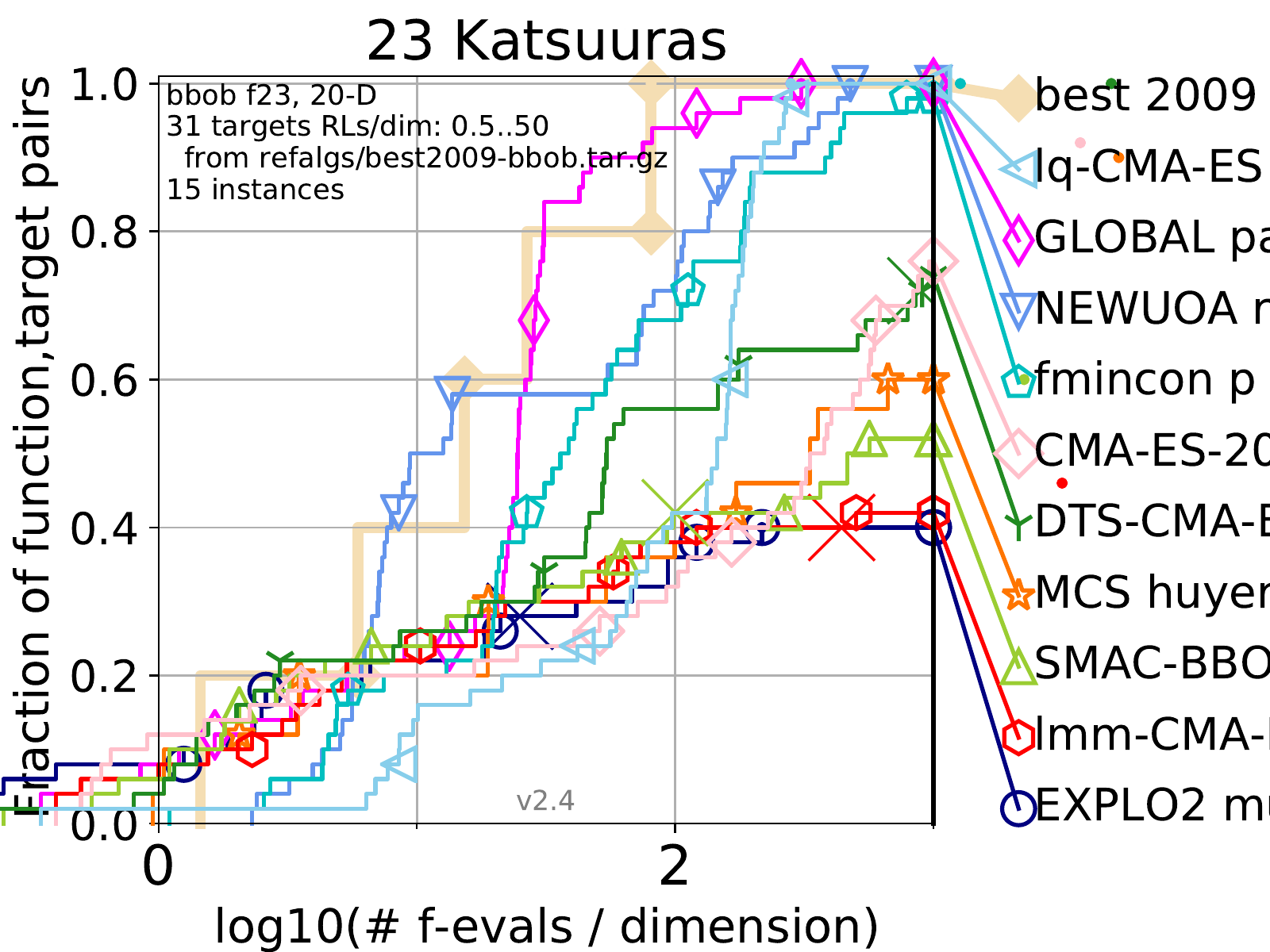}&
\includegraphics[width=0.238\textwidth]{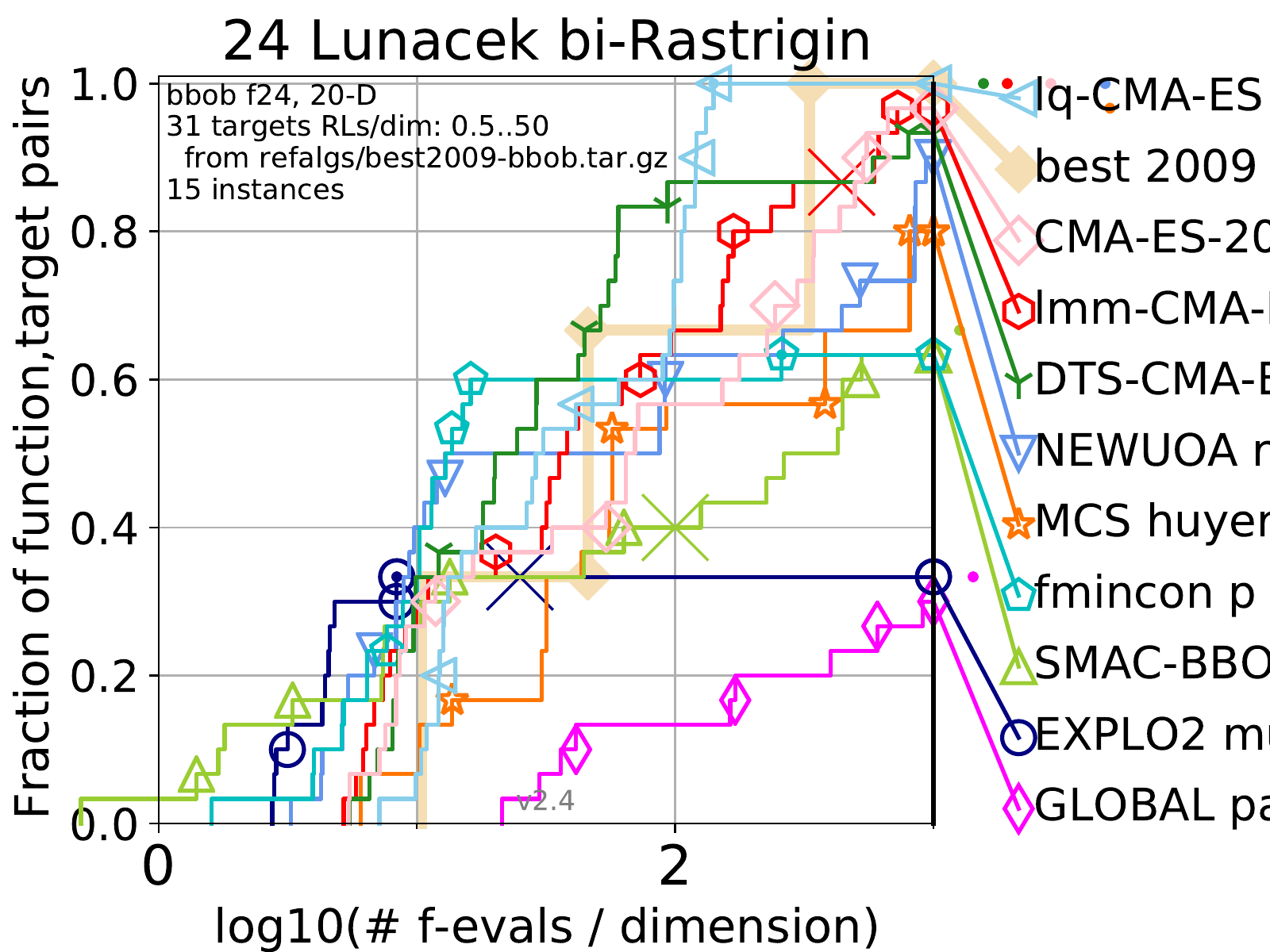}\\
\end{tabular}
\vspace*{-0.2cm}
\caption{
\bbobecdfcaptionsinglefunctionssingledim{20} {\bf Crosses ($\times$) indicate where experimental data ends and bootstrapping begins; algorithms are not comparable after this point.} \texttt{EXPLO2} used default options except for $n_\parallel = 32$. 
}
\end{figure*}

\clearpage

\section{\label{sec:ECDFs40}Runtime distributions (ECDFs) per function: $D = 40$}

\begin{figure*}
\centering
\begin{tabular}{@{}c@{}c@{}c@{}c@{}}
\includegraphics[width=0.238\textwidth]{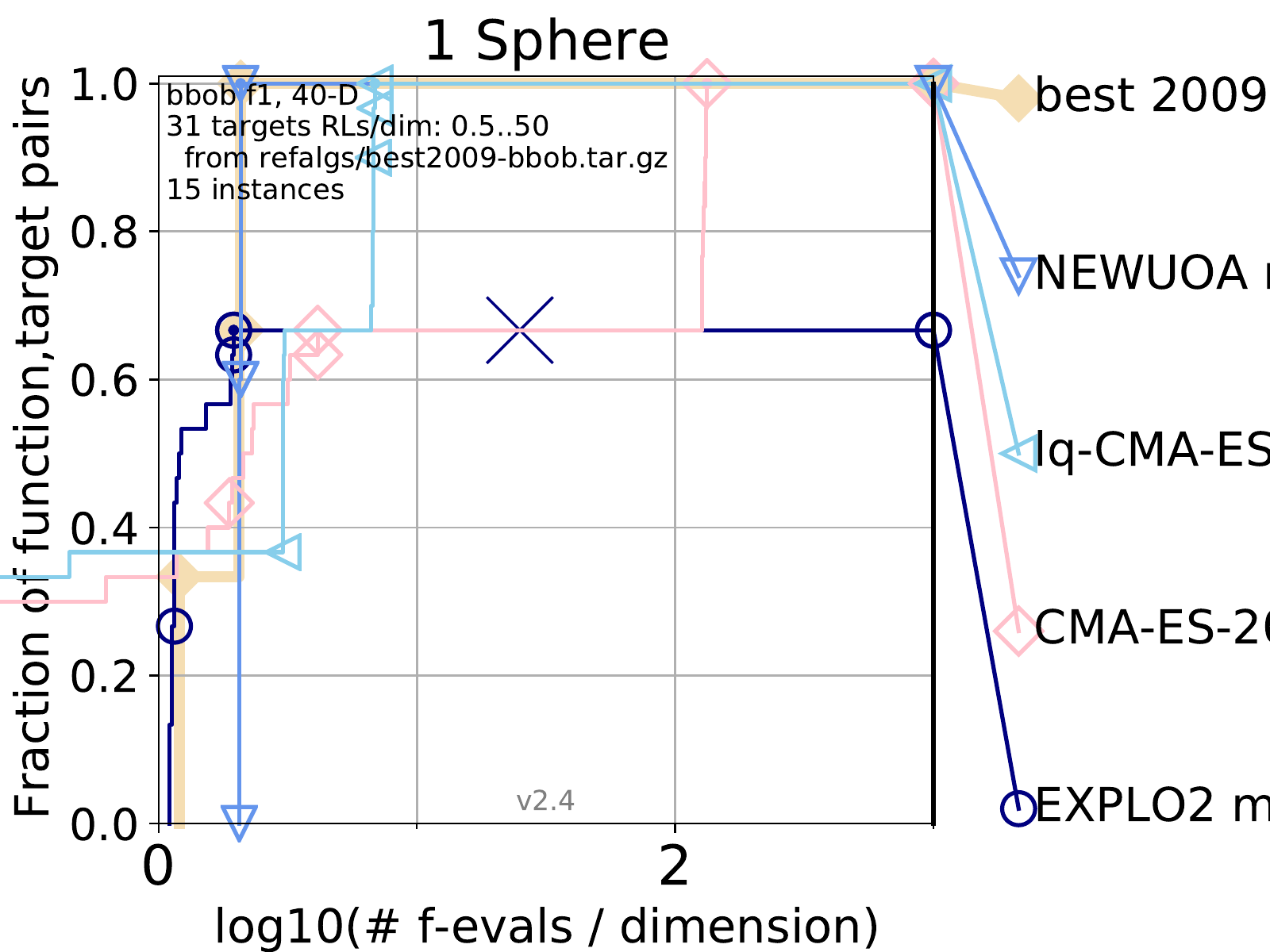}&
\includegraphics[width=0.238\textwidth]{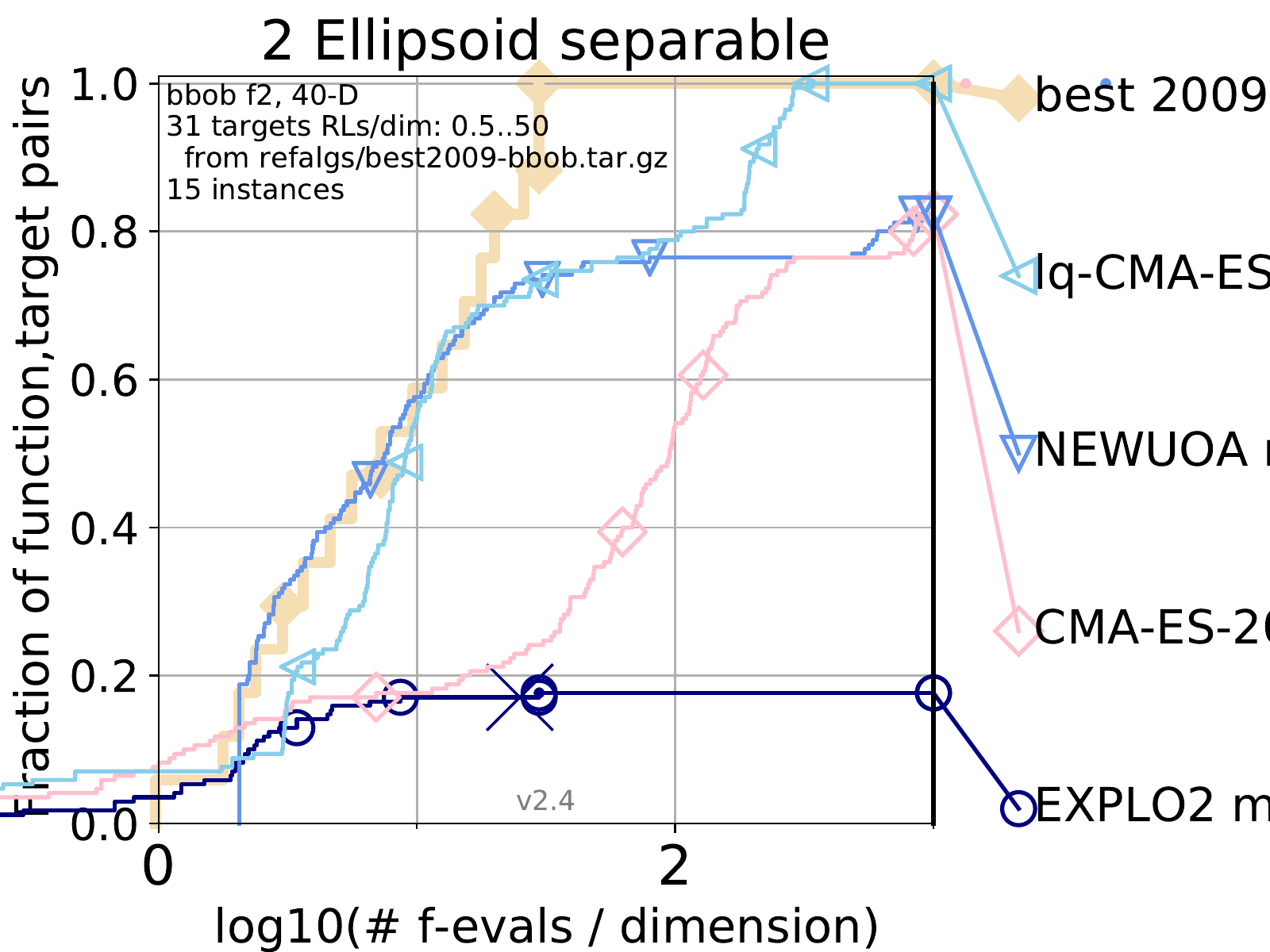}&
\includegraphics[width=0.238\textwidth]{EXPLO_GLOBA_MCS_h_NEWUO_lmm-C_SMAC-_fminc_et_al/pprldmany-single-functions/pprldmany_f003_40D}&
\includegraphics[width=0.238\textwidth]{EXPLO_GLOBA_MCS_h_NEWUO_lmm-C_SMAC-_fminc_et_al/pprldmany-single-functions/pprldmany_f004_40D}\\
\includegraphics[width=0.238\textwidth]{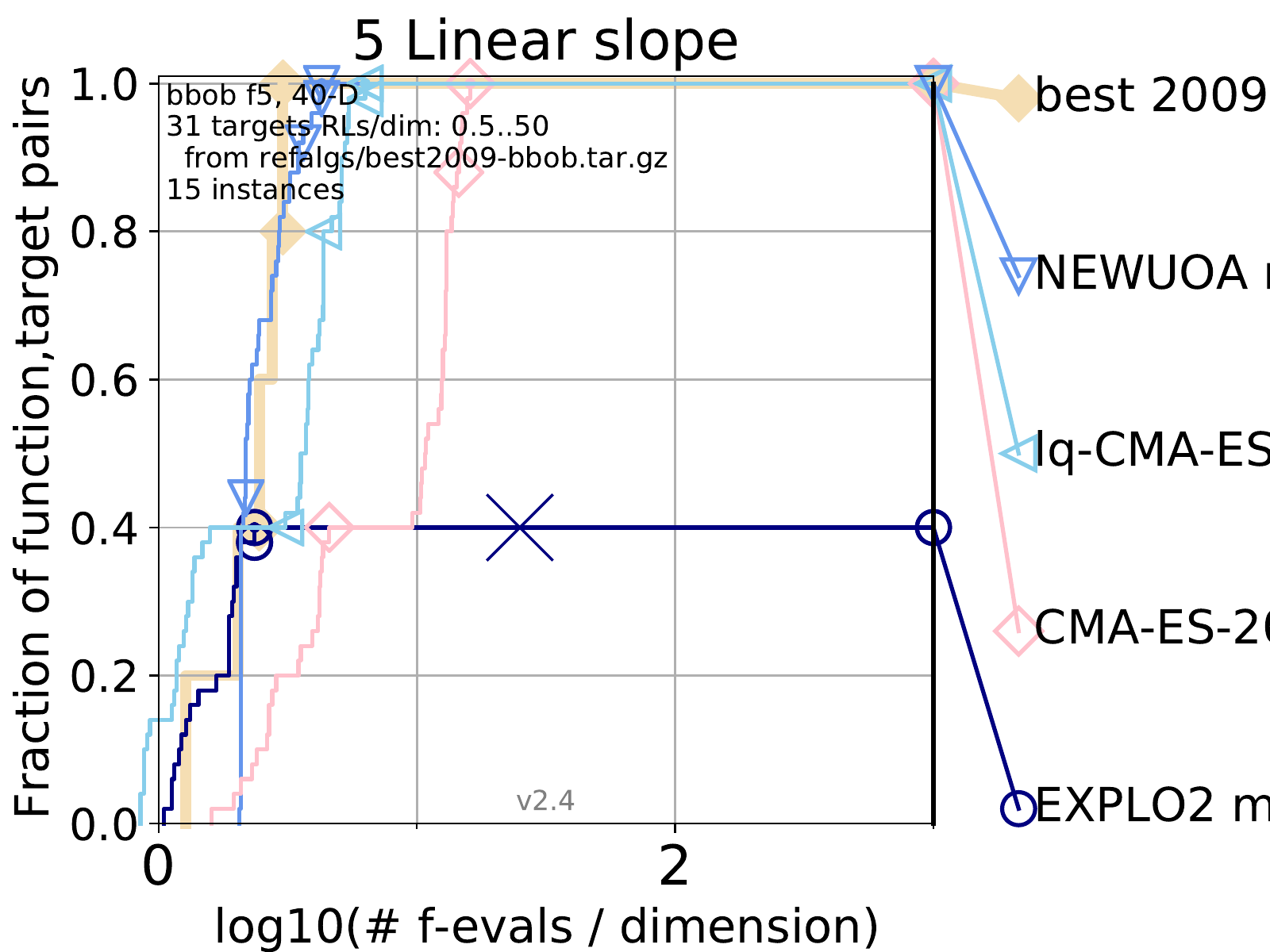}&
\includegraphics[width=0.238\textwidth]{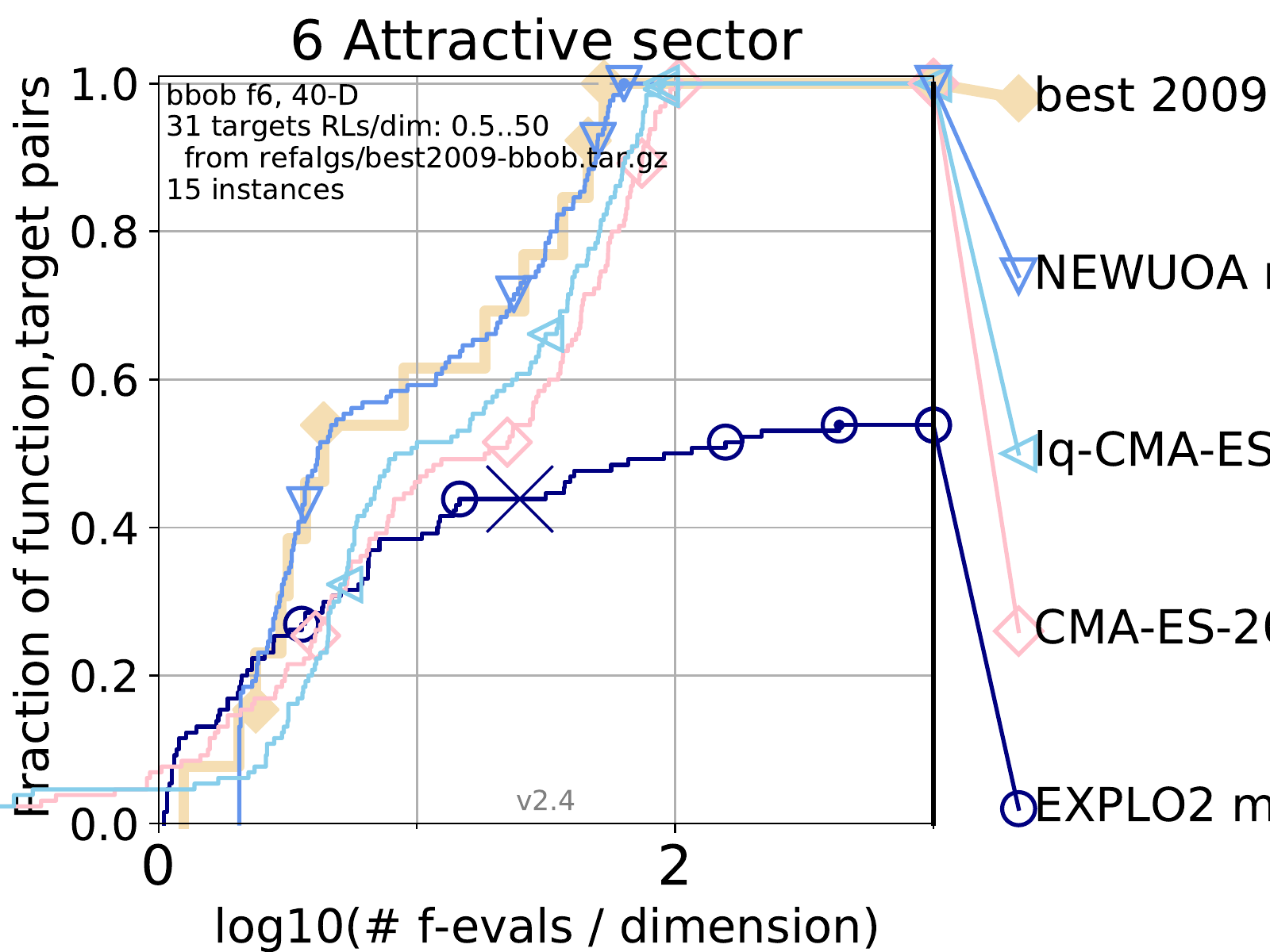}&
\includegraphics[width=0.238\textwidth]{EXPLO_GLOBA_MCS_h_NEWUO_lmm-C_SMAC-_fminc_et_al/pprldmany-single-functions/pprldmany_f007_40D}&
\includegraphics[width=0.238\textwidth]{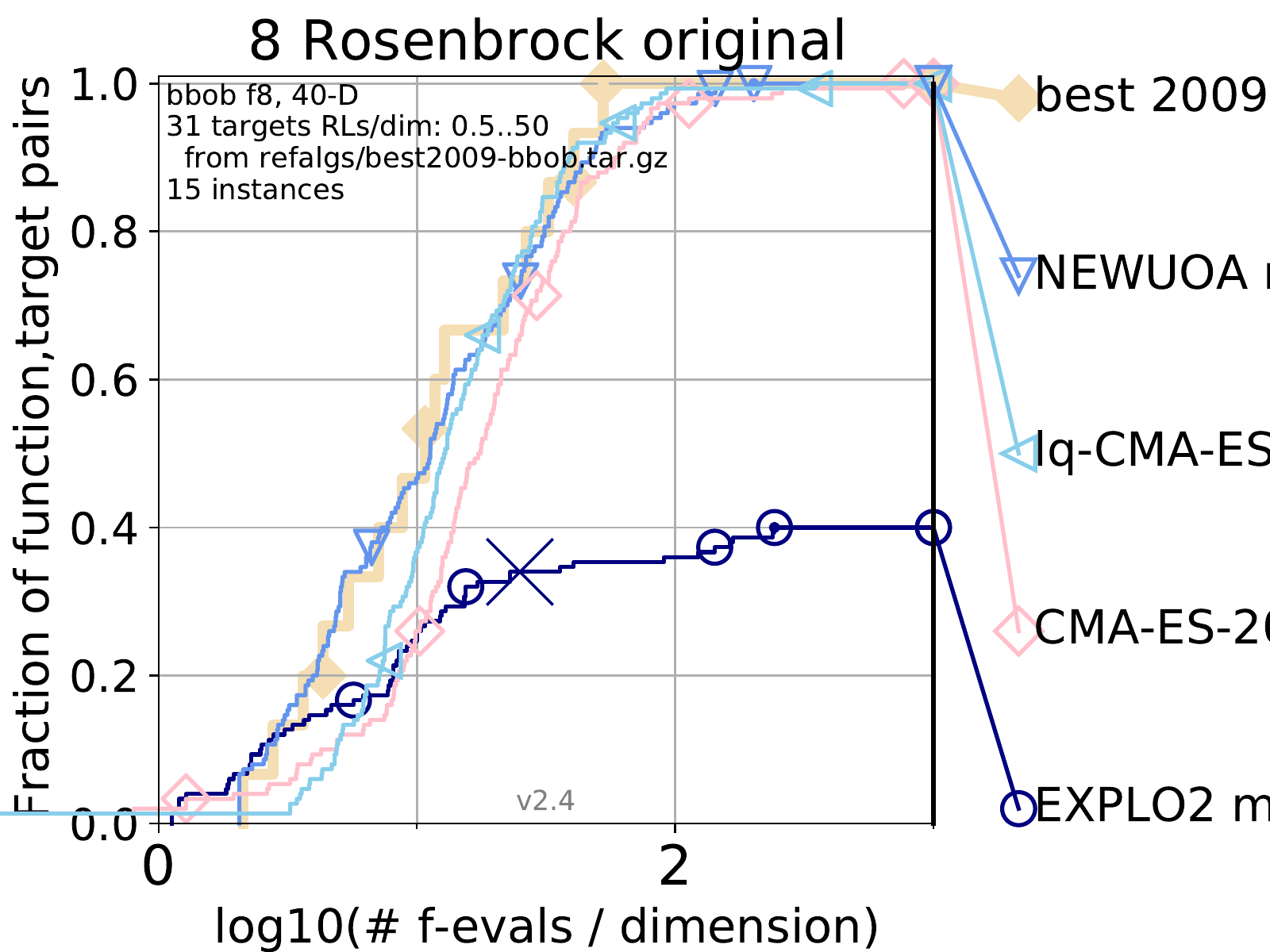}\\
\includegraphics[width=0.238\textwidth]{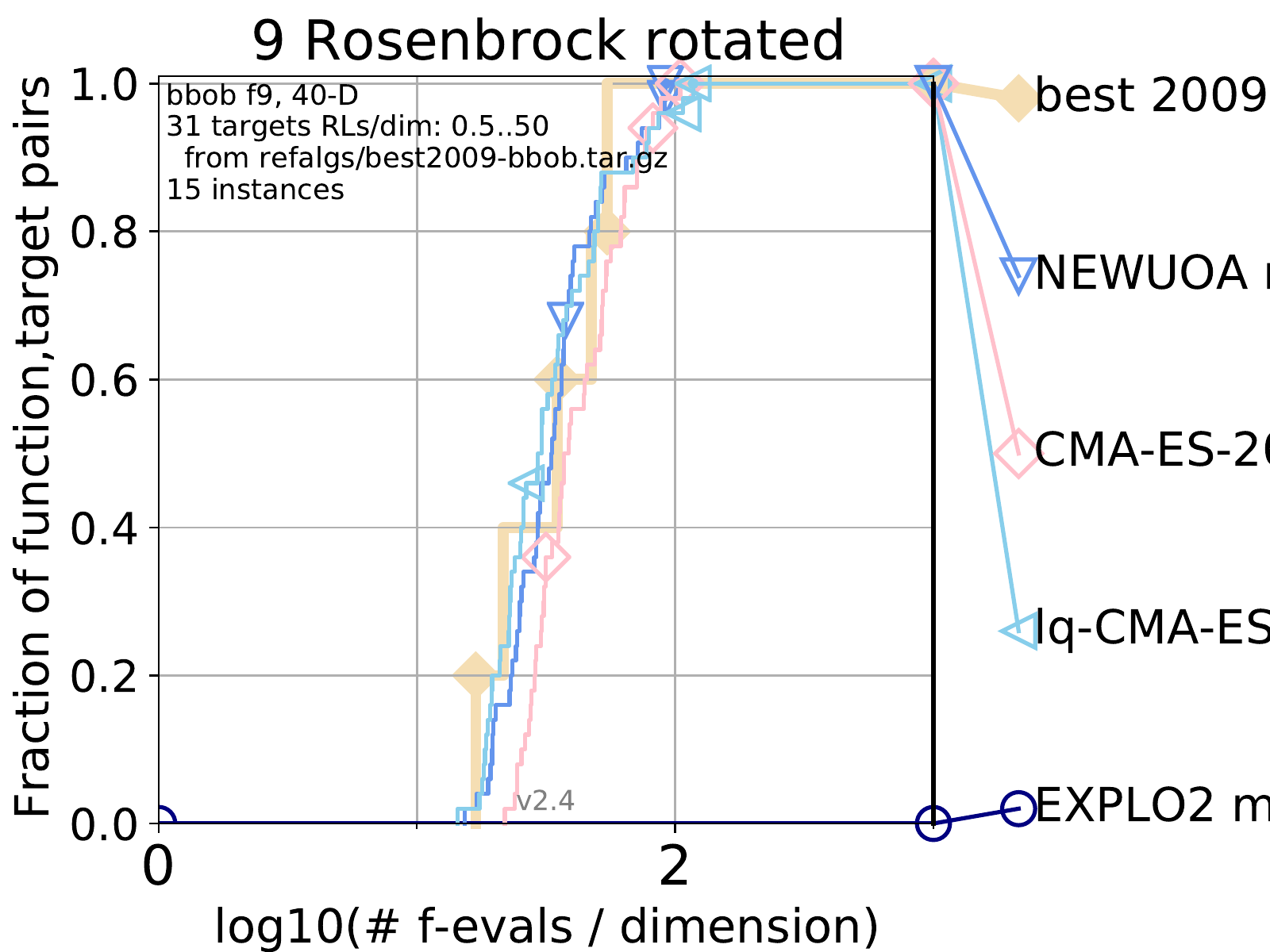}&
\includegraphics[width=0.238\textwidth]{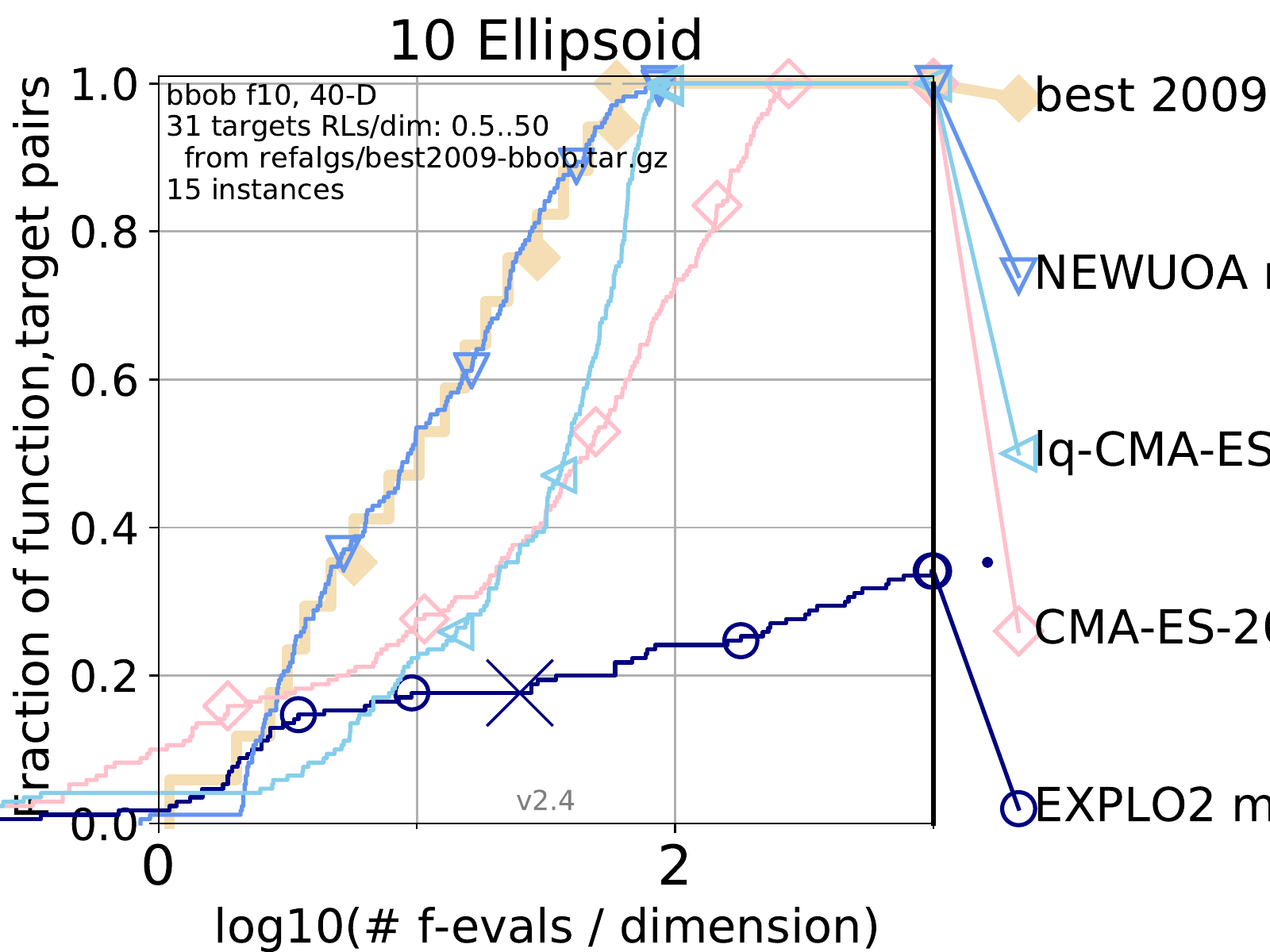}&
\includegraphics[width=0.238\textwidth]{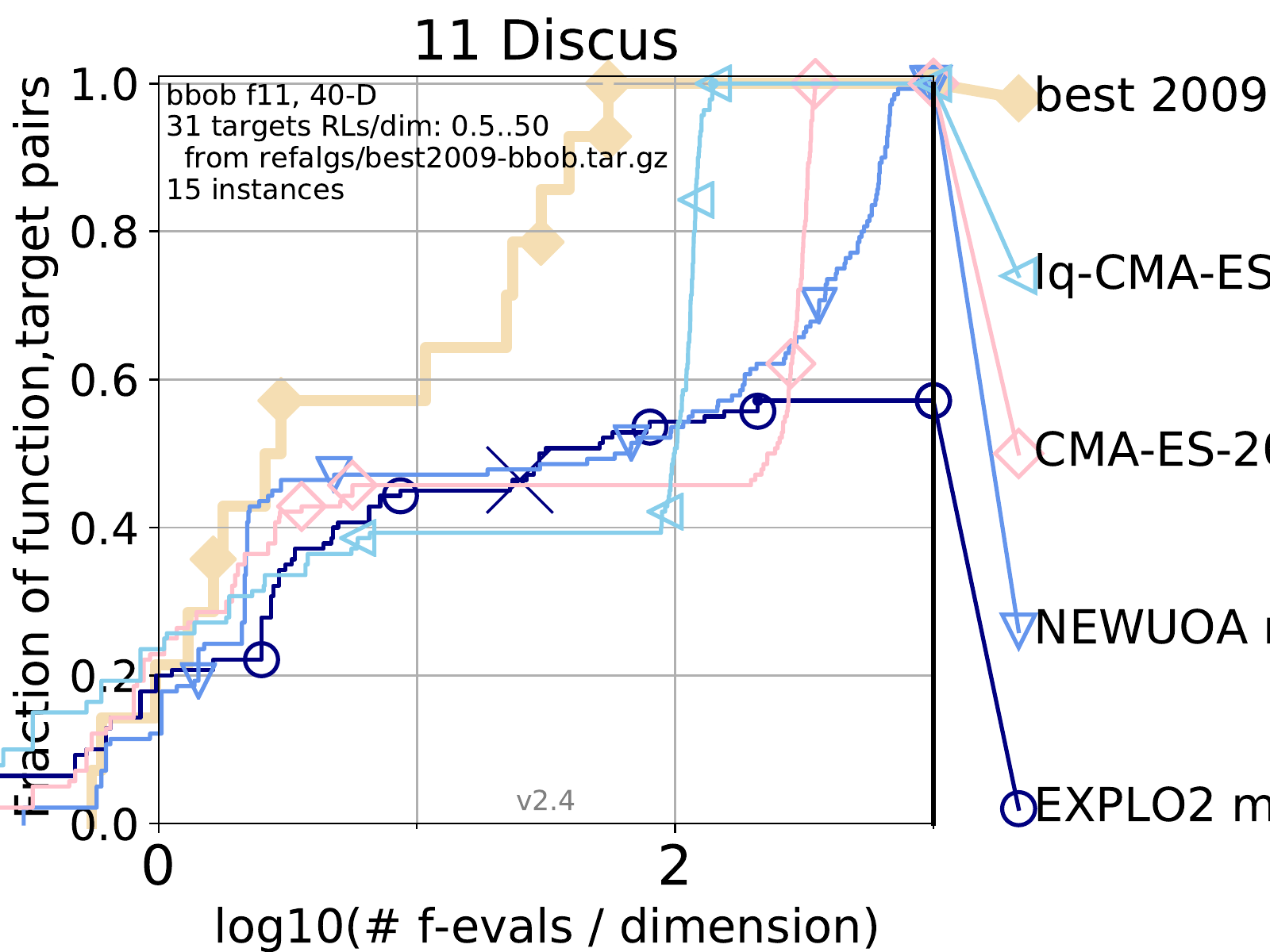}&
\includegraphics[width=0.238\textwidth]{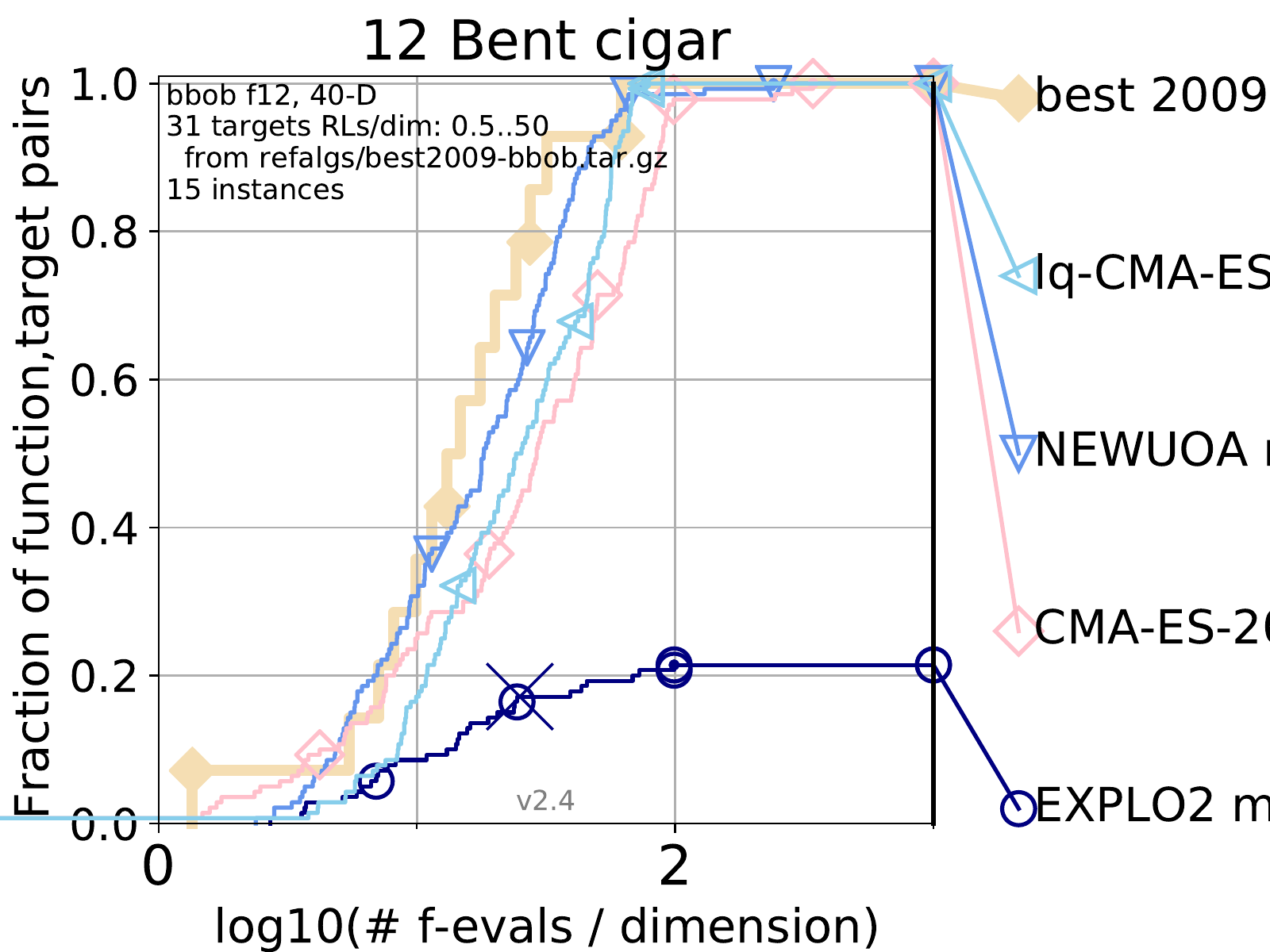}\\
\includegraphics[width=0.238\textwidth]{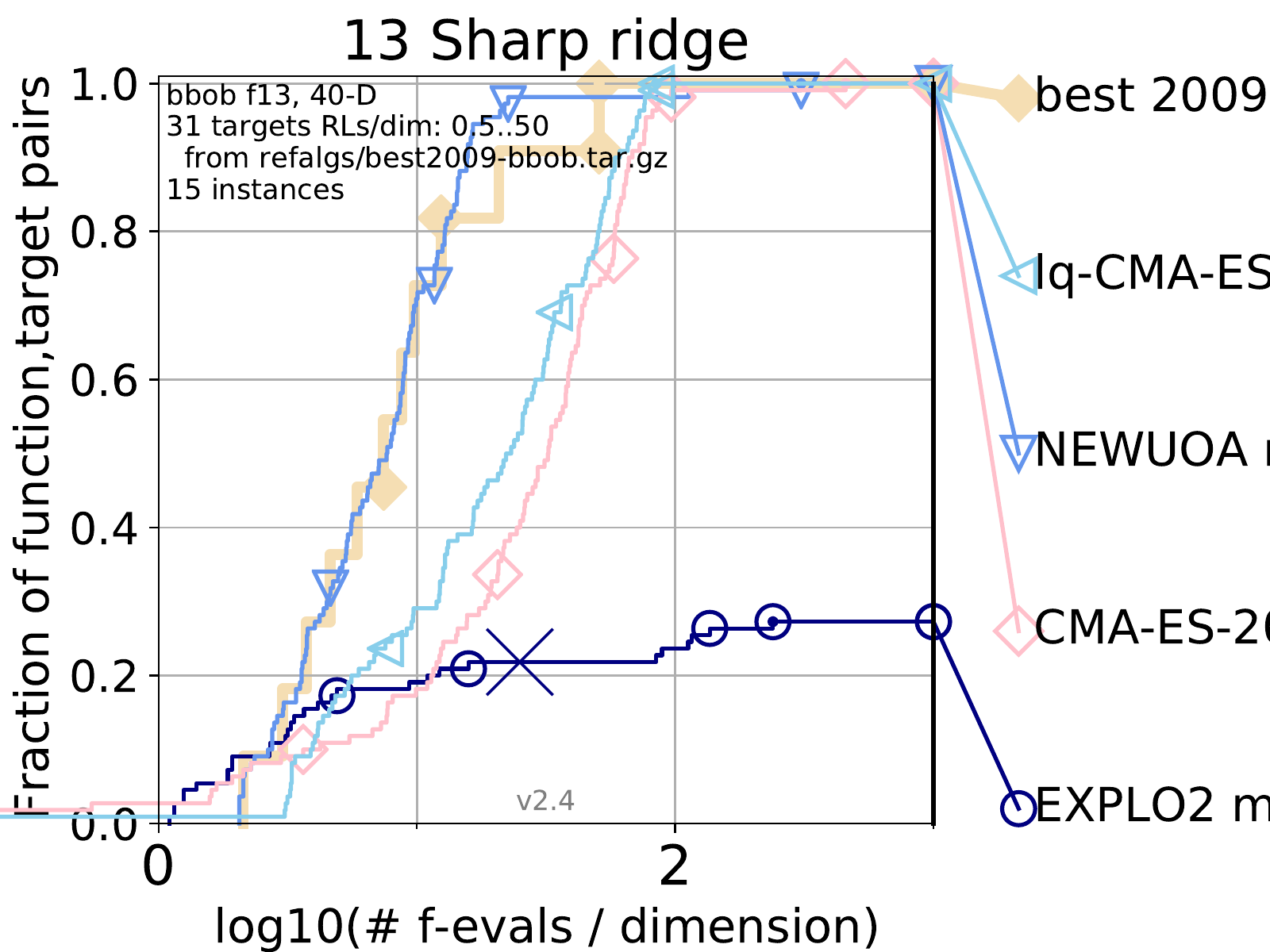}&
\includegraphics[width=0.238\textwidth]{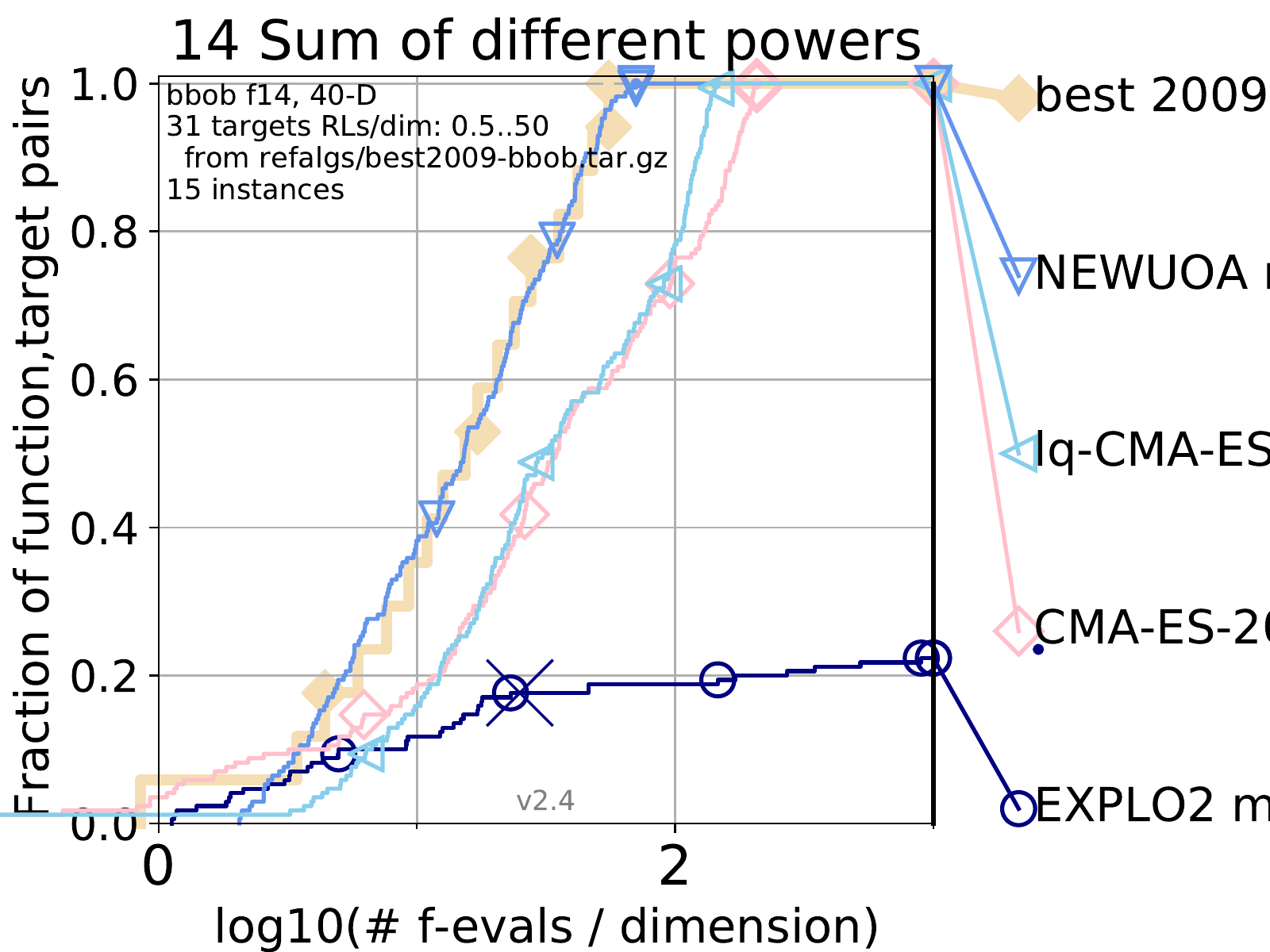}&
\includegraphics[width=0.238\textwidth]{EXPLO_GLOBA_MCS_h_NEWUO_lmm-C_SMAC-_fminc_et_al/pprldmany-single-functions/pprldmany_f015_40D}&
\includegraphics[width=0.238\textwidth]{EXPLO_GLOBA_MCS_h_NEWUO_lmm-C_SMAC-_fminc_et_al/pprldmany-single-functions/pprldmany_f016_40D}\\
\includegraphics[width=0.238\textwidth]{EXPLO_GLOBA_MCS_h_NEWUO_lmm-C_SMAC-_fminc_et_al/pprldmany-single-functions/pprldmany_f017_40D}&
\includegraphics[width=0.238\textwidth]{EXPLO_GLOBA_MCS_h_NEWUO_lmm-C_SMAC-_fminc_et_al/pprldmany-single-functions/pprldmany_f018_40D}&
\includegraphics[width=0.238\textwidth]{EXPLO_GLOBA_MCS_h_NEWUO_lmm-C_SMAC-_fminc_et_al/pprldmany-single-functions/pprldmany_f019_40D}&
\includegraphics[width=0.238\textwidth]{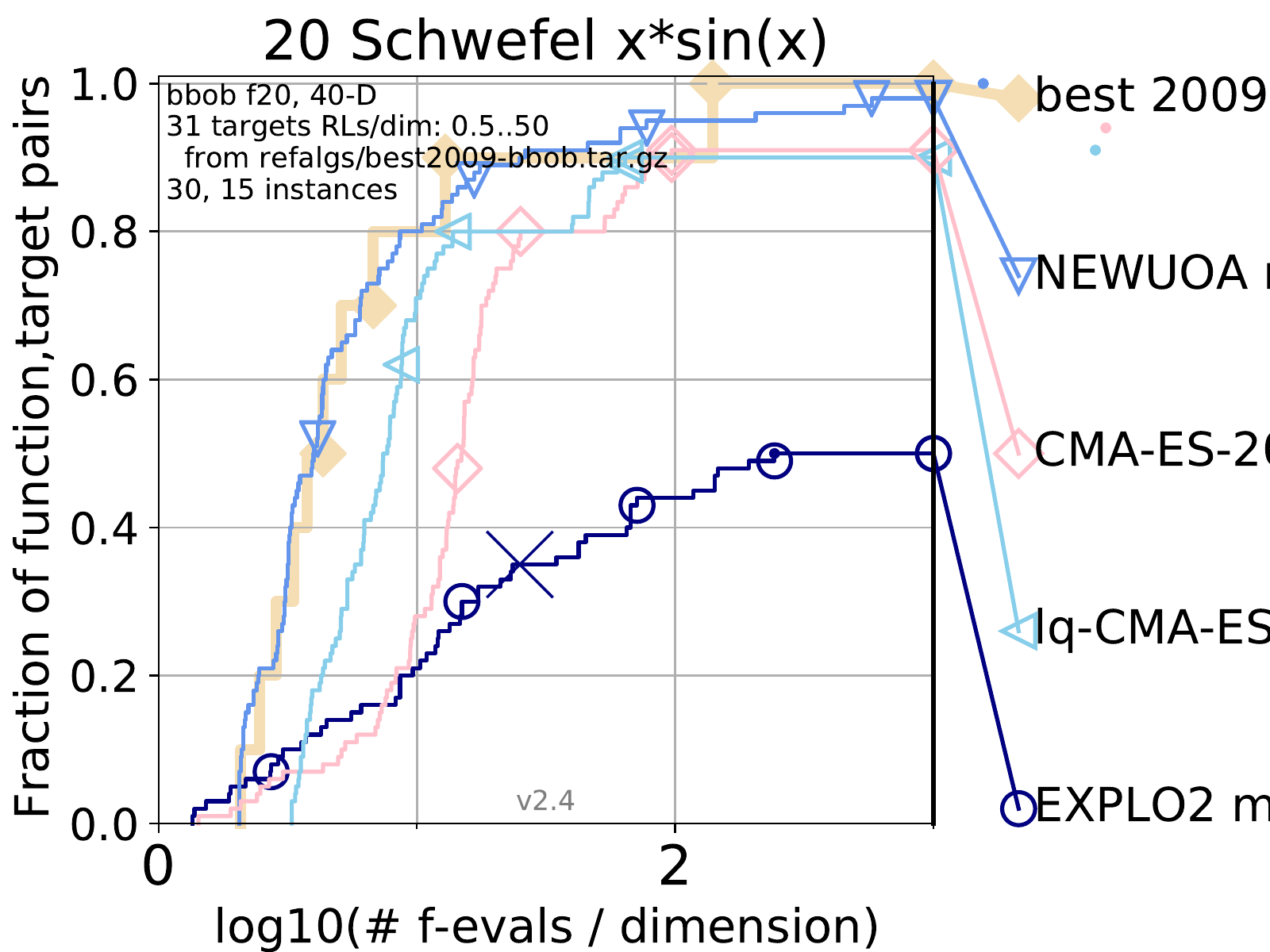}\\
\includegraphics[width=0.238\textwidth]{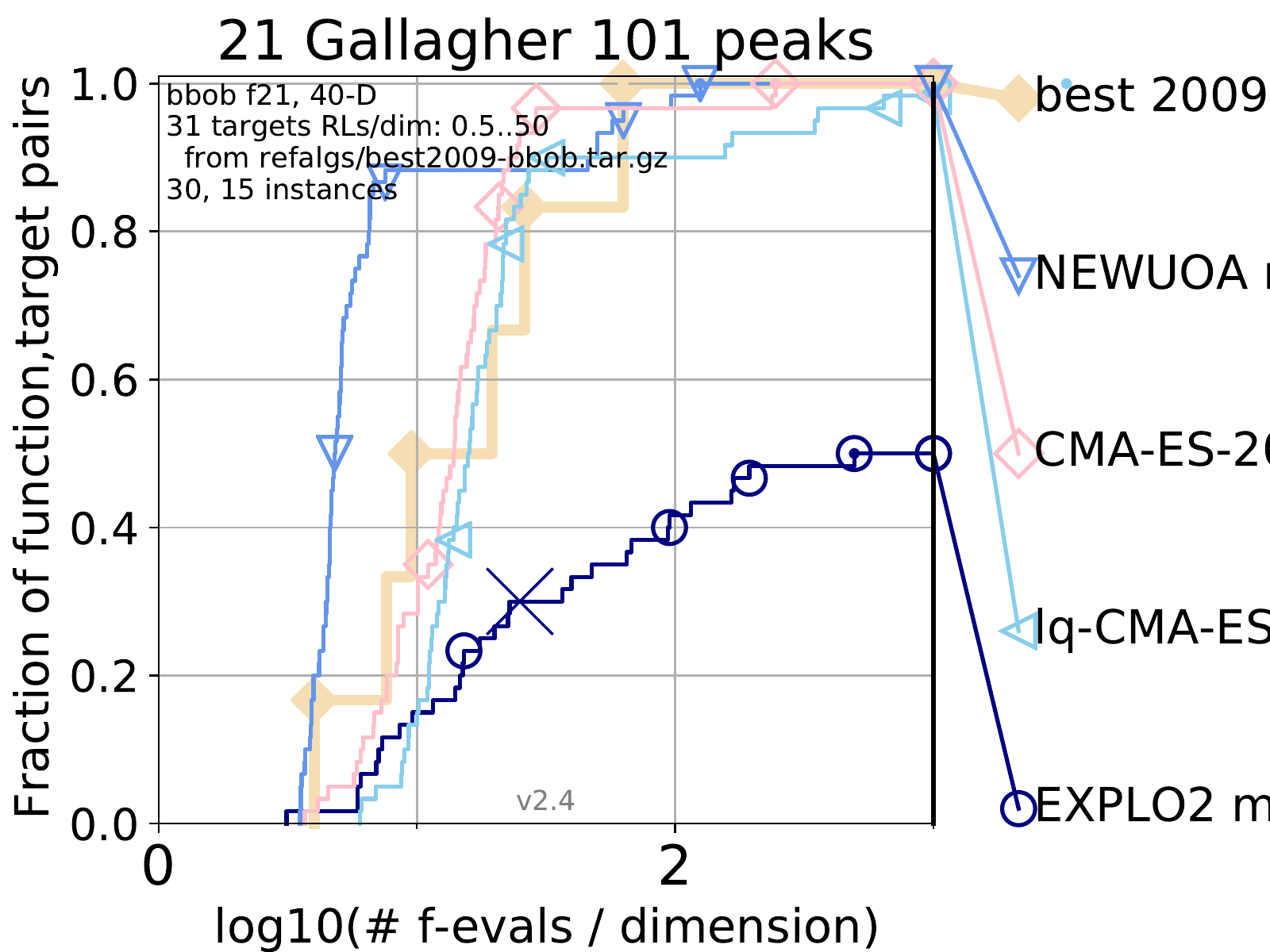}&
\includegraphics[width=0.238\textwidth]{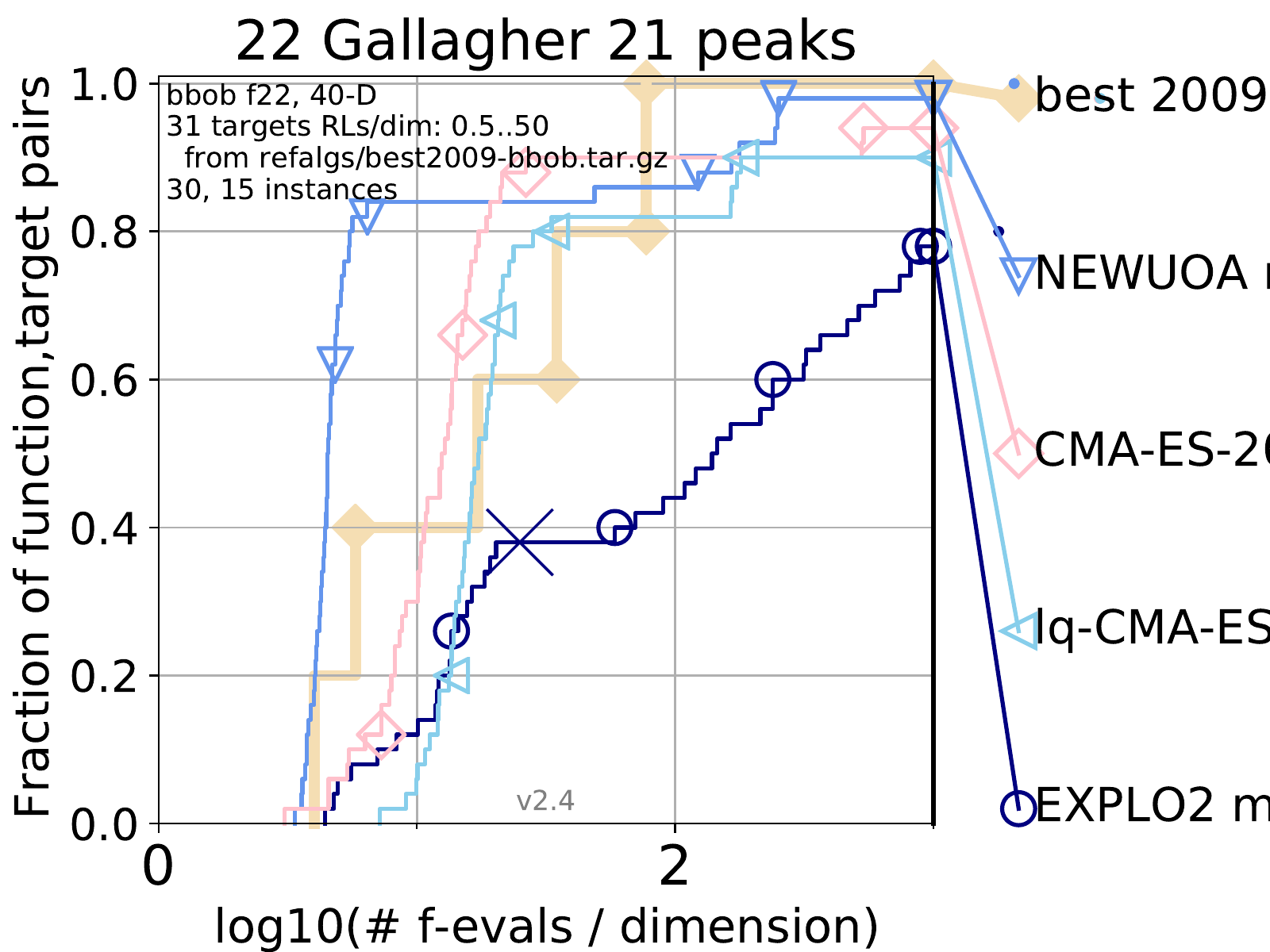}&
\includegraphics[width=0.238\textwidth]{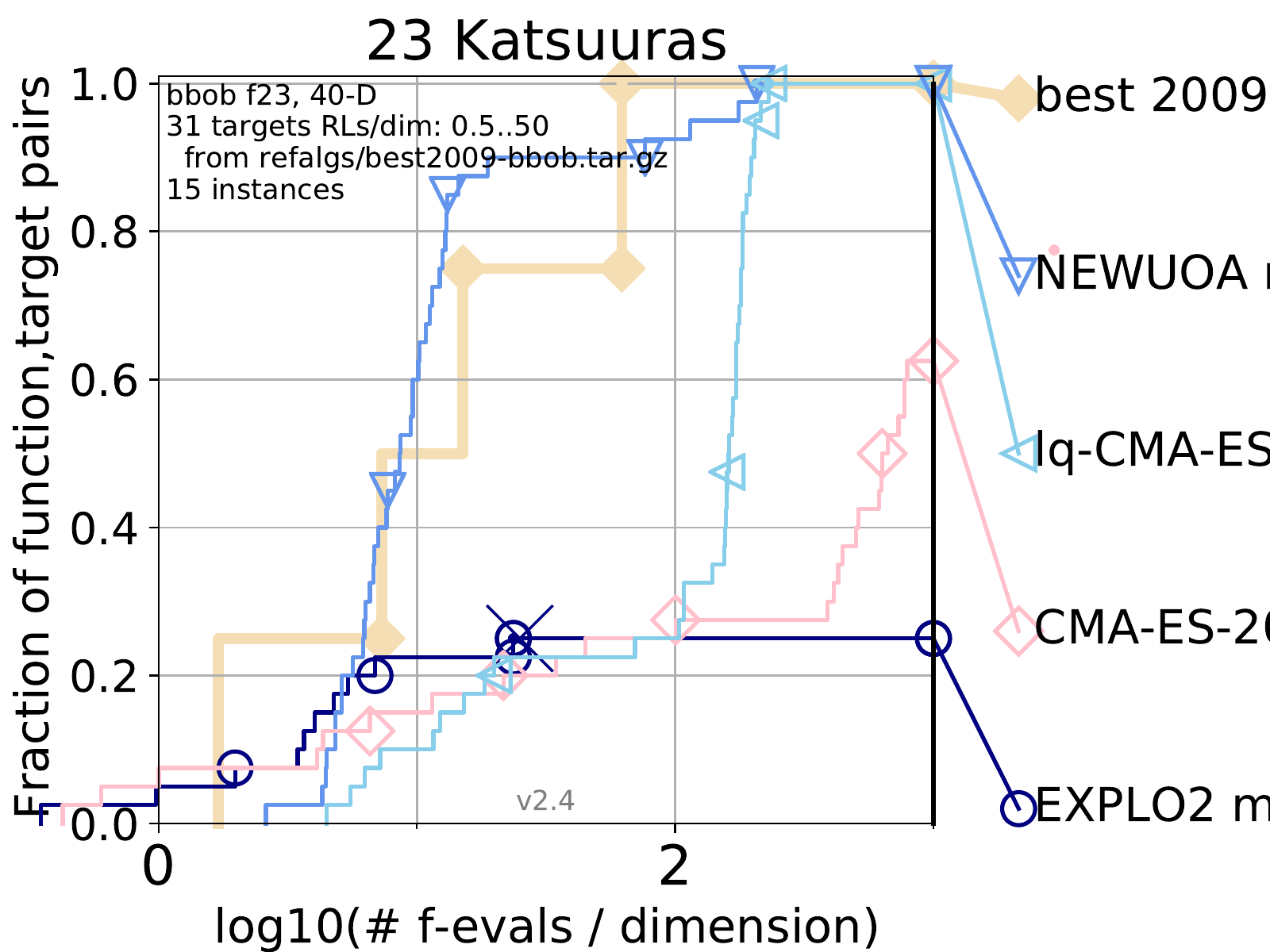}&
\includegraphics[width=0.238\textwidth]{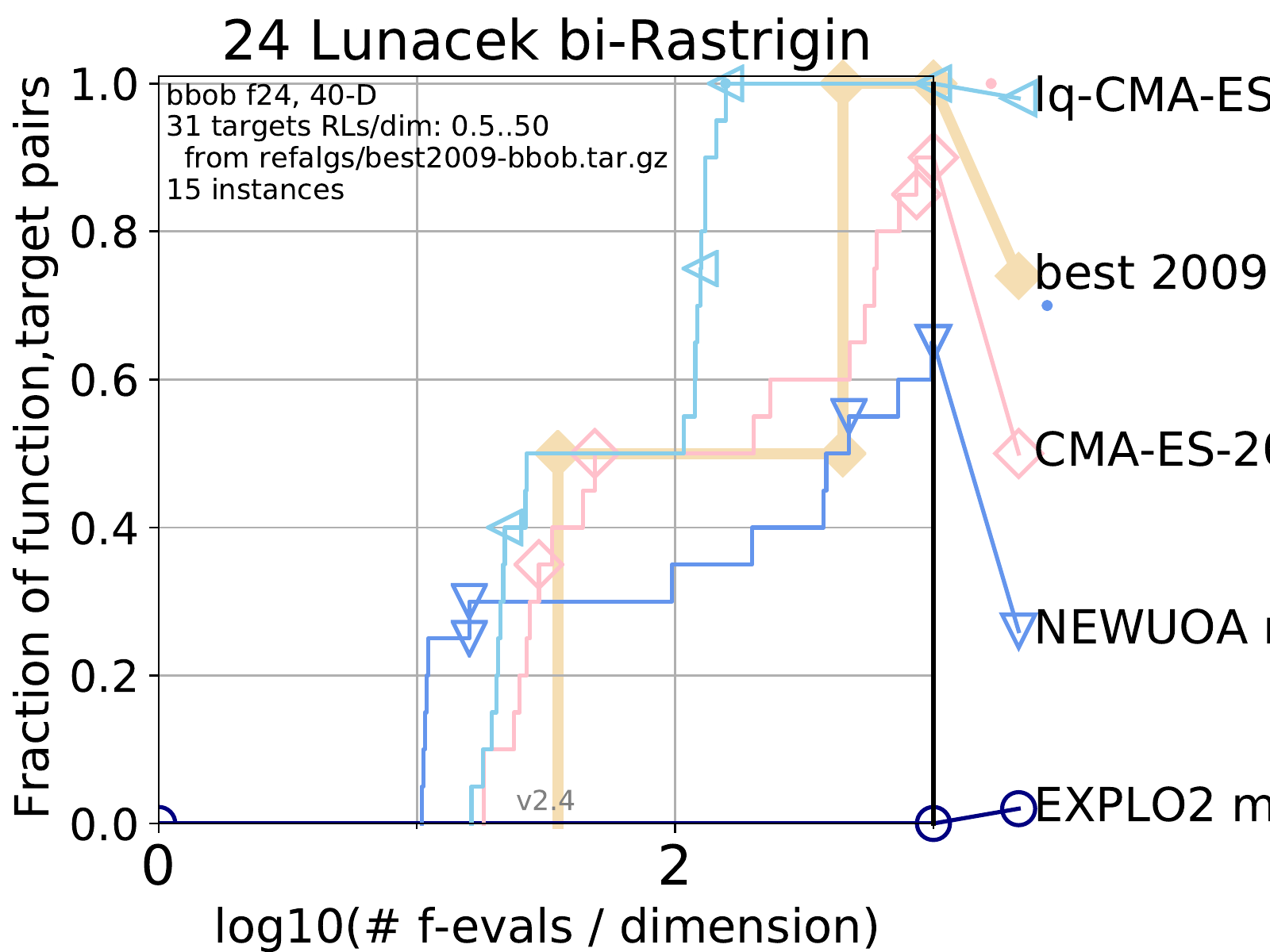}\\
\end{tabular}
\vspace*{-0.2cm}
\caption{
\bbobecdfcaptionsinglefunctionssingledim{40} {\bf Crosses ($\times$) indicate where experimental data ends and bootstrapping begins; algorithms are not comparable after this point.} \texttt{EXPLO2} used default options except for $n_\parallel = 32$. 
}
\end{figure*}

\clearpage

\section{\label{sec:largescale}Results from the large-scale \texttt{BBOB} suite}

\begin{figure*}
\centering
\begin{tabular}{@{}c@{}c@{}c@{}c@{}}
\includegraphics[width=0.48\textwidth]{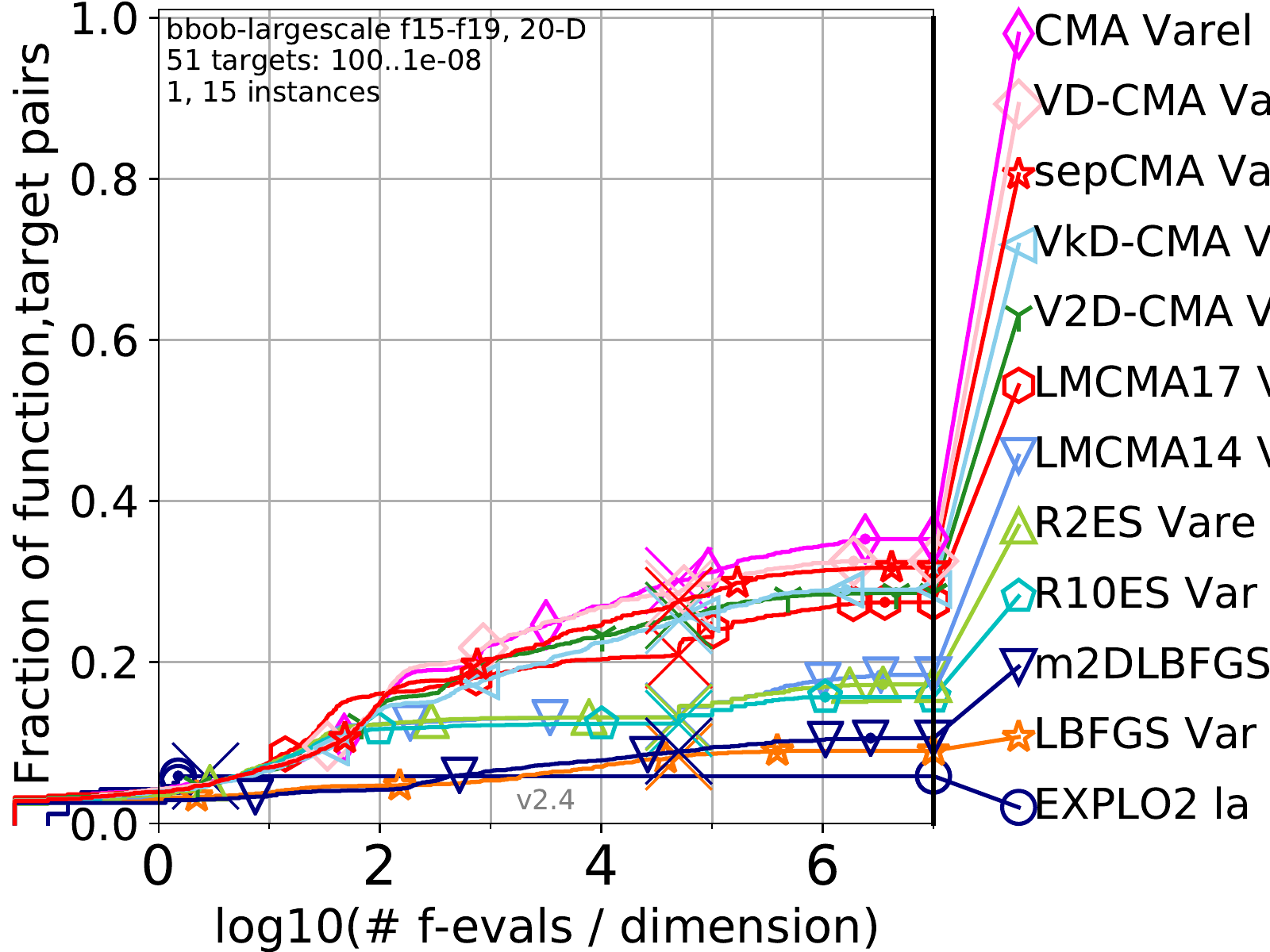}&
\includegraphics[width=0.48\textwidth]{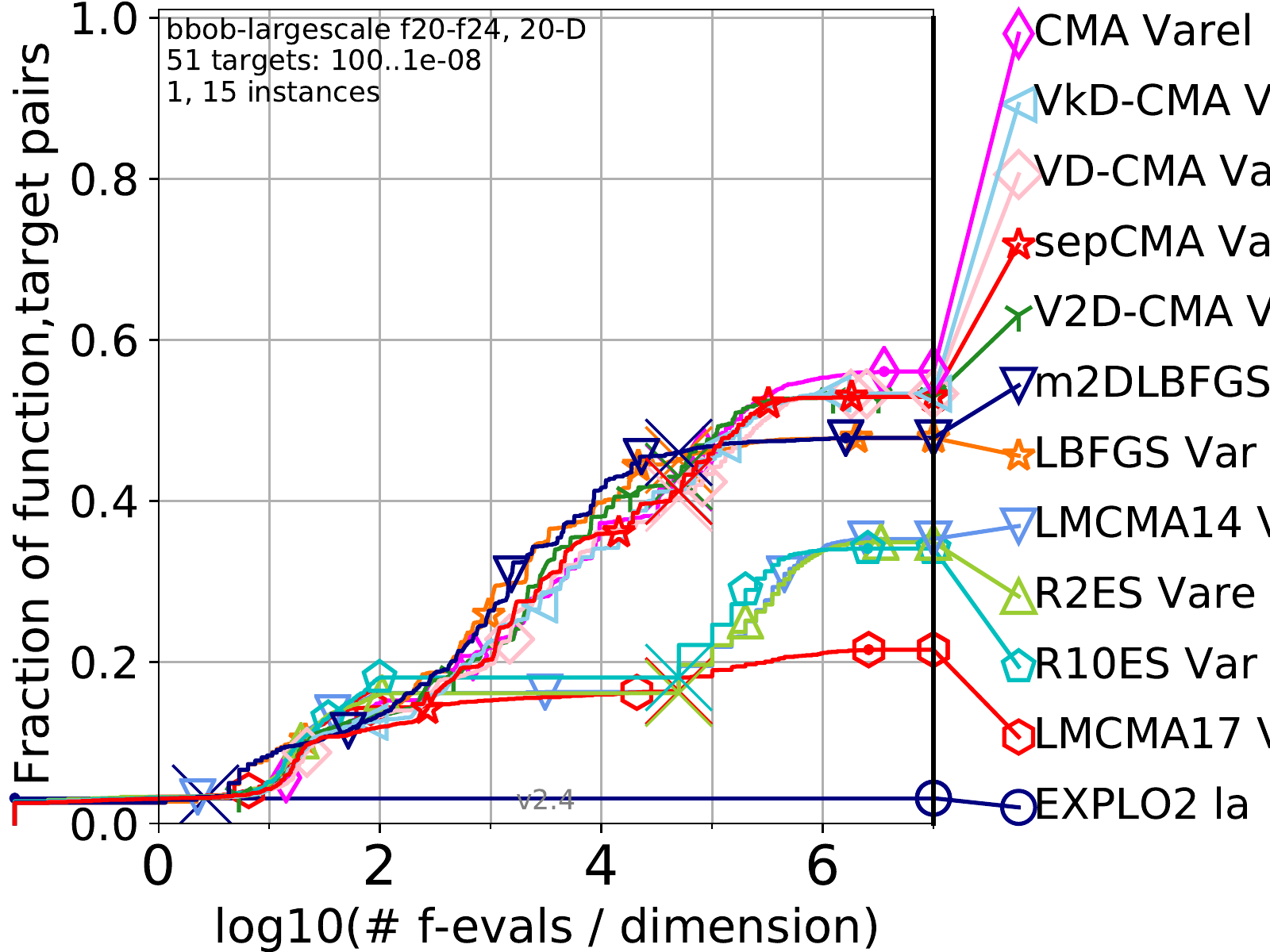}\\
\includegraphics[width=0.48\textwidth]{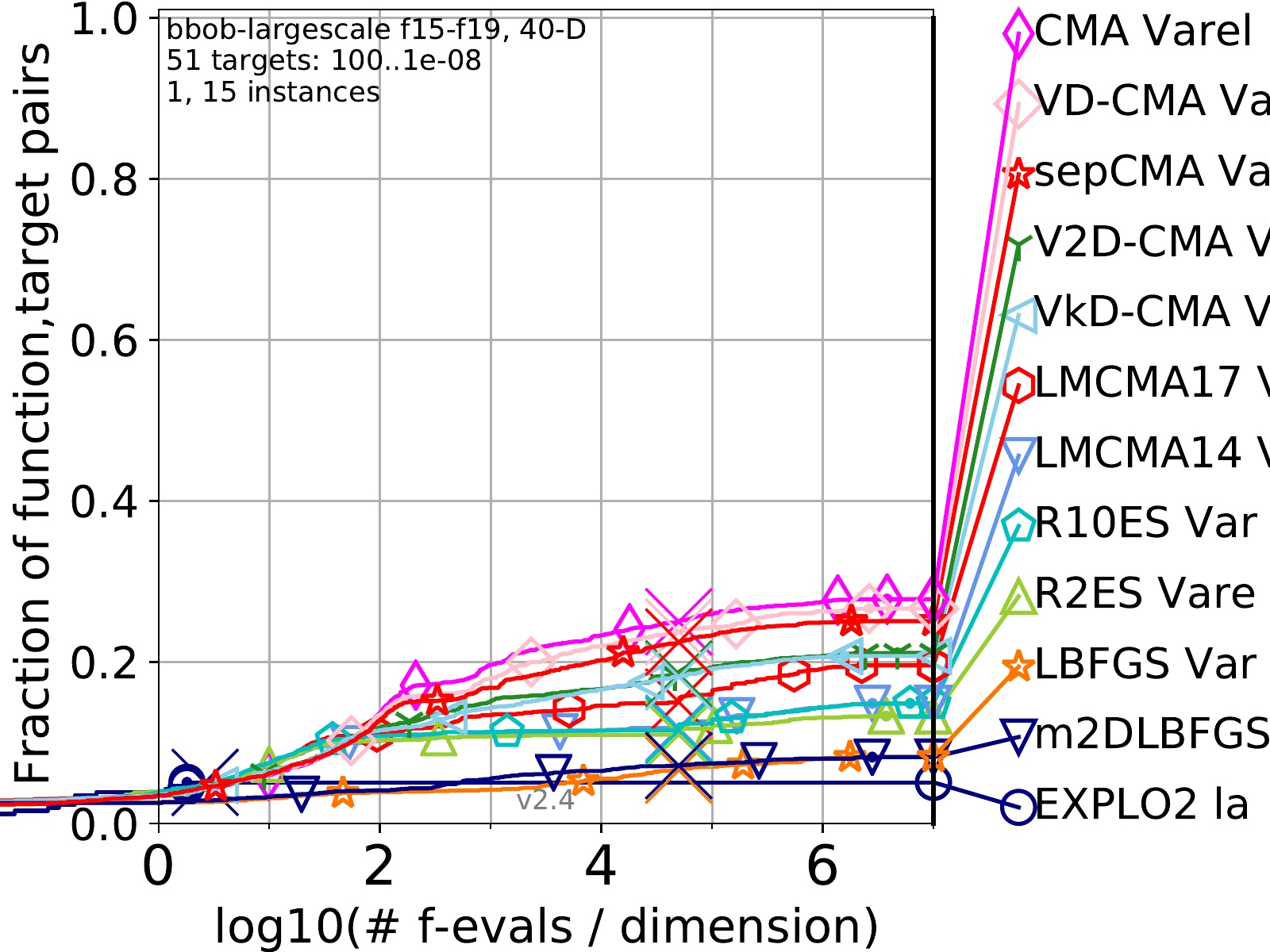}&
\includegraphics[width=0.48\textwidth]{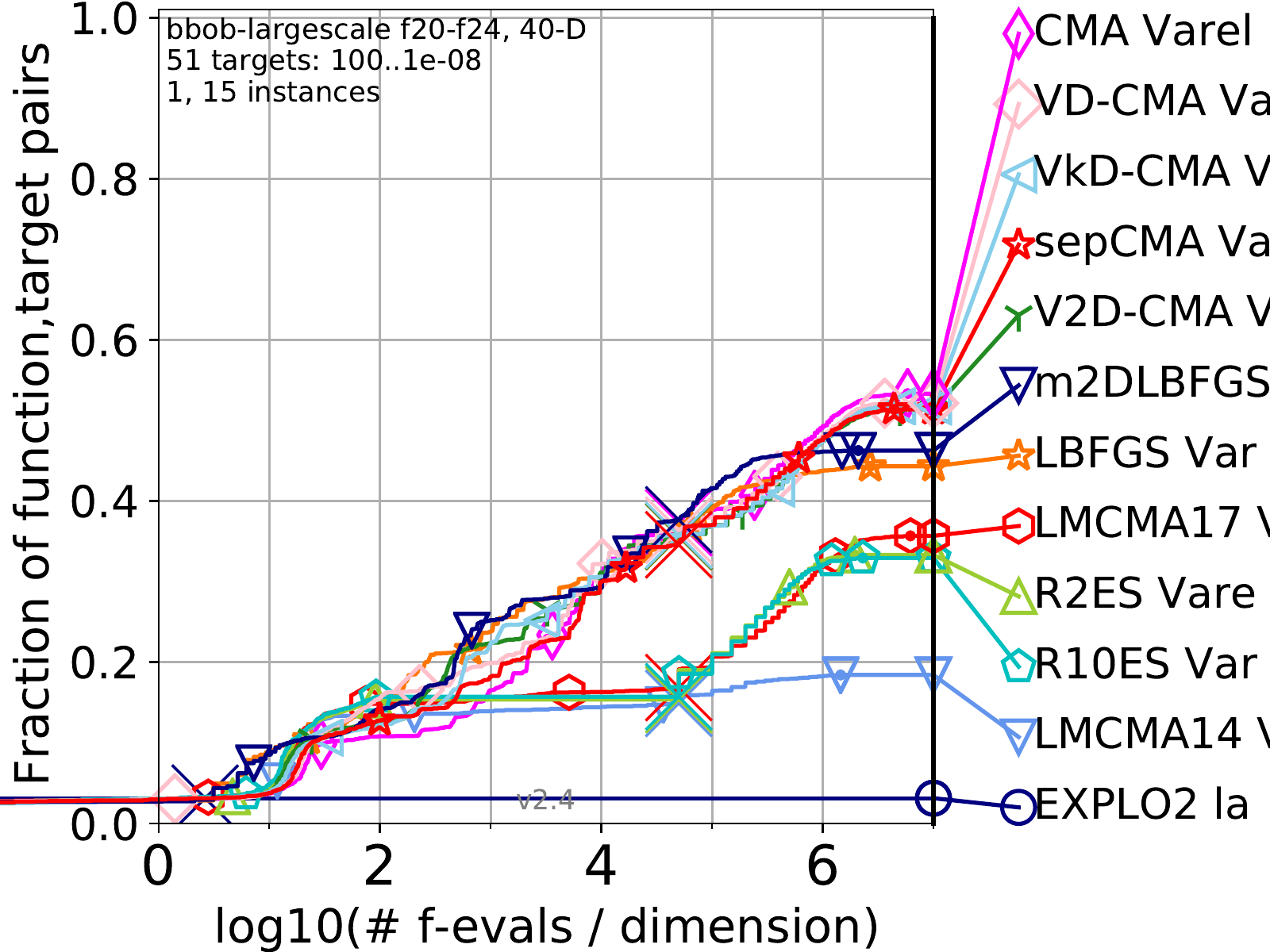}\\
\includegraphics[width=0.48\textwidth]{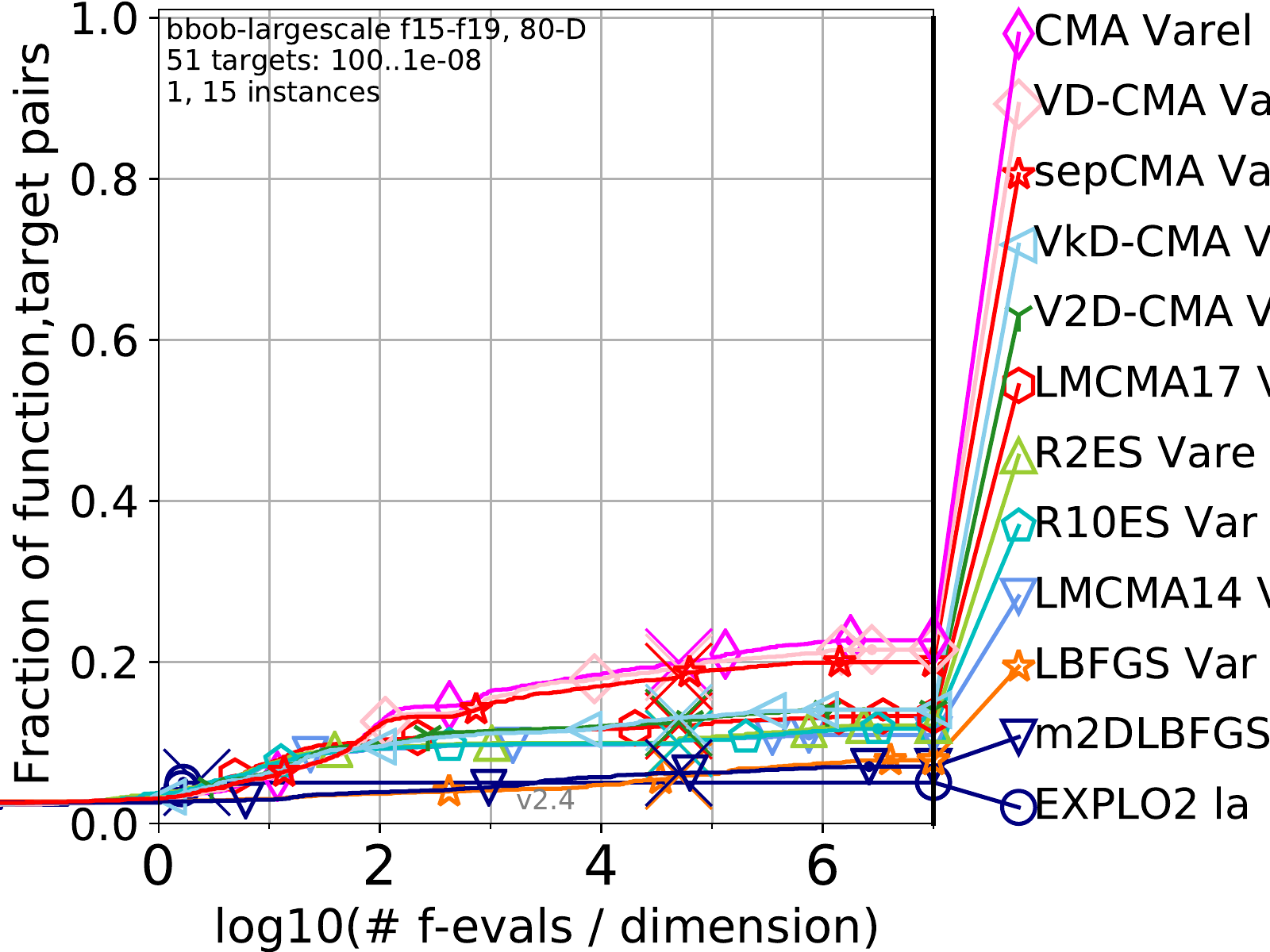}&
\includegraphics[width=0.48\textwidth]{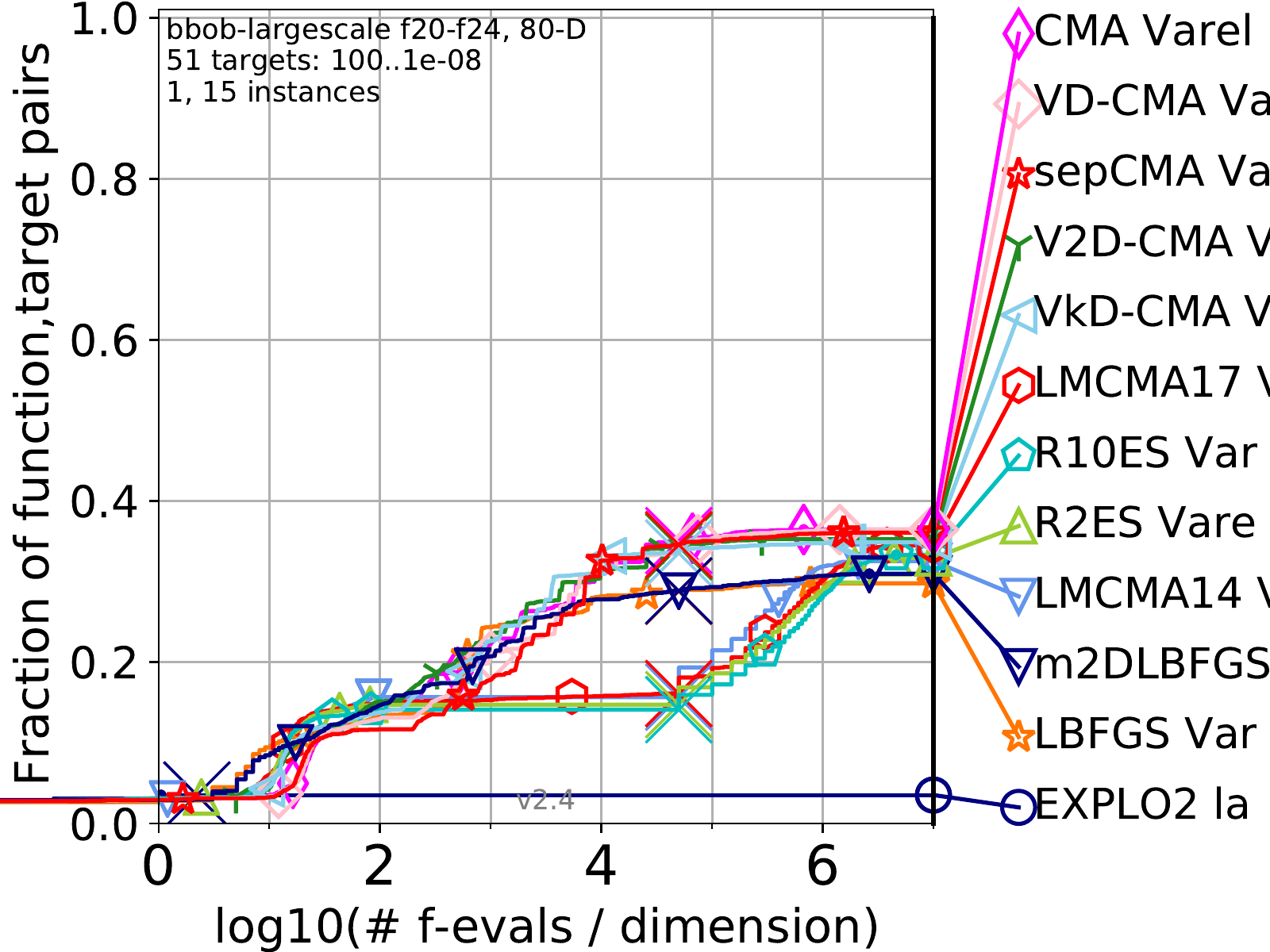}\\
\end{tabular}
\vspace*{-0.2cm}
\caption{
\label{fig:largescale80}
(Left panels) Bootstrapped empirical cumulative distribution of the number of objective function evaluations divided by dimension $D$ for structured multimodal functions $f_{15},\dots , f_{19}$ for $D \in \{20,40,80\}$. The targets are chosen from $10^{[-8 \dots 2]}$ such that the best algorithm from \texttt{BBOB 2009} just not reached them within a given budget of $k \times D$, with 31 different values of $k$ chosen equidistant in logscale within the interval $\{0.5,\dots,50\}$. As reference algorithm, the best algorithm from \texttt{BBOB 2009} is shown as light thick line with diamond markers. {\bf Crosses ($\times$) indicate where experimental data ends and bootstrapping begins; algorithms are not comparable after this point.} \texttt{EXPLO2} used default options except for $n_\parallel = 32$. (Right panels) As in the left panels, but for weakly structured multimodal functions $f_{20},\dots , f_{24}$.
}
\end{figure*}

\begin{figure*}
\centering
\begin{tabular}{@{}c@{}c@{}c@{}c@{}}
\includegraphics[width=0.48\textwidth]{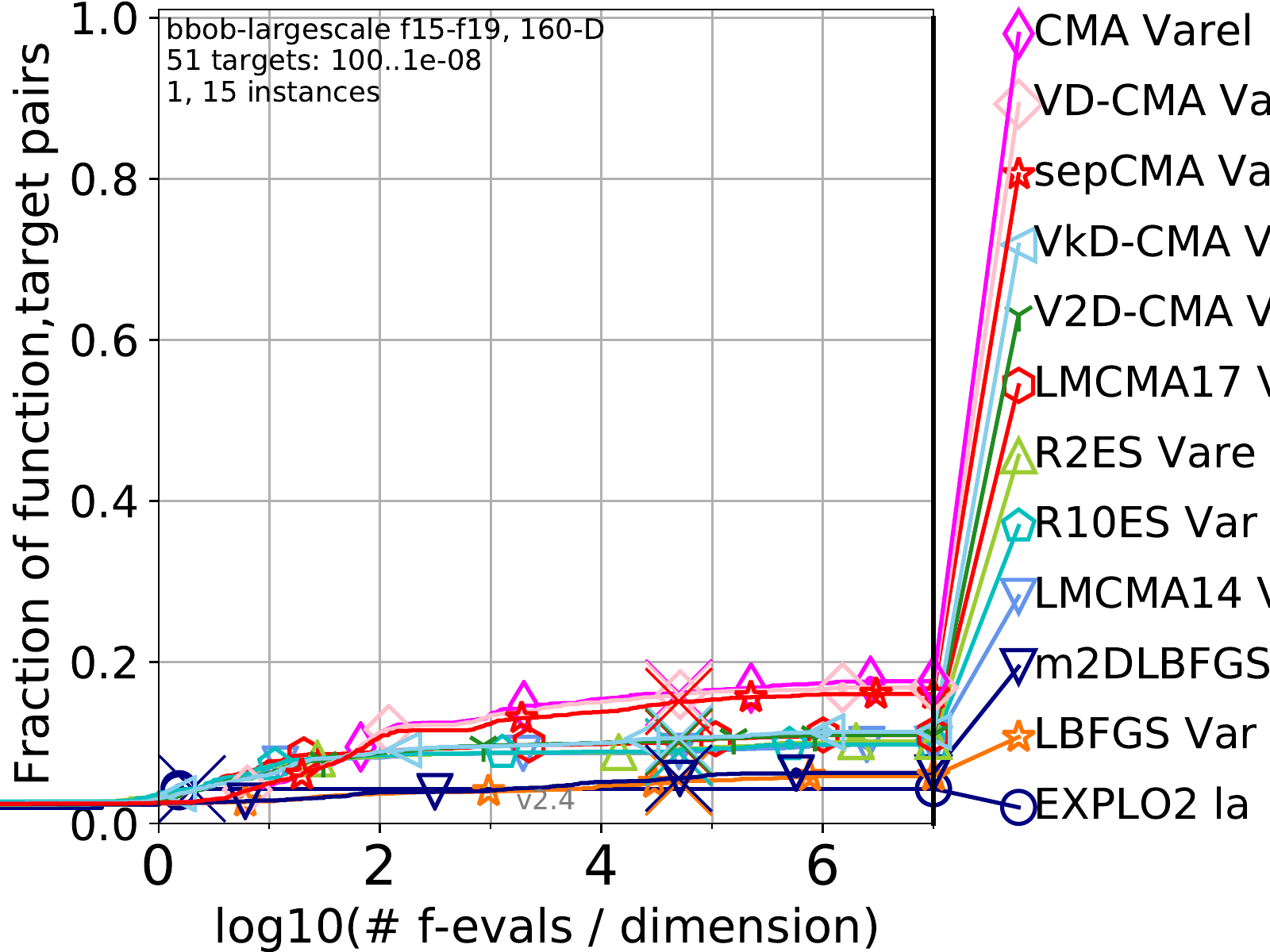}&
\includegraphics[width=0.48\textwidth]{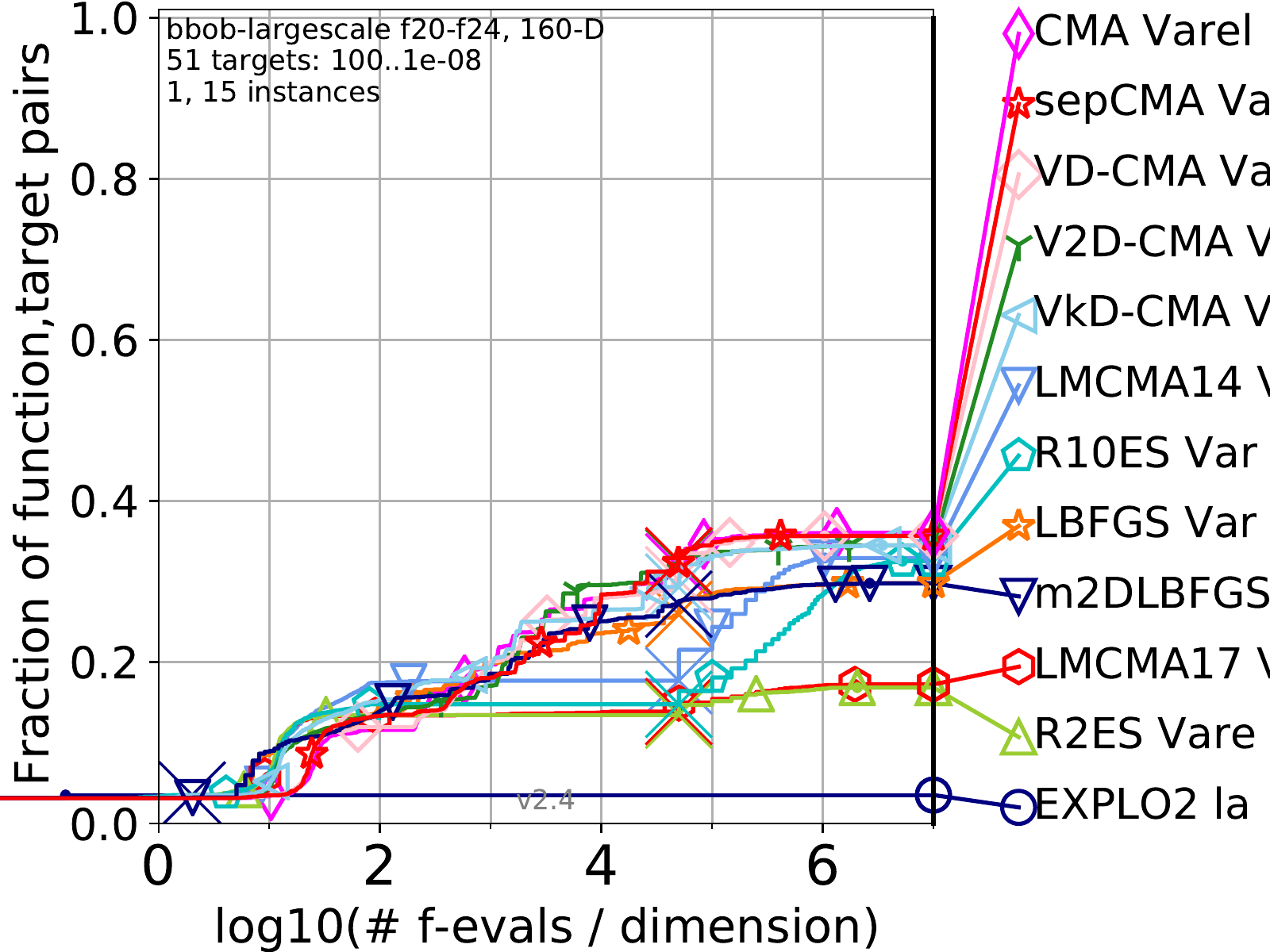}\\
\includegraphics[width=0.48\textwidth]{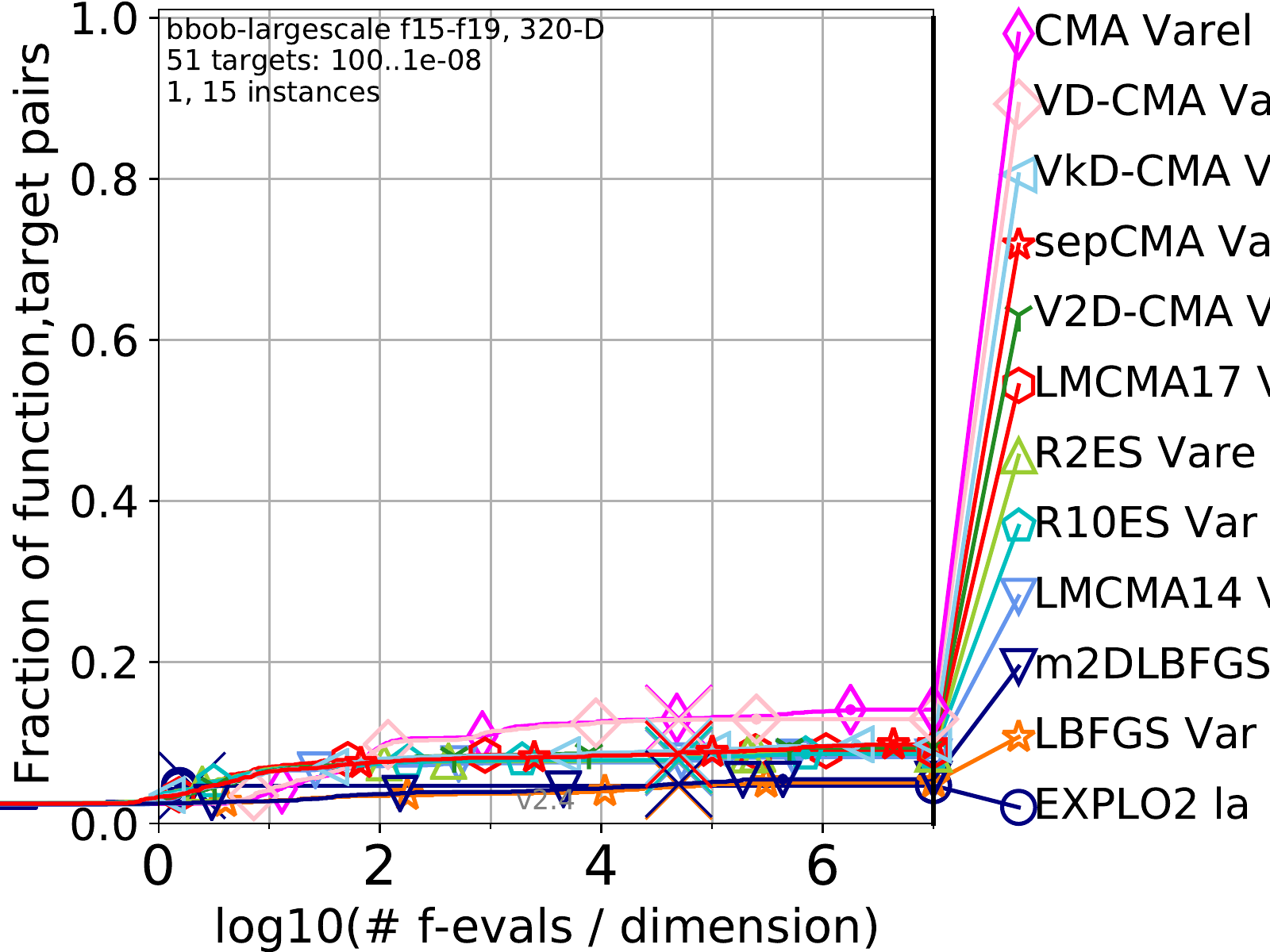}&
\includegraphics[width=0.48\textwidth]{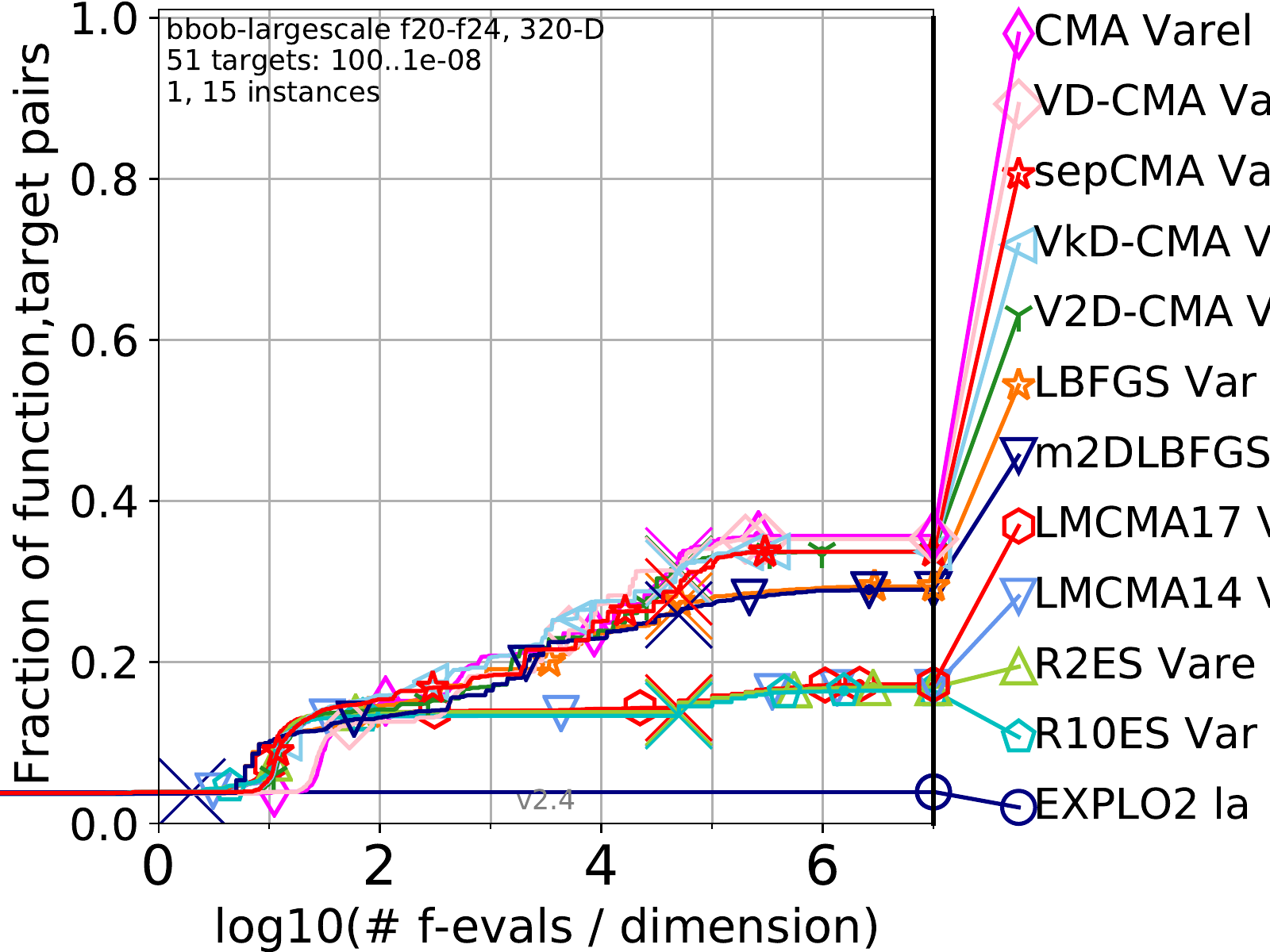}\\
\includegraphics[width=0.48\textwidth]{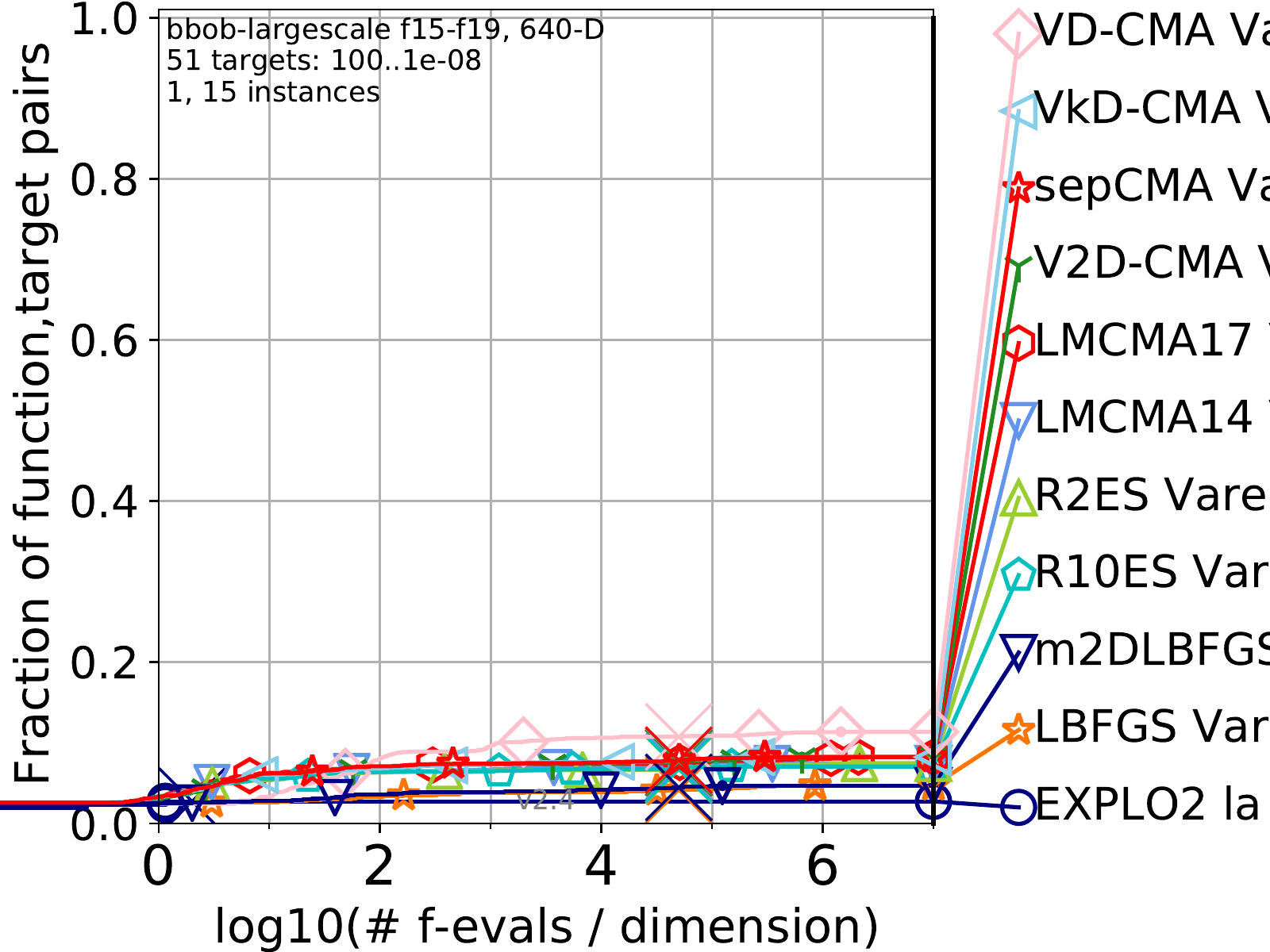}&
\includegraphics[width=0.48\textwidth]{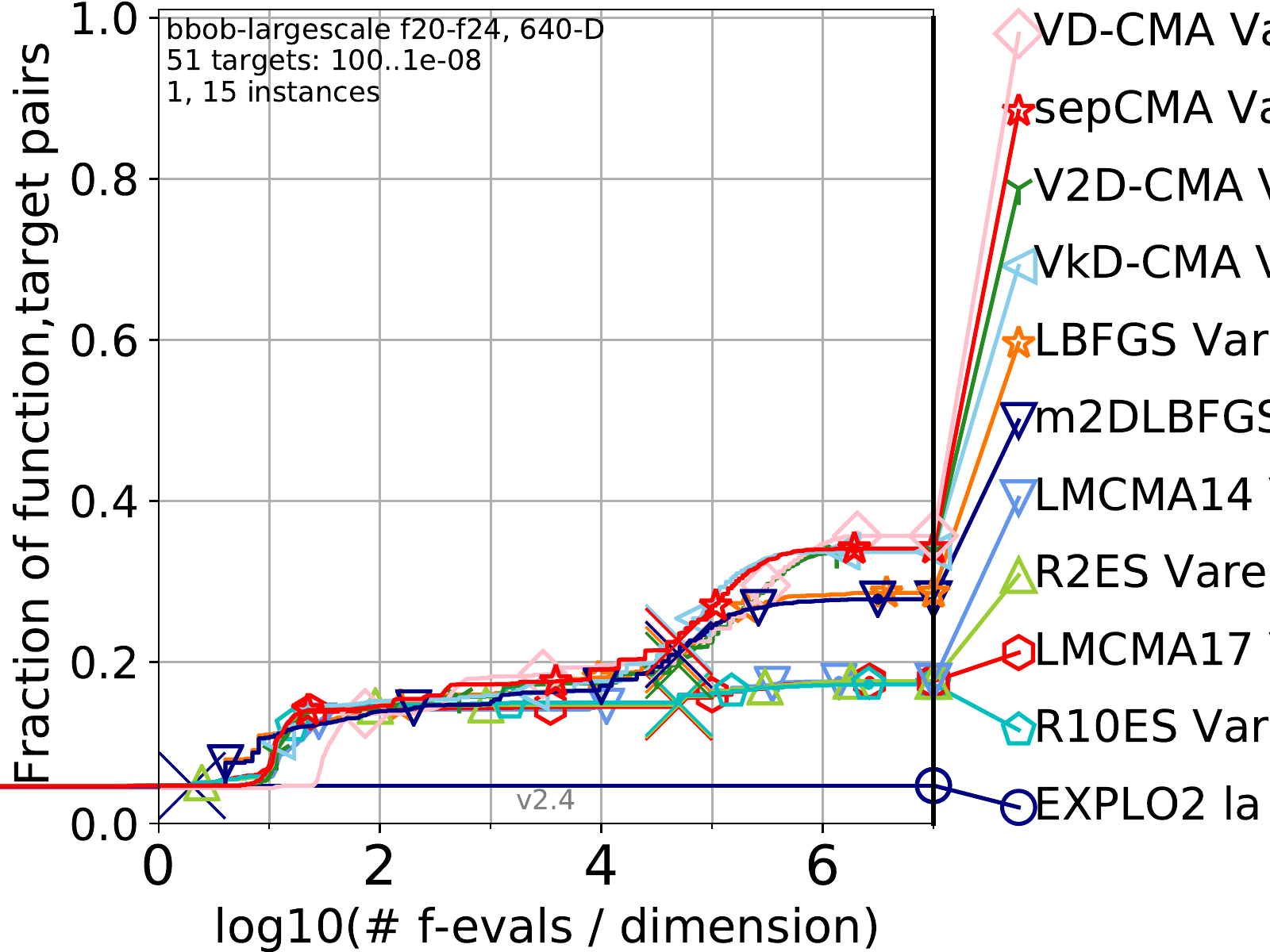}\\
\end{tabular}
\vspace*{-0.2cm}
\caption{
\label{fig:largescale640}
(Left panels) Bootstrapped empirical cumulative distribution of the number of objective function evaluations divided by dimension $D$ for structured multimodal functions $f_{15},\dots , f_{19}$ for $D \in \{160,320,640\}$. The targets are chosen from $10^{[-8 \dots 2]}$ such that the best algorithm from \texttt{BBOB 2009} just not reached them within a given budget of $k \times D$, with 31 different values of $k$ chosen equidistant in logscale within the interval $\{0.5,\dots,50\}$. As reference algorithm, the best algorithm from \texttt{BBOB 2009} is shown as light thick line with diamond markers. {\bf Crosses ($\times$) indicate where experimental data ends and bootstrapping begins; algorithms are not comparable after this point.} \texttt{EXPLO2} used default options except for $n_\parallel = 32$. (Right panels) As in the left panels, but for weakly structured multimodal functions $f_{20},\dots , f_{24}$.
}
\end{figure*}

\clearpage

\section{\label{sec:mixint}Result from the mixed-integer \texttt{BBOB} suite}

\begin{figure*}
\centering
\includegraphics[width=0.48\textwidth]{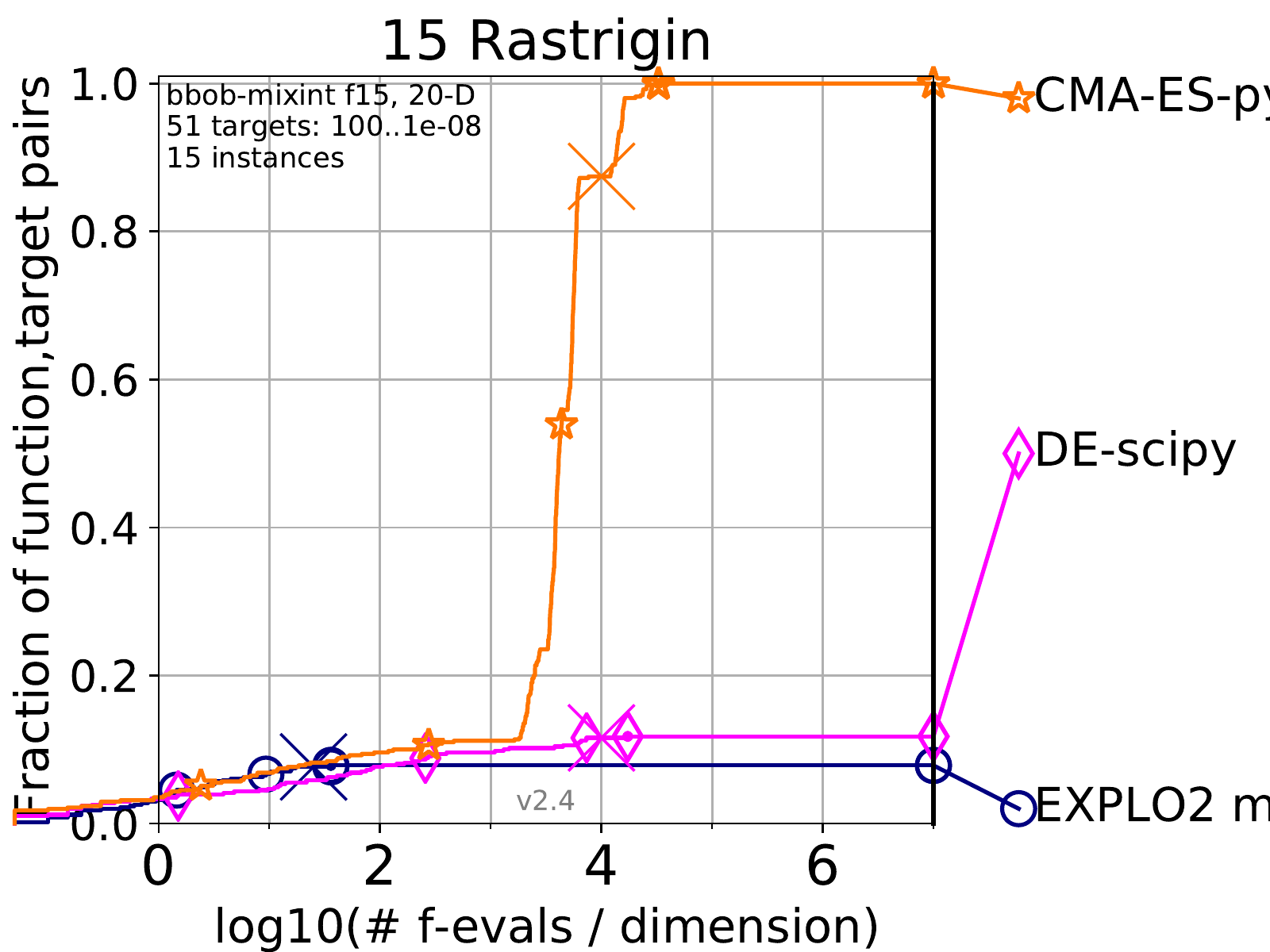}
\vspace*{-0.2cm}
\caption{
Comparison of differential evolution, \texttt{CMA-ES}, and \texttt{EXPLO2} on the mixed-integer version of the Rastrigin function $f_{15}$ for $D = 20$. {\bf Crosses ($\times$) indicate where experimental data ends and bootstrapping begins; algorithms are not comparable after this point.} \texttt{EXPLO2} used default options except for $n_\parallel = 32$.}
\end{figure*}

\clearpage

\section{\label{sec:source}Source code for \texttt{EXPLO2}}

NB. \texttt{parfor} loops are commented out in this code and replaced with \texttt{for} loops. This should be changed if not benchmarking.

{\fontsize{5}{6}
\begin{verbatim}
function [x,y] = explo2(f,lb,ub,N,varargin)

% Optimizer based on a principled tradeoff between exploration (in the
% guise of optimal differential magnitude as diversity proxy) and
% exploitation (in the guise of a convenient radial basis function
% interpolation that is a proxy for the true objective). 
%
% This function requires the optimization toolbox and optionally the
% parallel computing toolbox. (However, it should be relatively
% straightforward to port everything to Octave if needed.)
%
% This approach is only appropriate for functions whose evaluation is
% expensive: as such, it seeks to compete with, e.g., metaheuristics,
% deterministic algorithms such as DIRECT, MCS, or ADC; and other
% surrogate-assisted techniques. The key scaling desiderata are low
% function evaluation budget and high dimension.
%
% By default, execution is serial; this and other options can be set in a
% final input (see below) for which defaults are provided.
%
% Inputs:
%     f,    a handle to a scalar function with column input/output. It is
%           assumed that f takes arguments of the same size as lb/ub
%     lb,   a vector of lower bounds on the domain of f (must be finite)
%     ub,   a vector of upper bounds on the domain of f (must be finite)
%     N,	the number of function evaluations
%     (optional),  options structure with fields including
%         n_parallel,	number of parallel workers
%         n_sample,     number of points for downsampling
%         n_explore,	number of points for computing exploreRange
%         n_tries,      number of tries for optimizing surrogate
%         lambda,       regularization parameter for explore-exploit trade
%         init,         initial evaluation point selection strategy
%     See the default options for detailed explanations. The init option
%     must be one of 'uniform', 'near_corners', or 'corners'.
%
% Outputs: 
%     x,	evaluation points (array of size [dim,N])
%     y,	corresponding function values
%
% Examples:
%     Rastrigin function
%         f = @(x) 10*numel(x)+x(:)'*x(:)-10*sum(cos(2*pi*x(:))); 
%         dim = 2;
%         lb = -5.12*ones(dim,1);
%         ub = 5.12*ones(dim,1);
%     Rosenbrock function
%         f = @(x) sum(100*(x(2:end)-x(1:end-1).^2).^2+(1-x(1:end-1)).^2);
%         dim = 2;
%         lb = -1*ones(dim,1);
%         ub = 3*ones(dim,1);
%     Himmelblau function (dim = 2 only)
%         f = @(x) (x(1)^2+x(2)-11)^2+(x(1)+x(2)^2-7)^2;
%         lb = -5*ones(2,1);
%         ub = 5*ones(2,1);
% NB. Using an anonymous function handle as in these examples causes memory
% to get dragged along during broadcasting. In practice (vs testing),
% however, a function handle should not be anonymous, though likely it
% wouldn't be anyway for this optimization technique to make sense. See
% https://www.mathworks.com/matlabcentral/answers/339728
%
% AKA eeOptimize.m
%
% Last modified 20210727 by Steve Huntsman
%
% Copyright (c) 2021, 2022, Systems & Technology Research. All rights reserved.

%% %% %% %% %% %% HERE FOLLOW ~130 LINES OF PRELIMINARIES %% %% %% %% %% %%

%% Preliminary checks (except for options, which are treated below)
narginchk(4,5); % at most one extra argument, which must be options
if ~isa(f,'function_handle'), error('f must be a function handle'); end
% Bounds
if ~ismatrix(lb), error('lower bound not in matrix form'); end
if ~ismatrix(ub), error('upper bound not in matrix form'); end
if any(size(lb)~=size(ub)), error('lower/upper bound size mismatch'); end
if any(~isreal(lb(:))), error('lower bound not real'); end
if any(~isreal(ub(:))), error('upper bound not real'); end
if any(~isfinite(lb(:))), error('lower bound not finite'); end
if any(~isfinite(ub(:))), error('upper bound not finite'); end
if any(lb>ub), error('lower bound must be less than upper bound'); end
lb = lb(:);
ub = ub(:);
dim = numel(lb);
% Number of function evaluations
if ~isscalar(N), error('N must be scalar'); end
if ~isreal(N), error('N must be real'); end
if ~isfinite(N), error('N must be finite'); end
if N <= dim, error('N must be > dim'); end
if N ~= round(N), error('N must be integral'); end

%% Default options
% Number of parallel workers, set to 1 by default to avoid parfor
defaultOptions.n_parallel = 1;
% Number of points for downsampling. This is necessary to have constant
% runtime per loop iteration. With this in mind, we pick points with
% greatest interpolation error and least values: see main loop.
defaultOptions.n_sample = 100;
% Number of points for computing exploreRange. This is necessary because
% the simple far-field asymptotic behavior does not yield an adequate
% approximation. If n_explore >= 2^dim, then only the corners of the
% bounding box are used below.
defaultOptions.n_explore = 100;
% Number of trials for optimizing surrogate (starting from uniformly random
% location). Quoting [*]:
%     "In problems where the cost of fitting and evaluating a surrogate is
%     much less than the cost of evaluating the true objective function or
%     surrogates is a viable approach to find multiple optima."
% [*] Villanueva, Diane, et al. "Locating multiple candidate designs with
% surrogate-based optimization." WCSMO (2013).
defaultOptions.n_tries = 3;
% Regularization parameter for explore-exploit tradeoff. Alternatives might
% be any sigmoidal function decreasing from 1 at 0 to 0 at 1, e.g., 
%     bf = bumpfun(0,1,0); lambda = arrayfun(bf,linspace(0,1,N));
% The bumpfun function is available from Steve Huntsman.
% 
% Here, we use the following decidedly non-sigmoidal function:
% exp(-(log(2)/log(delta))*log(1-t)) takes the value 1/2 at 1-delta. Taking
% delta = 1/2 yields an affine function, which we have used. Assuming that
% N scales like C*dim, taking delta = 1/dim is also natural (and commented
% out).
delta = 1/2;    % corresponds to defaultOptions.lambda = 1-t
% if N/dim > 1
%     delta = 1/dim;
% else
%     delta = 1/sqrt(N);
% end
defaultOptions.lambda = exp(-(log(2)/log(delta))*log(linspace(1,0,N)));
% Initial evaluation point selection strategy
defaultOptions.init = 'uniform';

%% Set options, using defaults except as specified via input
availableOptions = fieldnames(defaultOptions);  % ignore anything else
if nargin == 5
    options = varargin{end};
    if ~isstruct(options), error('options must be struct'); end
    if numel(options) > 1
        error('options must have a single entry (multiple fields are OK)'); 
    end
    opts = fieldnames(options);
    % Ignore any fields not in availableOptions; check all options below
    ind = ismember(availableOptions,opts);
    for j = 1:numel(availableOptions)
        opt = availableOptions{j};
        if ind(j)
            % Option is specified via input, so do NOT use the default
            eval([opt,' = options.',opt,';']);  % tolerable use of eval
        else
            % Option is not specified via input, so use the default 
            eval([opt,' = defaultOptions.',opt,';']);   % tolerable use of eval
        end
    end
else
    % Use default options
    for j = 1:numel(availableOptions)
        opt = availableOptions{j};
        eval([opt,' = defaultOptions.',opt,';']);   % tolerable use of eval
    end
end

%% Preliminary checks for options--mostly boilerplate
% n_parallel
if ~isscalar(n_parallel), error('n_parallel must be scalar'); end
if ~isreal(n_parallel), error('n_parallel must be real'); end
if ~isfinite(n_parallel), error('n_parallel must be finite'); end
if n_parallel > 128, error('n_parallel must be <= 128'); end  % ad hoc
if n_parallel ~= round(n_parallel), error('n_parallel must be integral'); end
% n_sample
if ~isscalar(n_sample), error('n_sample must be scalar'); end
if ~isreal(n_sample), error('n_sample must be real'); end
if ~isfinite(n_sample), error('n_sample must be finite'); end
if n_sample < 16, error('n_sample must be >= 16'); end  % ad hoc
if n_sample ~= round(n_sample), error('n_sample must be integral'); end
% n_explore
if ~isscalar(n_explore), error('n_explore must be scalar'); end
if ~isreal(n_explore), error('n_explore must be real'); end
if ~isfinite(n_explore), error('n_explore must be finite'); end
if n_explore < 16, error('n_explore must be >= 16'); end  % ad hoc
if n_explore ~= round(n_explore), error('n_explore must be integral'); end
% n_tries
if ~isscalar(n_tries), error('n_tries must be scalar'); end
if ~isreal(n_tries), error('n_tries must be real'); end
if ~isfinite(n_tries), error('n_tries must be finite'); end
if n_tries < 0, error('n_tries must be >= 0'); end
if n_tries ~= round(n_tries)
    error('n_tries must be integral'); 
end
% lambda
% NB. MATLAB's static analyzer needlessly worries about this being defined
if ~ismatrix(lambda), error('lambda not matrix'); end
if numel(lambda) ~= N
    warning('numel(lambda) ~= N; truncating or padding with zeros');
    if numel(lambda) < N
        lambda = [lambda(:)',zeros(1,N-numel(lambda))];
    else
        lambda = lambda(1:N);
    end
end
if ~all(isfinite(lambda)&isreal(lambda))
    error('lambda must be real and finite'); 
end
if any((lambda<0)|(lambda>1))
    warning('lambda not between 0 and 1');
end
% initialization
initStrategies = {'uniform','near_corners','corners'};
if ~ischar(init), error('init must be a char array'); end
ind = ismember(initStrategies,init);
if nnz(ind) ~= 1
    error('init must be ''uniform'', ''near_corners'', or ''corners''');
end

%% %% %% %% %% %% ACTUAL SUBSTANTIVE CODE BEGINS HERE %% %% %% %% %% %%

%% Initial locations for function evaluations
% There are three available strategies:
%     'uniform',        uniform random sampling;
%     'near_corners',	uniform random sampling in small boxes inscribed in 
%                       bounding box corners; 
%     'corners',        deterministic sampling at bounding box corners
% 'uniform' has all the usual advantages and disadvantages, whereas
% 'corners' has the advantage of determinism and the feature of quickly
% putting pressure on the interior of the bounding box, i.e., it
% effectively foreshortens the exploration phase. 'near_corners' has the
% advantage of avoiding degeneracies and artifacts that 'corners' can
% cause, particularly on test problems with symmetry. Both 'near_corners'
% and 'corners' take points near/at lb and corners at locations from it
% that are parallel to the coordinate axes.
N0 = dim+1; % we need this many initial points for general position stuff
if strcmp(init,'uniform')
    x = diag(ub-lb)*rand(dim,N0)+diag(lb)*ones(dim,N0);
    % This assertion fails with measure zero--pretty safe bet...
    ux = unique(x','rows')'; 
    assert(size(ux,2)==size(x,2),'duplicate inputs!');	% ...check anyway
elseif strcmp(init,'near_corners')
    extent = ub-lb;
    margin = .1;
    x = [lb,lb*ones(1,dim)+diag(extent)*(1-margin)]...
        +margin*diag(extent)*rand(dim,dim+1);
elseif strcmp(init,'corners')   % deterministic
    extent = ub-lb;
    x = [lb,lb*ones(1,dim)+diag(extent)];
else
    error('bad init option: this was supposedly already checked!');
end

%% Finish initializing x and y
x = [x,nan(dim,N-N0)];
y = nan(1,N);
if n_parallel > 1
    for n = 1:N0    % use parfor in practice
        y(n) = f(x(:,n));   % broadcast overhead OK in practice
    end
else
    for n = 1:N0
        y(n) = f(x(:,n));
    end
end

%% Initialize t
% Using t >= [positive cutoff] would be problematic because of clustering
% that drives the cutoff to large values. So we work in the limit t -> 0.
% However, using the analytical results of the limit typically gives
% fmincon problems, so we avoid that below.
t = sqrt(eps);

%% Initialize relative error 
relErr = nan(1,N);

%% For estimating range of explore function in main loop
lbub = [lb,ub];
if n_explore >= 2^dim
    bits = double(dec2bin((1:2^dim)-1,dim))'-48;    % ASCII hack
    indCorners = (1:dim)'*ones(1,2^dim)+dim*bits;
end

%% Main loop
n = N0+1;
while n < N
    %%
    disp(['optimizer: function evaluation ',num2str(n),'/',num2str(N)]);
    
    %% Downsample 
    % We pick points with greatest interpolation error and least values.
    % Provided that lambda appropriately decreases to move from exploration
    % to exploitation, incorporating it here tends to explore boundaries
    % more at first; later, it avoids "superfluous" minima
    n_err = round(n_sample*min(1,lambda(n)/lambda(1)));
    if n > n_sample
        [~,ind_err] = sort(relErr,'descend');	% N0 NaNs at first is OK
        topErrors = ind_err(1:n_err);
        % Get indices for least values of y that aren't in topErrors
        [~,ind_y] = sort(y);
        leastValues = setdiff(ind_y,topErrors,'stable');
        ind_sample = [topErrors,leastValues(1:(n_sample-n_err))];
    else
        ind_sample = 1:(n-1);
    end
    x_sample = x(:,ind_sample);
    y_sample = y(ind_sample);
                
    %% Update d, Z, w
    d = squareform(pdist(x_sample'));
    Z = exp(-t*d);
    w = Z\ones(size(Z,1),1);
    
    %% Exploit using Laplacian RBF interpolation
    interpDist = @(xn) vecnorm(x_sample-xn*ones(1,size(x_sample,2)));
    zeta = @(xn) exp(-t*interpDist(xn))';
    interpCoeff = y_sample/Z;
    exploit = @(xn) interpCoeff*zeta(xn);   % surrogate (vs proxy below)
    exploitRange = max(y_sample)-min(y_sample);
    exploitRange = exploitRange+(exploitRange==0);  % for safety    
    
    %% Sequentially select n_parallel points for parallel evaluation later
    for j = 1:n_parallel

        %% Explore using change in magnitude
        explore = @(xn) ((1-zeta(xn)'*w)^2)/(1-zeta(xn)'/Z*zeta(xn));

        %% Estimate range of explore by sampling corners
        warning off;    % explore may be inaccurate; don't remind ourselves
        if n_explore < 2^dim
            % Sample because there are too many corners. Don't worry about
            % possible duplicates--these are unlikely to matter much
            exploreCorners = nan(1,n_explore);
            for k = 1:n_explore
                exploreCorners(k) = explore(lbub((randi(2,dim,1)-1)*dim+(1:dim)'));
            end
        else
            % No point in sampling: just evaluate at all corners
            exploreCorners = nan(1,2^dim);
            for k = 1:2^dim
                exploreCorners(k) = explore(lbub(indCorners(:,k)));
            end
        end
        exploreRange = max(exploreCorners);
        exploreRange = exploreRange+(exploreRange==0);  % for safety
        warning on;

        %% Well-behaved surrogate that incorporates exploration
        % Normalize explore and exploit to be approximately in unit
        % interval. The minus sign is because we want to minimize both
        % terms.
        surrogate = @(xn) exploit(xn)/exploitRange...
            -lambda(n)*explore(xn)/exploreRange;

        %% New evaluation point is local optimum of proxy objective
        optopt = optimoptions(@fmincon,'Display','off');   % no stdout
        try
            warning off;    % can be inaccurate; don't remind ourselves
            surrogateOptimum = Inf;
            % Do not bother with parfor on the surrogate
            for k = 1:n_tries
                x0 = diag(ub-lb)*rand(dim,1)+diag(lb)*ones(dim,1);  % uniform
                [x_sur,y_sur] = ...
                    fmincon(surrogate,x0,[],[],[],[],lb,ub,[],optopt);
                if y_sur < surrogateOptimum
                    xn = x_sur;
                    surrogateOptimum = y_sur;
                else
                    break;  % stop trying as soon as there's no improvement
                end
            end
            warning on;
        catch
            % We don't expect this to happen...
            disp('fmincon failed; evaluating at uniformly random point');
            pause(.01);
            x0 = diag(ub-lb)*rand(dim,1)+diag(lb)*ones(dim,1);  % uniform
            xn = x0;
        end
        
        %% Update x_sample
        x_sample = [x_sample,xn];
        
        %% Update d, Z, w
        d = squareform(pdist(x_sample'));
        Z = exp(-t*d);
        w = Z\ones(size(Z,1),1);
        
        %% Update function handles that feed into explore
        interpDist = @(xn) vecnorm(x_sample-xn*ones(1,size(x_sample,2)));
        zeta = @(xn) exp(-t*interpDist(xn))';
        
    end
       
    %% Memorialize and evaluate new points (in parallel if n_parallel > 1)
    % Unfortunately, running parfor inside a loop is slow, since the
    % parallel pool must be reinitialized on each iteration. There does not
    % seem to be any way around this, though.
    x(:,n-1+(1:n_parallel)) = x_sample(:,(end-n_parallel+1):end);
    temp_y = nan(1,n_parallel);	% this placeholder is needed for parfor
    if n_parallel > 1
        for j = 1:n_parallel	% use parfor in practice
            temp_y(j) = f(x(:,n-1+j));	% Broadcast overhead OK in practice
        end
    else
        temp_y = f(x(:,n));
    end
    y(n-1+(1:n_parallel)) = temp_y;

    %% Get relative errors of interpolation for downsampling 
    relErr = nan(1,n-1+n_parallel);
    for j = 1:numel(relErr)
        relErr(j) = abs(1-exploit(x(:,j))/y(j));
    end
    
    %% Update n
    n = n+n_parallel;

end
\end{verbatim}
}

\clearpage

\section{\label{sec:sourceBench}Source code for benchmarking}

\subsection{\label{sec:sourceBenchFunction}EXPLO2 version for COCO interface}

{\fontsize{5}{6}
\begin{verbatim}
function [x,y] = coco_explo2 (problem, ...
    lower_bounds, upper_bounds, num_integer_vars, num_constraints, budget)

% cf. my_optimizer.m

f = @(x) cocoEvaluateFunction(problem,x);
lb = lower_bounds(:);
ub = upper_bounds(:);
N = budget;
%% EXPLO2 options (don't comment all of these out)
options.n_parallel = 32;             % default: 1
% options.n_sample = 100;             % default: 100
% options.n_explore = 100;            % default: 100
% options.n_tries = 3;                % default: 3
% options.lambda = linspace(1,0,N);   % default: a bit complicated

% dim = cocoProblemGetDimension(problem);
% B = N; 
% lam = []; 
% while B > dim
%     B = floor(B/2); 
%     lam = [lam,linspace(0,1,B)]; 
% end
% lam = [lam,linspace(0,1,N-numel(lam))]; 
% lam = fliplr(lam);
% options.lambda = lam;

% NN = double(N);
% lam = -log(linspace(1/NN,1,NN));
% options.lambda = lam;

% dim = cocoProblemGetDimension(problem);
% B = double(N); 
% lam = []; 
% while B > dim
%     B = floor(B/2); 
%     lam = [lam,-log(linspace(1/B,1,B))]; 
% end
% lam = [lam,-log(linspace(1/(double(N)-numel(lam)),1,double(N)-numel(lam)))]; 
% options.lambda = double(lam);

dim = cocoProblemGetDimension(problem);
options.lambda = [ones(1,N-dim),linspace(1,0,dim)];

% options.init = 'uniform';           % default: uniform
%%
[x,y] = explo2(f,lb,ub,N,options);
\end{verbatim}
}

\subsection{\label{sec:sourceBenchScript}Benchmarking script}

{\fontsize{5}{6}
\begin{verbatim}
% COCO benchmarking of EXPLO2
%
% [From COCO documentation: This example experiment allows also for easy
% implementation of independent restarts by simply increasing
% NUM_OF_INDEPENDENT_RESTARTS. To make this effective, the algorithm should
% have at least one more stopping criterion than just a maximal budget.]
%
more off; % to get immediate output in Octave

%%%%%%%%%%%%%%%%%%%%%%%%%
% Experiment Parameters %
%%%%%%%%%%%%%%%%%%%%%%%%%
BUDGET_MULTIPLIER = 25; % algorithm runs for BUDGET_MULTIPLIER*dimension funevals
NUM_OF_INDEPENDENT_RESTARTS = 1e9; % max. number of independent algorithm
% restarts; if >0, make sure that the
% algorithm is not always doing the same thing
% in each run (which is typically trivial for
% randomized algorithms)

%%%%%%%%%%%%%%%%%%%%%%%%%
% Prepare Experiment    %
%%%%%%%%%%%%%%%%%%%%%%%%%

% choose a test suite and a matching logger, for
% example one of the following:
%
% bbob               24 unconstrained noiseless single-objective functions
% bbob-biobj         55 unconstrained noiseless bi-objective functions
% [bbob-biobj-ext     92 unconstrained noiseless bi-objective functions]
% bbob-largescale    24 unconstrained noiseless single-objective functions in large dimensions
% [bbob-constrained* 48 constrained noiseless single-objective functions]
% bbob-mixint        24 unconstrained noiseless single-objective functions with mixed-integer variables
% bbob-biobj-mixint  92 unconstrained noiseless bi-objective functions with mixed-integer variables
%
% Suites with a star are partly implemented but not yet fully supported.
%
suite_name = 'bbob';
observer_name = 'bbob';
observer_options = strcat('result_folder: EXPLO2_on_', ...
    suite_name, ...
    [' algorithm_name: EXPLO2 '...
    ' algorithm_info: EXPLO2 ']);

% initialize suite and observer with default options,
% to change the default, see 
% http://numbbo.github.io/coco-doc/C/#suite-parameters and
% http://numbbo.github.io/coco-doc/C/#observer-parameters
% for details.
% suite = cocoSuite(suite_name, '', '');
% observer = cocoObserver(observer_name, observer_options);
    suite_options = ['dimensions: 20 function_indices: 15-19'];
    suite = cocoSuite(suite_name, '', suite_options);
    observer = cocoObserver(observer_name, observer_options);

    % set log level depending on how much output you want to see, e.g. 'warning'
    % for fewer output than 'info'.
    cocoSetLogLevel('info');

    % keep track of problem dimension and #funevals to print timing information:
    printeddim = 1;
    doneEvalsAfter = 0; % summed function evaluations for a single problem
    doneEvalsTotal = 0; % summed function evaluations per dimension
    printstring = '\n'; % store strings to be printed until experiment is finished

    %%%%%%%%%%%%%%%%%%%%%%%%%
    % Run Experiment        %
    %%%%%%%%%%%%%%%%%%%%%%%%%
    while true
        % get next problem and dimension from the chosen suite:
        problem = cocoSuiteGetNextProblem(suite, observer);
        if ~cocoProblemIsValid(problem)
            break;
        end
        dimension = cocoProblemGetDimension(problem);

        % printing
        if printeddim < dimension
          if printeddim > 1
            elapsedtime = toc;
            printstring = strcat(printstring, ...
                sprintf('   COCO TIMING: dimension %d finished in %e seconds/evaluation\n', ...
                printeddim, elapsedtime/double(doneEvalsTotal)));
            tic;
          end
          doneEvalsTotal = 0;
          printeddim = dimension;
          tic;
        end  

        %
        lower_bounds = cocoProblemGetSmallestValuesOfInterest(problem);
        upper_bounds = cocoProblemGetLargestValuesOfInterest(problem);
        num_integer_vars = cocoProblemGetNumberOfIntegerVariables(problem);
        if num_integer_vars, warning('integer vars'); end
        num_constraints = cocoProblemGetNumberOfConstraints(problem);
        if num_constraints, warning('constraints'); end
        dimension = cocoProblemGetDimension(problem);
        budget = BUDGET_MULTIPLIER*dimension;

        rng('default'); % for reproducibility at every stage
        [x,y] = coco_explo2(problem, lower_bounds, upper_bounds, ...
            num_integer_vars, num_constraints, budget);

    doneEvalsTotal = budget;
    end

    elapsedtime = toc;
    printstring = strcat(printstring, ...
        sprintf('   COCO TIMING: dimension %d finished in %e seconds/evaluation\n', ...
        printeddim, elapsedtime/double(doneEvalsTotal)));
    fprintf(printstring);

    cocoObserverFree(observer);
    cocoSuiteFree(suite);
\end{verbatim}
}

\end{document}

%% file: cocopp_commands.tex
\providecommand{\bbobecdfcaptionsinglefunctionssingledim}[1]{
Empirical cumulative distribution of simulated (bootstrapped)
             runtimes, measured in number of objective function evaluations 
             divided by dimension (FEvals/DIM) in 
             dimension #1 and for those targets in
             $10^{[-8..2]}$ that have just not been reached by 
             the best algorithm from BBOB 2009 in a given budget of $k$ $\times$ DIM, with 
             $31$ different values of $k$ chosen 
             equidistant in logscale within the interval $\{0.5, \dots, 50\}$.
}

\providecommand{\bbobppfigslegend}[1]{
Expected running time (\ERT\ in number of $f$-evaluations
                        as $\log_{10}$ value) divided by dimension versus dimension. The target function value
                        is chosen such that the best algorithm from BBOB 2009 just failed to achieve
                        an \ERT\ of $10\times\DIM$. Different symbols correspond to different algorithms given in the legend of #1. Light symbols give the maximum number of function evaluations from the longest trial divided by dimension. Black stars indicate a statistically better result compared to all other algorithms with $p<0.01$ and Bonferroni correction number of dimensions (six).  
Legend: 
{\color{NavyBlue}$\circ$}: \algorithmA
, {\color{Magenta}$\diamondsuit$}: \algorithmB
, {\color{Orange}$\star$}: \algorithmC
, {\color{CornflowerBlue}$\triangledown$}: \algorithmD
, {\color{red}$\varhexagon$}: \algorithmE
, {\color{YellowGreen}$\triangle$}: \algorithmF
, {\color{cyan}$\pentagon$}: \algorithmG
, {\color{ForestGreen}$\downY$}: \algorithmH
, {\color{Lavender}$\Diamond$}: \algorithmI
, {\color{SkyBlue}$\triangleleft$}: \algorithmJ
}
\definecolor{NavyBlue}{HTML}{000080}
\definecolor{Magenta}{HTML}{FF00FF}
\definecolor{Orange}{HTML}{FFA500}
\definecolor{CornflowerBlue}{HTML}{6495ED}
\definecolor{YellowGreen}{HTML}{9ACD32}
\definecolor{Gray}{HTML}{BEBEBE}
\definecolor{Yellow}{HTML}{FFFF00}
\definecolor{GreenYellow}{HTML}{ADFF2F}
\definecolor{ForestGreen}{HTML}{228B22}
\definecolor{Lavender}{HTML}{FFC0CB}
\definecolor{SkyBlue}{HTML}{87CEEB}
\definecolor{NavyBlue}{HTML}{000080}
\definecolor{Goldenrod}{HTML}{DDF700}
\definecolor{VioletRed}{HTML}{D02090}
\definecolor{CornflowerBlue}{HTML}{6495ED}
\definecolor{LimeGreen}{HTML}{32CD32}

\providecommand{\pptablesfooter}{
\end{tabularx}
}
\providecommand{\pptablesheader}{
\begin{tabularx}{1.0\textwidth}{@{}c@{}|*{5}{@{}r@{}X@{}}|@{}r@{}@{}l@{}}
\#FEs/D & \multicolumn{2}{@{}c@{}}{0.5} & \multicolumn{2}{@{}c@{}}{1.2} & \multicolumn{2}{@{}c@{}}{3} & \multicolumn{2}{@{}c@{}}{10} & \multicolumn{2}{@{}c@{}}{50} & \multicolumn{2}{|@{}l@{}}{\#succ}\\\hline
}

\providecommand{\algsfolder}{EXPLO_GLOBA_MCS_h_NEWUO_lmm-C_SMAC-_fminc_et_al/}
\providecommand{\algorithmA}{EXPLO2 mul25par32lamLin}
\providecommand{\algorithmB}{GLOBAL pal noiseless}
\providecommand{\algorithmC}{MCS huyer noiseless}
\providecommand{\algorithmD}{NEWUOA ros noiseless}
\providecommand{\algorithmE}{lmm-CMA-ES auger noiseless}
\providecommand{\algorithmF}{SMAC-BBOB hutter noiseless}
\providecommand{\algorithmG}{fmincon pal noiseless}
\providecommand{\algorithmH}{DTS-CMA-ES Pitra}
\providecommand{\algorithmI}{CMA-ES-2019 Hansen}
\providecommand{\algorithmJ}{lq-CMA-ES Hansen}